\documentclass[singlespaced]{ut-thesis}

\usepackage{amsmath}%
\usepackage{amsfonts}%
\usepackage{amssymb}
\usepackage{amsthm}
\usepackage{subfigure}
\usepackage[margin=30pt,font=small,labelfont=bf,
               labelsep=endash]{caption}
\usepackage{cite}
\usepackage{hyperref}
\usepackage[capitalize]{cleveref}
\usepackage{tikz}
\usetikzlibrary{matrix,arrows,decorations.pathmorphing}

\theoremstyle{plain}
\newtheorem{theorem}{Theorem}[section]
\newtheorem*{theorem*}{Theorem}

\newtheorem{corollary}{Corollary}[section]
\newtheorem*{corollary*}{Corollary}
\newtheorem{lemma}{Lemma}[section]
\newtheorem*{lemma*}{Lemma}
\newtheorem{proposition}{Proposition}[section]
\newtheorem*{proposition*}{Proposition}

\theoremstyle{definition}
\newtheorem{definition}{Definition}[section]
\newtheorem{example}{Example}[section]

\theoremstyle{remark}
\newtheorem{remark}{Remark}[section]
\newtheorem{claim}{Claim}

\numberwithin{equation}{section}

\newcommand{\lb}{[\![}
\newcommand{\rb}{]\!]}
\newcommand{\Cour}[1]{\lb #1\rb}
\newcommand{\la}{\langle}
\newcommand{\ra}{\rangle}
\newcommand{\mf}{\mathfrak}
\newcommand{\mbb}{\mathbb}
\newcommand{\mbf}{\mathbf}
\newcommand{\mc}{\mathcal}
\newcommand{\on}{\operatorname}
\newcommand{\Lied}{\mathcal{L}}
\newcommand{\ad}{\mathbf{ad}}

\newcommand{\g}{\mathfrak{g}}
\newcommand{\h}{\mathfrak{h}}
\newcommand{\gr}{\on{gr}}
\newcommand{\ann}{\on{ann}}
\newcommand{\mmat}[2][3em]{\matrix (#2) [matrix of math nodes, row sep=#1,
  column sep=#1, text height=1.5ex, text depth=0.25ex]}

\tikzset{node distance=2cm, auto}
\degree{Doctor of Philosophy}
\department{Mathematics}
\gradyear{2012}
\author{David Scott Li-Bland}
\title{$\mc{LA}$-Courant algebroids and their applications.}

\crefname{pluralequation}{Eqs.}{Eqs.}
\Crefname{pluralequation}{Eqs.}{Eqs.}

\begin{document}


\begin{preliminary}

\maketitle


\begin{abstract}
In this thesis we develop the notion of $\mc{LA}$-Courant algebroids, the infinitesimal analogue of multiplicative Courant algebroids. Specific applications include the integration of q-Poisson $(\mf{d},\g)$-structures, and the reduction of Courant algebroids. We also introduce the notion of pseudo-Dirac structures, (possibly non-Lagrangian) subbundles   $W\subseteq \mbb{E}$ of a Courant algebroid such that the Courant bracket endows $W$ naturally with the structure of a Lie algebroid. Specific examples of pseudo-Dirac structures arise in the theory of q-Poisson $(\mf{d},\g)$-structures.
\end{abstract}


\begin{dedication}
\begin{center}To Esther, the love of my life, \\and to Wesley, the apple of our eyes.\\Thanks be to God for his blessings.\end{center}
\end{dedication}


\begin{acknowledgements}
I would like to thank my father, John Bland, for inspiring a love of mathematics in me.
I would like to thank my supervisor, Eckhard Meinrenken, for his guidance,  patience,  advice, encouragement and help over the years. I would like to thank Pavol \v{S}evera for many delightful conversations, explanations, perspectives, and the math he taught me. I would like to thank Alfonso Gracia-Saz and Rajan Mehta for teaching me supergeometry and the theory of double and $\mc{LA}$-vector bundles. I would like to thank Alejandro Cabrera for many interesting discussions and teaching me about tangent prolongations.  I would also like to thank Anton Alekseev, Henrique Bursztyn, Arlo Caine, Marco Gualtieri, Travis Li, Jiang Hua Lu, Brent Pym, and Ping Xu for many interesting conversations.

Finally, I would like to thank my beautiful wife, Esther, for all her love, her support, her encouragement, and her strength of character. She is the most wise, inspiring, fascinating, and wonderful person I have ever met.
\end{acknowledgements}

\tableofcontents


\listoffigures

\end{preliminary}


\chapter{Introduction}\label{chp:intro}

\subsection{A brief history}

Courant algebroids and Dirac structures were first introduced by  Courant \cite{Courant:1990uy,Courant:tm} as a geometric framework for  Dirac's theory of Hamiltonian systems with constraints \cite{Dirac:1967ug}.
Courant's original setup was generalized in \cite{ManinTriplesBi}, as a means of constructing doubles for Lie bialgebroids. Courant algebroids have now found many uses, from the theory of moment maps \cite{Bursztyn:2009wi,Bursztyn:2005te,Xu03,Bursztyn03-1,Alekseev:2009tg} to generalized complex geometry \cite{Hitchin:2003kx,Gualtieri:2004wh}.

\subsubsection{The Dirac bracket}\label{sec:DirBrk}
We shall go into some more detail. In Hamiltonian mechanics, the phase space - the space of all possible states of a physical system - is described by a smooth manifold $M$. Smooth functions $f\in C^\infty(M)$ on the phase space describe various quantities one might wish to measure, such as energy, position, or momentum. Additionally, $M$ carries a Poisson structure: a bivector field $\pi\in\mf{X}^2(M)$, such that the bracket $$\{f,g\}:=\pi(df,dg),\quad f,g\in C^\infty(M)$$ endows the vector space $C^\infty(M)$ with the structure of a Lie algebra.  

 Noether's first theorem - that conserved quantities correspond to symmetries - arises as follows: the Poisson structure associates to any function $f\in C^\infty(M)$ (a conserved quantity) the vector field $$X_f:=\big(g\to\{f,g\}\big)\in\mf{X}(M)$$ (describing the corresponding symmetry of phase space).
 
 Introducing constraints on this system corresponds to describing a submanifold $S\subseteq M$. Ideally, one would like the Poisson bracket $\{\cdot,\cdot\}$ on $C^\infty(M)$ to descend to a bracket $\{\cdot,\cdot\}_S$ on $C^\infty(S)$ so that the restriction of functions to $S$ is a morphism of Lie algebras, i.e.
  $$\{i^*f,i^*g\}_S=i^*\{f,g\},\quad f,g\in C^\infty(M),$$ where $i:S\to M$ is the inclusion. Unfortunately, this is impossible in general, since the vanishing ideal $Z(S)\subseteq C^\infty(M)$ of $S\subseteq M$ might not be a Lie algebra ideal. Equivalently, for an arbitrary function $c\in Z(S)$, the vector field $$X_c:=\{c,\cdot\}\in\mf{X}(M)$$ might not vanish when restricted to $S$ (nor even be tangent to $S$).
 
 Let $D\subseteq TM\rvert_S$ be the distribution spanned by the vector fields $X_c\rvert_S$ for $c\in Z(S)$ (see \cref{fig:NonTangDist}).
  A function $f\in C^\infty(S)$ is called \emph{admissible} if $$v\cdot f=0$$ for any  vector $v\in D\cap TS$. Similarly, a function $\tilde f\in C^\infty(M)$ is called an \emph{admissible extension of $f$} if $$i^*\tilde f=f$$ and $$v\cdot \tilde f=0$$ for any  vector $v\in D$. Dirac showed that there exists a Lie bracket $\{\cdot,\cdot\}_{DB}$ (called the \emph{Dirac bracket}) on the subspace $C^\infty_{adm}(S)$ of admissible functions such that 
 \begin{equation}\label{eq:DiracBracket}\{f,g\}_{DB}=\{i^*\tilde f,i^*\tilde g\}_{DB}=i^*\{\tilde f,\tilde g\}\end{equation} for any admissible extensions $\tilde f,\tilde g\in C^\infty(M)$ of $f,g\in C^\infty_{adm}(S)$.
 
 \begin{figure}
\begin{center}
\def\svgwidth{10cm}
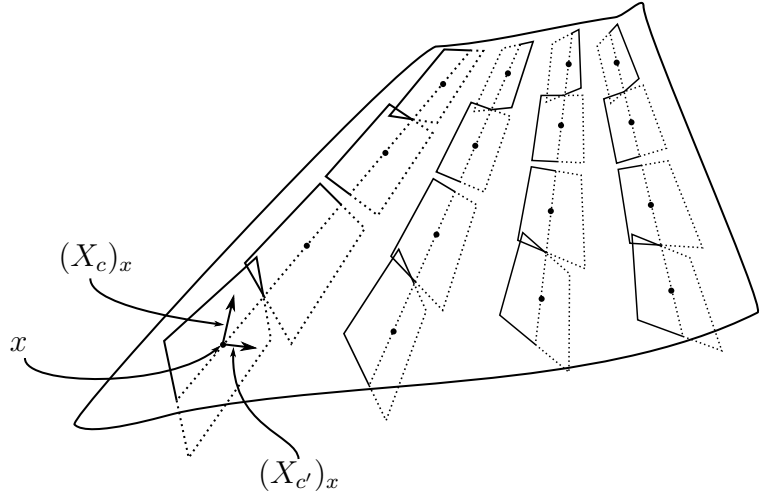
\end{center}
 \caption{\label{fig:NonTangDist}For functions $c,c'\in Z(S)$ vanishing on the submanifold $S\subseteq M$, the Hamiltonian vector fields $X_c\rvert_S$ and $X_{c'}\rvert_S$ may not be tangent to $S$. $D\subseteq TM$ is the distribution spanned by the restriction of all such vector fields to $S$.}
 \end{figure}
 
 \begin{remark}
One may interpret the distribution $D\cap TS$ as describing some gauge transformations of the physical system:
When $D\cap TS$ is of constant rank, then it is involutive, and thus, by Frobenius's theorem, defines a foliation. The admissible functions can then be interpreted as the algebra of functions on the leaf space. 

 \end{remark}
 
 \subsubsection{Courant algebroids}
 In 1986 Courant \cite{Courant:1990uy,Courant:tm} described a geometric framework for Dirac's theory of constrained Hamiltonian mechanics, which we shall briefly summarize. The \emph{Pontryagin bundle} $\mbb{T}M:=TM\oplus T^*M$ carries a  split signature metric defined by the natural paring
 $$\la(v,\mu),(w,\nu)\ra:=\mu(w)+\nu(v),\quad (v,\mu),(w,\nu)\in \mbb{T}M.$$
   Courant's insight was to consider Lagrangian subbundles of the Pontryagin bundle (i.e. $L\in\mbb{T}M$ such that $L^\perp=L$), replacing the Poisson bivector field $\pi\in \mf{X}^2(M)$ with its graph 
 $$\gr(\pi^\sharp):=\{(X,\alpha)\in TM\oplus T^*M\mid X=\pi(\alpha,\cdot)\}.$$
  He described a bracket, which is usually referred to as the \emph{Courant bracket} in the literature,\footnote{In fact, this bracket was introduced by Irene Dorfman in the context of two dimensional variational problems \cite{Dorfman:1993us}. Courant worked with its skew symmetrization instead.}  
 \begin{equation}\label{eq:DorfBrk}\Cour{(X,\alpha),(Y,\beta)}=\big([X,Y],\Lied_X\beta-\iota_Yd\alpha\big)\end{equation}
  on the space of sections $\Gamma(\mbb{T}M)$, and showed that this bracket restricts to a Lie bracket on $\Gamma\big(\gr(\pi^\sharp)\big)$. A \emph{Dirac structure} is defined to be an arbitrary Lagrangian subbundle $L\subseteq \mbb{T}M$ which is involutive with respect the Courant bracket. A function $f\in C^\infty(M)$ is called \emph{admissible} (with respect to $L$) if $$df\rvert_{L\cap TM}=0.$$ 
  Courant showed that the space  $C^\infty_{adm}(M)$ of admissible functions carries a well defined Lie bracket,
  \begin{equation}\label{eq:CourBrkFunct}\{f,g\}_L:=X_f\cdot g\end{equation}
   where $X \in\mf{X}(M)$ is any vector field such that $$(X_f,df)\in\Gamma(L).$$
 
 One of the key aspects of Courant's framework is the ability to impose constraints on a Dirac structure. More precisely, he described a procedure to \emph{restrict}\footnote{Courant calls this procedure \emph{reduction} in \cite{Courant:1990uy}. We refer to it as \emph{restriction} (as is also done in \cite{Jotz:2008wn}, for instance) or \emph{pull-back} (as is done in \cite{LiBland:2009ul}, for instance) to distinguish it from more general reduction procedures.} a Dirac structure $L\subseteq \mbb{T}M$ to a submanifold $S\subseteq M$, 
 \begin{equation}\label{eq:restProc}\big(L\subseteq \mbb{T}M\big)\to \big(L_S\subseteq \mbb{T}S),\end{equation} where $L_S\subseteq \mbb{T}S$ is computed by the formula 
\begin{equation}\label{eq:restForm}L_S:=\frac{L\cap TS\oplus T^*M\rvert_S}{L\cap \ann(TS)}\subseteq \mbb{T}S.\end{equation}
 Courant showed that  \cref{eq:restForm} describes a Dirac structure (under some cleanness assumptions).
In particular, when $L=\gr(\pi^\sharp)$ arises from a Poisson structure, then Courant's bracket between admissible functions, defined by \cref{eq:CourBrkFunct} for the Dirac structure $$\gr(\pi^\sharp)_S:=\frac{\gr(\pi^\sharp)\cap TS\oplus T^*M\rvert_S}{\gr(\pi^\sharp)\cap \ann(TS)}\subseteq \mbb{T}S,$$ coincides with Dirac's bracket \labelcref{eq:DiracBracket}.
 
 Courant's setup was generalized by Liu, Weinstein and Xu \cite{ManinTriplesBi}, who defined abstract Courant algebroids as a vector bundle $\mbb{E}\to M$ (replacing the Pontryagin bundle $\mbb{T}M$) carrying a fibrewise metric, $$\la\cdot,\cdot\ra:\mbb{E}\otimes\mbb{E}\to \mbb{R}\times M,$$ together with a bilinear bracket called the \emph{Courant bracket}, $$\Cour{\cdot,\cdot}:\Gamma(\mbb{E})\times\Gamma(\mbb{E})\to \Gamma(\mbb{E}),$$ satisfying certain axioms. 
 Dirac structures are defined to be Lagrangian subbundles $E\subseteq \mbb{E}$ which are involutive with respect to the Courant bracket. 
 
 As in Courant's original setup, there is a restriction procedure: Given a submanifold $S\subseteq M$, one may \emph{restrict}  both Courant algebroids $\mbb{E}\to M$,
 $$
 \begin{tikzpicture} 
  \matrix (m1)[matrix of math nodes,left delimiter=(,right delimiter=),row sep=1em,
  column sep=1em, text height=1.5ex, text depth=0.25ex] at (0,0)
  {\mbb{E}\\ M\\};
   \path[->] (m1-1-1) edge (m1-2-1);
   
  \matrix (m2)[matrix of math nodes,left delimiter=(,right delimiter=),row sep=1em,
  column sep=1em, text height=1.5ex, text depth=0.25ex] at (4,0)
  {\mbb{E}_S\\ S\\};
   \path[->] (m2-1-1) edge (m2-2-1);
   \path[->] (1.5,0) edge (2.5,0);
  \end{tikzpicture}$$
   and their Dirac structures $E\subseteq \mbb{E}$,
  $$\big(E\subseteq \mbb{E}\big)\to \big(E_S\subseteq \mbb{E}_S\big),$$ 
  generalizing Courant's original restriction procedure \labelcref{eq:restProc} (see \cite[Section~2.4.2]{LiBland:2009ul} for details). 
  
   Even more generally, restriction to a submanifold can be seen as a special case of composing a Dirac structure $E\subseteq\mbb{E}$ with certain linear relations $R:\mbb{E}\dasharrow\mbb{F}$ between two Courant algebroids, called \emph{Courant relations}. (Important examples of composing Dirac structures with Courant relations are called \emph{forward} and \emph{backward Dirac maps} \cite{Bursztyn03-1,Alekseev:2009tg,Bursztyn:2009wi,Bursztyn:2003ud,Bursztyn:2005wi}.  We shall review all this in more detail in \cref{chp:prelim}.)

\subsection{Multiplicative Courant algebroids, and their infinitesimal versions $\mc{LA}$-Courant algebroids}
In the early 80's Drinfel'd developed the theory of Lie groups  carrying compatible Poisson structures \cite{Drinfeld83}, and showed that these \emph{Poisson Lie groups} integrate Lie bialgebras. 
Shortly thereafter, it was noticed independently by Karasev and Weinstein \cite{Karasev:1986vg,Weinstein:1987ua} that Lie groupoids can carry non-degenerate Poisson structures (symplectic structures), and that these \emph{symplectic groupoids} integrate Poisson manifolds. 
Since then, many further examples of multiplicative Hamiltonian structures on Lie groupoids have appeared, such as Poisson groupoids, twisted symplectic groupoids, quasi-symplectic groupoids etc. \cite{weinstein87,SCattaneo:2004fe,LiBland:2010wi,Bursztyn03-1,Ponte:2005txa,Xu03}. These have found applications to quantization \cite{Weinstein:1987ua,Hawkins:2006vl,Cattaneo:2000du,Cattaneo:2008vf}, reduction and symmetries of Hamiltonian systems \cite{Semenov-Tian-Shansky85,thesis-3,Mikami:1988tv,Xu03}, and Morita equivalence \cite{Xu:1991vb,Xu:1992tb,Bursztyn07,Xu03}, among others.

Since Dirac structures provide a unified framework for these various Hamiltonian structures, it was natural to investigate multiplicative Dirac structures and Courant algebroids on Lie groupoids, as was first done in \cite{Bursztyn03-1,Ponte:2005txa,Xu03} (though general definitions did not appear until \cite{Ortiz:2009ux,LiBland:2010wi,Mehta:2009js}). In this thesis, we develop $\mc{LA}$-Courant algebroids, the infinitesimal version of multiplicative Courant algebroids. (Note that infinitesimal versions of multiplicative Courant algebroids and Dirac structures were already studied in \cite{Ortiz:2009ux,LiBland:2010wi,Mehta:2009to}, but always using the language of supergeometry). 

In more detail, an $\mc{LA}$-Courant algebroid consists of a Courant algebroid $\mbb{A}\to A$ such that $\mbb{A}$  also has the structure of a Lie algebroid $\mbb{A}\to V$ over a different base space $V$. The Lie algebroid and Courant algebroid structures must be compatible in an appropriate sense. In particular $\mbb{A}$ is the total space of a double vector bundle
$$\begin{tikzpicture}
\mmat{m}{\mbb{A}&V\\ A&M\\};
\draw[->] (m-1-1) edge (m-1-2)
			edge (m-2-1);
\draw[<-] (m-2-2) edge (m-1-2)
			edge (m-2-1);
\end{tikzpicture}$$
(a concept due to Pradines \cite{Pradines:1974tc}).
Dirac structures $L\subseteq \mbb{A}$ which are also subbundles of the vector bundle $\mbb{A}\to V$ are called $\mc{VB}$-Dirac structures. They are called $\mc{LA}$-Dirac structures if they are also Lie subalgebroids of $\mbb{A}\to V$.

As a first example, if $\mbb{E}$ is any Courant algebroid, then its tangent bundle
$$\begin{tikzpicture}
\mmat{m}{T\mbb{E}&\mbb{E}\\ TM&M\\};
\draw[->] (m-1-1) edge (m-1-2)
			edge (m-2-1);
\draw[<-] (m-2-2) edge (m-1-2)
			edge (m-2-1);
\end{tikzpicture}$$
is canonically an $\mc{LA}$-Courant algebroid, called the \emph{tangent prolongation} of $\mbb{E}$. This example was first studied by Courant \cite{Courant:1999ho} for the special case where $\mbb{E}=\mbb{T}M$. Later, it was studied in full generality  by  Boumaiza and  Zaalani \cite{Boumaiza:2009eg}.

$\mc{LA}$-Courant algebroids have a variety of applications, and we describe two main applications in this thesis: reduction, and integration. We shall summarize these applications later in the introduction, after first describing pseudo-Dirac structures.

\subsection{pseudo-Dirac structures}
In general, the Courant bracket \labelcref{eq:DorfBrk} does not define a Lie bracket  on the sections of $\mbb{T}M$, since it fails to be skew symmetric. 
In order to obtain a Lie bracket from the Courant bracket, Courant restricted the bracket to sections of Lagrangian subbundles $L\subseteq \mbb{T}M$, which ensures skew symmetry of the bracket. In this thesis, we will take a different approach to obtain a Lie bracket from the Courant bracket: we modify the Courant bracket itself.


In \cref{chp:TngProAndLieSub} we introduce the notion of a \emph{pseudo-Dirac structure} in the Courant algebroid $\mbb{T}M$: a subbundle $$W\subseteq \mbb{T}M$$ together with a map $$\nabla:\Omega^0(M,W)\to\Omega^1(M,W^*)$$ satisfying certain axioms, including
\begin{subequations}\label[pluralequation]{eq:PsConIntAll}
\begin{align}
\label{eq:PsConInt}\nabla f\sigma&=f\nabla\sigma+df\otimes \la\sigma,\cdot\ra,&\sigma\in\Gamma(W),\\
\label{eq:derPairInt}d\la\sigma,\tau\ra&=\la\nabla\sigma,\tau\ra+\la\sigma,\nabla\tau\ra,&\sigma,\tau\in\Gamma(W).
\end{align}
\end{subequations}
Since \cref{eq:PsConIntAll} reduce to the definition of a metric connection when $\la\cdot,\cdot\ra\rvert_W$ is non-degenerate, we call $\nabla$ a \emph{pseudo-connection}. 
\Cref{eq:derPairInt} guarantees that the modification to the Courant bracket, given by 
\begin{equation}\label{eq:ModBrkInt}[\sigma,\tau]:=\Cour{\sigma,\tau}-\mbf{a}^*\la\nabla\sigma,\tau\ra,\end{equation}
is skew symmetric. 
More importantly, in \cref{thm:LieSubIsVBDir}, we prove that \cref{eq:ModBrkInt} endows $W$ with the structure of a Lie algebroid (a concept due to Pradines \cite{Pradines:1967wn}). In particular, the sections of $W$ form a Lie algebra.

\Cref{lem:qIsNab} states that  subbundles of $\mbb{T}M$ endowed with a pseudo-connection are in one-to-one correspondence with Lagrangian double vector subbundles of the $\mc{LA}$-Courant algebroid $$\begin{tikzpicture}
\mmat{m}{T\mbb{T}M&\mbb{T}M\\ TM&M\\};
\draw[->] (m-1-1) edge (m-1-2)
			edge (m-2-1);
\draw[<-] (m-2-2) edge (m-1-2)
			edge (m-2-1);
\end{tikzpicture}$$ Thus it is quite natural to study them. Similarly, pseudo-Dirac structures in $\mbb{T}M$ are in one-to-one correspondence with $\mc{VB}$-Dirac structures in $T\mbb{T}M$.

As a consequence, it is easy to impose constraints on a pseudo-Dirac structure in $\mbb{T}M$:
Suppose that $S\subseteq M$ is the constraint submanifold. Courant's restriction procedure \labelcref{eq:restProc} allows us to restrict $\mc{VB}$-Dirac structures of $T\mbb{T}M$ to $\mc{VB}$-Dirac structure of $T\mbb{T}S$:
\begin{center}
\begin{tikzpicture}
\node (B) at (-4,-1) [text width=4cm,align=center] {$\mc{VB}$-Dirac structures in $T\mbb{T}M$};
\node (D) at (4,-1) [text width=4cm,align=center] {$\mc{VB}$-Dirac structures in $T\mbb{T}S$};
\draw [->] (B) edge node  [text width=4cm,align=center,below] {Courant's restriction procedure \labelcref{eq:restProc}} (D);
\end{tikzpicture}
\end{center} 
Composing this restriction procedure with the equivalence between $\mc{VB}$-Dirac structures in $T\mbb{T}M$ and pseudo-Dirac structures in $\mbb{T}M$ allows us to impose constraints on pseudo-Dirac structures in $\mbb{T}M$:
\begin{center}
\begin{tikzpicture}
\node (A) at (-4,1) [text width=4cm,align=center] {pseudo-Dirac structures in $\mbb{T}M$};
\node (B) at (-4,-1) [text width=4cm,align=center] {$\mc{VB}$-Dirac structures in $T\mbb{T}M$};
\node (C) at (4,1) [text width=4cm,align=center] {pseudo-Dirac structures in $\mbb{T}S$};
\node (D) at (4,-1) [text width=4cm,align=center] {$\mc{VB}$-Dirac structures in $T\mbb{T}S$};
\draw [->] (B) edge node  [text width=4cm,align=center,below] {Courant's restriction procedure \labelcref{eq:restProc}} (D);
\draw [->,dashed] (A) edge node [text width=4cm,align=center] {restriction procedure for pseudo-Dirac structures} (C);
\draw [<->] (A) edge (B);
\draw [<->] (C) edge (D);
\end{tikzpicture}
\end{center}



Finally, there are many examples of pseudo-Dirac structures, to list a few:
\begin{itemize}
\item Dirac structures,
\item Lie algebroid structures on $T^*M$,
\item Lie subalgebras of a quadratic Lie algebra,
\item action Courant algebroids, as introduced by Meinrenken and the author \cite{LiBland:2009ul},
\item quasi-Poisson structures on $M$, in the sense of Alekseev and Kosmann-Schwarzbach \cite{Alekseev99},
\item etc.
\end{itemize}

In particular, the concept of pseudo-Dirac structures allows us to unify the treatment of Poisson and q-Poisson structures on a manifold $M$: Both correspond to a pseudo-Dirac structure whose underlying bundle is $W=\gr(\mbf{a})\subseteq \mbb{T}M$ where $\mbf{a}:T^*M\to TM$ is the anchor map for a Lie algebroid structure on $T^*M$.

\subsection{Applications}

\subsubsection{Reduction of Courant algebroids}
Roughly speaking, reduction is the process of reducing the dimensions of a physical system (a Poisson structure, a Courant algebroid, etc.) through a combination of imposing constraints and quotienting out symmetries. %
In \cref{sec:reduc} we will present  a novel approach to reducing Courant algebroids in terms of $\mc{LA}$-Dirac and $\mc{VB}$-Dirac structures. Our framework is quite general, encompassing many known reduction procedures. However, to provide the reader with some context, we shall first summarize Marsden and Ratiu's reduction procedure for Poisson structures \cite{Marsden:1986vs}, before summarizing our approach to the reduction of Courant algebroids.
\begin{remark}
There are many interesting reduction procedures for both Poisson structures (see \cite{Falceto:2008gg,Marsden:1986vs,Liu:2000vf,Cattaneo:2010th,Mackenzie:2000ut,Mikami:1988tv,weinstein87,Cattaneo:2010vr}, and the references therein) and Courant algebroids (see \cite{Hu:2007wq,Vaisman:2007gg,Zambon:2008wj,Bursztyn:2007ko,Bursztyn:gcFguuB1,Yoshimura:2007gw,Stienon:2008cl,Jotz:2008wn,Hu:2009wl,Mehta:2010ux,Jotz:2011cz,Goldberg:2010wj,Calvo:2010bj}, and the references therein). Regrettably, we will not have space to discuss all these treatments here.
\end{remark}


Let $M$ be a Poisson manifold. Constraints are described by a submanifold $S\subseteq M$, while symmetries are described by a collection of vector fields on $S$ which are closed under the Lie bracket and which (we assume to) span a subbundle $$F\subseteq TS\subseteq TM\rvert_S.$$ Equivalently, $F\subseteq TM$ is a \emph{sub Lie algebroid}. By Frobenius's theorem, $F$ arises from a foliation of $S$ by maximal integral submanifolds, as pictured in \cref{fig:FolInt}.

 \begin{figure}
\begin{center}
\subfigure[A foliation of the submanifold $S\subseteq M$. $F\subseteq TS$ is the  subbundle tangent to the foliation.]
{\label{fig:FolInta}
\def\svgwidth{5cm}
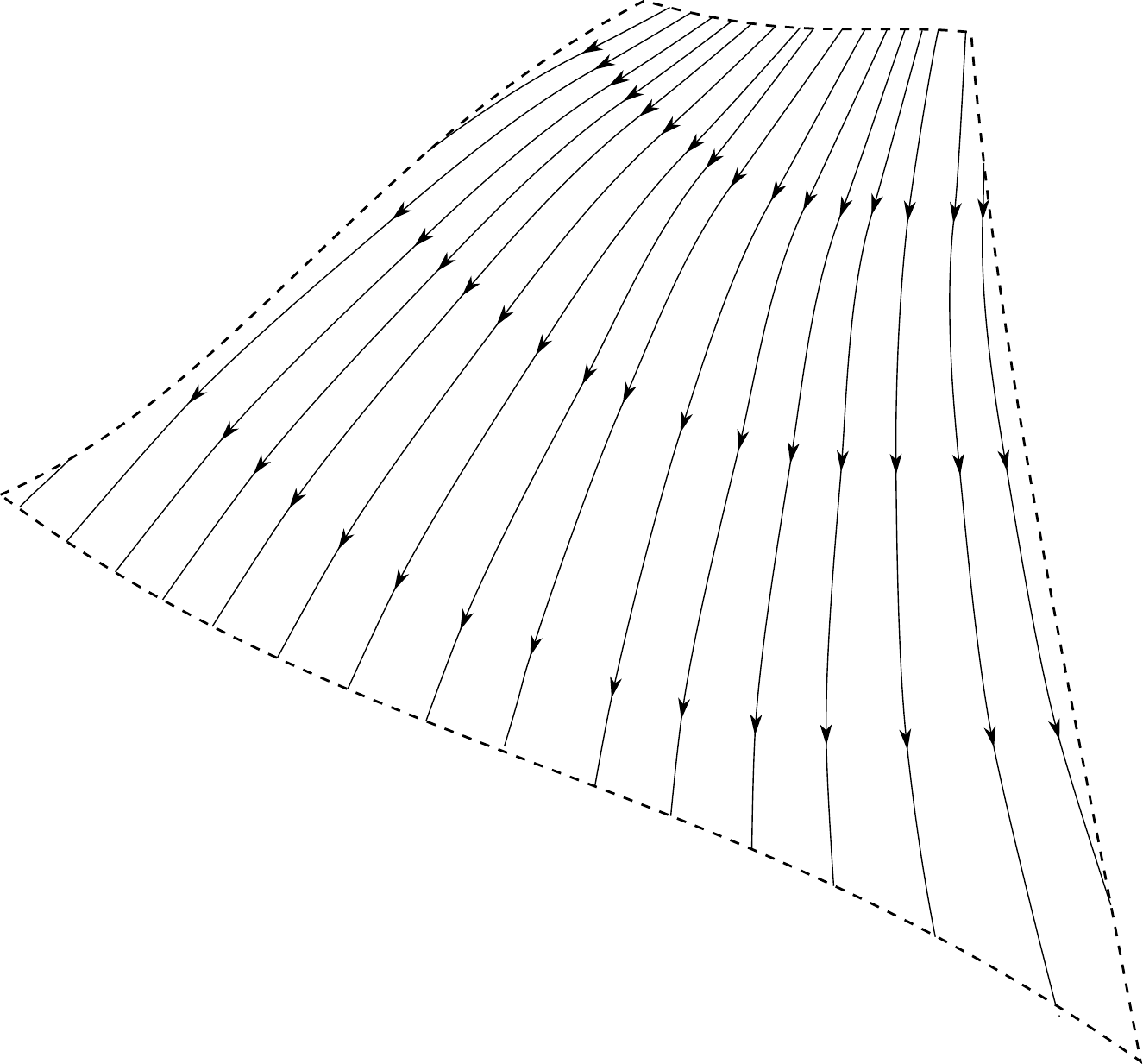}
\hskip 1cm
\subfigure[$N$ is the leaf space of the foliation and $\phi:S\to N$ the natural projection.]
{\label{fig:phiInt}
\def\svgwidth{5cm}
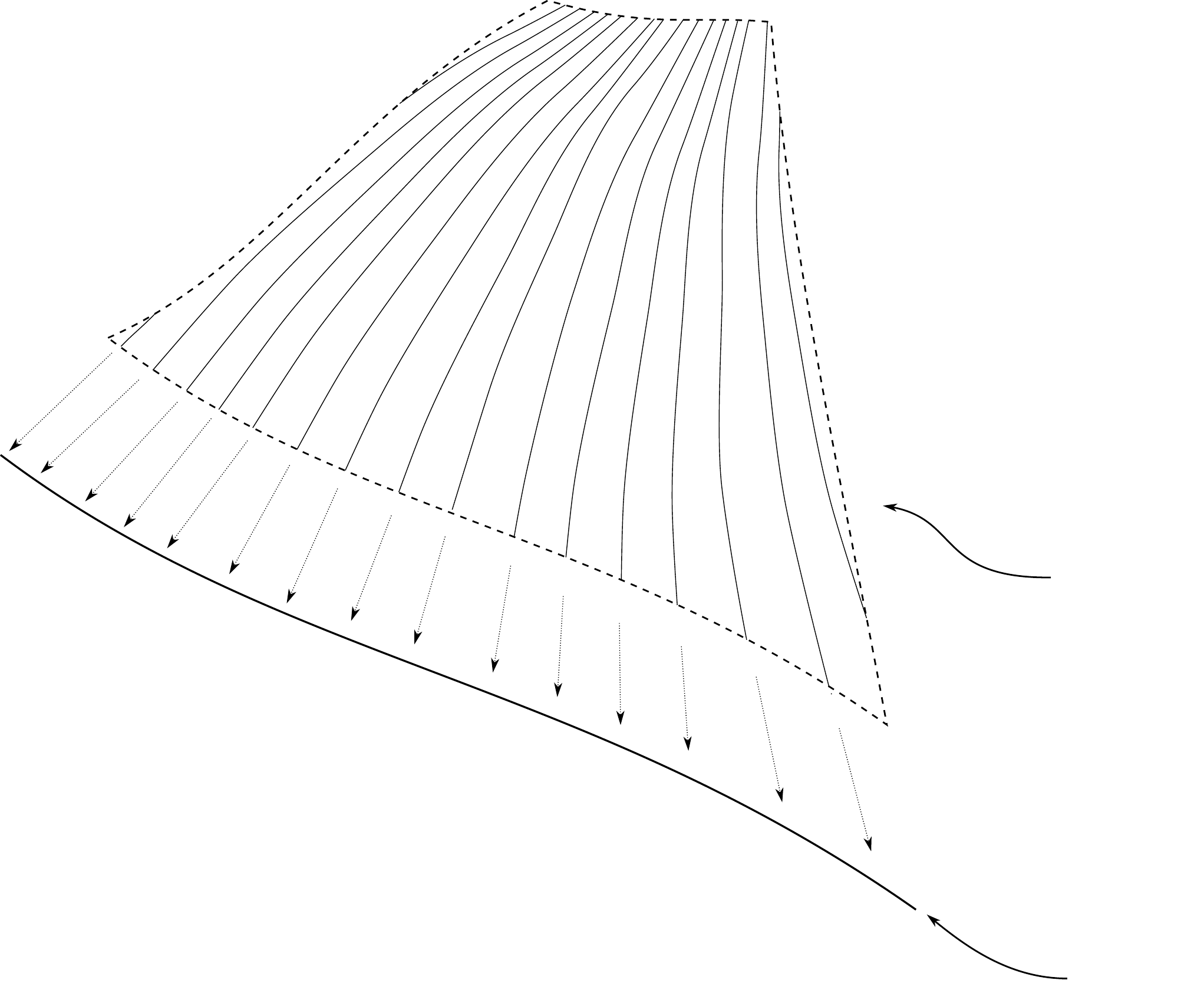}
\end{center}
 \caption{\label{fig:FolInt}The submanifold $S\subseteq M$ describes constraints on the physical system, while vector fields tangent to the foliation describe symmetries.}
 \end{figure}
 
Suppose that $N$ is the leaf space of the foliation, that is there is a surjective submersion $\phi:S\to N$, such that $F=\on{ker}(d\phi)$. A function $\tilde f\in C^\infty(M)$ is said to \emph{extend} a function $f\in C^\infty(N)$ if $$\phi^*f=i^*\tilde f,$$ where $i:S\to N$ is the inclusion.
If there exists a Poisson structure on $N$ such that
\begin{equation}\label{eq:RedInt} \phi^*\{f,g\}=i^*\{\tilde f,\tilde g\}\end{equation}  for any functions $f,g\in C^\infty(N)$ and extensions $\tilde f,\tilde g\in C^\infty(M)$, then the Poisson structure on $M$ is said to \emph{reduce} to a Poisson structure on $N$. 

As explained by Courant and Sanchez de Alvarez \cite{Courant:1999ho,SanchezdeAlvarez:1989vf}, the Poisson structure on $M$ induces a Poisson structure on $TM$ for which the equation $$\{df,dg\}=d\{f,g\}$$ is satisfied for any functions $f,g\in C^\infty(M)$ (where we interpret 1-forms on $M$ as linear functions on $TM$). This  \emph{tangent prolongation of the Poisson structure} is very relevant to reduction. Indeed, \cite[Remark 2.2]{Falceto:2008gg} implies the following theorem:

\begin{theorem*}[Coisotropic Reduction\footnote{This is a special case of results found in \cite{weinstein87,ManinTriplesBi,Falceto:2008gg}}]\label{prop:IntBscRed}
Let $M$ be a Poisson manifold, $S\subseteq M$ a submanifold, and $F\subseteq TS$ a regular foliation with leaf space $N$.
 The Poisson structure on $M$ reduces to a Poisson structure on $N$ if and only if $F\subseteq TM$ is coisotropic.\footnote{As defined by Weinstein \cite{weinstein87}, a submanifold $C\subseteq M$ of a Poisson manifold is called \emph{coisotropic} if the ideal $Z(C)\subseteq C^\infty(M)$ of functions vanishing on $C$ is closed under the Poisson bracket.}
 \end{theorem*}
 
 We will refer this this reduction procedure as \emph{coisotropic reduction}, since \cref{eq:RedInt} holds if and only if the composition $R:M\dashleftarrow S\dashrightarrow N$ of relations is coisotropic \cite[Proposition~(2.3.3)]{weinstein87}. As explained by Weinstein \cite{weinstein87}, coisotropic reduction fits snugly within the framework of coisotropic calculus.
Unfortunately, it is not broad enough to encompass certain important examples of reduction. For instance, if the Poisson structure on $M$ is non-degenerate and $S\subseteq M$ is a symplectic submanifold, then $S$ inherits a Poisson structure via the Dirac bracket. The definition of the Dirac bracket \labelcref{eq:DiracBracket} is structurally the same as that of  the reduced Poisson bracket \labelcref{eq:RedInt} described above (with $\phi:S\to S$ the identity map), except that it must be calculated via \emph{admissible extensions} of the functions $f,g\in C^\infty(S)$, rather than arbitrary extensions.

To incorporate examples such as this, Marsden and Ratiu\footnote{Marsden and Ratiu's reduction procedure \cite{Marsden:1986vs} predates coisotropic reduction \cite{weinstein87}.} considered a subbundle $D\subseteq TM\rvert_S$, as in \cref{fig:MRDistInt}, whose role is to prescribe how one extends functions on $S$ to all of $M$. They assume $F=D\cap TS$ is a constant rank involutive subbundle arising from a regular foliation. Let $N$ denote the corresponding leaf space, and $\phi:S\to N$ the quotient map. A function $\tilde f\in C^\infty(M)$ is called an \emph{admissible extension} of $f\in C^\infty(N)$ if 
$$\phi^* f= i^*\tilde f.$$
 \begin{figure}
\begin{center}
\subfigure[$D\subseteq TM\rvert_S$]
{
\def\svgwidth{4cm}
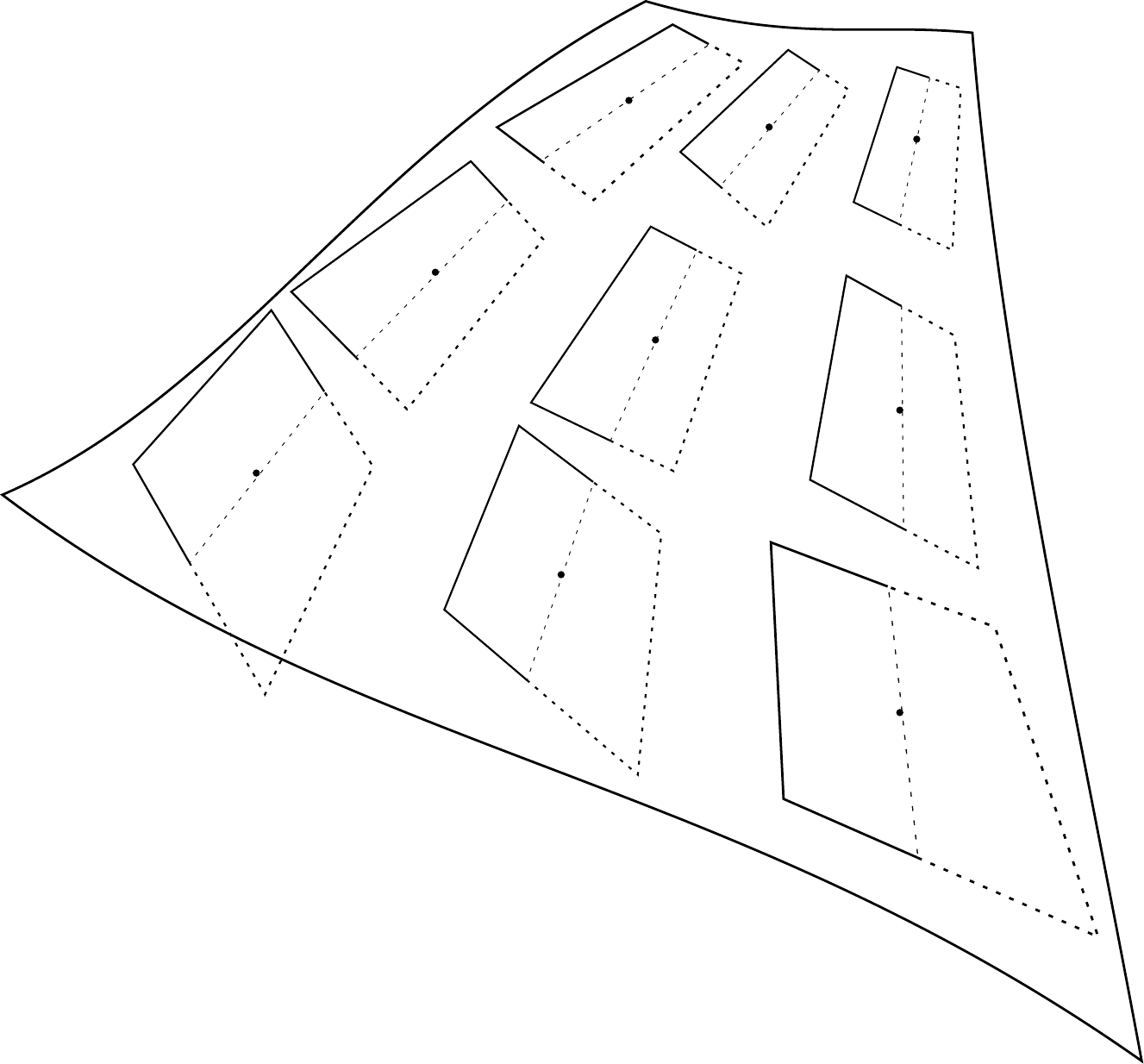}
\hskip 1cm
\subfigure[$F=D\cap TS$ is required to be involutive.]
{\label{fig:MRFolInt}
\def\svgwidth{4cm}
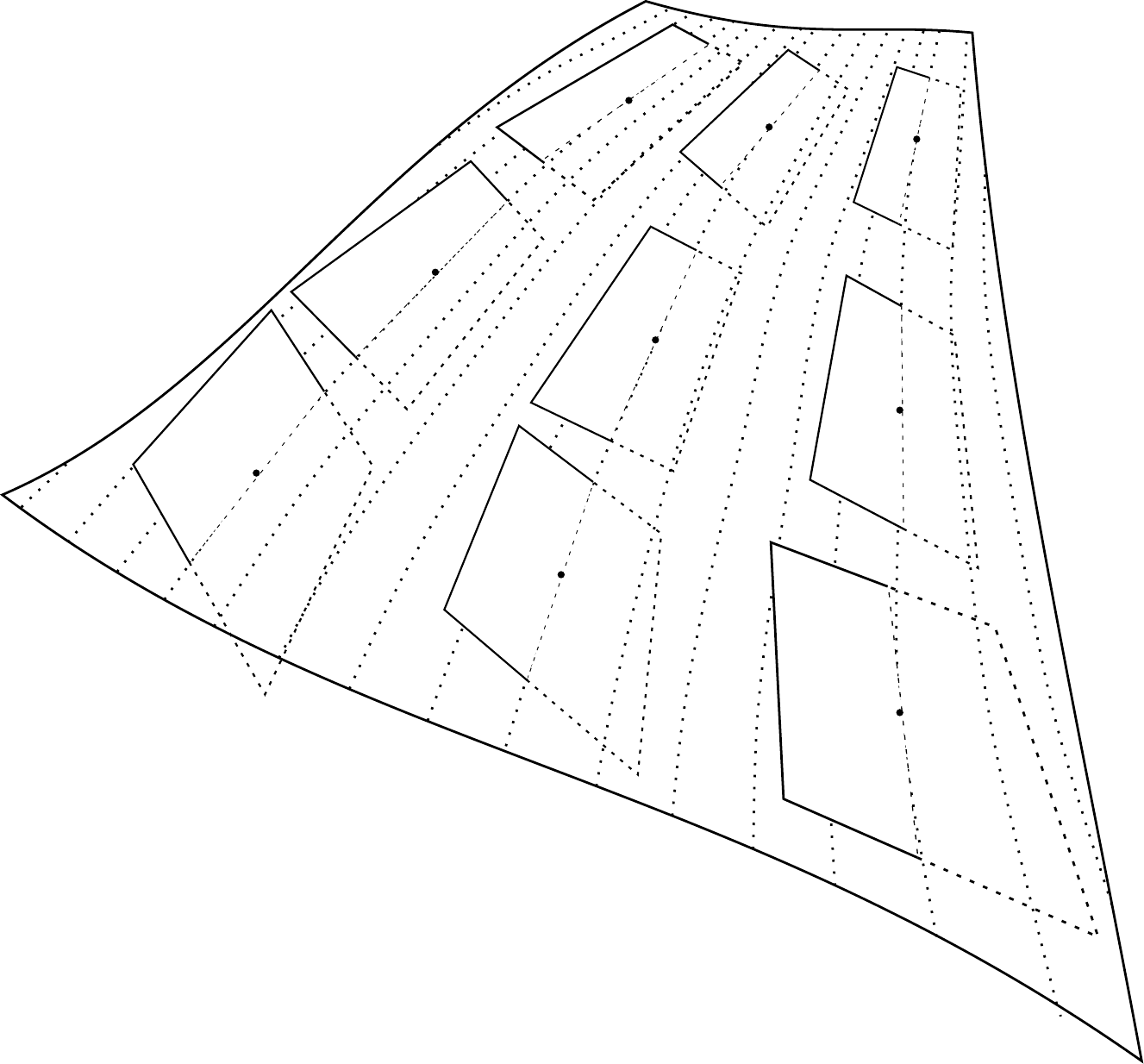}
\hskip 1cm
\subfigure[$N$ is the leaf space of the foliation corresponding to $F=D\cap TS$, and $\phi:S\to N$ the natural projection.]
{\def\svgwidth{4cm}
\input{Phi.pdf_tex}}
\end{center}
\caption{\label{fig:MRDistInt}The subbundle $D\subseteq TM\rvert_S$ of the ambient tangent space controls how functions on $S$ are extended to $M$}
 \end{figure}
If there is a Poisson structure on $N$ so that 
\begin{equation}\label{eq:MRPoisRedInt} \phi^*\{f,g\}=i^*\{\tilde f,\tilde g\}\end{equation} for any functions $f,g\in C^\infty(N)$ and any admissible extensions $\tilde f,\tilde g\in C^\infty(M)$, then the triple $(M,N,D)$ is said to be \emph{reducible} \cite{Marsden:1986vs}. Assuming $D\neq 0$, \cite[Remark 2.2]{Falceto:2008gg} allows us to rephrase Marsden and Ratiu's reduction theorem \cite{Marsden:1986vs} as follows:

\begin{theorem*}[Marsden-Ratiu Reduction]\label{prop:IntMRRed}
Let $M$ be a Poisson manifold, $S\subseteq M$ a submanifold, and $D\subseteq TM\rvert_S$ a non trivial subbundle such that $F=D\cap TS$ is a regular foliation with leaf space $N$.
 The triple $(M,N,D)$ is reducible if and only if $D\subseteq TM$ is coisotropic.
 \end{theorem*}

Summarizing, we have the following table:
\begin{center}
\begin{tabular}{p{.2\textwidth}|p{.5\textwidth}|p{.3\textwidth}}
Reduction Procedure & Reduction Data & Reduction Requirements \\\hline\hline
Coisotropic Reduction & A  Lie subalgebroid $F\subseteq TM$ & $F\subseteq TM$ is coisotropic \\\hline
Marsden-Ratiu Reduction & A subbundle $D\subseteq TM\rvert_S$ such that $D\cap TS\subseteq TS$ is a  Lie subalgebroid & $D\subseteq TM$ is coisotropic
\end{tabular}
\end{center}

Having described the reduction procedures for Poisson structures, we now summarize the main features of our reduction procedure for Courant algebroids. More details will be given in \cref{sec:reduc}.
Our reduction procedure for a Courant algebroid $\mbb{E}$ is formulated in terms of double vector subbundles of the tangent prolongation of $\mbb{E}$:
$$\begin{tikzpicture}
\mmat{m1} at (-2,0) {L&W\\ D&S\\};
\draw[->] (m1-1-1) edge (m1-1-2)
			edge (m1-2-1);
\draw[<-] (m1-2-2) edge (m1-1-2)
			edge (m1-2-1);

\draw (0,0) node {$\subseteq$};

\mmat{m2} at (2,0) {T\mbb{E}&\mbb{E}\\ TM&M\\};
\draw[->] (m2-1-1) edge (m2-1-2)
			edge (m2-2-1);
\draw[<-] (m2-2-2) edge (m2-1-2)
			edge (m2-2-1);
\end{tikzpicture}$$
We give sufficient conditions on $L$ for the quotient vector bundle
 $$\begin{tikzpicture}
 \node (F) at (-2,1) {$\mbb{F}$};
 \node (N) at (-2,-1) {$N$};
\node (VB) at (0,1) {$W/(L\cap TW)$};
\node (Bs) at (0,-1) {$S/(D\cap TS)$};
\draw[->] (VB) edge (Bs);
\draw[->] (F) edge (N);
\node at (-1.5,-1) {$:=$};
\node at (-1.5,1) {$:=$};
\end{tikzpicture}$$
to inherit the structure of a Courant algebroid, which we summarize in the following table:
\begin{center}
\begin{tabular}{p{.2\textwidth}|p{.5\textwidth}|p{.3\textwidth}}
Reduction Procedure & Reduction Data & Reduction Requirements \\\hline\hline
Coisotropic-type Reduction & A double vector subbundle $L\subseteq T\mbb{E}$ which is also a  Lie subalgebroid & $L\subseteq T\mbb{E}$ is a Dirac structure (with support) \\\hline
Marsden-Ratiu-type Reduction & A double vector subbundle $L\subseteq T\mbb{E}$ such that $L\cap TW\subseteq TW$ is a  Lie subalgebroid & $L\subseteq T\mbb{E}$ is a Dirac structure (with support)
\end{tabular}
\end{center}
%

\begin{remark}
There is a clear similarity between the two tables above which summarize the reduction procedures for Poisson manifolds and Courant algebroids, respectively.
In fact, one could obtain the second table from the first by using \v{S}evera's dictionary to translate between Poisson manifolds and Courant algebroids \cite{Severa:2005vla}:
\begin{center}
\begin{tikzpicture}
\node (a) at (-4,1) {Poisson manifold};
\node (b) at (4,1) {Courant algebroid};
\node (c) at (-4,0) {Coisotropic submanifold};
\node (d) at (4,0) {Dirac structure (with support)};
\draw [<->,decorate,
     decoration={snake,amplitude=.4mm,segment length=4mm,post length=4mm,pre length=4mm}] (a) -- (b);
\draw [<->,decorate,
     decoration={snake,amplitude=.4mm,segment length=4mm,post length=4mm,pre length=4mm}] (c) -- (d);
\end{tikzpicture}
\end{center}
\end{remark}

\subsubsection{Integration of q-Poisson structures}

Roughly speaking, \emph{integration} is the process of concatenating an infinite string of infinitesimal symmetries to yield a definite symmetry: for example, passing from a Lie algebra to the corresponding simply connected Lie group. 

One of the earliest examples of a Lie algebra is the ring of functions $C^\infty(M)$ on a Poisson manifold $M$. Unfortunately, $C^\infty(M)$ is generally infinite dimensional, complicating integration. Nevertheless, Karasev and Weinstein \cite{Karasev:1986vg,Weinstein:1987ua} independently discovered  that there are finite dimensional objects, called symplectic groupoids, which integrate Poisson structures in an appropriate sense: 

Let $M$ be a Poisson manifold and $\pi\in\mf{X}^2(M)$ the corresponding bivector field. The space of 1-forms, $\Omega^1(M)$, carries the Lie bracket $$[\alpha,\beta]=d\pi(\alpha,\beta)+\iota_{\pi^\sharp\alpha}d\beta-\iota_{\pi^\sharp\beta}d\alpha,\quad\alpha,\beta\in\Omega^1(M)$$ called the Koszul bracket \cite{Koszul:1985wb}. With this bracket on its sections, the cotangent bundle $T^*M$ becomes a Lie algebroid. Lie algebroids are a concept due to Pradines,  who proved that they integrate to local Lie groupoids \cite{Pradines:1967wn,Pradines:1966ux,Pradines:1968wi}. Suppose that $K$ is a (local) Lie groupoid integrating $T^*M$. The canonical symplectic structure on $T^*M$ is compatible with the Lie algebroid structure. Using this fact, Mackenzie and Xu showed it induces a symplectic structure on $K$, which is compatible with the multiplication \cite{Mackenzie97,Weinstein:1987ua}. 


q-Poisson structures were introduced by Alekseev and Kosmann-Schwarzbach \cite{Alekseev99}, as a generalization of Poisson structures which could describe new examples of Lie group valued moment maps \cite{Alekseev97,Alekseev00}. Let $\mf{d}$ be a quadratic Lie algebra and $\g\subseteq \mf{d}$ a Lie subalgebra which is equal to its orthogonal complement, $\g=\g^\perp$. A q-Poisson $(\mf{d},\g)$ structure on a manifold $M$ consists of 
\begin{itemize}
\item an action of $\g$ on $M$, determined by a map $\rho:\g\times M\to TM$, and
\item a bivector field $\pi\in\mf{X}^2(M)$
\end{itemize} satisfying certain axioms. 

Consider the bracket
\begin{equation}\label{eq:qPbracketInt} \{f,g\}=\pi(df,dg),\quad f,g\in C^\infty(M),\end{equation}
which does not generally satisfy the Jacobi identity. One says that a function $f\in C^\infty(M)$ is \emph{admissible} if it is $\g$-invariant.  
Alekseev and Kosmann-Schwarzbach \cite{Alekseev99} prove that \labelcref{eq:qPbracketInt} defines a Lie bracket on the space $C^\infty_{adm}(M)$ of admissible functions.

q-Poisson $(\mf{d},\g)$-structures are most interesting when there exists a Lie subalgebra $\h\subseteq\mf{d}$ such that $\mf{d}=\g\oplus\h$ as a vector space. Two special cases are when $\h\subseteq\mf{d}$ is Lagrangian (i.e. $\h=\h^\perp$), or when $\h\subseteq\mf{d}$ is a quadratic ideal (i.e. $\h\cap\h^\perp=0$). In both cases, interesting things happen:
\begin{center}
\begin{tabular}{p{.5\textwidth}|p{.5\textwidth}}
$\h\subseteq\mf{d}$ is Lagrangian. & $\h\subseteq\mf{d}$ is a quadratic ideal. \\ \hline\hline
The triple $(\mf{d},\g,\h)$ is a \emph{Manin triple} \cite{Drinfeld:1988fq}, and $\g$ is a \emph{Lie bialgebra}. &
 $\h\cong \g$ and $\mf{d}=\g\oplus \bar\g$, with $\g$ embedded diagonally and  $\h$ embedded as the second factor. \\ \hline
\cref{eq:qPbracketInt} defines a \emph{Poisson} structure on $M$. &\\\hline
While $\pi\in\mf{X}^2(M)$ is not $\g$ invariant, the action of $\g$ on $M$ is a \emph{infinitesimal Poisson action} \cite{thesis-3,lu90}. &
 $\pi\in\mf{X}^2(M)$ is $\g$-invariant.
\end{tabular}
\end{center}

Moreover, when $\h\subseteq\mf{d}$ is Lagrangian, there is a known integration procedure due to Xu \cite{Xu95}. Meanwhile when $\h\subseteq\mf{d}$ is an ideal  and $\h\cap\h^\perp=0$, there is a known integration procedure due to \v{S}evera and the author \cite{LiBland:2010wi}. As a first step,  both integration procedures take advantage of a Lie algebroid morphism $\mu:T^*M\to \h$, constructed as follows:

 The dual of the action map $\rho:\g\times M\to TM$ is a map $T^*M\to \g^*\times M$. Composing this with the projection to the first factor defines a map $T^*M\to \g^*$. After identifying $\g^*\cong\h$ via the quadratic form on $\mf{d}$, we obtain the map $$\mu:T^*M\to \h.$$ 
 We outline the properties of $\mu$ in the table below.
\begin{center}
\begin{tabular}{p{.5\textwidth}|p{.5\textwidth}}
$\h\subseteq\mf{d}$ is Lagrangian. & $\h\subseteq\mf{d}$ is a quadratic ideal. \\ \hline\hline
$T^*M$ is a Lie bialgebroid \cite{Mackenzie-Xu94}. & $T^*M$ is a Lie algebroid \cite{LiBland:2010wi}\footnotemark.\\\hline
$\mu:T^*M\to \h$ is a morphism of Lie bialgebroids \cite{Xu95}. & $\mu:T^*M\to \h$ is a morphism of Lie algebroids \cite{LiBland:2010wi}.\\
\end{tabular}
\end{center}
\footnotetext{This was already shown by Bursztyn and Crainic in the presence of a moment map for the q-Poisson structure \cite{Bursztyn:2005te}.}

Suppose $K$ is the source simply connected groupoid integrating $T^*M$ (or use a local integration instead), and $H$ is the simply connected Lie group integrating $\h$. Let $\mu:K\to H$ denote the morphism of Lie groupoids integrating $\mu:T^*M\to \h$. We outline the main features of the relevant integration procedures in the table below.
\begin{center}
\begin{tabular}{p{.5\textwidth}|p{.5\textwidth}}
$\h\subseteq\mf{d}$ is Lagrangian. & $\h\subseteq\mf{d}$ is a quadratic ideal. \\ \hline \hline
$H$ is a Poisson Lie group \cite{Drinfeld83}. & $H$ carries a multiplicative Dirac structure, the \emph{Cartan Dirac} structure \cite{Severa:2001,Bursztyn03-1}.\\ \hline
The Poisson structure on $M$ determines a compatible symplectic structure on $K$ \cite{Mackenzie97,Karasev:1986vg,Weinstein:1987ua}. & The q-Poisson structure on $M$ determines a compatible quasi-symplectic structure on $K$ \cite{LiBland:2010wi}.\\ \hline
$\mu:K\to H$ is a moment map for the natural action of $\g$ on $K$ \cite{Xu95}. & $\mu:K\to H$ is a moment map for the natural action of $\g$ on $K$ \cite{LiBland:2010wi}.
\end{tabular}
\end{center}

In \cref{chp:Outlook}, we will describe an integration procedure for general q-Poisson $(\mf{d},\g)$-spaces, which unifies both of the procedures outlined above. 
Before summarizing our result, we make the following general remark:
Extrapolating from the two special cases above, such an integration procedure should result in a moment map $\mu:K\to H$ where $K$ is a Lie groupoid, and $H$ is some Lie group. Moreover, one expects that $H$ will carry a multiplicative Dirac structure. In \cite{LiBland:2010wi,LiBland:2011vqa} multiplicative Dirac structures\footnote{For this classification result, the authors  assume that multiplication is a strong forward Dirac morphism, a desirable property in our context. See \cite{Jotz:2009va,Ortiz:2008bd,Affane:2009dx} for a slightly different concept of Dirac Lie groups.} over a simply connected Lie group $H$ were classified in terms of triples $(\mf{d},\g;\h)$ of Lie algebras where:
\begin{itemize}
\item $\mf{d}$ is a quadratic Lie algebra,
\item $\g\subseteq \mf{d}$ is a Lagrangian Lie subalgebra,
\item $\h\subseteq \mf{d}$ is a Lie subalgebra which integrates to $H$, and
\item $\mf{d}=\g\oplus\h$ as a vector space.
\end{itemize}

Our main result in \cref{chp:Outlook} is the following: Given a q-Poisson $(\mf{d},\g)$-structure on $M$, together with a Lie subalgebra $\h\subseteq \mf{d}$ which is a vector space complement to $\g$, 
\newcounter{saveenum}
\begin{enumerate}
\item there exists a Lie algebroid structure on $T^*M$, and
\item the map $\mu:T^*M\to \h$ is a morphism of Lie algebroids.
  \setcounter{saveenum}{\value{enumi}}
\end{enumerate}
Let $K$ denote the source simply connected Lie groupoid integrating $T^*M$ (or use a local integration instead), then 
 \begin{enumerate}
  \setcounter{enumi}{\value{saveenum}}
 \item $K$ carries a q-Poisson $(\mf{d},\g)$-structure, and 
 \item $\mu:K\to H$ is a moment map,
 \end{enumerate}
  where the simply connected Lie group $H$ carries the multiplicative Dirac structure determined by the triple $(\mf{d},\g,\h)$.

The theory of pseudo-Dirac structures, which we will introduce in \cref{chp:TngProAndLieSub}, play an important role in determining the Lie algebroid structure on $T^*M$. As explained by Iglesias Ponte-Xu \cite{PonteXu:08} and Bursztyn-Iglesias Ponte-\v{S}evera \cite{Bursztyn:2009wi}, the q-Poisson $(\mf{d},\g)$ structure on $M$ defines a Courant morphism $$R:\mbb{T}M\dasharrow \mf{d}.$$ 

As we shall prove in \cref{sec:FBimage}, since $\h$ is a pseudo-Dirac structure in the Courant algebroid $\mf{d}$, the composition $$L:=\h\circ R\subseteq\mbb{T}M$$ is a pseudo-Dirac structure in $\mbb{T}M$.  $L$ is canonically isomorphic to $T^*M$, so this defines a Lie algebroid structure on $T^*M$. In general $L\subseteq\mbb{T}M$ will not be a Dirac structure. However, in the special case where $\h\subseteq \mf{d}$ is Lagrangian, $L\subseteq \mbb{T}M$ is the graph of the Poisson structure on $M$.

\begin{remark}
Suppose $D$ is a Lie group with Lie algebra $\mf{d}$ and $G\subseteq D$ is a closed Lie subgroup integrating the inclusion $\g\subseteq \mf{d}$. 
An equivariant map $\phi:M\to D/G$ satisfying certain compatibility conditions with the q-Poisson structure is called a \emph{moment map}.

The peculiar feature of the bracket \labelcref{eq:qPbracketInt} - that it is only well behaved on the subalgebra of admissible functions - is reminiscent of the Dirac bracket, discussed in \cref{sec:DirBrk} above. Since Dirac structures form the geometric framework for the Dirac bracket, it is reasonable to try to incorporate q-Poisson geometry into Dirac geometry. 
 There has been a substantial body of work in this direction. Some key papers include \cite{Bursztyn:2009vq,Bursztyn:2005te,PonteXu:08,Bursztyn:2009wi}. To summarize the relevant results: In the presence of a moment map $\phi:M\to D/G$, they construct a Dirac structure $$L\subseteq \mbb{E},$$ in some exact Courant algebroid $\mbb{E}$,  encoding the q-Poisson structure\footnote{This was first explained by Bursztyn and Crainic in \cite{Bursztyn:2005te} for the special case where $\mf{d}=\g\oplus\bar\g$, and in full generality in \cite{Bursztyn:2009wi}}. Integrating this Dirac structure via the techniques described in \cite{Bursztyn03-1,Ponte:2005txa}, one obtains a (local) Lie groupoid $G$ which carries a compatible (non-degenerate) q-Poisson structure. In the special case where $\mf{d}$ is trivial, this procedure specializes to the integration of Poisson structures to symplectic groupoids. 
 \end{remark}

\chapter{Preliminaries}\label{chp:prelim}

\section{Linear relations}\label{sec:LinRel}
We recall some background on linear relations.

A (linear) \emph{relation} $R\colon V_1\dasharrow V_2$ 
between vector spaces $V_1,\ V_2$ is a subspace $R\subseteq V_2\times V_1$. Write $v_1\sim_R v_2$ if $(v_2,v_1)\in R$. The graph of any linear map $A\colon V_1\to V_2$ defines a relation $\on{gr}(A)$. In particular, the identity map 
of $V$ defines the diagonal relation $\on{gr}(\on{id}_V)=V_\Delta\subseteq V\times V$. 

The \emph{transpose relation} $R^\top \colon V_2\dasharrow V_1$ consists of all $(v_1,v_2)$ such that 
$(v_2,v_1)\in R$. We define 
\[\ker(R)=\{v_1\in V_1|\ v_1\sim 0\},\ \ \  
\on{ran}(R)=\{v_2\in V_2|\ \exists v_1\in V_1\colon (v_2,v_1)\in R\}.\]
Given another relation $R'\colon V_2\dasharrow V_3$, the composition $R'\circ R\colon V_1\dasharrow V_3$ consists of all $(v_3,v_1)$ such that $v_1\sim_{R}v_2$ and $v_2\sim_{R'} v_3$ for some $v_2\in V_2$. 

We let $\ann^\natural(R)\colon V_1^*\dasharrow V_2^*$ be the relation such that $\mu_1\sim_{\ann^\natural(R)}\mu_2$ if $\la\mu_1,\,v_1\ra=\la\mu_2,\,v_2\ra$ whenever $v_1\sim_R v_2$. 
Thus $(\mu_2,\mu_1)\in \ann^\natural(R)\Leftrightarrow (\mu_2,-\mu_1)\in \ann(R)$. Note 
$\ann^\natural(V_\Delta)=(V^*)_\Delta$, and more generally 
\begin{equation}\label{eq:annRel} \ann^\natural(\on{gr}(A))=\on{gr}(A^*)^\top\end{equation}
for linear maps $A\colon V_1\to V_2$. 

We will frequently find the following Lemma useful.

\begin{lemma}\label{lem:ann}
For any relations $R\colon V_1\dasharrow	 V_2$ and $R'\colon V_2\dasharrow V_3$, 
one has  
$\ann^\natural(R'\circ R)=\ann^\natural(R')\circ \ann^\natural(R)$. 
\end{lemma}

More generally, a smooth relation $S\colon M_1\dasharrow M_2$ between manifolds is an immersed submanifold 
$S\subseteq M_2\times M_1$. We will write $$m_1\sim_S m_2,$$ if $(m_2,m_1)\in S\subseteq M_2\times M_1$, and for functions $f_i\in C^\infty(M_i)$, we will write $$f_1\sim_S f_2$$ if $f_2\oplus(-f_1)$ vanishes on $S$.
Given smooth relations $S\colon M_1\dasharrow M_2$ and 
$S'\colon M_2\dasharrow M_3$, the set-theoretic composition $S'\circ S$ is the image of 
\begin{equation}\label{eq:intersection}
 S'\diamond S=(S'\times S)\cap (M_3\times (M_2)_\Delta \times M_1)
 \end{equation}
under projection to $M_3\times M_1$. Here $(M_2)_\Delta\subset M_2\times M_2$ denotes the diagonal.

If $V_1,V_2$ are two vector bundles over $M_1$ and $M_2$ respectively, then a \emph{$\mc{VB}$-relation} $R:V_1\dasharrow V_2$ is a subbundle
$R\subseteq V_2\times V_1$  along a submanifold $S\subset M_2\times M_1$.

We define $\ker(R)\subseteq p_{M_1}^*V_1,\ \on{ran}(R)\subseteq p_{M_2}^*V_2$ to be the kernel and range of the bundle map $R\to p_{M_2}^*V_2,\ 
(v_2,v_1)\mapsto v_2$ (where $p_{M_i}\colon S\to M_i,\ (m_2,m_1)\mapsto m_i$). 
Finally, we define the $\mc{VB}$-relation $\ann^\natural(R):V_1^*\dasharrow V_2^*$ by $$\ann^\natural(R)=\{(\mu_2,-\mu_1)\mid (\mu_2,\mu_1)\in \ann(R)\}\subseteq V_2^*\times V_1^*.$$

If  $V_i\to M_i$ are vector bundles, $R:V_1\dasharrow V_2$ is a $\mc{VB}$-relation over $S$,  and $\sigma_i\in\Gamma(V_i)$, then we write $$\sigma_1\sim_R\sigma_2$$ whenever $(\sigma_2,\sigma_1)\rvert_S\in \Gamma(R)$.

\begin{example}\label{ex:GrAdd}
Let $p:V\to M$ be a vector bundle. Then the $\mc{VB}$-relation $$\gr(+):V\times V\dasharrow V$$ defined by $$(v_1,v_2)\sim_{\gr(+)} v_1+v_2,$$ whenever $p(v_1)=p(v_2)$ is called the \emph{graph of addition}. A (fibrewise) linear function $f \in C^\infty(V)$ satisfies  $$f\oplus f\sim_{\gr(+)}f$$ while a  fibrewise constant function satisfies $$f\oplus 0\sim_{\gr(+)}f.$$
\end{example}

%

\section{Lie algebroids and Courant algebroids}

\subsection{Poisson manifolds}
Poisson geometry is both a useful tool in the theory of Lie algebroids and Courant algebroids, and a source of important examples. For the reader's convenience, we summarize the basic concepts here. 

A manifold $M$ is called a \emph{Poisson manifold} if the space of functions $C^\infty(M)$ is equipped with a Lie bracket $$\{\cdot,\cdot\}:C^\infty(M)\times C^\infty(M)\to C^\infty(M)$$ satisfying the Leibniz rule, $$\{f,gh\}=\{f,g\}h+g\{f,h\}, \quad f,g,h\in C^\infty(M).$$
Consequently, for any function $f\in C^\infty(M)$, the operator $\{f,\cdot\}:C^\infty(M)\to C^\infty(M)$ is a derivation, and so it defines a vector field $X_f=\{f,\cdot\}$ on $M$ called the \emph{Hamiltonian vector field} associated to $f$. By skew symmetry the Poisson bracket is also a derivation in the first variable. 

\begin{remark}[Multivector fields]
We let $\mf{X}(M):=\Gamma(TM)$ denote the space of vector fields, and $\mf{X}^k(M)=\Gamma(\wedge^k TM)$ the space of $k$-vector fields (note $\mf{X}^0(M)=C^\infty(M)$). The space $$\mf{X}^\bullet(M):=\bigoplus_k\mf{X}^k(M)$$ of \emph{multivector fields} is a $\mbb{Z}$-graded algebra with respect to the wedge product, and carries a unique bilinear operation, called the Schouten bracket \cite{Schouten:1954uj,Nijenhuis:1955td}, which extends the Lie bracket for vector fields and satisfies
\begin{itemize}
\item $[X,Y]=-(-1)^{(k-1)(l-1)}[Y,X]$
\item $[X,Y\wedge Z]=[X,Y]\wedge Z+(-1)^{(k-1)l}Y\wedge[X,Z]$ (derivation)
\item $[X,[Y,Z]]=[[X,Y],Z]+(-1)^{(k-1)(l-1)}[Y,[X,Z]]$ (graded Jacobi identity)
\end{itemize}
for $X\in\mf{X}^k(M)$, $Y\in\mf{X}^l(M)$, and $Z\in\mf{X}^\ast(M)$.
\end{remark}

Since $\{f,g\}$ depends only on the differentials $df$ and $dg$, there exists a bivector field $\pi\in\mf{X}^2(M)$ such that \begin{equation}\label{eq:bivField}\{f,g\}=\pi(df,dg).\end{equation} We let $\pi^\sharp:T^*M\to TM$ be the associated skew-symmetric map, that is $\pi^\sharp(df)=X_f$. It is known \cite{Lichnerowicz:1977wp} that the bracket \labelcref{eq:bivField} satisfies the Jacobi identity if and only if the Schouten bracket  of $\pi$ with itself vanishes $$[\pi,\pi]=0.$$

A map between Poisson manifolds $\phi:M\to N$ is said to be a \emph{Poisson morphism} if the associated map $\phi^*:C^\infty(N)\to C^\infty(M)$ is a map of Lie algebras.

Let $C\subset M$ be a submanifold, and let $\ann(TC)=\{\alpha\in T^*M\rvert_C\text{ such that } \alpha\rvert_{TC}=0\}$ denote its \emph{conormal} bundle. Then $C\subset M$ is said to be a \emph{coisotropic submanifold} if $\pi^\sharp\big(\ann(TC)\big)\subset TC$ \cite{weinstein87}.

If $M$ is any Poisson manifold with Poisson bracket $\{\cdot,\cdot\}_M$, then we let $\bar M$ denote the manifold $M$ with Poisson bracket $\{\cdot,\cdot\}_{\bar M}=-\{\cdot,\cdot\}_M$.

\begin{example}
Suppose $\phi:M\to N$ is a smooth map between Poisson manifolds, and let $\on{gr}(\phi)\subset N\times \bar M$ denote its graph. Then $\phi$ is a morphism of Poisson manifolds if and only if $\on{gr}(\phi)\subset N\times \bar M$ is a coisotropic submanifold \cite{weinstein87}.
\end{example}

Generalizing this example, a \emph{coisotropic relation} $R:M\dasharrow N$ between Poisson manifolds is a relation such that $R\subseteq N\times \bar M$ is a coisotropic submanifold \cite{weinstein87}. This implies that 
\begin{equation}\label{eq:CoisoRelFP}
f_1\sim_R g_1, \quad f_2\sim_R g_2\Rightarrow \{f_1,f_2\}\sim_R \{g_1,g_2\}
\end{equation}
for any $f_i\in C^\infty(M)$ and $g_i\in C^\infty(N)$.

\begin{example}[Lie algebras]\label{ex:LieAlgPois1}
If $\g$ is a finite dimensional Lie algebra, then $\g^*$ carries the Kirillov Poisson structure \cite{Kirillov:1976ud}. 
We may identify $\g$ with the subspace of linear functions on $\g^*$, and the Poisson bracket restricts to the Lie bracket for $\g$ on this subspace. This property is enough to define the bivector field $\pi$, which in turn defines the Poisson bracket of two arbitrary functions by the formula \labelcref{eq:bivField}.

If $\h$ is a second Lie algebra, then a linear map $\psi:\g\to\h$ is a Lie algebra morphism if and only if $\phi=\psi^*:\h^*\to\g^*$ is a morphism of Poisson manifolds.

In this way, there is a one-to-one correspondence between finite-dimensional Lie algebras and \emph{linear} Poisson structures on vector spaces.
\end{example}


As a slight generalization of \cref{ex:LieAlgPois1}, we have the following.
\begin{definition}[Linear Poisson structures]
Let $V$ be a vector bundle. A Poisson structure on $V$ is called \emph{linear} if the graph of addition \begin{equation}\label{eq:LinPoisStr}\on{gr}(+): V\times V\dasharrow V\end{equation} is a coisotropic relation.

\begin{remark}


If $f,g\in C^\infty(V)$ satisfy $$f\oplus f\sim_{\gr(+)}f,$$ and $$g\oplus g\sim_{\gr(+)} g$$ then \cref{eq:CoisoRelFP,eq:LinPoisStr} imply that $$\{f,g\}\oplus\{f,g\}\sim_{\gr(+)}\{f,g\}.$$ 
In other words, $\{f,g\}$ is a linear function whenever $f,g\in C^\infty(V)$ are linear. Similarly, $\{f,g\}$ is a fibrewise constant function whenever $f\in C^\infty(V)$ is linear and $g\in C^\infty(V)$ is fibrewise constant (or vice versa), and $\{f,g\}=0$ whenever $f,g\in C^\infty(V)$ are both fibrewise constant. In particular, the bundle projection $V\to M$ is a Poisson morphism to the zero Poisson structure.

Since the Poisson bracket satisfies the Leibniz rule, linear Poisson structures are entirely determined by the Poisson brackets of linear functions. For instance the Poisson bracket between a linear function $f$ and a constant function $g$ is determined by the Leibniz rule, $$\{f,g\}h=\{f,gh\}-g\{f,h\},$$ where $h$ is some arbitrarily chosen linear function.
\end{remark}
\end{definition}

A Poisson manifold $M$ is said to be \emph{symplectic} whenever $\pi^\sharp:T^*M\to TM$ is an isomorphism.

\begin{example}[The cotangent bundle]\label{ex:SympCot}
If $N$ is any manifold, then the cotangent bundle, $T^*N$, is a symplectic manifold. Let $p:T^*N\to N$ denote the bundle projection. For a vector field $X\in\mf{X}(N)$, let $\la X,\cdot\ra$ denote the corresponding linear function on $T^*N$.
The linear Poisson bracket on $T^*N$ is defined by the equation $$\{\la X,\cdot\ra ,\la Y,\cdot\ra\}=\la[X,Y],\cdot\ra$$ for $X,Y\in \mf{X}(N)$. Note that for $f,g\in C^\infty(N)$, one has
\begin{align*}
 \{\la X,\cdot\ra,p^*f\}&=p^*(X\cdot f), \\ \{p^*f,p^*g\}&=0.
\end{align*}
\end{example}


\subsection{Lie algebroids}
Lie algebroids, which are a common generalization of Lie algebras and the tangent bundle, will be an important concept in this paper. We recall their definition and basic properties, and refer the reader to \cite{moerdijk03,Mackenzie05} for more comprehensive expositions.

\begin{definition}
A \emph{Lie algebroid} is a vector bundle $A\to M$ together with a Lie bracket $[ \cdot,\cdot ]$
on its space of sections $\Gamma(A)$ and a 
bundle map $\mbf{a}:A\to TM$ called the \emph{anchor map} 
%
such that the following Leibniz identity is satisfied:
$$[\sigma,f\tau]=(\mbf{a}(\sigma)\cdot f)\tau+f[\sigma,\tau],\quad \sigma,\tau\in\Gamma(A), f\in C^\infty(M).$$

\begin{remark}
Note that the anchor map is determined by the Lie bracket. Indeed $$(\mbf{a}(\sigma)\cdot f)\sigma=[\sigma,f\sigma], \quad \sigma\in\Gamma(A),f\in C^\infty(M).$$ Additionally, the anchor map intertwines the Lie brackets:
$$\mbf{a}[\sigma,\tau]=[\mbf{a}(\sigma),\mbf{a}(\tau)].$$
\end{remark}

\end{definition}

\begin{example}[The tangent bundle and the Lie algebroid of a foliation.]\label{ex:TangBundLie}
The tangent bundle $A=TM\to M$ is a Lie algebroid, where the bracket on $\Gamma(TM)=\mf{X}(M)$ is the Lie bracket of vector fields, and the anchor map $\mbf{a}:TM\to TM$ is the identity map.

Lie algebroids with injective anchor map are precisely the integrable distributions $F\subset TM$, 
i.e. sub-bundles for which $\Gamma(F)$ is a 
Lie subalgebra of $\Gamma(TM)$.  


\end{example}

\begin{example}[Lie algebras and action Lie algebroids]
Any Lie algebra $\g$, regarded as a vector bundle over the point, is a Lie algebroid. 

More generally, if $\g$ acts on a manifold $M$ via the Lie algebra morphism $\rho:\g\to \mf{X}(M)$, then $\g\times M$ is a Lie algebroid with bracket 
\begin{equation}
\label{eq:actionlie} [\xi_1,\xi_2]_{\g\times M}=[\xi_1,\xi_2]_\g+\Lied_{\rho(\xi_1)}\xi_2-\Lied_{\rho(\xi_2)}\xi_1,\quad \xi_1,\xi_1\in\Gamma(\g\times M)\cong C^\infty(M,\g).\end{equation} The anchor map $\mbf{a}:\g\times M\to TM$ is defined to extend the map $\rho:\g\to\mf{X}(M)$ on constant sections. With this structure, $\g\times M$ is called the \emph{action Lie algebroid} for the action of $\g$ on $M$.
\end{example}

\begin{example}[Cotangent Lie algebroid of a Poisson manifold]\label{ex:CotLieAlg}
If $M$ is a Poisson manifold, with bivector field $\pi$, then $T^*M\to M$ is a Lie algebroid. The anchor map is $\pi^\sharp:T^*M\to TM$, and the Lie algebroid bracket is the \emph{Koszul} bracket
\begin{equation}\label{eq:Koszul}[\alpha,\beta]_\pi=d\pi(\alpha,\beta)+\iota_{\pi^\sharp\alpha}d\beta-\iota_{\pi^\sharp\beta}d\alpha.\end{equation}
\end{example}

\begin{example}[Atiyah Lie algebroid]\label{eq:AtiyahLieAlg}
Let $G$ be a Lie group, and $P\to M$ a principal $G$ bundle. Then $TP/G\to M$ is a Lie algebroid, with the Lie bracket induced by the embedding $$\Gamma(TP/G)\cong \Gamma(TP)^G\subset \mf{X}(P)$$ of sections of $TP/G$ as $G$-invariant vector fields on $P$.


\end{example}

An observation of Courant \cite[Theorem 2.1.4]{Courant:1990uy} provides an alternative characterization of Lie algebroids in Poisson geometrical terms which generalizes \cref{ex:LieAlgPois1}. We will find this alternative characterization useful in several places, but most notably to define Lie subalgebroids and morphisms of Lie algebroids. For the reader's convenience, we restate it here.

\begin{theorem}[\!\!{\cite[Theorem 2.1.4]{Courant:1990uy}}]\label{thm:linPois} Let $A\to M$ be a vector bundle and $A^*\to M$ the dual vector bundle.
There is a one-to-one correspondence between Lie algebroid structures on $A$ and linear Poisson structures on $A^*$. 
\end{theorem}
Recall that any section $\sigma\in\Gamma(A)$ defines a linear function $\la \sigma,\cdot\ra$ on $A^*$.
The Lie bracket on $\Gamma(A)$ is related to the linear Poisson structure on $A^*$ by the formula $$\la[\sigma,\tau],\cdot\ra=\{\la \sigma,\cdot\ra,\la \tau,\cdot\ra\},\quad \sigma,\tau\in\Gamma(A).$$ 

\begin{example}
The correspondence described in \cref{thm:linPois} relates the canonical Lie algebroid structure on the tangent bundle (i.e. \cref{ex:TangBundLie}) to the canonical linear symplectic structure on the cotangent bundle (i.e. \cref{ex:SympCot}).
\end{example}

\begin{example}[Tangent lift of a Poisson structure]\label{ex:PoisTngPro}
Suppose that $M$ is a Poisson manifold, then \cref{ex:CotLieAlg} shows that $T^*M$ is a Lie algebroid. By \cref{thm:linPois}, $TM$ carries a linear Poisson structure, called the \emph{tangent prolongation} (or \emph{tangent lift}) of the Poisson structure on $M$. In fact, the tangent lift is a functor from the category of Poisson manifolds to the category of linear Poisson structures on vector bundles. This construction is originally due to Courant \cite{Courant:1999ho}, and independently to S\'{a}nchez de Alvarez\cite{SanchezdeAlvarez:1989vf}. See also Grabowski \cite{Grabowski:1999ce}.
\end{example}

Next we recall the notion of a Lie subalgebroid of a Lie algebroid, originally due to Higgins and Mackenzie \cite{Higgins:1990gq}, though we use the formulation due to Mackenzie and Xu \cite{Mackenzie-Xu94}.

\begin{definition}\label{def:SubLA}
Let $A\to M$ be a Lie algebroid. A subbundle 
$B\subseteq A$ over a submanifold $N\subseteq M$ is called a Lie subalgebroid if its annihilator $$\ann(B)\subseteq A^*\rvert_N$$ is a coisotropic submanifold of $A^*$ (with respect to the linear Poisson structure on $A^*$).
\end{definition}

\begin{definition}\label{def:MorphLA}
Suppose that $A_1\to M_1$ and $A_2\to M_2$ are two Lie algebroids.
\begin{itemize}
 \item A morphism of vector bundles $\Phi:A_1\to A_2$ is called a Lie algebroid morphism if its graph $\on{gr}(\Phi)\subset A_2\times A_1$ is a Lie subalgebroid.

\item Suppose $\Psi:\psi^* A_2\to A_1$ is a base-preserving map of vector bundles, where $\psi^*A_2$ is the pullback of $A_2$ along a map $\psi:M_1\to M_2$. If $\on{gr}(\Psi)\subset A_2\times A_1$ is a Lie subalgebroid, then the vector bundle relation $\gr(\Psi):A_1\dasharrow A_2$ is said to be a \emph{comorphism} of Lie algebroids.

\item More generally, a relation $R:A_1\dasharrow A_2$ is called an $\mc{LA}$ relation if $R\subseteq A_2\times A_1$ is a Lie subalgebroid.
\end{itemize}

%

\begin{remark}
\begin{itemize}
\item The definition of base preserving morphisms of Lie algebroids in terms of coisotropic calculus is due to Courant \cite{Courant:1990uy}, following the work of Kirillov for Lie algebras \cite{Kirillov:1976ud}. 

\item The general definition of a morphism of Lie algebroids is originally due to Higgins and Mackenzie \cite{Higgins:1990gq}, the definition above is due to Mackenzie and Xu \cite{Mackenzie-Xu94}.

\item The concept of a comorphism between Lie algebroids is due to Higgins and Mackenzie \cite{Higgins:1993bc}.
Note that comorphisms can be described more naturally as Poisson morphisms $A_1^*\to A_2^*$ between the corresponding linear Poisson structures.

\item 
If $R:A_1\dasharrow A_2$ is an $\mc{LA}$-relation, then for any $\sigma_i,\tau_i\in\Gamma(A_i)$ satisfying $\sigma_1\sim_R\sigma_2$ and $\tau_1\sim_R\tau_2$, we have
$$[\sigma_1,\tau_1]\sim_R[\sigma_2,\tau_2].$$
\end{itemize}
\end{remark}
\end{definition}

\begin{remark}[Supergeometric definition of a Lie algebroid]\label{rem:SupLie}
There is an elegant definition of Lie algebroids in terms of supergeometry, due to Va\u{\i}ntrob \cite{LieAlgebroidsH}. Recall that an \emph{$N$-manifold} is a supermanifold carrying an action of the multiplicative semigroup $\mbb{R}$ such that $-1$ acts as the parity operator \cite{Severa:2005vla}. The degree of a function (or vector field etc.) is the weight of this action, i.e. $t\cdot f=t^{\on{deg}(f)}f$, $t\in\mbb{R}$. An $N$-manifold is said to be of \emph{degree $k$} if the highest degree of a coordinate function is $k$. 

If $X$ is an $N$-manifold, then $0\cdot X$ is a regular smooth manifold called the \emph{base of $X$}. We may consider the sheaf of functions on $0\cdot X$ generated by elements of $C^\infty(X)$ of degree no greater than $k$. This sheaf of functions describes an $N$-manifold called the \emph{$k^{th}$ truncation of $X$}. For example, the $0^{th}$ truncation is $0\cdot X$. 

Finally, a \emph{homological vector field} is a degree 1 self-commuting vector field, usually denoted $Q$. Since $$2Q^2=[Q,Q]=0,$$ the operator $Q$ defines a differential of degree 1 on $C^\infty(X)$ which is also a derivation with respect to the multiplication (i.e. $C^\infty(X)$ is a \emph{differential graded algebra}). Supermanifolds (resp. $N$-manifolds) carrying a homological vector field are referred to as \emph{$Q$-manifolds} (resp. \emph{$NQ$-manifolds}) \cite{Schwarz:1993,Alexandrov:1997jj}.

 If $A\to M$ is a vector bundle, then the supermanifold $A[1]$ whose sheaf of functions is $C^\infty(A[1]):=\Gamma(\wedge^\bullet A^*)$, is a degree 1 $N$-manifold (where $\mbb{R}$ acts on the fibres by scalar multiplication).

As shown by Va\u{\i}ntrob \cite{LieAlgebroidsH}, a Lie algebroid structure on $A$ is equivalent to a \emph{homological} vector field, $Q$, on the supermanifold $A[1]$ (i.e. $\Gamma(\wedge^\bullet A^*)$ is a differential graded algebra).

If $\sigma\in \Gamma(A)$, then contraction with $\sigma$ defines a degree -1 derivation, $\iota_\sigma$, of $C^\infty(A[1]):=\Gamma(\wedge^\bullet A^*)$. In this way, degree -1 vector fields on $A[1]$ are in one-to-one correspondence with sections of $A$.
The Lie bracket of two sections $\sigma,\tau\in\Gamma(A)$ is given by the so-called \emph{derived bracket construction}
$$\iota_{[\sigma,\tau]}=[[\iota_\sigma,Q],\iota_\tau].$$

If $A\to M$ and $B\to N$ are two Lie algebroids, then a map $\phi:A[1]\to B[1]$ which intertwines the homological vector fields is equivalent to a morphism $A\to B$ of Lie algebroids.

\end{remark}

\subsection{Courant algebroids}
Dirac structures were introduced by T. Courant \cite{Courant:1990uy} as a unified framework from which to study constrained mechanical systems. Liu-Weinstein-Xu \cite{ManinTriplesBi}
generalized Courant's original set-up, replacing $\mbb{T}M$ with a more general notion of a \emph{Courant algebroid}
$\mbb{E}\to M$.  We recall the basic theory, and refer the reader to \cite{Dorfman:1993us,LetToWein,Roytenberg99,Severa:2005vla,Roytenberg:2002,Uchino02,Bursztyn:2009wi}.

\begin{definition}\label{def:CA}
A \emph{Courant algebroid} over a manifold $M$ is a vector bundle $\mbb{E}\to M$, together with a bundle 
map $\mbf{a}\colon \mbb{E}\to TM$ called the \emph{anchor}, a bundle metric\footnote{In this thesis, we take 
`metric' to mean a non-degenerate symmetric bilinear form.} $\la\cdot,\cdot\ra$, and 
 a bilinear bracket $\Cour{\cdot,\cdot}$ on its space of sections $\Gamma(\mbb{E})$. These are required 
to satisfy the following axioms, for all sections $\sigma_1,\sigma_2,\sigma_3\in\Gamma(\mbb{E})$:
\begin{enumerate}
\item[c1)] $\Cour{\sigma_1,\Cour{\sigma_2,\sigma_3}}=\Cour{\Cour{\sigma_1,\sigma_2},\sigma_3}
+\Cour{\sigma_2,\Cour{\sigma_1,\sigma_3}}$, 
\item[c2)] $\mbf{a}(\sigma_1)\la \sigma_2,\sigma_3\ra=\la \Cour{\sigma_1,\sigma_2},\,\sigma_3\ra+\la \sigma_2,\,\Cour{\sigma_1,\sigma_3}\ra$,
\item[c3)] $\Cour{\sigma_1,\sigma_2}+\Cour{\sigma_2,\sigma_1}=\mbf{a}^*(d \la \sigma_1,\sigma_2\ra)$.
\end{enumerate}
Here $\mbf{a}^*\colon T^*M\to \mbb{E}^*\cong\mbb{E}$ is the dual map to $\mbf{a}$. The axioms c1)-c3) imply various other properties, 
in particular
\begin{enumerate}
\item[c4)] $\Cour{\sigma_1,f\sigma_2}=f\Cour{\sigma_1,\sigma_2}+\mbf{a}(\sigma_1)(f)\sigma_2$,
\item[c5)] $\Cour{f\sigma_1,\sigma_2}=f\Cour{\sigma_1,\sigma_2}-\mbf{a}(\sigma_2)(f)\sigma_1+\la \sigma_1,\sigma_2\ra \mbf{a}^*(d f)$, 
\item[c6)] $\mbf{a}(\Cour{\sigma_1,\sigma_2})=[\mbf{a}(\sigma_1),\mbf{a}(\sigma_2)]$,
\end{enumerate}
for sections $\sigma_i\in\Gamma(\mbb{E})$ and functions $f\in C^\infty(M)$. We
will refer to the bracket $\Cour{\cdot,\cdot}$ as the \emph{Courant
bracket} (some authors refer to $\Cour{\cdot,\cdot}$ as the Dorfman bracket (after Dorfman \cite{Dorfman:1993us}, who introduced it) and its skew-symmetric part as the Courant bracket). 

For any Courant algebroid $\mbb{E}$, we denote by $\overline{\mbb{E}}$ the Courant algebroid with the same 
bracket and anchor, but with the bundle metric, $\la\cdot,\cdot\ra$, negated.\footnote{To see that axiom c3) holds for $\overline{\mbb{E}}$, it can be useful to rewrite it as $$\la\Cour{\sigma_1,\sigma_2}+\Cour{\sigma_2,\sigma_1},\sigma_3\ra=\mbf{a}(\sigma_3)\cdot\la \sigma_1,\sigma_2\ra,$$ an equation whose validity is manifestly preserved when the bundle metric is negated. Note that  we abused notion when writing the original equation, denoting  the composition $T^*M\xrightarrow{\mbf{a}^*}\mbb{E}^*\xrightarrow{\la\cdot,\cdot\ra}\mbb{E}$ simply by $\mbf{a}^*$. }

\end{definition}

A subbundle $E\subseteq \mbb{E}$ along a submanifold $S\subseteq M$ is called \emph{involutive} if it has the 
property 
\[ \sigma_1|_S,\ \sigma_2|_S\in\Gamma(E) 
\Rightarrow \Cour{\sigma_1,\sigma_2}|_S\in \Gamma(E),\]  
for any $\sigma_1,\sigma_2\in\Gamma(\mbb{E})$.
It is important to note that this does not define a bracket on sections of $E$, in general.

We let $L^\perp\subseteq \mbb{E}$ denote the orthogonal complement of $L$ with respect to the fibre metric.
A subbundle $L\subset \mbb{E}$ is called \emph{Lagrangian} if $L^\perp=L$, and coisotropic if $L^\perp\subseteq L$.
An involutive Lagrangian subbundle $E\subseteq \mbb{E}$  along $S\subseteq M$ is called
a \emph{Dirac structure along $S$}.   

A Dirac structure, $E$, along $S=M$ is
simply called a Dirac structure, and the pair $(\mbb{E},E)$ is called a \emph{Manin pair} \cite{PonteXu:08,Bursztyn:2009wi}. In this case, the restriction of the Courant bracket and the anchor map endows $E$ with the structure of a Lie algebroid \cite{Courant:1990uy,ManinTriplesBi}.
Dirac structures were introduced by
Courant \cite{Courant:1990uy} and Liu-Weinstein-Xu
\cite{ManinTriplesBi}. The notion of a Dirac structure along a
submanifold goes back to \v{S}evera \cite{LetToWein} and was developed
in \cite{Alekseev:2002tn,Bursztyn:2009wi,PonteXu:08}.

\begin{example}[The standard Courant algebroid]\label{ex:StdCourAlg}
The \emph{standard Courant algebroid} over $M$ is $\mbb{T} M=TM\oplus T^*M$ with anchor map the projection to the first factor and bilinear form 
$\la (X,\alpha),(Y,\beta)\ra=\la\beta,X\ra+\la \alpha,Y\ra$. The Courant bracket reads
\begin{equation}\label{eq:StdCourBrack} \Cour{(X,\alpha),(Y,\beta)}=([X,Y],\Lied_{X}\beta-\iota_{Y}d\alpha),\end{equation}
for vector fields $X,Y\in\mf{X}(M)$ and 1-forms $\alpha,\beta\in\Omega^1(M)$.

Both $TM$ and $T^*M$ are Dirac structures in $\mbb{T} M$. More generally, if $\pi\in\mf{X}^2(M)$ is a bivector field, then the graph $\on{gr}(\pi^\sharp)\subset \mbb{T}M$ of the associated skew symmetric map $\pi^\sharp:T^*M\to TM$ is a Dirac structure if and only if $\pi$ is Poisson. If $\omega\in \Omega^2(TM)$ is a 2-form, then the graph $\on{gr}(\omega^\flat)\subset\mbb{T} M$ of the associated skew symmetric map $\omega^\flat:TM\to T^*M$ is a Dirac structure if and only if $\omega$ is closed.

The standard Courant algebroid was introduced by Courant \cite{Courant:1990uy}.
\end{example}

\begin{example}[Exact Courant algebroids]The theory of \emph{exact Courant algebroids} was developed by \v{S}evera \cite{Severa:2001}.
A Courant algebroid $\mbb{E}\to M$ is called \emph{exact}, if the sequence $$0\to T^*M\xrightarrow{\mbf{a}^*}\mbb{E}\xrightarrow{\mbf{a}}TM\to0$$ is  exact. In this case, any Lagrangian splitting $s:TM\to\mbb{E}$ of this sequence defines a closed 3-form $\gamma\in\Omega^3(M)$ by the formula $$\gamma(X,Y,Z)=\la\Cour{s(X),s(Y)},s(Z)\ra,\quad X,Y,Z\in\Gamma(TM)$$
Additionally, the splitting defines a trivialization $$(\mbf{a}\times s^*):\mbb{E}\to TM\times T^*M,$$ which intertwines the bundle metric on $\mbb{E}$ with the natural pairing on $TM\oplus T^*M$. 

Under this identification, the Courant bracket becomes
$$ \Cour{(X,\alpha),(Y,\beta)}_\gamma=([X,Y],\Lied_{X}\beta-\iota_{Y}d\alpha+\iota_X\iota_Y\gamma),$$
for vector fields $X,Y\in\mf{X}(M)$ and 1-forms $\alpha,\beta\in\Omega^1(M)$.

 Up to isomorphism, exact Courant algebroids are classified by the de Rham cohomology class $[\gamma]\in H^3_{dR}(M)$, often referred to as the \emph{\v{S}evera class} of the Courant algebroid.
\end{example}

\begin{example}[Quadratic Lie algebras]\label{ex:ActCourAlg}
A Lie algebra together with an invariant metric is called a \emph{quadratic Lie algebra}.
Courant algebroids over a point correspond to \emph{quadratic Lie algebras}.

Suppose $\mf{d}$ is a quadratic Lie algebra,
acting on a manifold $M$. Let $\rho\colon \mf{d}\times M\to TM$ be the
action map. Let $\mbb{E}=\mf{d}\times M$ with anchor map $\mbf{a}=\rho$ and with
the bundle metric coming from the metric on $\mf{d}$. As shown in
\cite{LiBland:2009ul}, the Lie bracket on constant sections $\mf{d}\subseteq
C^\infty(M,\mf{d})=\Gamma(\mbb{E})$ extends to a Courant bracket if and only if
the action has coisotropic stabilizers $\ker(\rho_m)\subseteq
\mf{d}$. Explicitly, for $\sigma_1,\sigma_2\in \Gamma(\mbb{E})=C^\infty(M,\mf{d})$ the
Courant bracket reads (see \cite[$\mathsection$ 4]{LiBland:2009ul})
\begin{equation}
\label{eq:actioncourant} \Cour{\sigma_1,\sigma_2}=[\sigma_1,\sigma_2]+\Lied_{\rho(\sigma_1)}\sigma_2-\Lied_{\rho(\sigma_2)}\sigma_1+\rho^*\la d\sigma_1,\sigma_2\ra.\end{equation}
Here $\rho^*\colon T^*M\to \mf{d}\times M$ is the dual map to the action map, 
using the metric to identify $\mf{d}^*\cong \mf{d}$. 
Note that the first three terms give the Lie algebroid bracket \labelcref{eq:actionlie} for the action Lie algebroid $\mf{d}\times M$. The correction term \begin{equation}\label{eq:ActCourCorr}\rho^*\la d\sigma_1,\sigma_2\ra\end{equation} turns the Lie algebroid bracket into a 
Courant bracket. We refer to $\mf{d}\times M$ with bracket \labelcref{eq:actioncourant} as an \emph{action Courant algebroid}.
\end{example}

\begin{remark}[Supergeometric perspective]\label{rem:SupCour} We recall the one-to-one correspondence, due to Roytenberg \cite{Roytenberg:2002}, between vector bundles carrying a bundle metric, and $N$-manifolds carrying a symplectic form of degree 2 (this fact was discovered independently by \v{S}evera \cite{LetToWein, Severa:2005vla}).

Suppose that $V\to M$ is a vector bundle. Then a bundle metric $\la\cdot,\cdot\ra$ on $V$ defines a degree $-2$ Poisson structure on $V[1]$, as follows. Any degree 1 function on $V[1]$ is equal to $\la\sigma,\cdot\ra$ for some section $\sigma\in\Gamma(V)$. The Poisson bracket of two such functions is defined to be $$\{\la\sigma,\cdot\ra,\la\tau,\cdot\ra\}=\la\sigma,\tau\ra,\quad\sigma,\tau\in\Gamma(V),$$ and it extends to functions of arbitrary degree via the Leibniz rule. 
Let $X$ be the fibred product defined by the following diagram 
%
$$\begin{tikzpicture}
\mmat[2em]{m}
{X&T^*[2]V[1]\\
V[1]&(V\oplus V^*)[1]\\};
\path[->]
(m-1-1) 	edge  (m-1-2)
		edge  (m-2-1);
\path[->] (m-1-2) edge node {$p$} (m-2-2);
\path[->] (m-2-1) edge node {$i$} (m-2-2);
\end{tikzpicture}$$

 where $p$ is the canonical projection and $i$ is the embedding given by $$i:v\to v\oplus\frac{1}{2}\la v,\cdot\ra,\quad v\in V.$$ Then $X$ is a symplectic submanifold of $T^*[2]V[1]$, and is the minimal symplectic realization of $V[1]$.

Next, we recall that this correspondence extends to a one-to-one correspondence between Lagrangian subbundles $W\subseteq V$ and Lagrangian sub $N$-manifolds $L\subseteq X$, as observed by \v{S}evera \cite{LetToWein, Severa:2005vla}.
Let $W\subset V$ be a Lagrangian subbundle over $S\subseteq M$, then $\ann[2](TW[1])\subset T^*[2]V[1]$ is a Lagrangian submanifold. The corresponding Lagrangian submanifold $L\subseteq X$ is the intersection $L:=X\cap \ann[2](TW[1])$. 

It was discovered independently by Roytenberg and \v{S}evera \cite{Roytenberg:2002, LetToWein, Severa:2005vla} that Courant brackets on $V\to M$ are in one-to-one correspondence with homological vector fields on $X$ which preserve the symplectic structure, in which case $X$ is called a \emph{degree 2 symplectic $NQ$-manifold}. Moreover, as described by \v{S}evera \cite{LetToWein,Severa:2005vla}, Dirac structures along a submanifold are in one-to-one correspondence with Lagrangian submanifolds of $X$ which are tangent to the homological vector field.

\end{remark}

\subsubsection{Courant relations, and morphisms of Manin pairs}\label{sec:CARel}

\begin{definition}[Courant morphisms and relations]
Let $\mbb{E}_1,\mbb{E}_2$ be two Courant algebroids over $M_1$ and $M_2$, respectively. A \emph{Courant relation} $R:\mbb{E}_1\dasharrow\mbb{E}_2$ is a Dirac structure $R\subseteq \mbb{E}_2\times\overline{\mbb{E}_1}$ along a submanifold $S\subset M_2\times M_1$.
If $S$ is the graph of a map $M_1\to M_2$, then $R$ is called a \emph{Courant morphism}.
\end{definition}

As a consequence of the definition, if $\sigma_i\in \Gamma(\mbb{E}_i)$ and $\tau_i\in\Gamma(\mbb{E}_i)$ satisfy 
$\sigma_1\sim_R\sigma_2$, and $\tau_1\sim_R\tau_2$, then
\begin{equation*}\begin{split}
\Cour{\sigma_1,\tau_1}&\sim_R\Cour{\sigma_2,\tau_2},\\
\la\sigma_1,\tau_1\ra&\sim_R\la\sigma_2,\tau_2\ra.
\end{split}
\end{equation*}

If two Courant relations $R_1:\mbb{E}_1\dasharrow \mbb{E}_2$ and $R_2:\mbb{E}_2\dasharrow\mbb{R}_3$ compose \emph{cleanly}, then their composition $R_2\circ R_1:\mbb{E}_1\dasharrow \mbb{R}_3$ is a Courant relation (see \cite[Proposition~1.4]{LiBland:2011vqa}). In particular, since Dirac structures $E_i\subseteq \mbb{E}_i$ define Courant relations (to or from the trivial Courant algebroid), if the compositions $R_1\circ E_1\subseteq \mbb{E}_2$ and $E_2\circ R_1\subseteq \mbb{E}_1$ are clean, they define Dirac structures (with support).

\begin{example}\label{ex:diagDirStr}
Suppose $\mbb{E}$ is a Courant algebroid over $M$. Then the diagonal $\mbb{E}_\Delta\subseteq \mbb{E}\times\overline{\mbb{E}}$ is a Dirac structure with support along the diagonal $M_\Delta\subseteq M\times M$. The corresponding Courant relation, $$\mbb{E}_\Delta:\mbb{E}\dasharrow\mbb{E}$$ is just the identity map.
\end{example}


\begin{definition}[Morphisms of Manin pairs]\label{def:MorphMP}
Suppose $E_1\subset \mbb{E}_1$ and $E_2\subset\mbb{E}_2$ are two Dirac structures. A \emph{morphism of Manin pairs} 
$$R:(\mbb{E}_1,E_1)\dasharrow(\mbb{E}_2,E_2)$$ 
is a Courant morphism $R:\mbb{E}_1\dasharrow\mbb{E}_2$ such that
\begin{enumerate} 
\item[m1)] $E_1\cap \on{ker}(R)=0$, 
\item[m2)] $R\circ E_1\subseteq E_2$. 
\end{enumerate}
\end{definition}

Courant morphisms were introduced by Alekseev and Xu in  \cite{Alekseev:2002tn} while morphisms of Manin pairs were introduced by Bursztyn, Iglesias Ponte and \v{S}evera in \cite{Bursztyn:2009wi}.
We recall from \cite{Bursztyn:2009wi} some of the important properties of morphisms of Manin pairs. Let
$$R:(\mbb{E}_1,E_1)\dasharrow(\mbb{E}_2,E_2)$$ be a morphism of Manin pairs over the map $\phi:M_1\to M_2$. Then \cref{def:MorphMP} implies that $$R/\big((E_2\times E_1)\cap R\big)\subseteq (\mbb{E}_2/E_2)\times(\mbb{E}_1/E_1)$$ is the graph of a morphism $\mbb{E}_1/E_1\to\mbb{E}_2/E_2$, which we can identify with a map $$\Phi_R:E_1^*\to E_2^*.$$ In fact $\ann^\natural(\Phi_R)\subset E_2\times E_1$ is the graph of a comorphism $\phi^*E_2\to E_1$ of Lie algebroids.

Next, if $F_2\subset\mbb{E}_2$ is a Dirac structure which is transverse to $E_2$ (i.e. $F_2\oplus E_2=\mbb{E}_2$), then $F_1:=F_2\circ R\subset\mbb{E}_1$ is a Dirac structure transverse to $E_1$ \cite[Remark~2.12]{Bursztyn:2009wi} (see also \cite[Proposition~1.4]{LiBland:2011vqa}). 
Under the identification $F_i\cong E_i^*$ given by the pairing in $\mbb{E}_i$, the map $\Phi_R:E_1^*\to E_2^*$ can be identified with a map $$F_1\to F_2.$$ By\cite[Remark~2.12]{Bursztyn:2009wi}, this latter map is a morphism of Lie algebroids.
\begin{remark}
Another way to see this fact is as follows. Since $R$ is involutive, $\gr(\Phi_R)\cong R\cap(F_2\times F_1)$ is a Lie subalgebroid. Thus, by \cref{def:MorphLA}, $\Phi_R:F_1\to F_2$ is a morphism of Lie algebroids.
\end{remark}

As described in \cite{Bursztyn:2009wi}, the composition of morphisms of Manin pairs is defined to be the composition of the underlying relation, which is guaranteed to be clean by \cref{def:MorphMP}.

\begin{example}[Standard lift of a relation]\label{ex:StdDiracStr}
Let $S\subset M$ be an embedded submanifold. Then $TS\oplus N^{0}(S)\subseteq \mbb{T} M$ is a Dirac structure along $S$. Moreover, if $S:M_1\dasharrow M_2$ is a relation, then $R_S:=TS\oplus\ann^\natural(TS)\subseteq \mbb{T}M_2\times \overline{\mbb{T} M_1}$ defines a Courant relation $$R_S:\mbb{T}M_1\dasharrow \mbb{T}M_2.$$
\end{example}

\begin{example}[q-Poisson structures]\label{ex:qPstr}
Let $\mf{d}$ be a quadratic Lie algebra and $\g\subset\mf{d}$ a Lagrangian subalgebra. We will refer to a morphism of Manin pairs 
\begin{equation}\label{eq:qPstr}R:(\mbb{T}M,TM)\dasharrow (\mf{d},\g)\end{equation}
as a \emph{q-Poisson $(\mf{d},\g)$ structure on $M$}.

It was shown in \cite[\S~3.2]{Bursztyn:2009wi} that once a Lagrangian complement $\mf{p}\subset\mf{d}$ to $\g$ has been chosen, a q-Poisson $(\mf{d},\g)$-structure on $M$ defines a q-Poisson structure on $M$ in the sense of \cite{Alekseev99}, for the Manin quasi-triple $(\mf{d},\g,\mf{p})$.


If $\mf{p}\subset\mf{d}$ is a Dirac structure (i.e. a Lagrangian Lie subalgebra), then $F=\mf{p}\circ R\subset\mbb{T} M$ is a Dirac structure. Since $F$ is transverse to $TM$, it is isomorphic to $T^*M$ as a vector bundle. In this way, any choice of Lagrangian Lie subalgebra $\mf{p}\subset\mf{d}$ endows $T^*M$ with the structure of a Lie algebroid (cf. \cref{ex:qPCotLie}).
\end{example}

\begin{example}[{\!\!\cite[Example~1.6]{LiBland:2011vqa}}]\label{ex:canonicalmorphism}
For any Manin pair $(\mbb{E},E)$ over $M$, there is a morphism of Manin pairs 
$$ J_E\colon (\mbb{T}M,TM)\dasharrow (\mbb{E},E)$$
where $v+\mu\sim_R x$ if and only if $v=\mbf{a}(x)$ and $x-\mbf{a}^*(\mu)\in E$, where $v\in TM$, $\mu\in T^*M$ and $x\in \mbb{E}$.
\end{example}

\begin{remark}[Supergeometric perspective]\label{rem:SupMP}
We now describe an interpretation, due to \v{S}evera \cite{NonComDiffForm}, of Manin pairs in terms of supergeometry (see also \cite{LetToWein,Bursztyn:2009wi}).

Recall that a degree $k$ Poisson structure on a supermanifold $Y$ corresponds to a quadratic function $\pi$ of total degree $2-k$ on $T^*[1-k]Y$. The function $\pi$ is called the \emph{Poisson bivector field}. Let $\{\cdot,\cdot\}$ denote the canonical Poisson bracket on $T^*[1-k]Y$, as described in \cref{ex:SympCot}, and $p:T^*[1-k]Y\to Y$ the bundle projection. Then for $f,g\in C^\infty(Y)$, the Poisson bracket $\{f,g\}_\pi$ defined by $\pi$ is defined by the \emph{derived bracket construction}
$$p^*(\{f,g\}_\pi)=\{\{p^*f,\pi\},p^*g\}.$$ The bracket $\{\cdot,\cdot\}_\pi$ satisfies the Jacobi identity if and only if $\{\pi,\pi\}=0$.

As described in \cite{NonComDiffForm}, Manin pairs $(\mbb{E},E)$ are in one-to-one correspondence with principal $\mbb{R}[2]$ bundles $P\to E^*[1]$ carrying an $\mbb{R}[2]$ invariant Poisson structure of degree -1. To describe the correspondence, suppose $P\to E^*[1]$ is such a bundle and let $\pi$ denote the degree 3 function on $T^*[2]P$ corresponding to the Poisson bivector field. The cotangent lift of the $\mbb{R}[2]$ action is Hamiltonian, with moment map $\mu:T^*[2]P\to \mbb{R}$.
 Let $X=\mu^{-1}(1)/\mbb{R}[2]$ be the Marsden-Weinstein quotient taken at moment level 1. Since $\pi$ is an $\mbb{R}[2]$ invariant function satisfying $\{\pi,\pi\}=0$, it reduces to a degree 3 function $\Phi\in C^\infty(X)$ on the symplectic quotient satisfying $\{\Phi,\Phi\}=0$. Let $Q=\{\Phi,\cdot\}$ be the homological vector field corresponding to $\Phi$.  Since $X$ is an $NQ$-manifold carrying a symplectic form of degree 2, it corresponds to a Courant algebroid $\mbb{E}$ (see \cite{LetToWein,Roytenberg:2002} or \cref{rem:SupCour} for details).

The zero set of the ideal of functions of positive degree on $P$ is precisely the base $M$ of the supermanifold.
Since the Poisson structure on $P$ is of degree -1, it follows that $M$ is a coisotropic submanifold of $P$. Hence $\ann[2](TM)\subset T^*[2]P$ is a Lagrangian submanifold on which $\pi$ vanishes. Hence it reduces to define a Lagrangian submanifold $L\subset X$ tangent to the homological vector field $Q$. As explained in \cite{LetToWein,Severa:2005vla} or \cref{rem:SupCour}, this defines a Dirac structure in $\mbb{E}$, which is canonically isomorphic to $E$ as a vector bundle. In this way, one associates a Manin pair $(\mbb{E},E)$ to the Poisson principal $\mbb{R}[2]$ bundle $P\to E^*[1]$.

\v{S}evera observed  \cite{NonComDiffForm,PoissonActions,Bursztyn:2009wi} that trivializations $$P\to \mbb{R}[2]\times E^*[1]$$ of the $\mbb{R}[2]$ bundle are in one-to-one correspondence with Lagrangian complements to $E\subseteq\mbb{E}$.
Indeed, let $t:P\to\mbb{R}[2]$ be such a trivialization. The graph of $dt$ is a Lagrangian submanifold of $T^*[2]P$. Hence it reduces to a Lagrangian submanifold $L'$ of $X$ which, as explained in \cref{rem:SupCour}, corresponds to a Lagrangian complement $F\subset\mbb{E}$ to $E$. Moreover, if $\{t,t\}_\pi=0$, then $\pi$ vanishes on the graph of $dt$, and therefore $L'\subset X$ is tangent to the homological vector field, so $F\subset\mbb{E}$ is a Dirac structure \cite{NonComDiffForm,PoissonActions}.


Suppose $(\mbb{E}_1,E_1)$ and $(\mbb{E}_2,E_2)$ are two Manin pairs corresponding to the Poisson principal $\mbb{R}[2]$ bundles $P_1\to E_1^*[1]$ and $P_2\to E_2^*[1]$. Then, as explained by \v{S}evera \cite[Remark~2.12]{Bursztyn:2009wi}, there is a one-to-one correspondence between morphisms of Manin pairs $$K:(\mbb{E}_1,E_1)\dasharrow(\mbb{E}_2,E_2)$$ and $\mbb{R}[2]$ equivariant Poisson morphisms $\phi:P_1\to P_2$. If $t:P_2\to\mbb{R}[2]$ is a trivialization of the $\mbb{R}[2]$ bundle corresponding to a Lagrangian complement $F_2\subset\mbb{E}_2$ to $E_2$, then $\phi^*(t):P_1\to\mbb{R}[2]$ is a trivialization corresponding to the Lagrangian complement $F_1=F_2\circ K$. Since $\{\phi^*t,\phi^*t\}_{\pi_0}=\phi^*\{t,t\}_{\pi_1}$, $F_1:=F_2\circ K$ is a Dirac structure whenever $F_2$ is.
\end{remark}

\subsubsection{Coisotropic reduction and  Pull-backs of Courant algebroids}\label{sec:pullback}
Reduction of exact Courant algebroids was developed in \cite{Zambon:2008wj,Bursztyn:2007ko} (see also \cite{Stienon:2008cl,Vaisman:2007gg,Hu:2009wl,Lin:2006ku,Yoshimura:2007gw}). Some aspects of this construction were extended to arbitrary Courant algebroids in  \cite[Sections~2.1 and 2.2]{LiBland:2009ul}, as we recall now.
\begin{proposition}[\!\!{\cite[Proposition~2.1]{LiBland:2009ul}}]
Let $S\subset M$ be a submanifold, and $C\subset \mbb{E}|_S$ an involutive coisotropic subbundle
such that 
$$\mbf{a}(C)\subset TS,\ \mbf{a}(C^\perp)=0.$$ 
Then the anchor map, bracket and inner product on $C$ descend to
$\mbb{E}_C=C/C^\perp$, and make $\mbb{E}_C$ into a Courant algebroid over $S$.
The inclusion $\phi\colon S\hookrightarrow M$ lifts to a  
Courant   morphism 
\[ R_\phi\colon \mbb{E}_C\dasharrow \mbb{E},\ \  y\sim_{R_\phi} x\ \Leftrightarrow\ x\in C,\ y=p(x)\]
where $p\colon C\to \mbb{E}_C$ is the quotient map. 
\end{proposition}

As a special case of the coisotropic reduction procedure, we can define the pull back of a Courant algebroid along a map.
\begin{definition}\label{def:CAPullback}
Suppose $\phi\colon S\to M$ is a smooth map whose
differential $d\phi\colon TS\to TM$ is transverse to $\mbf{a}\colon \mbb{E}\to
TM$.  We define the \emph{pull-back Courant algebroid} $\phi^!\mbb{E}\to S$ as 
$$\phi^!\mbb{E}=C/C^\perp,$$ where 
$$C=\{(x;v,\mu)\in \mbb{E}\times \mbb{T} S\mid (d\phi)(v)=\mbf{a}(x)\}$$
is an involutive coisotropic subbundle of $\mbb{E}\times\mbb{T}S$ over $\gr(\phi)\subset M\times S$.
(The shriek notation is used to
distinguish $\phi^!A$ from the pull-back as a vector bundle.)
One has $\on{rk}(\phi^!\mbb{E})=\on{rk}(\mbb{E})-2(\dim M-\dim S)$. 
\end{definition}

  There is a canonical Courant morphism
\begin{equation}\label{eq:pphi} P_\phi\colon \phi^!\mbb{E}\dasharrow \mbb{E}\end{equation}
lifting $\phi\colon S\to M$. Explicitly, 
\begin{equation}\label{eq:ysimIx} y\sim_{P_\phi} x\ \Leftrightarrow\ \exists v\in TS\colon
\mbf{a}(x)=d\phi(v),\ (x;v,0)\in C \mbox{ maps to }y.\end{equation}
Note that as a space, $P_\phi$ is the fibred product of $\mbb{E}$ and $TS$
over $TM$. It is a smooth vector bundle over $S\cong \gr(\phi)$
since $\mbf{a}$ is transverse
to $d\phi$ by assumption.

%

For exact Courant algebroids, the pull-back can be computed quite simply using the following proposition.

\begin{proposition}[\!\!{\cite[Proposition~2.9]{LiBland:2009ul}}]
For any smooth map $\phi\colon S\to M$, one has a canonical
isomorphism 
\[ \phi^!(\mbb{T}_\eta M)=\mbb{T}_{\phi^*\eta} S.\]
\end{proposition}

\section{Multiplicative Courant algebroids, and multiplicative Manin pairs}

%

For any groupoid $G\rightrightarrows G^{(0)}$ we let $s,t:G\to G^{(0)}$ and $\mbf{1}:G^{(0)}\to G$ denote the source, target and unit maps. 
Let
$$\on{gr}(\on{Mult}_G)=\{(g_1\circ g_2,g_1,g_2)\mid s(g_1)=t(g_2)\}\subseteq G^3$$ denote the graph of multiplication, which we regard as a relation $$\on{gr}(\on{Mult}_G):G\times G\dasharrow G.$$

\begin{definition}\label{def:multDef}Let $G\rightrightarrows G^{(0)}$ be a Lie groupoid.
\begin{itemize}
\item A $\mc{VB}$-groupoid over $G$ is a Lie groupoid $V\rightrightarrows V^{(0)}$ such that $V\to G$ is a vector bundle and $\on{gr}(\on{Mult}_V):V\times V\to V$ is a $\mc{VB}$-relation over $\on{gr}(\on{Mult}_G)$.
\item An $\mc{LA}$-groupoid over $G$ is a Lie groupoid $A\rightrightarrows A^{(0)}$ such that $A\to G$ is a Lie algebroid and $\on{gr}(\on{Mult}_A):A\times A\to A$ is an $\mc{LA}$-relation over $\on{gr}(\on{Mult}_G)$.
\item A $\mc{CA}$-groupoid (or multiplicative Courant algebroid) over $G$ is a Lie groupoid $\mbb{G}\rightrightarrows \mbb{G}^{(0)}$ such that $\mbb{G}\to G$ is a Courant algebroid and $\on{gr}(\on{Mult}_\mbb{G}): \mbb{G}\times\mbb{G}\dasharrow\mbb{G}$ is a Courant relation over $\on{gr}(\on{Mult}_G)$.
\item A Manin pair $(\mbb{G},E)$ is called \emph{multiplicative} if $\mbb{G}$ is a $\mc{CA}$-groupoid, and $E\subset\mbb{G}$ is a $\mc{VB}$-subgroupoid. In this case, $E$ is called a multiplicative Dirac structure.
\item
A Courant morphism (or relation) $R:\mbb{G}_1\dasharrow\mbb{G}_2$ between $\mc{CA}$-groupoids is called \emph{multiplicative} if $R\subset \mbb{G}_2\times\overline{\mbb{G}_1}$ is a $\mc{VB}$-subgroupoid.
\end{itemize}
\end{definition}

\begin{remark}
$\mc{VB}$ and $\mc{LA}$-groupoids were first defined by Pradines \cite{Pradines:1988td} and Mackenzie \cite{Mackenzie:2003wj,Mackenzie:2008tz}, respectively. The definitions above \cite{LiBland:2011vqa} are equivalent, but somewhat shorter.

The concepts of  multiplicative Courant algebroids and multiplicative Dirac structures were suggested in terms of supergeometry by Mehta \cite[Example~3.8]{Mehta:2009js} and Ortiz \cite[\S~7.2]{Ortiz:2009ux}, respectively. More precisely, they suggest that a $\mc{CA}$-groupoid should be a degree 2 symplectic $NQ$-groupoid, and a multiplicative Dirac structure should be a Lagrangian $NQ$-subgroupoid.
The above definitions are restatements of this without the use of supergeometry.
\end{remark}

\begin{example}[Standard $\mc{CA}$-groupoid]\label{ex:StdCAGr}
If $G\rightrightarrows G^{(0)}$ is a Lie groupoid, then applying the tangent functor to the structure maps endows $TG\rightrightarrows TG^{(0)}$ with the structure of an $\mc{LA}$-groupoid, with 
$$\on{gr}(\on{Mult}_{TG})=T\on{gr}(\on{Mult}_G):TG\times TG\dasharrow TG.$$
 As observed by Weinstein \cite{Weinstein:1987ua}, $T^*G\rightrightarrows \ann(TG^{(0)})$ also inherits the structure of a Lie groupoid over the conormal bundle of $G^{(0)}\subset G$, with 
 $$\on{gr}(\on{Mult}_{T^*G})=\ann^\natural(\on{gr}(\on{Mult}_{TG})):T^*G\times T^*G\dasharrow T^*G.$$

 Therefore $\mbb{T} G$ is naturally a Lie groupoid. The graph of multiplication is given explicitly as $$\on{gr}(\on{Mult}_{\mbb{T}G}):=T\on{gr}(\on{Mult}_G)\oplus\ann^\natural(T\on{gr}(\on{Mult}_G)).$$ By \cref{ex:StdDiracStr} it defines a Courant relation $$\on{gr}(\on{Mult}_{\mbb{T}G}):\mbb{T}G\times\mbb{T}G\dasharrow\mbb{T}G$$  over $\on{gr}(\on{Mult}_G)$, so $\mbb{T}G$ is a $\mc{CA}$-groupoid. This fact was first explicitly observed by Mehta \cite[Example~3.8]{Mehta:2009js}, though it was used implicitly by Bursztyn, Crainic, Weinstein and Zhu in \cite{Bursztyn03-1}.

Suppose $\pi$ is the bivector field of a Poisson structure on $G$. As explained by Ortiz \cite{Ortiz:2009ux}, following Mackenzie and Xu \cite{Mackenzie97}, $\on{gr}(\pi^\sharp)\subset \mbb{T}G$ is a multiplicative Dirac structure if and only if $\pi$ is a multiplicative bivector field on $G$, i.e. $\pi^\sharp:T^*G\to TG$ is a morphism of Lie groupoids. Similarly, if $\omega\in\Omega^2(G)$ is a closed 2-form on $G$, then $\on{gr}(\omega^\flat)\subset \mbb{T}G$ is a multiplicative Dirac structure if and only if $\omega$ is a multiplicative 2-form, i.e. $\omega^\flat:TG\to T^*G$ is a morphism of Lie groupoids.
\end{example}

\begin{example}[Cartan Dirac Structure]\label{ex:CartanDirac}
Let $\mf{d}$ be a quadratic Lie algebra, and $D$ any Lie group with Lie algebra $\mf{d}$ which preserves the quadratic form. Then $\mf{d}\oplus\bar{\mf{d}}$ acts on $D$ by $$\rho:(\xi_1,\xi_2)\to-\xi_1^R+\xi_2^L,\quad\xi_i\in\mf{d}$$
where $\xi^L,\xi^R$ are the left-,right- invariant vector fields on $D$ which are equal to $\xi\in\mf{d}$ at the unit.
At the unit, the stabilizer for this action is the diagonal subalgebra $\mf{d}_\Delta\subset\mf{d}\oplus\bar{\mf{d}}$, which is Lagrangian. Since the stabilizer at any other point in $D$ is conjugate to $\mf{d}_\Delta$, it follows that this action has coisotropic stabilizers.
Hence we may form the `action Courant algebroid' $$(\mf{d}\oplus\bar{\mf{d}})\times D,$$ as in \cref{ex:ActCourAlg}.

The Courant algebroid $(\mf{d}\oplus\bar{\mf{d}})\times D$ is actually a $\mc{CA}$-groupoid. As a groupoid, it is the cross product of the Lie group $D$ and the pair groupoid $\mf{d}\oplus\bar{\mf{d}}$. See \cite{Alekseev:2009tg,LiBland:2009ul,LiBland:2010wi} for details.

The diagonal $\mf{d}_\Delta\subset\mf{d}\oplus\bar{\mf{d}}$ is a Lagrangian subalgebra, so $\mf{d}_\Delta\times D$ is a Dirac structure. Since it is also a subgroupoid, it defines a multiplicative Dirac structure. This Dirac structure, known as the Cartan-Dirac structure,  was introduced independently by Alekseev, \v{S}evera and Strobl \cite{Alekseev:2009tg,Severa:2001,Kotov:2005fe}.
\end{example}

\section{Double structures}
In this section, we recall the concepts of a double vector bundle and an $\mc{LA}$-vector bundle, and we define a $\mc{CA}$-vector bundle. Each of these definitions is quite similar, and we give them here to highlight the relationships. In subsequent sections, we shall describe them in more detail.

\begin{definition}\label{def:DLACAVB}
Suppose that $D\to A$ and $B\to M$ are vector bundles, and
\begin{equation}\label{eq:DVB}\begin{tikzpicture}
\mmat{m}{
D&B\\
A&M\\
};
\path[->] (m-1-1) edge  (m-1-2);
\path[->] (m-1-1) edge  (m-2-1);
\path[->] (m-1-2) edge  (m-2-2);
\path[->] (m-2-1) edge  (m-2-2);
\end{tikzpicture}\end{equation}
is a morphism of vector bundles. We let $\gr(+_{D/A})\subseteq D\times D\times D$ denote the graph of addition for $D\to A$, as in \cref{ex:GrAdd}.
\begin{itemize}
\item $D$ is called a \emph{double vector bundle} if $D$ is a vector bundle over $B$ and $\gr(+_{D/A})$ is a vector subbundle of $D^3\to B^3$.
\item $D$ is called an \emph{$\mc{LA}$-vector bundle} if $D$ is a Lie algebroid over $B$ and $\gr(+_{D/A})\subseteq D^3$ is Lie subalgebroid.
\item $D$ is called an \emph{$\mc{CA}$-vector bundle} if $D$ is a Courant algebroid over $B$ and $\gr(+_{D/A})\subseteq D\times\overline{D\times D}$ is a Dirac structure with support.
\end{itemize}

\end{definition}
\begin{remark}
Notice the similarity between \cref{def:multDef,def:DLACAVB}. Indeed double vector bundles, $\mc{LA}$-vector bundles and  $\mc{CA}$-vector bundles are special cases of $\mc{VB}$-groupoids,  $\mc{LA}$-groupoids, and  $\mc{CA}$-groupoids.

\end{remark}

During the rest of this chapter, we shall recall some of the basic theory of double vector bundles and $\mc{LA}$-vector bundles. While none of the results in this chapter are new, the use of relations in some of the exposition and proofs is novel. 

\section{Double vector bundles}
Double vector bundles will be a frequent concept in this thesis, so we begin by describing some basic examples, and recalling their definition and basic properties. In essence, double vector bundles are vector bundles in the category of vector bundles. They were first introduced by Pradines \cite{Pradines:1974tc} and further studied in \cite{Mackenzie:2005tc,Konieczna:1999vh,Grabowski:2009dc}.

\begin{example}\label{ex:DirectSumDVB}
Suppose that $A$, $B$ and $C$ are all vector bundles over the manifold $M$. Then $A\times_M B\times _M C$ is the total space of a vector bundle over $A$ and of a vector bundle over $B$ with the respective additions given by
$$(a,b_1,c_1)+_{D/A}(a,b_2,c_2)=(a,b_1+b_2,c_1+c_2),$$ and
$$(a_1,b,c_1)+_{D/B}(a_2,b,c_2)=(a_1+a_2,b,c_1+c_2),$$ where $a,a_1,a_2\in A$, $b,b_1,b_2\in B$ and $c,c_1,c_2\in C$ all lie over the same point in $M$. With these vector bundle structures, 
$$\begin{tikzpicture}
\mmat{m}{
A\times_M B\times _M C&B\\
A&M\\
};
\path[->] (m-1-1) edge  (m-1-2);
\path[->] (m-1-1) edge  (m-2-1);
\path[->] (m-1-2) edge  (m-2-2);
\path[->] (m-2-1) edge  (m-2-2);
\end{tikzpicture}$$
is a double vector bundle.
\end{example}

\begin{example}[Tangent bundle of a vector bundle]\label{ex:TngDVB}
The following example, which will be of central importance to this thesis, is due to Pradines \cite{Pradines:1974tc}.
Suppose that $B\to M$ is a vector bundle, then 
\begin{equation}\label{eq:TngDVB}\begin{tikzpicture}
\mmat{m}{
TB&B\\
TM&M\\
};
\path[->] (m-1-1)	edge (m-1-2)
				edge (m-2-1);
\path[<-] (m-2-2)	edge  (m-1-2)
				edge (m-2-1);
\end{tikzpicture}\end{equation}
is a double vector bundle. The graph of the addition in the vector bundle $TB\to TM$ is obtained by applying the tangent functor $$\gr(+_{TB/TM})=T\gr(+{B/M})$$ to the graph of the addition in the vector bundle $B\to M$. We will revisit this example in more detail in \cref{ex:TngDVB2,ex:TngLAVB}  below.
\end{example}

It would appear from \cref{def:DLACAVB} that there is an asymmetry between the two vector bundle structures on $D$. This is not the case, as the following proposition shows.

\begin{proposition}\label{prop:DVBSym}
Let $D$ be a manifold which is canonically identified with the total space of two vector bundles over manifolds $A\subseteq D$ and $B\subseteq D$, respectively. We let 
\begin{align*}\gr(+_{D/A}):D\times D&\dasharrow D,\\
\gr(+_{D/B}):D\times D&\dasharrow D
\end{align*}
denote the graph of the two additions. Let $s_{(1324)}:D^4\to D^4$ denote the permutation $$s_{(1324)}(d_1,d_2,d_3,d_4)=(d_1,d_3,d_2,d_4),$$  and $M=A\cap B$. 
Then the following are equivalent. 
\begin{enumerate}
\item[E1)] The following diagram of relations commutes
\begin{equation}\label{eq:DVBComDiagRel}
\begin{tikzpicture}
\mmat[2em]{m}{D^4& &&D^2&&&\\
&&&&&&D\\
D^4&&&D^2&&&\\};
\path [dashed,<->] (m-1-1) edge node [swap] {$\gr(s_{(1324)})$} (m-3-1);
\path [dashed,->] (m-1-1) edge node {$\gr(+_{D/B})^2$} (m-1-4);
\path [dashed,->] (m-3-1) edge node [swap]{$\gr(+_{D/A})^2$} (m-3-4);
\path [dashed,<-] (m-2-7) edge node  {$\gr(+_{D/B})$} (m-3-4)
			edge node [swap] {$\gr(+_{D/A})$} (m-1-4);
\end{tikzpicture}
\end{equation}
\item[E2)] The following diagram is a double vector bundle
\begin{equation}\label{eq:DVB1}\begin{tikzpicture}
\mmat{m}{
D&B\\
A&M\\
};
\path[->] (m-1-1) edge node {$q_{D/B}$} (m-1-2);
\path[->] (m-1-1) edge node[swap] {$q_{D/A}$} (m-2-1);
\path[->] (m-1-2) edge node {$q_{B/M}$} (m-2-2);
\path[->] (m-2-1) edge node[swap] {$q_{A/M}$} (m-2-2);
\end{tikzpicture}\end{equation}
where $q_{B/M}:=(q_{D/A})\rvert_B$ and $q_{A/M}:=(q_{D/B})\rvert_A$.
\item[E3)] The following diagram (the `diagonal flip' of \labelcref{eq:DVB1}) is a double vector bundle
\begin{equation}\label{eq:DVBflip1}\begin{tikzpicture}
\mmat{m}{
D&A\\
B&M\\
};
\path[->] (m-1-1) edge node {$q_{D/A}$} (m-1-2);
\path[->] (m-1-1) edge node[swap] {$q_{D/B}$} (m-2-1);
\path[->] (m-1-2) edge node {$q_{A/M}$} (m-2-2);
\path[->] (m-2-1) edge node[swap] {$q_{B/M}$} (m-2-2);
\end{tikzpicture}\end{equation}
where $q_{B/M}:=(q_{D/A})\rvert_B$ and $q_{A/M}:=(q_{D/B})\rvert_A$.
\end{enumerate}
\end{proposition}
Although the proof of this proposition is not complicated, we have placed it in \cref{app:dvbSymProof} for the sake of brevity.

 As an immediate consequence of \cref{prop:DVBSym}, if \labelcref{eq:DVB} is a double vector bundle, both $A\to M$ and $B\to M$ are vector bundles, which we call the \emph{horizontal} and \emph{vertical side bundles}, respectively.

Let $\cdot_{D/A}$ and $\cdot_{D/B}$ denote the scalar multiplication operations on $D$, viewed as a vector bundle over $A$ and $B$, respectively.
Moreover, we let $\tilde 0_{D/M}:M\to D$ be the composition of $0_{D/B}:B\to D$ with the zero section $M\to B$.

We let $D^{flip}$
denote the reflection of \labelcref{eq:DVB} across the diagonal, 
$$\begin{tikzpicture}
\mmat{m}{
D&A\\
B&M\\
};
\path[->] (m-1-1) edge node {$q_{D/A}$} (m-1-2);
\path[->] (m-1-1) edge node[swap] {$q_{D/B}$} (m-2-1);
\path[->] (m-1-2) edge node {$q_{A/M}$} (m-2-2);
\path[->] (m-2-1) edge node[swap] {$q_{B/M}$} (m-2-2);
\end{tikzpicture}$$
which, as a consequence of \cref{prop:DVBSym}, is also a double vector bundle.

As an immediate consequence of \cref{prop:DVBSym} and \cite[Proposition~2.1]{gracia2010lie}, we get the following corollary. 
\begin{corollary}\label{cor:AltDVBdef}
 $D$ is a double vector bundle if and only if \labelcref{eq:DVB} is a commutative diagram where all four sides are vector bundles and the following equations hold
\begin{align}\label{eq:DVBinterchg}(d_1+_{D/B} d_2)+_{D/A} (d_3+_{D/B} d_4)&=(d_1+_{D/A} d_3)+_{D/B}(d_2+_{D/A} d_4),\\
 t\cdot_{D/A}(d_1+_{D/B} d_2)&=t\cdot_{D/A} d_1+_{D/B} t\cdot_{D/A} d_2,\\
 t\cdot_{D/B}(d_1+_{D/A} d_3)&=t\cdot_{D/B} d_1 +_{D/A} t\cdot_{D/B} d_3,\end{align}
for any $d_1,d_2,d_3,d_4\in D$  satisfying $(d_1,d_2)\in D\times_B D$, $(d_3,d_4)\in D\times_B D$, and $(d_1,d_3)\in D\times_A D$, $(d_2,d_4)\in D\times_A D$.
\end{corollary}

\begin{remark}
The alternative definition of a double vector bundle given by \cref{cor:AltDVBdef} is the standard one found in \cite{Pradines:1974tc,Mackenzie:2005tc,Mackenzie05}.
Also, note that \cref{eq:DVBinterchg} is equivalent to the commutativity of \labelcref{eq:DVBComDiagRel}.
\end{remark}

\begin{remark}[Euler vector fields]\label{rem:Euler}
One may also define a double vector bundle in terms of Euler vector fields, as is done by Grabowski and Rotkiewicz \cite{Grabowski:2009dc}. Suppose $D$ is a manifold which is the total space of two different vector bundles $D\to A$ and $D\to B$. Let $\epsilon_{D/A}$ and $\epsilon_{D/B}$ the respective Euler vector fields. Then $D$ is a double vector bundle if and only if $[\epsilon_{D/A},\epsilon_{D/B}]=0$.

One major advantage of this definition is that it extends easily to $n$-vector bundles. An $n$-vector bundle $D$ is manifold with $n$ vector bundle structures $D\to M_i$ (for $i=1,\dots n$) such that the respective Euler vector fields commute pairwise.

\end{remark}

\subsection{The core}
Consider the submanifold $C:=\on{ker}(q_{D/A})\cap\on{ker}(q_{D/B})$.
If in \cref{eq:DVBinterchg}, we let $d_1,d_4\in C$ and $d_2,d_3=\tilde 0_{D/M}$, then we get 
\begin{align*}
d_1+_{D/A} d_4 =& (d_1+_{D/B} \tilde 0_{D/M})+_{D/A}(\tilde 0_{D/M}+_{D/B} d_4)\\
=&(d_1+_{D/A} \tilde 0_{D/M})+_{D/B} (\tilde 0_{D/M}+_{D/A} d_4)\\
=&d_1+_{D/B} d_4.
\end{align*} 
So both $+_{D/A}$ and $+_{D/B}$ define the same additive structure on $C$. Similarly both $\cdot_{D/A}$ and $\cdot_{D/B}$ restrict to the same scalar multiplication on $C$. Therefore, with either choice of addition and scalar multiplication, $C$ is a vector bundle over $M$, called the \emph{core} of $D$. 

We may occasionally display the core explicitly in a double vector bundle diagram, as follows:
$$\begin{tikzpicture}
\mmat{m}{
D&B\\
A&M\\
};
\path[->] (m-1-1) edge  (m-1-2);
\path[->] (m-1-1) edge  (m-2-1);
\path[->] (m-1-2) edge  (m-2-2);
\path[->] (m-2-1) edge  (m-2-2);
\draw (0,0) node (c) {$C$} ;
\path[left hook->] (c) edge (m-1-1);
\end{tikzpicture}$$

\begin{example}
Suppose $A,B$ and $C$ are vector bundles over $M$. The core in the double vector bundle
$$\begin{tikzpicture}
\mmat{m}{
A\times_M B\times _M C&B\\
A&M\\
};
\path[->] (m-1-1) edge  (m-1-2);
\path[->] (m-1-1) edge  (m-2-1);
\path[->] (m-1-2) edge  (m-2-2);
\path[->] (m-2-1) edge  (m-2-2);
\end{tikzpicture}$$
 (c.f \cref{ex:DirectSumDVB}) is the subbundle $C\subseteq A\times_M B\times _M C$.
\end{example}

Note that the abelian groupoid $C$ acts on $D$, via the formula
\begin{equation}\label{eq:coreaction}d+c:=(d+_{D/A}0_{D/A})+_{D/B}(c+_{D/A}0_{D/B})=(d+_{D/B}0_{D/B})+_{D/A}(c+_{D/B}0_{D/A}).\end{equation}
 for any $(d,c)\in D\times_M C$.

We let $i:C\to D$ denote the inclusion and 
$i_A:q_{A/M}^*C\to D$
denote  the composition 
$$\begin{tikzpicture}
\mmat[5em]{m}{A\times_M C&D\times_B D&D.\\};
\path[->] (m-1-1) edge node {$ 0_{D/A}\times i$} (m-1-2);
\path[->] (m-1-2) edge node {$+_{D/B}$} (m-1-3);
\end{tikzpicture}$$
The exact sequence of vector bundles 
\begin{equation}\label{eq:iA}\begin{tikzpicture}
\mmat{m}{q_{A/M}^* C & D& q_{A/M}^* B\\ A& A& A\\};
\path[->] (m-1-1) 	edge node {$i_A$} (m-1-2)
				edge (m-2-1);
\path[->] (m-1-2)	edge (m-1-3)
				edge (m-2-2);
\path[->] (m-1-3)	edge (m-2-3);
\path[->] (m-2-1)	edge node[swap] {$\on{id}$} (m-2-2);
\path[->] (m-2-2)	edge node[swap] {$\on{id}$} (m-2-3);
\end{tikzpicture}\end{equation} is called the \emph{core sequence}.

Among the sections $\Gamma(D,A)$, there are two subspaces of sections: the \emph{core} sections $\Gamma_C(D,A)$, and the \emph{linear} sections $\Gamma_l(D,A)$. 
For any section $\sigma\in\Gamma(C)$, we let  the section $\sigma_{c_A}:A\to D$ be given by
$$\begin{tikzpicture}
\mmat{m}{A&&q_{A/M}^*C&D\\};
\path[->] (m-1-1) edge node {$q_{A/M}^*\sigma$} (m-1-3);
\path[->] (m-1-3) edge node {$i_A$} (m-1-4);
\end{tikzpicture}$$
The map $\sigma\to\sigma_{c_A}$ embeds $\Gamma(C)$ into $\Gamma(D,A)$, as the space $\Gamma_C(D,A)$ of core sections.

Meanwhile a section $\sigma\in\Gamma(D,A)$ is called linear, if there is a section $\sigma_0\in\Gamma(B)$ such that the horizontal map
$$\begin{tikzpicture}
\mmat{m}{
A&D\\
M&B\\
};
\path[->] (m-1-1) edge node {$\sigma$} (m-1-2);
\path[->] (m-1-1) edge (m-2-1);
\path[->] (m-1-2) edge (m-2-2);
\path[->] (m-2-1) edge node {$\sigma_0$} (m-2-2);
\end{tikzpicture}$$
is a morphism of vector bundles (in this case, $\sigma_0$ is unique). The subspace of linear sections is denoted by $\Gamma_l(D,A)\subset\Gamma(D,A)$, and the map $\sigma\to\sigma_0$ defines a linear surjection $\Gamma_l(D,A)\to\Gamma(B)$.

It is clear that there are analogous notions with the roles of $A$ and $B$ replaced.

\begin{remark}
Let $\epsilon_{D/A}$ and $\epsilon_{D/B}$ denote the Euler vector fields on the vector bundles $D\to A$ and $D\to B$, as in \cref{rem:Euler}.

Any section $\sigma\in\Gamma(D,A)$ defines a vertical vector field $\iota_\sigma$ on $D$ by fibrewise translation,\footnote{We use the notation $\iota_\sigma$ since it acts by contraction with $\sigma$ for fibrewise linear functions.} i.e, $$(\iota_\sigma f)(d)=\partial_t f(d +_{D/A} t\cdot_{D/A} \sigma)\rvert_{t=0},$$  such that $[\epsilon_{D/A},\iota_\sigma]=-\iota_\sigma$. Indeed there is a one to one correspondence between vector fields $X\in\mf{X}(D)$ such that $[\epsilon_{D/A},X]=-X$ and sections $\sigma\in\Gamma(D,A)$ given by the condition $X=\iota_\sigma$. The core sections $\Gamma_C(D,A)$ are precisely the sections $\sigma\in\Gamma(D,A)$ such that $[\epsilon_{D/B},\iota_\sigma]=-\iota_\sigma$, while the linear sections $\Gamma_l(D,A)$ are precisely those sections $\sigma\in\Gamma(D,A)$ such that $[\epsilon_{D/B},\iota_\sigma]=0$.
\end{remark}

\begin{example}[Tangent bundle of a vector bundle (cont.)]\label{ex:TngDVB2}
Suppose that $B\to M$ is a vector bundle, then as explained in \cref{ex:TngDVB},
\begin{equation}\label{eq:TngDVB}\begin{tikzpicture}
\mmat{m}{
TB&B\\
TM&M\\
};
\path[->] (m-1-1)	edge (m-1-2)
				edge (m-2-1);
\path[<-] (m-2-2)	edge node[swap] {$q_{B/M}$} (m-1-2)
				edge node {$q_{TM/M}$}(m-2-1);
\end{tikzpicture}\end{equation}
is a double vector bundle. For a point $x\in M$, the core fibre over $x$ consists of vectors based at $x$ which are tangent to the fibre, $q_{B/M}^{-1}(x)$, of $B$ over $x$ (see \cref{fig:TngCore}). So the core is canonically isomorphic to $B$:
$$\begin{tikzpicture}
\mmat{m}{
TB&B\\
TM&M\\
};
\path[->] (m-1-1)	edge (m-1-2)
				edge (m-2-1);
\path[<-] (m-2-2)	edge (m-1-2)
				edge (m-2-1);
\draw (0,0) node (c) {$B$};
\path[left hook->] (c) edge (m-1-1);
\end{tikzpicture}$$
If $\sigma\in\Gamma(B)$ is a section of $B$, then we have $\sigma_{c_{B}}=\iota_\sigma$, so we may denote $\sigma_{c_{TM}}$ by $\sigma_C$ to simplify notation. 

\begin{figure}
\begin{tabular}{cc}
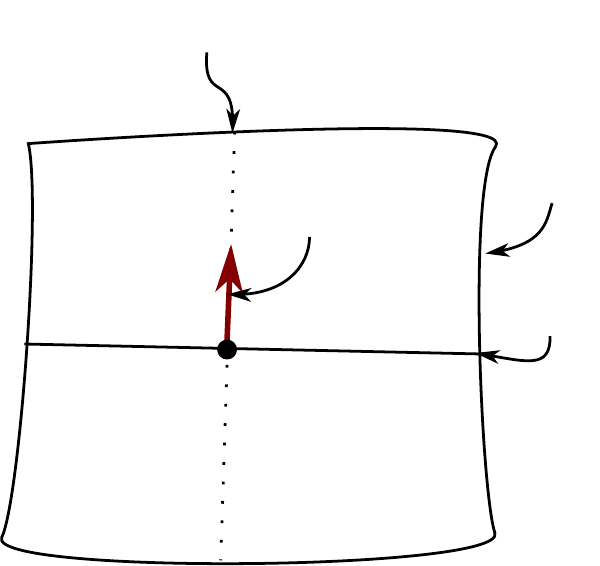 & 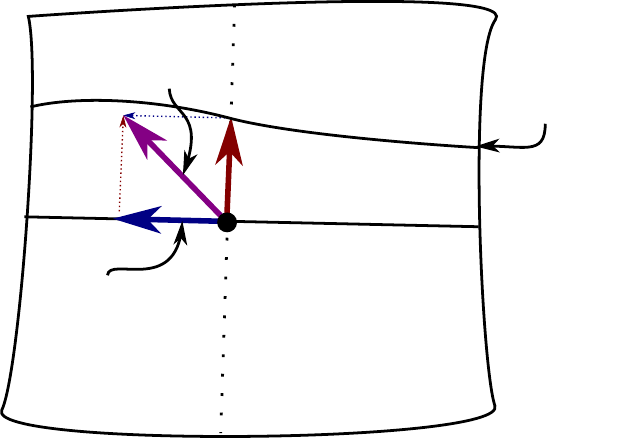
\end{tabular}
\caption{\label{fig:TngCore}The vector $\xi\in T_xB$ is tangent to the fibre $q_{B/M}^{-1}(x)$, and therefore defines a core element of $TB$ at $x\in M$. If $\sigma\in\Gamma(B)$ is such that $\xi=\sigma(x)$, then for any $X\in TM$, $\sigma_C(X)=X+_{TB/B}\xi$, (here the addition takes place in the vector space $T_xB$).}
\end{figure}

Meanwhile,  the differential $d\sigma:TM\to TB$ is a linear section of $TB$ over $TM$ canonically associated to the section $\sigma:M\to B$, which we call the tangent lift $\sigma_T$ of $\sigma$ (see \cref{fig:TngLiftSec}). Note that, while the tangent lift of any section is a linear section, the converse does not hold.

\begin{figure}
\begin{center}
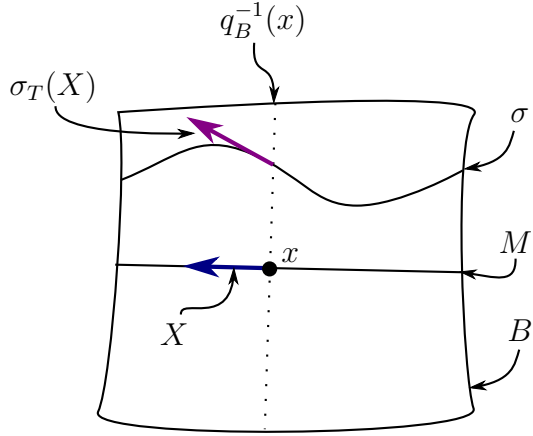
\end{center}
\caption{\label{fig:TngLiftSec}For any section $\sigma\in\Gamma(B)$, the tangent lift $\sigma_T\in\Gamma(TB,TM)$ of $\sigma$ takes $X\in TM$ to $d\sigma(X)\in TB$.}
\end{figure}

For $f\in C^\infty(M)$, we introduce the notation $f_C:=q_{TM/M}^*f\in C^\infty(TM)$ and $f_T:=df\in C^\infty(TM)$.\footnote{Here, we understand the 1-form, $df\in\Omega^1(M)$, as defining a function on $TM$.} Then we have the following rules 
\begin{subequations}\label{eq:TCDer}
\begin{align}
(f\cdot\sigma)_C&=f_C\cdot\sigma_C,\\
(f\cdot\sigma)_T&=f_T\cdot\sigma_C+f_C\cdot\sigma_T.
\end{align}\end{subequations}

Finally, we point out that if $\{\sigma^i\}\subset\Gamma(B)$ is a local basis of sections of $B$, then $\{\sigma^i_C,\sigma^i_T\}\subset\Gamma(TB,TM)$ is a local basis of sections of $TB$.

\end{example}

\subsection{Duals of double vector bundles}
A double vector bundle
\begin{equation}\label{eq:DVBdualpic}\begin{tikzpicture}
\mmat{m}{
D&B\\
A&M\\
};
\path[->] (m-1-1) edge  (m-1-2);
\path[->] (m-1-1) edge  (m-2-1);
\path[->] (m-1-2) edge  (m-2-2);
\path[->] (m-2-1) edge  (m-2-2);
\draw (0,0) node (c) {$C$} ;
\path[left hook->] (c) edge (m-1-1);
\end{tikzpicture}\end{equation}
is the total space for two ordinary vector bundles: 
$$\begin{tikzpicture}
\draw (-3,0) node (d) {$D$};
\draw (-2,0) node (b) {$B$};
\draw (2,.5) node (d') {$D$};
\draw (2,-.5) node (a) {$A$};
\draw (0,0) node {and};
\path[->] (d) edge (b);
\path[->] (d') edge (a);
\end{tikzpicture}$$
 We denote the duals of these two ordinary vector bundles by
$$\begin{tikzpicture}
\draw (-3,0) node (d) {$D^{*_{\!x}}$};
\draw (-2,0) node (b) {$B$};
\draw (2,.5) node (d') {$D^{*_{\!y}}$};
\draw (2,-.5) node (a) {$A$};
\draw (0,0) node {and};
\path[->] (d) edge (b);
\path[->] (d') edge (a);
\end{tikzpicture}$$
The notation $*_{\!x}$ and $*_{\!y}$, which we have adapted from \cite{GraciaSaz:2009ck},\footnote{Gracia-Saz and Mackenzie denote $*_{\!x}$ and $*_{\!y}$ by $X$ and $Y$, respectively.} identifies the arrow along which we dualize with the appropriate axis:
$$\begin{tikzpicture}
\path[->] (0,0) edge node {$x$} (1,0)
		edge node[swap] {$y$} (0,-1);
\end{tikzpicture}$$
We call $D^{*_{\!y}}$  the \emph{vertical dual} or \emph{dual over $A$}, and $D^{*_{\!x}}$ the \emph{horizontal dual} or \emph{dual over $B$}.

Pradines \cite{Pradines:1988td}, and later Konieczna-Urba\'{n}ski \cite{Konieczna:1999vh}, Mackenzie \cite{Mackenzie:2005tc,Mackenzie:1999vk} and  Grabowski-Rotkiewicz\cite{Grabowski:2009dc}, studied the duals $D^{*_{\!x}}$ and $D^{*_{\!y}}$, proving that they form the total space for double vector bundles themselves. We now explain this fact in terms of relations.

Note that the following relations are all $\mc{VB}$-relations, where $D$ is viewed as a vector bundle over $A$.
\begin{itemize}
\item $\gr(+_{D/B}):D\times D\dasharrow D$ (addition)
\item $B:\ast\dasharrow D$ (the unit, $\ast\sim_{B}0_{D/B}(b)$ for any $b\in B$)
\item $\gr(t\cdot_{D/B}):D\dasharrow D$ (scalar multiplication by $t\in\mbb{R}$)
\item $D_\Delta:D\dasharrow D$ (the identity map, $d\sim_{\Delta_D}d$)
\item $\gr(\on{Swap}):D\times D\dasharrow D\times D$ ($(d_1,d_2)\sim_{\gr(\on{Swap})}(d_2,d_1)$)
\end{itemize}
The defining axioms for $D$ to be a vector bundle over $B$ can be written entirely in terms of these relations. That is,
\begin{equation*}
\begin{split}
\gr(+_{D/B})\circ(\gr(+_{D/B})\times D_\Delta)&=\gr(+_{D/B})\circ(D_\Delta\times\gr(+_{D/B}))\text{ (associativity of addition)}\\
\gr(+_{D/B})\circ \gr(\on{Swap})&=\gr(+_{D/B}) \text{ (commutativity of addition)}\\
\gr(+_{D/B})\circ(D_\Delta\times B)&=D_\Delta \text{ (identity element of addition)}\\
\gr(+_{D/B})\circ(\gr(t\cdot_{D/B})\times \gr(t\cdot_{D/B}))&=\gr(t\cdot_{D/B})\circ\gr(+_{D/B}) \text{ (distributivity)}\\
etc.&
\end{split}
\end{equation*}

Since $D\to A$ decomposes as $C\oplus B$ along $M\subset A$ (embedded as the zero section), we have $C^*\cong \ann(B)\subset D^{*_{\!y}}$.
 We define \begin{align}\gr(+_{D^{*_{\!y}}/C^*})&:=\ann^\natural(\gr(+_{D/B})):D^{*_{\!y}}\times D^{*_{\!y}}\dasharrow D^{*_{\!y}}\\
\gr(t\cdot_{D^{*_{\!y}}/C^*})&:=\ann^\natural(\gr(t\cdot_{D/B})):D^{*_{\!y}}\dasharrow D^{*_{\!y}}.\end{align} An application of \cref{lem:ann} then shows that $D^{*_{\!y}}$ also satisfies the axioms to be a vector bundle over $C^*$, that is 
\begin{equation*}
\begin{split}
\gr(+_{D^{*_{\!y}}/C^*})\circ(\gr(+_{D^{*_{\!y}}/C^*})\times (D^{*_{\!y}})_\Delta)&=\gr(+_{D^{*_{\!y}}/C^*})\circ((D^{*_{\!y}})_\Delta\times\gr(+_{D^{*_{\!y}}/C^*}))\\
\gr(+_{D^{*_{\!y}}/C^*})\circ \gr(\on{Swap})&=\gr(+_{D^{*_{\!y}}/C^*})\\
\gr(+_{D^{*_{\!y}}/C^*})\circ((D^{*_{\!y}})_\Delta\times C^*)&=(D^{*_{\!y}})_\Delta\\
\gr(+_{D^{*_{\!y}}/C^*})\circ(\gr(t\cdot_{D^{*_{\!y}}/C^*})\times \gr(t\cdot_{D^{*_{\!y}}/C^*}))&=\gr(t\cdot_{D^{*_{\!y}}/C^*})\circ\gr(+_{D^{*_{\!y}}/C^*}) \\
etc.&
\end{split}
\end{equation*}

Since, by definition, $\gr(+_{D^{*_{\!y}}/C^*})$ is a vector subbundle of $(D^{*_{\!y}})^3\to A^3$, it follows from \cref{def:DLACAVB} that
\begin{equation}\label{eq:DVBDual}\begin{tikzpicture}
\mmat{m}{D^{*_{\!y}}&C^*\\ A& M\\};
\path[->] (m-1-1)	edge (m-1-2)
				edge (m-2-1);
\path[<-] (m-2-2)	edge (m-1-2)
				edge (m-2-1);
\draw (0,0) node (b) {$B^*$};
\path[left hook->] (b) edge (m-1-1);
\end{tikzpicture}\end{equation}
is a double vector bundle (see \cite{Mackenzie:2005tc,Grabowski:2009dc} for details). The core of $D^{*_{\!y}}$ is $B^*$, and the core sequence is the dual of \labelcref{eq:iA}. 

A similar discussion applies for the dual $D^{*_{\!x}}\to B$. 

\begin{example}[Cotangent bundle of a vector bundle]\label{ex:CotDVB}
Taking the dual of the tangent double vector bundle \labelcref{eq:TngDVB} over $B$, and using the identification $(TB)^{*_{\!x}}\cong T^*B$, we get the cotangent double vector bundle
\begin{equation}\label{eq:CotDVB}\begin{tikzpicture}
\mmat{m}{
T^*B&B\\
B^*&M\\
};
\path[->] (m-1-1)	edge (m-1-2)
				edge (m-2-1);
\path[<-] (m-2-2)	edge (m-1-2)
				edge (m-2-1);
\draw (0,0) node (c) {$T^*M$};
\path[left hook->] (c) edge (m-1-1);
\end{tikzpicture}\end{equation}

\end{example}

\begin{example}\label{ex:SecDualTng}
If we take the dual of the tangent double vector bundle \labelcref{eq:TngDVB} over $TM$ instead, the result is canonically isomorphic to 
$$\begin{tikzpicture}
\mmat{m}{
TB^*&B^*\\
TM&M\\
};
\path[->] (m-1-1)	edge (m-1-2)
				edge (m-2-1);
\path[<-] (m-2-2)	edge (m-1-2)
				edge (m-2-1);
\draw (0,0) node (c) {$B^*$};
\path[left hook->] (c) edge (m-1-1);
\end{tikzpicture}$$
\end{example}

\subsection{Total kernels and quotients of double vector bundles}

Suppose that the  map
\begin{equation}\label{eq:TotQuo}
\begin{tikzpicture}
\mmat{m1} at (-2,0){D&B\\ A& M\\};
\path[->] (m1-1-1)	edge (m1-1-2)
				edge (m1-2-1);
\path[<-] (m1-2-2)	edge (m1-1-2)
				edge (m1-2-1);
\mmat{m2} at (2,0) {Q&M\\ M& M\\};
\path[->] (m2-1-1)	edge (m2-1-2)
				edge (m2-2-1);
\path[<-] (m2-2-2)	edge node[swap]{$\on{id}$} (m2-1-2)
				edge node{$\on{id}$} (m2-2-1);
\path[->] (m1) edge node {$q$} (m2);
\draw (-2,0) node (c) {$C$};
\path[left hook->] (c) edge (m1-1-1);
\draw (2,0) node (c2) {$Q$};
\path[left hook->] (c2) edge (m2-1-1);
\end{tikzpicture}
\end{equation}
is a morphism of double vector bundles
such that the restriction of $q$ to the core $C$ of $D$ is a surjection.
Since $Q$ has a common zero section over both side bundles, it makes sense to define the \emph{total kernel} $$\on{ker}(q):=\{x\in D \mid q(x)=0\},$$
a double vector subbundle of $D$. 

On the other hand, suppose we have a double vector subbundle of $D$,
\begin{equation}\label{eq:TotKer}
\begin{tikzpicture}
\mmat{m1} at (-2,0){D'&B\\ A& M\\};
\path[->] (m1-1-1)	edge (m1-1-2)
				edge (m1-2-1);
\path[<-] (m1-2-2)	edge (m1-1-2)
				edge (m1-2-1);
\mmat{m2} at (2,0) {D&B\\ A& M\\};
\path[->] (m2-1-1)	edge (m2-1-2)
				edge (m2-2-1);
\path[<-] (m2-2-2)	edge (m2-1-2)
				edge (m2-2-1);
\draw (0,0) node {$\subseteq$} (m2);
\draw (-2,0) node (c) {$C'$};
\path[left hook->] (c) edge (m1-1-1);
\draw (2,0) node (c2) {$C$};
\path[left hook->] (c2) edge (m2-1-1);
\end{tikzpicture}
\end{equation}
%
containing both side bundles $A$ and $B$. Recall that the core $C$ acts on $D$ via \cref{eq:coreaction}. 
Consequently, we have a well defined map $q:D\to C/C'$, given by the condition $$q(d)=\tilde c\Leftrightarrow d+c\in D',$$ where $c\in C$ is any lift of $\tilde c$. The map $q:D\to C/C'$ is called the \emph{total quotient} of $D$ by $D'$.

It is clear that these two constructions invert each other, yielding the following proposition.

\begin{proposition}\label{prop:totkerVtotquo}
The operations of taking the total kernel or the total quotient define a one-to-one correspondence between surjective morphisms of double vector bundles of the form \labelcref{eq:TotQuo} and inclusions of the form \labelcref{eq:TotKer}.
\end{proposition}

\section{Triple vector bundles}
In this section we briefly summarize the theory of \emph{triple vector bundles}, first studied by Mackenzie \cite{Mackenzie:2005tc}. The following definition is an abbreviated\footnote{The abbreviations are due to the work of Grabowski and Rotkiewicz \cite{Grabowski:2009dc}: Suppose $E$ is the total space for $n$ vector bundles $E\to V_1,\dots, E\to V_n$. They show that $E$ is an $n$-vector bundle if and only if the $n$ Euler vector fields commute pairwise. As a consequence, an $n$-dimensional cube whose arrows (with correct orientations) are vector bundles is an $n$-vector bundle if each face is a double vector bundle.} version of the one found in \cite{Mackenzie:2005tc}.

\begin{definition}\label{def:TVB}A triple vector bundle is a cube with edges oriented as follows,
$$\begin{tikzpicture}[
        back line/.style={densely dotted},
        cross line/.style={preaction={draw=white, -,
           line width=6pt}}]
\mmat[1em]{m}{
	&E	&	&D_1\\
D_2	&		&V_3	&\\
	&D_3	&	&V_2\\
V_1	&		&M	&\\
}; 
\path[->]
        (m-1-2) edge (m-1-4)
                edge (m-2-1)
                edge [back line] (m-3-2)
        (m-1-4) edge (m-3-4)
                edge (m-2-3)
        (m-2-1) edge [cross line] (m-2-3)
                edge (m-4-1)
        (m-3-2) edge [back line] (m-3-4)
                edge [back line] (m-4-1)
        (m-4-1) edge (m-4-3)
        (m-3-4) edge (m-4-3)
        (m-2-3) edge [cross line] (m-4-3);
\end{tikzpicture}$$
such that
\begin{itemize}
\item each edge is a vector bundle, and 
\item each face is a double vector bundle.
\end{itemize}


\end{definition}


\begin{example}[\!\!\cite{Mackenzie:2005tc}]\label{ex:TngTVB}
Suppose that
$$\begin{tikzpicture}
\mmat{m}{
D&B\\
A&M\\
};
\path[->] (m-1-1)	edge (m-1-2)
				edge (m-2-1);
\path[<-] (m-2-2)	edge (m-1-2)
				edge (m-2-1);
\end{tikzpicture}$$
is a double vector bundle. Then the tangent bundle $TD$ is naturally a triple vector bundle,
$$\begin{tikzpicture}[
        back line/.style={densely dotted},
        cross line/.style={preaction={draw=white, -,
           line width=6pt}}]
\mmat[1em]{m}{
	&TD	&	&TB\\
D	&		&B	&\\
	&TA	&	&TM\\
A	&		&M	&\\
}; 
\path[->]
        (m-1-2) edge (m-1-4)
                edge (m-2-1)
                edge [back line] (m-3-2)
        (m-1-4) edge (m-3-4)
                edge (m-2-3)
        (m-2-1) edge [cross line] (m-2-3)
                edge (m-4-1)
        (m-3-2) edge [back line] (m-3-4)
                edge [back line] (m-4-1)
        (m-4-1) edge (m-4-3)
        (m-3-4) edge (m-4-3)
        (m-2-3) edge [cross line] (m-4-3);
\end{tikzpicture}$$
called the \emph{tangent triple vector bundle} of $D$.
\end{example}

\begin{example}[\!\!\cite{Mackenzie:2005tc}]\label{ex:CotTVB}
Suppose that
$$\begin{tikzpicture}
\mmat{m}{
D&B\\
A&M\\
};
\path[->] (m-1-1)	edge (m-1-2)
				edge (m-2-1);
\path[<-] (m-2-2)	edge (m-1-2)
				edge (m-2-1);
\end{tikzpicture}$$
is a double vector bundle. 
The three double vector bundles $D, D^{*_{\!y}}$ and $D^{*_{\!x}}$ form faces of the triple vector bundle $T^*D$,
$$\begin{tikzpicture}[
        back line/.style={densely dotted},
        cross line/.style={preaction={draw=white, -,
           line width=6pt}}]
\mmat[1em]{m}{
	&T^*D	&	&D^{*_{\!x}}\\
D	&		&B	&\\
	&D^{*_{\!y}}	&	&C^*\\
A	&		&M	&\\
}; 
\path[->]
        (m-1-2) edge (m-1-4)
                edge (m-2-1)
                edge [back line] (m-3-2)
        (m-1-4) edge (m-3-4)
                edge (m-2-3)
        (m-2-1) edge [cross line] (m-2-3)
                edge (m-4-1)
        (m-3-2) edge [back line] (m-3-4)
                edge [back line] (m-4-1)
        (m-4-1) edge (m-4-3)
        (m-3-4) edge (m-4-3)
        (m-2-3) edge [cross line] (m-4-3);
\end{tikzpicture}$$
called the \emph{cotangent triple vector bundle} of $D$.
Here the left, top and back faces are cotangent double vector bundles, as in \cref{ex:CotDVB}.\footnote{However, the back face can be viewed as a cotangent double vector bundle in two different ways: as $T^*(D^{*_{\!y}})$ or $T^*(D^{*_{\!x}})$. There is a canonical isomorphism $T^*(D^{*_{\!y}})\cong T^*(D^{*_{\!x}})$, which restricts to the identity map on both $D^{*_{\!x}}$ and $D^{*_{\!y}}$, and to minus the identity on the core $T^*C^*$. }
\end{example}

\subsection{Duality}\label{sec:TVBDual}
Since a triple vector bundle 
$$\begin{tikzpicture}[
        back line/.style={densely dotted},
        cross line/.style={preaction={draw=white, -,
           line width=6pt}}]
\mmat[1em]{m} at (-3,0){
	&E	&	&D_{2,3}\\
D_{1,3}	&	&V_3	&\\
	&D_{1,2}	&	&V_2\\
V_1	&		&M	&\\
}; 
\path[->]
        (m-1-2) edge (m-1-4)
                edge (m-2-1)
                edge [back line] (m-3-2)
        (m-1-4) edge (m-3-4)
                edge (m-2-3)
        (m-2-1) edge [cross line] (m-2-3)
                edge (m-4-1)
        (m-3-2) edge [back line] (m-3-4)
                edge [back line] (m-4-1)
        (m-4-1) edge (m-4-3)
        (m-3-4) edge (m-4-3)
        (m-2-3) edge [cross line] (m-4-3);

\path [->] (3,1) edge node {$x$} (5,1)
			edge node {$y$} (3,-1)
			edge node[swap] {$z$} (2,0);
\end{tikzpicture}$$
 carries three vector bundle structures, one can dualize $E$ in three different ways: along the $x$, $y$ and $z$ axis depicted above. Each of these duals is a triple vector bundle, which we denote by $E^{*_{\!x}}$, $E^{*_{\!y}}$ and $E^{*_{\!z}}$, respectively (adapting Gracia-Saz and Mackenzie's  \cite{GraciaSaz:2009ck} notation).
 
 \begin{example}\label{ex:dualX}
 Suppose 
$$\begin{tikzpicture}
\mmat{m}{
D&B\\
A&M\\
};
\path[->] (m-1-1)	edge (m-1-2)
				edge (m-2-1);
\path[<-] (m-2-2)	edge (m-1-2)
				edge (m-2-1);
\draw (0,0) node (c) {$C$};
\path[left hook->] (c) edge (m-1-1);
\end{tikzpicture}$$ is a double vector bundle. Consider the tangent triple vector bundle 
 $$\begin{tikzpicture}[
        back line/.style={densely dotted},
        cross line/.style={preaction={draw=white, -,
           line width=6pt}}]
\mmat[1em]{m}{
	&TD	&	&TB\\
D	&	&B	&\\
	&TA	&	&TM\\
A	&	&M	&\\
}; 
\path[->]
        (m-1-2) edge (m-1-4)
                edge (m-2-1)
                edge [back line] (m-3-2)
        (m-1-4) edge (m-3-4)
                edge (m-2-3)
        (m-2-1) edge [cross line] (m-2-3)
                edge (m-4-1)
        (m-3-2) edge [back line] (m-3-4)
                edge [back line] (m-4-1)
        (m-4-1) edge (m-4-3)
        (m-3-4) edge (m-4-3)
        (m-2-3) edge [cross line] (m-4-3);
\end{tikzpicture}$$
described in \cref{ex:TngTVB}. Dualizing along the $z$-axis yields the cotangent triple vector bundle described in \cref{ex:CotTVB}, i.e. $(TD)^{*_{\!z}}= T^*D$. On the other hand, dualizing $TD$ along the $x$-axis yields the triple vector bundle 
 $$\begin{tikzpicture}[
        back line/.style={densely dotted},
        cross line/.style={preaction={draw=white, -,
           line width=6pt}}]
\mmat[1em]{m}{
	&TD^{*_{\!x}}	&	&TB\\
D^{*_{\!x}}	&		&B	&\\
	&TC^*	&	&TM\\
C^*	&	&M	&\\
}; 
\path[->]
        (m-1-2) edge (m-1-4)
                edge (m-2-1)
                edge [back line] (m-3-2)
        (m-1-4) edge (m-3-4)
                edge (m-2-3)
        (m-2-1) edge [cross line] (m-2-3)
                edge (m-4-1)
        (m-3-2) edge [back line] (m-3-4)
                edge [back line] (m-4-1)
        (m-4-1) edge (m-4-3)
        (m-3-4) edge (m-4-3)
        (m-2-3) edge [cross line] (m-4-3);
\end{tikzpicture}$$
 That is $(TD)^{*_{\!x}}= T(D^{*_{\!x}})$.
 
 As a final example, dualizing the cotangent triple vector bundle $T^*D$ along the $x$-axis yields $T^*D^{*_{\!x}}= TD^{*_{\!z}*_{\!x}}$:
 $$\begin{tikzpicture}[
        back line/.style={densely dotted},
        cross line/.style={preaction={draw=white, -,
           line width=6pt}}]
\mmat[1em]{m}{
	&TD^{*_{\!x}}	&	&D^{*_{\!x}}\\
TB&		&B	&\\
	&TC^*		&	&C^*\\
TM	&		&M	&\\
}; 
\path[->]
        (m-1-2) edge (m-1-4)
                edge (m-2-1)
                edge [back line] (m-3-2)
        (m-1-4) edge (m-3-4)
                edge (m-2-3)
        (m-2-1) edge [cross line] (m-2-3)
                edge (m-4-1)
        (m-3-2) edge [back line] (m-3-4)
                edge [back line] (m-4-1)
        (m-4-1) edge (m-4-3)
        (m-3-4) edge (m-4-3)
        (m-2-3) edge [cross line] (m-4-3);
\end{tikzpicture}$$
 \end{example}

\section{$\mc{LA}$-vector bundles and double linear Poisson vector bundles}
\subsection{$\mc{LA}$-vector bundles}
$\mc{LA}$-vector bundles are a concept due to Mackenzie \cite{Mackenzie:1998te,Mackenzie:1998ge} which encode the notion of a vector bundle in the category of Lie algebroids, or equivalently a Lie algebroid in the category of vector bundles. In \cite{gracia2010lie} it was shown that $\mc{LA}$-vector bundles are the correct context from which to study 2-term representations of Lie algebroids `up to homotopy' (note that \cite{gracia2010lie} also provides a very nice summary of $\mc{LA}$-vector bundle theory).


Suppose that 
\begin{equation}\label{eq:LAvb}\begin{tikzpicture}
\mmat{m}{D&B\\ A& M\\};
\path[->] (m-1-1)	edge (m-1-2)
				edge node[swap] {$q^D_A$} (m-2-1);
\path[<-] (m-2-2)	edge (m-1-2)
				edge (m-2-1);
\end{tikzpicture}\end{equation}
is an $\mc{LA}$-vector bundle. That is, $D\to A$ and $B\to M$ are vector bundles, \labelcref{eq:LAvb} is a morphism of vector bundles, $D\to B$ is a Lie algebroid, and $\gr{+_{D/A}}\subset D^3$ is a Lie subalgebroid. 

\begin{proposition}\label{prop:AisLA}
There is a unique Lie algebroid structure on $A\to M$ such that the inclusion $0_{D/A}:A\to D$ and the projection $q_{D/A}:D\to A$ are both morphisms of Lie algebroids.
\end{proposition}
\begin{proof}

We let $\gr(-_{D/A}):D\times D\dasharrow D$ be the relation defined by $( d_1, d_2)\sim_{\gr(-_{D/A})} d$ if and only if $ d+_{D/A} d_2= d_1$. Since $\on{gr}(-_{D/A})=\{( d; d_1, d_2)\mid d+_{D/A} d_2= d_1\}\subseteq D\times D\times D$ is obtained from $\on{gr}(+_{D/A})$ by permuting factors, it is a Lie subalgebroid of $D^3\to B^3$. Moreover, the diagonal $D_\Delta\subset D\times D$ is Lie subalgebroid of $D^2\to B^2$. Therefore, the composition $\gr(-_{D/A})\circ D_\Delta$ is a Lie subalgebroid of $D\to B$. Since $0_{D/A}:A\to D$ embeds $A$ as $\gr(-_{D/A})\circ D_\Delta$, there exists a unique Lie algebroid structure on $A$ such that $0_{D/A}:A\to D$ is a morphism of Lie algebroids.

Moreover, since diagonal embedding $\Delta_D:D\to D\times D$, given by $ d\to( d, d)$ is a morphism of Lie algebroids from $D\to B$ to $D^2\to B^2$, the composition of relations $\gr(-_{D/A})\circ \gr(\Delta_D):D\dasharrow D$ is a $\mc{LA}$-relation. However $\gr(-_{D/A})\circ \gr(\Delta_D)=\gr(0_{D/A}\circ q_{D/A})$, which shows that $q_{D/A}:D\to A$ is a morphism of Lie algebroids.
\end{proof}

\begin{remark}
\Cref{prop:AisLA} shows that the Lie algebroid structure on $A$ is completely determined by the Lie algebroid structure on $D$. Indeed, if $\sigma_D,\tau_D\in\Gamma(D,B)$ and $\sigma_A,\tau_A\in\Gamma(A)$ are such that $\sigma_A=q_{D/A}(\sigma_D)$ and $\tau_A=q_{D/A}(\tau_D)$, then $$[\sigma_A,\tau_A]=q_{D/A}[\sigma_D,\tau_D].$$
\end{remark}


\begin{example}[Tangent prolongation]\label{ex:TngLAVB}
If $B$ is any vector bundle then writing the structural equations for $B$ as diagrams and applying the tangent functor we get corresponding diagrams in the category of Lie algebroids. For example, applying the tangent functor to the addition relation $\gr(+_{B/M}):B\times B\dasharrow B$ yields addition for $TB\to TM$, $\gr(+_{TB/TM}):=T\gr(+_{B/M}):TB\times_{TM} TB\to TB$. Since $T\gr(+_{B/M})\subset (TB)^3$ is a subalgebroid, 
 it follows directly from \cref{def:DLACAVB} that then the tangent double vector bundle \labelcref{eq:TngDVB}, $TB$ is a $\mc{LA}$-vector bundle. 

Meanwhile, if $A\to M$ is a Lie algebroid, then the flip, $(TA)^{flip}$, of the tangent double vector bundle is naturally an $\mc{LA}$-vector bundle, called the \emph{tangent prolongation} (or \emph{tangent lift}) of $A$.
$$\begin{tikzpicture}
\mmat{m}{
TA&TM\\
A&M\\
};
\path[->] (m-1-1)	edge (m-1-2)
				edge (m-2-1);
\path[<-] (m-2-2)	edge (m-1-2)
				edge (m-2-1);
\end{tikzpicture}$$
The Lie bracket on $\Gamma(TA,TM)$ can be defined as follows. For $\sigma,\tau\in\Gamma(A)$ we define
\begin{subequations}\label{eq:LATngPro}
\begin{align}
[\sigma_T,\tau_T]&=[\sigma,\tau]_T,\\
[\sigma_T,\tau_C]&=[\sigma,\tau]_C,\\
[\sigma_C,\tau_C]&=0.
\end{align}\end{subequations}
The bracket of arbitrary sections is defined by means of the Leibniz rule.

As an example, when $A=\g$ is a Lie algebra, then $(TA)^{flip}=\g\ltimes\g$, as a Lie algebra.
\end{example}

\subsection{Double linear Poisson vector bundles}\label{sec:DLPV}
We recall the theory of double linear Poisson vector bundles, as introduced in \cite{gracia2010lie}.
Suppose that \labelcref{eq:LAvb} is a double vector bundle, and $D\to B$ is a Lie algebroid. Taking the horizontal dual, we obtain 
$$\begin{tikzpicture}
\mmat{m}{D^{*_{\!x}}&B\\ C^*& M\\};
\path[->] (m-1-1)	edge (m-1-2)
				edge (m-2-1);
\path[<-] (m-2-2)	edge (m-1-2)
				edge (m-2-1);
\end{tikzpicture}$$
Since $D\to B$ was a Lie algebroid, we recall from \cref{thm:linPois} that $D^{*_{\!x}}$ carries a linear Poisson structure. 
That is, $$\gr(+_{D^{*_{\!x}}/B}):D^{*_{\!x}}\times D^{*_{\!x}}\dasharrow D^{*_{\!x}}$$ is a coisotropic relation. 

Now $D$ is a $\mc{LA}$-vector bundle if and only if $\gr(+_{D/A}):D\times D\dasharrow D$ is an $\mc{LA}$-relation, which is equivalent to $$\gr(+_{D^{*_{\!x}}/C^*})=\ann^\natural(\gr(+_{D/A})):D^{*_{\!x}}\times D^{*_{\!x}}\dasharrow D^{*_{\!x}}$$ being a coisotropic relation.
 In other words, $D$ is a $\mc{LA}$-vector bundle if and only if the induced Poisson structure on $D^{*_{\!x}}$ is linear with respect to both vector bundle structures. Double vector bundles equipped with such a Poisson structure are called \emph{double linear Poisson vector bundles}.

It is clear that the definition of double linear Poisson vector bundles is symmetric with respect to the two vector bundle structures on the total space. More precisely, if $D^{*_{\!x}}$ is a double linear Poisson vector bundle, then so is $(D^{*_{\!x}})^{flip}$. Taking the horizontal dual of $(D^{*_{\!x}})^{flip}$, we get a second $\mc{LA}$-vector bundle. Using the canonical isomorphism $\big((D^{*_{\!x}})^{*_{\!y}}\big)^{flip}\cong D^{*_{\!y}}$, we see that 
$$\begin{tikzpicture}
\mmat{m}{D^{*_{\!y}}&C^*\\ A& M\\};
\path[->] (m-1-1)	edge (m-1-2)
				edge (m-2-1);
\path[<-] (m-2-2)	edge (m-1-2)
				edge (m-2-1);
\end{tikzpicture}$$
is canonically an $\mc{LA}$-vector bundle, called the dual $\mc{LA}$-vector bundle to $D$.

\begin{example}\label{ex:CotLAVB}
Suppose $A\to M$ is a Lie algebroid, then \cref{ex:TngLAVB} shows that the tangent prolongation, $TA^{flip}$, is an $\mc{LA}$-vector bundle. 
Taking the horizontal dual, we get 
$$\begin{tikzpicture}
\mmat{m}{TA^*&A^*\\ TM& M\\};
\path[->] (m-1-1)	edge (m-1-2)
				edge (m-2-1);
\path[<-] (m-2-2)	edge (m-1-2)
				edge (m-2-1);
\end{tikzpicture}$$
whose double linear Poisson structure is the tangent prolongation (i.e. \cref{ex:PoisTngPro}) of the linear Poisson structure on $A^*$ (see  \cref{thm:linPois} for details).

On the other hand, the vertical dual of $TA^{flip}$ is
$$\begin{tikzpicture}
\mmat{m}{T^*A&A^*\\ A& M\\};
\path[->] (m-1-1)	edge (m-1-2)
				edge (m-2-1);
\path[<-] (m-2-2)	edge (m-1-2)
				edge (m-2-1);
\end{tikzpicture}$$
which, as explained above is naturally an $\mc{LA}$-vector bundle. Indeed, under the isomorphism $T^*A\cong T^*A^*$, it is just the cotangent Lie algebroid for the linear Poisson structure on $A^*$.

Notice that the canonical symplectic structure on $T^*A$ is a double linear Poisson structure. The two duals of $T^*A$ are the tangent Lie algebroids $TA\to A$ and $TA^*\to A^*$.
\end{example}

%
%

\chapter{$\mc{VB}$-Courant algebroids}\label{chp:VBCour}

\section{$\mc{VB}$-Courant algebroids}
For the reader's convenience we recall the definition of a $\mc{CA}$-vector bundle, and its diagonal reflection, a $\mc{VB}$-Courant algebroid.
\begin{definition}\label{def:CAvb}
A \emph{$\mc{VB}$-Courant algebroid} is a double vector bundle
\begin{equation}\label{eq:CAvb}\begin{tikzpicture}
\mmat{m}{\mbb{E}&V\\ E& M\\};
\path[->] (m-1-1)	edge (m-1-2)
				edge (m-2-1);
\path[<-] (m-2-2)	edge (m-1-2)
				edge (m-2-1);
\end{tikzpicture}\end{equation}
such that $\mbb{E}\to E$ is a Courant algebroid and 
\begin{equation}\label{eq:VBCAcond}\gr(+_{\mbb{E}/V}):\mbb{E}\times\mbb{E}\dasharrow\mbb{E}\end{equation} is a Courant relation. 
The diagonal reflection, $\mbb{E}^{flip}$ is called a \emph{$\mc{CA}$-vector bundle}. 
\end{definition}

\begin{example}
Suppose $\mbb{E}\to M$ is a Courant algebroid. In \cite{Boumaiza:2009eg}  Boumaiza and  Zaalani showed that $T\mbb{E}\to TM$ is naturally a Courant algebroid. With this structure, the double vector bundle 
$$\begin{tikzpicture}
\mmat{m}{T\mbb{E}&\mbb{E}\\ TM& M\\};
\path[->] (m-1-1)	edge (m-1-2)
				edge (m-2-1);
\path[<-] (m-2-2)	edge (m-1-2)
				edge (m-2-1);
\end{tikzpicture}$$ is a $\mc{VB}$-Courant algebroid, as we shall show in \cref{chp:TngProAndLieSub}.
\end{example}

\begin{example}
Suppose $\mbb{E}\to M$ is a Courant algebroid, then the (trivial) double vector bundle
$$\begin{tikzpicture}
\mmat{m}{\mbb{E}&\mbb{E}\\ M& M\\};
\path[->] (m-1-1)	edge (m-1-2)
				edge (m-2-1);
\path[<-] (m-2-2)	edge (m-1-2)
				edge (m-2-1);
\end{tikzpicture}$$
is a $\mc{VB}$-Courant algebroid.
\end{example}

\begin{example}\label{ex:VBCAoverpt}
Let $\g$ be a Lie algebra. Then
$$\begin{tikzpicture}
\mmat{m}{\g\ltimes\g^*&\g\\ \ast&\ast\\};
\path[->] (m-1-1)	edge (m-1-2)
				edge (m-2-1);
\path[<-] (m-2-2)	edge (m-1-2)
				edge (m-2-1);
\end{tikzpicture}$$
is a $\mc{VB}$-Courant algebroid, where $\g\ltimes\g^*$ carries the quadratic form induced by the natural pairing and the double vector bundle structure is as in \cref{ex:DirectSumDVB}.
\end{example}

\begin{proposition}\label{prop:BasVBCAfacts}
Suppose that 
\begin{equation}\label{eq:CAvb2}\begin{tikzpicture}
\mmat{m}{\mbb{E}&V\\ E& M\\};
\path[->] (m-1-1)	edge (m-1-2)
				edge (m-2-1);
\path[<-] (m-2-2)	edge (m-1-2)
				edge (m-2-1);
\end{tikzpicture}\end{equation}
 is a $\mc{VB}$-Courant algebroid. Then the following facts hold.
\begin{enumerate}
\item The pairing on $\mbb{E}$ is linear. That is, the map 
\begin{equation}\label{eq:VBCALinPair}
\begin{tikzpicture}
\mmat{m1} at (-2,0) {\mbb{E}&V\\ E& M\\};
\path[->] (m1-1-1)	edge (m1-1-2)
				edge (m1-2-1);
\path[<-] (m1-2-2)	edge (m1-1-2)
				edge (m1-2-1);
\mmat{m2} at (2,0) {\mbb{E}^{*_{\!y}}&C^*\\ E& M\\};
\path[->] (m2-1-1)	edge (m2-1-2)
				edge (m2-2-1);
\path[<-] (m2-2-2)	edge (m2-1-2)
				edge (m2-2-1);
				
\path[->] (m1) edge node {$\cong$} (m2);
\draw (-2,0) node (c) {$C$};
\draw (2,0) node (v) {$V^*$};
\path[left hook->] (c) edge (m1-1-1);
\path[left hook->] (v) edge (m2-1-1);
\end{tikzpicture}
\end{equation}
 defined by the pairing $\la\cdot,\cdot\ra$ is an isomorphism of double vector bundles, where $C\to M$ is the core of $\mbb{E}$.

In particular, the core $C\to M$ of $\mbb{E}$ is canonically isomorphic to $V^*\to M$. 

\item The anchor map is linear. That is, 
\begin{equation}\label{eq:VBCALinAnc}
\begin{tikzpicture}
\mmat{m1} at (-2,0) {\mbb{E}&V\\ E& M\\};
\path[->] (m1-1-1)	edge (m1-1-2)
				edge (m1-2-1);
\path[<-] (m1-2-2)	edge (m1-1-2)
				edge (m1-2-1);
\mmat{m2} at (2,0) {TE&TM\\ E& M\\};
\path[->] (m2-1-1)	edge (m2-1-2)
				edge (m2-2-1);
\path[<-] (m2-2-2)	edge (m2-1-2)
				edge (m2-2-1);
				
\path[->] (m1) edge node {$\mbf{a}$} (m2);
\end{tikzpicture}
\end{equation}
 is a morphism of double vector bundles.

\item The Courant bracket is linear. That is 
\begin{subequations}\label[pluralequation]{eq:CBlin}
\begin{align}
\label{eq:VBCALinBrk1}\Cour{\Gamma_l(\mbb{E},E),\Gamma_l(\mbb{E},E)}&\subseteq\Gamma_l(\mbb{E},E),\\
\label{eq:VBCALinBrk2}\Cour{\Gamma_l(\mbb{E},E),\Gamma_C(\mbb{E},E)}&\subseteq\Gamma_C(\mbb{E},E),\\
\label{eq:VBCALinBrk3}\Cour{\Gamma_C(\mbb{E},E),\Gamma_C(\mbb{E},E)}&=0.
\end{align}\end{subequations}
\end{enumerate}

Conversely, if \labelcref{eq:CAvb2} is any double vector bundle such that $\mbb{E}$ is a Courant algebroid such that \cref{eq:CBlin} holds and \cref{eq:VBCALinPair,eq:VBCALinAnc} are morphisms of double vector bundles, then $\mbb{E}$ is a $\mc{VB}$-Courant algebroid.
\end{proposition}

\begin{proof}
\begin{enumerate}
\item First we show that the pairing is linear. This is equivalent to showing that the map $\phi_{\la\cdot,\cdot\ra}:\mbb{E}\to\mbb{E}^{*_{\!y}}$ defined by the pairing satisfies  \begin{equation}\label{eq:VBCAprop1}\phi_{\la\cdot,\cdot\ra}^3(\gr(+_{\mbb{E}/V}))\subseteq \gr(+_{\mbb{E}^{*_{\!y}}/C^*}),\end{equation} where $C$ is the core of $\mbb{E}$. Since $\mbb{E}$ is a $\mc{VB}$-Courant algebroid, the relation \begin{equation}\label{eq:VBCArel}\gr(+_{\mbb{E}/V}):\mbb{E}\times\mbb{E}\dasharrow\mbb{E}\end{equation} is a Courant relation, in particular, it is isotropic. Hence $$\phi_{\la\cdot,\cdot\ra}^3(\gr(+_{\mbb{E}/V}))\subseteq\ann^\natural(\gr(+_{\mbb{E}/V})).$$ But $\gr(+_{\mbb{E}^{*_{\!y}}/C^*})=\ann^\natural(\gr(+_{\mbb{E}/V}))$, by definition, which implies \cref{eq:VBCAprop1}.

Since the pairing identifies $\mbb{E}$ with its dual double vector bundle $\mbb{E}^{*_{\!y}}$, it identifies the core of $\mbb{E}$ with the core $V^*\to M$ of $\mbb{E}^{*_{\!y}}$. In particular, this implies that 
\begin{equation}\label{eq:CorePairing}\la x,c+_{\mbb{E}/V}0_{\mbb{E}/E}\ra=\la q_{\mbb{E}/V}(x),c\ra\end{equation}
 for any $x\in\mbb{E}$ and $c\in C$.

\item Next, we show that the anchor map is linear. Let $$\on{gr}(+_{E/M}):E\times E\dasharrow E$$ denote the graph of addition for $E$, and $$\on{gr}(+_{TE/TM}):=T\big(\on{gr}(+_{E/M})\big):TE\times TE\dasharrow TE$$ denote the graph of addition for $TE$, viewed as a vector bundle over $TM$.
Since \labelcref{eq:VBCArel} is a Courant relation supported on the submanifold $\on{gr}(+_{E/M})\subseteq E^3$, it follows that $$\mbf{a}(\gr(+_{\mbb{E}/V}))\subseteq T\big(\on{gr}(+_{E/M})\big)=\on{gr}(+_{TE/TM}),$$ which in turn shows that $\mbf{a}:\mbb{E}\to TE$ is a morphism of double vector bundles.

\item Finally, we show that the Courant bracket is linear. A section $\sigma\in\Gamma(\mbb{E},E)$ is linear if and only if 
\begin{subequations}\label[pluralequation]{eq:LinCoreSects}
\begin{equation}\label{eq:LinSect}\sigma \oplus \sigma\sim_{\gr(+_{\mbb{E}/V})}\sigma.\end{equation}
 Since  $\gr(+_{\mbb{E}/V})$ is a Courant relation, the Courant bracket of two sections satisfying \cref{eq:LinSect} also satisfies \cref{eq:LinSect}, which implies \cref{eq:VBCALinBrk1}.
Meanwhile, a section $\tau\in\Gamma(\mbb{E},E)$ is a core section if and only if \begin{equation}\label{eq:CoreSec}\tau\oplus 0_{\mbb{E}/E}\sim_{\gr(+_{\mbb{E}/V})}\tau.\end{equation} Taking the Courant bracket of \labelcref{eq:LinSect} and \labelcref{eq:CoreSec}, we see that 
$$\Cour{\sigma,\tau}\oplus 0_{\mbb{E}/E}\sim_{\gr(+_{\mbb{E}/V})}\Cour{\sigma,\tau}.$$ 
Hence $\Cour{\sigma,\tau}$ is a core section, which proves \cref{eq:VBCALinBrk2}. Finally, from \cref{eq:CoreSec} and the commutativity of addition, we see that a second core section $\tau'\in\Gamma_C(\mbb{E},E)$ satisfies 
\begin{equation}0_{\mbb{E}/E}\oplus\tau'\sim_{\gr(+_{\mbb{E}/V})}\tau.\end{equation}\end{subequations}
 Taking the Courant bracket of this expression with \labelcref{eq:CoreSec}, we see that $$0_{\mbb{E}/E}\oplus 0_{\mbb{E}/E}\sim_{\gr(+_{\mbb{E}/V})}\Cour{\tau,\tau'}.$$ In other words, $\Cour{\tau,\tau'}=0_{\mbb{E}/E}$, which implies \cref{eq:VBCALinBrk3}.
 
 \end{enumerate}

Now we prove the converse. Suppose \labelcref{eq:CAvb2} is a double vector bundle and $\mbb{E}\to E$ is a Courant algebroid. Suppose further that \cref{eq:CBlin} holds and \cref{eq:VBCALinPair,eq:VBCALinAnc} are morphisms of double vector bundles. We must show that $$\gr(+_{\mbb{E}/V})\subseteq \mbb{E}\times\overline{\mbb{E}\times\mbb{E}}$$ is a Dirac structure with support on $\gr(+_{E/M})$.

As explained above, \cref{eq:VBCALinPair} is a morphism of double vector bundles if and only if 
\begin{equation}\label{eq:phiissubset}\phi_{\la\cdot,\cdot\ra}^3(\gr(+_{\mbb{E}/V}))\subseteq \gr(+_{\mbb{E}^{*_{\!y}}/C^*}):=\ann^\natural(\gr(+_{\mbb{E}/V})),\end{equation} holds where $\phi_{\la\cdot,\cdot\ra}:\mbb{E}\to\mbb{E}_E^*$ is the map defined by the pairing. Since the pairing is non-degenerate \cref{eq:phiissubset} must be an equality (rather than a strict inclusion). In particular, $$\gr(+_{\mbb{E}/V})\subseteq \mbb{E}\times\overline{\mbb{E}\times\mbb{E}}$$ is Lagrangian.

Now $\Gamma(\mbb{E},E)$ is spanned by the linear and core sections, while $\gr(+_{\mbb{E}/V})$ is spanned by linear and core sections of the forms found in \cref{eq:LinCoreSects}.
Hence \cref{eq:CBlin} show that sections of the form \labelcref{eq:LinCoreSects} are closed under the Courant bracket. Since the anchor map is linear, we have $$\mbf{a}\big(\gr(+_{\mbb{E}/V})\big)\subseteq T\gr(+_{E/M}),$$
 which shows that the $C^\infty$-linear span of  sections of the form \labelcref{eq:LinCoreSects} is also closed under the Courant bracket. Therefore $\gr(+_{\mbb{E}/V})$ is involutive with respect to the Courant bracket.
\end{proof}

\begin{remark}
Gracia-Saz and Mehta showed that morphisms of double vector bundles can be characterized in terms of linear and core sections (see \cite[Lemma~2.8]{gracia2010lie}). Thus \cref{eq:VBCALinPair} is a morphism of double vector bundles if and only if 
\begin{subequations}\label[pluralequation]{eq:CPlin}
\begin{align}
\label{eq:VBCALinP1}\la\Gamma_l(\mbb{E},E),\Gamma_l(\mbb{E},E)\ra&\subseteq C^\infty_l(E),\\
\label{eq:VBCALinP2}\la\Gamma_l(\mbb{E},E),\Gamma_C(\mbb{E},E)\ra&\subseteq C^\infty_C(E),\\
\label{eq:VBCALinP3}\la\Gamma_C(\mbb{E},E),\Gamma_C(\mbb{E},E)\ra&=0,
\end{align}\end{subequations} 
where $C^\infty_l(E):=\Gamma_l(\mbb{R}\times E,E)\cong \Gamma(E^*)$ denotes the set of linear functions on $E$, and $C^\infty_C(E):=\Gamma_C(\mbb{R}\times E,E)=q_{E/M}^* C^\infty(M)$ the set of fibrewise constant functions. Here $q_{E/M}:E\to M$ is the bundle projection.

Similarly, \cref{eq:VBCALinAnc} is a morphism of double vector bundles if and only if 
\begin{subequations}\label[pluralequation]{eq:ANClin}
\begin{align}
\label{eq:VBCALinA1}\mbf{a}(\Gamma_l(\mbb{E},E))&\subseteq \Gamma_l(TE,E),\\
\label{eq:VBCALinA2}\mbf{a}(\Gamma_C(\mbb{E},E))&\subseteq \Gamma_C(TE,E).
\end{align}\end{subequations}
\end{remark}

\begin{definition}
Suppose that \begin{equation*}\label{eq:VBDirWSup}\begin{tikzpicture}
\mmat{m1} at (-2,0){L&W\\ S& N\\};
\path[->] (m1-1-1)	edge (m1-1-2)
				edge (m1-2-1);
\path[<-] (m1-2-2)	edge (m1-1-2)
				edge (m1-2-1);
\mmat{m2} at (2,0) {\mbb{E}&V\\ E& M\\};
\path[->] (m2-1-1)	edge (m2-1-2)
				edge (m2-2-1);
\path[<-] (m2-2-2)	edge (m2-1-2)
				edge (m2-2-1);
\draw (0,0) node {$\subseteq$};
\end{tikzpicture}\end{equation*}
is a double vector subbundle of a $\mc{VB}$-Courant algebroid.
If $L\subset \mbb{E}$ is also a Dirac structure with support on $S$, we call it a \emph{$\mc{VB}$-Dirac structure with support on $S$}. When $S=E$, we simply call it a \emph{$\mc{VB}$-Dirac structure}.
\end{definition}

Every $\mc{VB}$-Courant algebroid has two distinguished $\mc{VB}$-Dirac structures (with support):
\begin{proposition}\label{prop:FreeVBDir} Suppose that $$\begin{tikzpicture}
\mmat{m}{\mbb{E}&V\\ E& M\\};
\path[->] (m-1-1)	edge node{$q_{\mbb{E}/V}$} (m-1-2)
				edge (m-2-1);
\path[<-] (m-2-2)	edge (m-1-2)
				edge (m-2-1);
\end{tikzpicture}$$ is a $\mc{VB}$-Courant algebroid. Then
\begin{enumerate}
\item $\mbb{E}_C:=\on{ker}(q_{\mbb{E}/V})$ is a Dirac structure. Moreover, for any $\sigma\in\Gamma_l(\mbb{E},E)$ and $\tau\in\Gamma_l(\mbb{E}_C,E)$, we have \begin{equation}\label{eq:ECalmostideal}\Cour{\sigma,\tau}\in\Gamma_l(\mbb{E}_C,E).\end{equation}
\item The subbundle $V\subseteq \mbb{E}$ is a Dirac structure with support on $M\subseteq E$. 
\end{enumerate}
\end{proposition}
\begin{proof}
\begin{enumerate}
\item  Since $\mbb{E}_C\subseteq \mbb{E}$ is spanned by the core sections, the fact that it is a Dirac structure follows immediately from \cref{eq:VBCALinP3,eq:VBCALinBrk3}.

 Given $\sigma\in\Gamma_l(\mbb{E},E)$ and $\tau\in\Gamma_l(\mbb{E}_C,E)$, choose a decomposition $$\tau=\sum_i f_i\tau_i,$$ where each $f_i$ is a linear function on $E$ and each $\tau_i\in\Gamma_C(\mbb{E},E)$ is a core section. Then $$\Cour{\sigma,\tau}=\sum_i\big(f_i\Cour{\sigma,\tau_i}+(\mbf{a}(\sigma)f_i)\tau_i\big).$$ Since  $\tau_i$ is a core section, the second term on the right hand side certainly lies in $\Gamma_l(\mbb{E}_C,E)$. Moreover,  since \cref{eq:VBCALinBrk2} implies that $\Cour{\sigma,\tau_i}$ is a core section, we also have $f_i\Cour{\sigma,\tau_i}\in\Gamma_l(\mbb{E}_C,E)$.
 
\item  Let $\gr(-_{\mbb{E}/V})=\{(x,x_1,x_2)\mid x+_{\mbb{E}/V}x_1=x_2\}$ be the graph of subtraction in $\mbb{E}\to V$. Since it is just a permutation of $\gr(+_{\mbb{E}/V})$, it defines a Courant morphism $$\gr(-_{\mbb{E}/V}):\mbb{E}\times\overline{\mbb{E}}\dasharrow\mbb{E}.$$ Let $\mbb{E}_\Delta\subset\mbb{E}\times\overline{\mbb{E}}$ denote the diagonal, viewed as a Courant relation from the trivial Courant algebroid to $\mbb{E}\times\overline{\mbb{E}}$.
Since $V=\gr(-_{\mbb{E}/V})\circ\mbb{E}_\Delta$ is a clean composition of Courant relations, its graph  $V\subset \mbb{E}$ is a Dirac structure with support.

\end{enumerate}
\end{proof}

\begin{example}\label{ex:VBDSoverpt}
Let $\g$ be a Lie algebra and $\h\subseteq\g$ a subalgebra. Then
$$\begin{tikzpicture}
\mmat{m1} at (-3,0){\h\ltimes\on{ann}(\h)&\h\\ \ast& \ast\\};
\path[->] (m1-1-1)	edge (m1-1-2)
				edge (m1-2-1);
\path[<-] (m1-2-2)	edge (m1-1-2)
				edge (m1-2-1);
\mmat{m2} at (2,0) {\g\ltimes\g^*&\g\\ \ast& \ast\\};
\path[->] (m2-1-1)	edge (m2-1-2)
				edge (m2-2-1);
\path[<-] (m2-2-2)	edge (m2-1-2)
				edge (m2-2-1);
\draw (0,0) node {$\subseteq$};
\end{tikzpicture}$$
is a $\mc{VB}$-Dirac structure of the $\mc{VB}$-Courant algebroid described in \cref{ex:VBCAoverpt}.
\end{example}

$\mc{VB}$-Dirac structures have some important properties, which we collect in the following proposition.

\begin{proposition}\label{prop:VBDirIsLaVB}
Suppose that $$\begin{tikzpicture}
\mmat{m1} at (-2,0){L&W\\ E& M\\};
\path[->] (m1-1-1)	edge (m1-1-2)
				edge (m1-2-1);
\path[<-] (m1-2-2)	edge (m1-1-2)
				edge (m1-2-1);
\mmat{m2} at (2,0) {\mbb{E}&V\\ E& M\\};
\path[->] (m2-1-1)	edge (m2-1-2)
				edge (m2-2-1);
\path[<-] (m2-2-2)	edge (m2-1-2)
				edge (m2-2-1);
\draw (0,0) node {$\subseteq$};
\end{tikzpicture}$$
is a $\mc{VB}$-Dirac structure.
Then
\begin{enumerate}
\item $L^{flip}$ is an $\mc{LA}$-vector bundle. In particular, $L\to E$ and $W\to M$ are Lie algebroids, and $L\to W$ is a Lie algebroid morphism.
\item the core of $L$ is  $\ann(W)\subseteq V^*$.
\end{enumerate}
\end{proposition}
\begin{proof}
\begin{enumerate}
\item Since $\mbb{E}$ is a $\mc{VB}$-Courant algebroid, $$\gr(+_{\mbb{E}/V}):\mbb{E}\times\mbb{E}\dasharrow\mbb{E}$$ is a Courant morphism. 
It follows that $$\gr(+_{L/W})=\gr(+_{\mbb{E}/V})\cap L^3:L\times L\to L$$ is an involutive subbundle of $L$, and hence a Lie algebroid relation. Hence $L^{flip}$ is an $\mc{LA}$-vector subbundle.

\item Let $C_L$ denote the core of $L$. Dualizing the inclusion
 $$\begin{tikzpicture}
\mmat{m1} at (-2,0){L&W\\ E& M\\};
\path[->] (m1-1-1)	edge (m1-1-2)
				edge (m1-2-1);
\path[<-] (m1-2-2)	edge (m1-1-2)
				edge (m1-2-1);
\mmat{m2} at (2,0) {\mbb{E}&V\\ E& M\\};
\path[->] (m2-1-1)	edge (m2-1-2)
				edge (m2-2-1);
\path[<-] (m2-2-2)	edge (m2-1-2)
				edge (m2-2-1);
\draw (0,0) node {$\subseteq$};
\draw (-2,0) node (c1) {$C_L$};
\draw (2,0) node (c2) {$V^*$};
\draw[left hook->] (c1) edge (m1-1-1);
\draw[left hook->] (c2) edge (m2-1-1);
\end{tikzpicture}$$
we get the projection
 $$\begin{tikzpicture}
\mmat{m1} at (2,0){L^{*_{\!y}}&C^*_L\\ E& M\\};
\path[->] (m1-1-1)	edge (m1-1-2)
				edge (m1-2-1);
\path[<-] (m1-2-2)	edge (m1-1-2)
				edge (m1-2-1);
\mmat{m2} at (-2,0) {\mbb{E}&V\\ E& M\\};
\path[->] (m2-1-1)	edge (m2-1-2)
				edge (m2-2-1);
\path[<-] (m2-2-2)	edge (m2-1-2)
				edge (m2-2-1);
\draw[->] (m2) edge (m1);
\draw (2,0) node (c1) {$W^*$};
\draw (-2,0) node (c2) {$V^*$};
\draw[left hook->] (c1) edge (m1-1-1);
\draw[left hook->] (c2) edge (m2-1-1);
\end{tikzpicture}$$
 whose kernel is $L$. Restricting to the cores yields the map $V^*\to W^*$, whose kernel is $\ann(W)$, the core of $L$.
\end{enumerate}
\end{proof}

\begin{remark}
Suppose $$\begin{tikzpicture}
\mmat{m1} at (-2,0){L&W\\ E& M\\};
\path[->] (m1-1-1)	edge (m1-1-2)
				edge (m1-2-1);
\path[<-] (m1-2-2)	edge (m1-1-2)
				edge (m1-2-1);
\mmat{m2} at (2,0) {\mbb{E}&V\\ E& M\\};
\path[->] (m2-1-1)	edge (m2-1-2)
				edge (m2-2-1);
\path[<-] (m2-2-2)	edge (m2-1-2)
				edge (m2-2-1);
\draw (0,0) node {$\subseteq$};
\end{tikzpicture}$$ is a $\mc{VB}$-Dirac structure.
While the restriction of the Courant bracket to $L$ projects to a bracket on $W$, the Courant bracket on $\mbb{E}$ does not (in general) project to a bracket on $V$. Indeed, let $\mbf{a}':V^*\to E$ be the restriction of the anchor map 
$$\begin{tikzpicture}
\mmat{m1} at (-2,0) {\mbb{E}&V\\ E& M\\};
\path[->] (m1-1-1)	edge (m1-1-2)
				edge (m1-2-1);
\path[<-] (m1-2-2)	edge (m1-1-2)
				edge (m1-2-1);
\mmat{m2} at (2,0) {TE&TM\\ E& M\\};
\path[->] (m2-1-1)	edge (m2-1-2)
				edge (m2-2-1);
\path[<-] (m2-2-2)	edge (m2-1-2)
				edge (m2-2-1);
				
\path[->] (m1) edge node {$\mbf{a}$} (m2);

\draw (-2,0) node (c1) {$V^*$};
\draw (2,0) node (c2) {$E$};
\draw[left hook->] (c1) edge (m1-1-1);
\draw[left hook->] (c2) edge (m2-1-1);
\end{tikzpicture}$$
to the core. Then for any $\gamma\in\Gamma(E^*)$, $q_{\mbb{E}/V}(\mbf{a}^*d\la\gamma,\cdot\ra)=\mbf{a}'^*\gamma$. Hence, if $\sigma\in\Gamma_l(\mbb{E},E)$, and $\tau\in\Gamma_C(\mbb{E},E)$, then 
$$q_{\mbb{E}/V}\big(\la\gamma,\cdot\ra\tau\big)=0,\text{ but }q_{\mbb{E}/V}\big(\Cour{\sigma,\la\gamma,\cdot\ra\tau}+\Cour{\la\gamma,\cdot\ra\tau,\sigma})=\la\sigma,\tau\ra\mbf{a}'^*\gamma.$$
 In fact, the Courant bracket on $\mbb{E}$ projects to a bracket on $V$ if and only if $\mbf{a}':V^*\to E$ is trivial (as is the case, for instance, in \cref{ex:VBCAoverpt}).
\end{remark}

It is well known that Courant algebroids over a point are just quadratic Lie algebras. $\mc{VB}$-Courant algebroids over a point are characterized similarly.
\begin{proposition}\label{prop:VBCAoverpt}
The only $\mc{VB}$-Courant algebroids $$\begin{tikzpicture}
\mmat{m}{\mbb{E}&V\\ E& M\\};
\path[->] (m-1-1)	edge (m-1-2)
				edge (m-2-1);
\path[<-] (m-2-2)	edge (m-1-2)
				edge (m-2-1);
\end{tikzpicture}$$ such that $E=\ast$ is a single point are of the form given in \cref{ex:VBCAoverpt}. In this case, the only $\mc{VB}$-Dirac structures in $\mbb{E}$ are of the form given in \cref{ex:VBDSoverpt}. 
\end{proposition}
\begin{proof}
Since $E=\ast$, the Courant algebroid $\mbb{E}$ must be a quadratic Lie algebra. Moreover, the core sequence \labelcref{eq:iA} becomes the short exact sequence of vector spaces
$$V^*\to\mbb{E}\to V,$$
which is split by the inclusion $V\to \mbb{E}$ of the zero section.
Thus $\mbb{E}=V\oplus V^*$ as a vector space. 
\Cref{prop:FreeVBDir} then show that $V$ and $V^*$ are transverse Lagrangian Lie subalgebras. It follows that the quadratic form on $V\oplus V^*$ is given by the canonical pairing.  Moreover, \cref{eq:CPlin} show that $V^*$ is an abelian ideal. 
Therefore $$\mbb{E}\cong \g\ltimes\g^*$$ for some Lie algebra $\g$. Hence there is a one-to-one correspondence between Lie algebras and $\mc{VB}$-Courant algebroids over a point.

Suppose 
$$\begin{tikzpicture}
\mmat{m1} at (-2,0){L&W\\ \ast& \ast\\};
\path[->] (m1-1-1)	edge (m1-1-2)
				edge (m1-2-1);
\path[<-] (m1-2-2)	edge (m1-1-2)
				edge (m1-2-1);
\mmat{m2} at (2,0) {\g\ltimes\g^*&\g\\ \ast& \ast\\};
\path[->] (m2-1-1)	edge (m2-1-2)
				edge (m2-2-1);
\path[<-] (m2-2-2)	edge (m2-1-2)
				edge (m2-2-1);
\draw (0,0) node {$\subseteq$};
\end{tikzpicture}$$ is a $\mc{VB}$-Dirac structure. Then $W=L\cap\g$ is identified with a subalgebra $\h\subseteq\g$. Since the core of $L$ is $\ann(W)=\ann(\h)$, we must have $L=\h\ltimes\ann(\h)$. So $\mc{VB}$-Dirac structures of $\g\ltimes\g^*$ are in one-to-one correspondence with subalgebras of $\g$.
\end{proof}

The above example shows that $\mc{VB}$-Dirac structures of $\mc{VB}$-Courant algebroids have a surprisingly simple structure when $E=\ast$ is a point. In general, they can be more complicated, but we still have the following proposition.

\begin{proposition}\label{prop:clasLagDB}
Suppose that $\mbb{E}$ is a $\mc{VB}$-Courant algebroid and
\begin{equation}\label{eq:LagDVB}
\begin{tikzpicture}
\mmat{m1} at (-2,0){L&W\\ E& M\\};
\path[->] (m1-1-1)	edge (m1-1-2)
				edge (m1-2-1);
\path[<-] (m1-2-2)	edge (m1-1-2)
				edge (m1-2-1);
\mmat{m2} at (2,0) {\mbb{E}&V\\ E& M\\};
\path[->] (m2-1-1)	edge node {$q_{\mbb{E}/V}$} (m2-1-2)
				edge (m2-2-1);
\path[<-] (m2-2-2)	edge (m2-1-2)
				edge (m2-2-1);
\draw (0,0) node {$\subseteq$};
\end{tikzpicture}
\end{equation}
is a Lagrangian double vector subbundle.
 Let $\mbb{E}\rvert_{W}:=q_{\mbb{E}/V}^{-1}(W)$ denote the restriction of $\mbb{E}$ to $W$.
The total quotient of $\mbb{E}\rvert_W$ by $L$ defines a morphism of double vector bundles  
\begin{equation}\label{eq:LagQuo}
\begin{tikzpicture}
\mmat{m1} at (-2,0){\mbb{E}\rvert_W&W\\ E& M\\};
\path[->] (m1-1-1)	edge node {$q_{\mbb{E}/V}$} (m1-1-2)
				edge (m1-2-1);
\path[<-] (m1-2-2)	edge (m1-1-2)
				edge (m1-2-1);
\draw (-2,0) node (c) {$V^*$};
\path[left hook->] (c) edge (m1-1-1);

\mmat{m2} at (2,0) {W^*&M\\ M& M\\};
\path[->] (m2-1-1)	edge (m2-1-2)
				edge (m2-2-1);
\path[<-] (m2-2-2)	edge node[swap]{$\on{id}$} (m2-1-2)
				edge node{$\on{id}$} (m2-2-1);
\draw (2,0) node (c2) {$W^*$};
\path[left hook->] (c2) edge (m2-1-1);
				
\path[->] (m1) edge node {$q$} (m2);
\end{tikzpicture}
\end{equation}
such that 
\begin{enumerate}
\item[(i)] \label{itm:1} the restriction $q\rvert_{V^*}:V^*\to W^*$ of $q$ to the core of $\mbb{E}\rvert_W$ is dual to the inclusion $W\to V$, and
\item[(ii)] for any $x,y\in \mbb{E}_W$, we have \begin{equation}\label{eq:restPair}\la x,y\ra=\la q(x),q_{\mbb{E}/V}(y)\ra+\la q(y),q_{\mbb{E}/V}(x)\ra.\end{equation}
\end{enumerate}
 In this way, there is a one-to-one correspondence between Lagrangian double vector subbundles of the form \labelcref{eq:LagDVB}, and surjective morphisms of the form \labelcref{eq:LagQuo} which satisfy \cref{eq:restPair}.
\end{proposition}
\begin{proof}
\Cref{prop:totkerVtotquo} states that there is a one-to-one correspondence between double vector subbundles $L\subseteq\mbb{E}$ with side bundle $W\subseteq V$, and surjective morphisms \labelcref{eq:LagQuo}. We need only show that this correspondence relates Lagrangian $L$ to morphisms \labelcref{eq:LagQuo} such that (i) and (ii) hold.

Now $\mbb{E}\rvert_W^\perp=\ann(W)\subseteq V^*$. Let $\mbb{F}=\mbb{E}\rvert_W/\ann(W)$, and let $Q:\mbb{F}\dasharrow \mbb{E}$ be the canonical reduction relation, namely $$y\sim_Q x\quad\Leftrightarrow\quad x\in \mbb{E}\rvert_W,\quad y\in\mbb{F},\text{ and }y\equiv x\mod\ann(W).$$ Let $\mc{F}$ be the set of double vector subbundles $L'\subseteq \mbb{F}$ with side bundles $E$ and $W$, and $\mc{E}$ be the set of double vector subbundles $L\subseteq \mbb{E}$ with side bundles $E$ and $W$ satisfying $\ann(W)\subseteq L$.
 The map \begin{equation}\label{eq:Qrel}(L'\to Q\circ L'):\mc{F}\to\mc{E}\end{equation} is a bijection. Moreover,  since $Q\subseteq \mbb{E}\times\overline{\mbb{F}}$ is a Lagrangian relation, \labelcref{eq:Qrel} identifies the Lagrangian double vector subbundles. 
 
The double vector subbundle $L':=\on{ker}(q)/\ann(W)$ is Lagrangian if and only if (ii) holds. In this case 
$L$ is Lagrangian, and \cref{prop:VBDirIsLaVB} implies that the core of $L$ is $\on{ann}(W)$. Therefore the kernel of $q_{V^*}$, the restriction of $q$ to the core of $\mbb{E}\rvert_W$, is $\on{ann}(W)$. This implies that (i) holds as well, establishing the desired correspondence.
\end{proof}

By \cref{prop:VBDirIsLaVB}, if $L\subseteq\mbb{E}$ is a $\mc{VB}$-Dirac structure \labelcref{eq:LagDVB}, then the side bundle $W$ carries a Lie algebroid structure. The following proposition describes the Lie algebroid bracket on $W$ in terms of the Courant bracket on $\mbb{E}$ and the map $q:\mbb{E}\rvert_W\to W^*$ described in \cref{prop:clasLagDB}.

\begin{proposition}\label{prop:qIndBrk}
Suppose that 
$$\begin{tikzpicture}
\mmat{m1} at (-2,0){L&W\\ E& M\\};
\path[->] (m1-1-1)	edge (m1-1-2)
				edge (m1-2-1);
\path[<-] (m1-2-2)	edge (m1-1-2)
				edge (m1-2-1);
\mmat{m2} at (2,0) {\mbb{E}&V\\ E& M\\};
\path[->] (m2-1-1)	edge node {$q_{\mbb{E}/V}$} (m2-1-2)
				edge (m2-2-1);
\path[<-] (m2-2-2)	edge (m2-1-2)
				edge (m2-2-1);
\draw (0,0) node {$\subseteq$};
\end{tikzpicture}$$
 is a $\mc{VB}$-Dirac structure, and let $q:\mbb{E}\rvert_W\to W^*$ be the morphism \labelcref{eq:LagQuo} described in \cref{prop:clasLagDB}.

 Suppose that $\sigma,\tau\in\Gamma(W)$ are two sections, and $\tilde\sigma,\tilde\tau\in\Gamma_l(\mbb{E},E)$ are  lifts,  i.e. the following diagrams commute:
$$\label{eq:WhatIsALift}\begin{tikzpicture}
\mmat{m1} at (-3,0) {\mbb{E}\rvert_W&W\\ E& M\\};
\path[->] (m1-2-1)	edge node {$\tilde\sigma$} (m1-1-1)
				edge node[swap] {$q_{E/M}$}(m1-2-2);
\path[<-] (m1-1-2)	edge node[swap] {$q_{\mbb{E}/V}$} (m1-1-1)
				edge node {$\sigma$} (m1-2-2);
\mmat{m2} at (3,0) {\mbb{E}\rvert_W&W\\ E& M\\};
\path[->] (m2-2-1)	edge node {$\tilde\tau$} (m2-1-1)
				edge node[swap] {$q_{E/M}$} (m2-2-2);
\path[<-] (m2-1-2)	edge node[swap] {$q_{\mbb{E}/V}$} (m2-1-1)
				edge node {$\tau$} (m2-2-2);
				
\end{tikzpicture}$$
  Then $$[\sigma,\tau]=q_{\mbb{E}/V}\Cour{\tilde\sigma,\tilde\tau}-\mbf{a}'^*\la q(\tilde\sigma),\tau\ra,$$ where $\mbf{a}':V^*\to E$ denotes the restriction of the anchor map
$$\begin{tikzpicture}
\mmat{m1} at (-2,0) {\mbb{E}&V\\ E& M\\};
\path[->] (m1-1-1)	edge (m1-1-2)
				edge (m1-2-1);
\path[<-] (m1-2-2)	edge (m1-1-2)
				edge (m1-2-1);
\mmat{m2} at (2,0) {TE&TM\\ E& M\\};
\path[->] (m2-1-1)	edge (m2-1-2)
				edge (m2-2-1);
\path[<-] (m2-2-2)	edge (m2-1-2)
				edge (m2-2-1);
				
\path[->] (m1) edge node {$\mbf{a}$} (m2);

\draw (-2,0) node (c1) {$V^*$};
\draw (2,0) node (c2) {$E$};
\draw[left hook->] (c1) edge (m1-1-1);
\draw[left hook->] (c2) edge (m2-1-1);
\end{tikzpicture}$$
to the core.
\end{proposition}
\begin{proof}
Recall the $\mc{VB}$-Dirac structure $\mbb{E}_C:=\on{ker}(q_{\mbb{E}/V})$ described in \cref{prop:FreeVBDir}. Let $\sigma',\tau'\in\Gamma_l(\mbb{E}_C,E)$ be chosen so that $q(\sigma')=q(\tilde\sigma)$ and $q(\tau')=q(\tilde\tau)$. Thus $\tilde\sigma-\sigma'\in\Gamma_l(L,E)$ and $\tilde\tau-\tau'\in\Gamma_l(L,E)$. Hence, we have
\begin{align*}
[\sigma,\tau]=&q_{\mbb{E}/V}\Cour{\tilde\sigma-\sigma',\tilde\tau-\tau'}\\
=&q_{\mbb{E}/V}\big(\Cour{\tilde\sigma,\tilde\tau}-\Cour{\tilde\sigma,\tau'}+\Cour{\tilde\tau,\sigma'}+\Cour{\sigma',\tau'}-\mbf{a}^*d\la\sigma',\tilde\tau\ra\big)\\
=&q_{\mbb{E}/V}\big(\Cour{\tilde\sigma,\tilde\tau}-\mbf{a}^*d\la\sigma',\tilde\tau\ra\big),
\end{align*}
where the last line follows from \cref{eq:ECalmostideal}. Dualizing the anchor map, 
$$\begin{tikzpicture}
\mmat{m1} at (2,0) {\mbb{E}&V\\ E& M\\};
\path[->] (m1-1-1)	edge (m1-1-2)
				edge (m1-2-1);
\path[<-] (m1-2-2)	edge (m1-1-2)
				edge (m1-2-1);
\mmat{m2} at (-2,0) {T^*E&E^*\\ E& M\\};
\path[->] (m2-1-1)	edge (m2-1-2)
				edge (m2-2-1);
\path[<-] (m2-2-2)	edge (m2-1-2)
				edge (m2-2-1);
				
\path[->] (m2) edge node {$\mbf{a}^*$} (m1);

\draw (2,0) node (c1) {$V^*$};
\draw (-2,0) node (c2) {$T^*M$};
\draw[left hook->] (c1) edge (m1-1-1);
\draw[left hook->] (c2) edge (m2-1-1);
\end{tikzpicture}$$
we see that $\mbf{a}'^*:E^*\to V$ is the restriction of $\mbf{a}^*:T^*E\to \mbb{E}$ to the vertical side bundle. Thus $q_{\mbb{E}/V}\circ \mbf{a}^*=\mbf{a}'^*\circ q_{T^*E/E^*}$, which yields the equation
$$[\sigma,\tau]=q_{\mbb{E}/V}\Cour{\tilde\sigma,\tilde\tau}-\mbf{a}'^*q_{T^*E/E^*}d\la\sigma',\tilde\tau\ra.$$
By \cref{eq:VBCALinP1}, we see that $\la\sigma',\tilde\tau\ra$ is a linear function on $E$, and $$q_{T^*E/E^*}d\la\sigma',\tilde\tau\ra\in\Gamma(E^*)$$ is its partial derivative along the fibres of $E$. Therefore, after identifying linear functions on $E$ with sections of $E^*$,  $$q_{T^*E/E^*}d\la\sigma',\tilde\tau\ra=\la\sigma',\tilde\tau\ra.$$ It follows that
\begin{align*}
[\sigma,\tau]&=q_{\mbb{E}/V}\Cour{\tilde\sigma,\tilde\tau}-\mbf{a}'^*\la\sigma',\tilde\tau\ra\\
&=q_{\mbb{E}/V}\Cour{\tilde\sigma,\tilde\tau}-\mbf{a}'^*\la q(\sigma'),q_{\mbb{E}/V}\tilde\tau\ra\\
&=q_{\mbb{E}/V}\Cour{\tilde\sigma,\tilde\tau}-\mbf{a}'^*\la q(\tilde\sigma),\tau\ra,
\end{align*}
where the second line follows from \cref{eq:restPair}.
\end{proof}

\section{Lie 2-algebroids}
In this section, we briefly discuss the supergeometric interpretation of $\mc{VB}$-Courant algebroids.

\begin{definition}[\v{S}evera \cite{Severa:2005vla}]
A  Lie 2-algebroid is a degree 2 $NQ$ manifold. 
\end{definition}

\begin{remark}
Lie 2-algebroids are a generalization of 2-term $L_\infty$-algebras, which first appeared in the work of Schlessinger and Stasheff \cite{Schlessinger:1985fv} on pertubations of rational homotopy types.
Roytenberg showed that Courant algebroids are a special case of Lie 2-algebroids \cite{Roytenberg:1998ku,Roytenberg99,Roytenberg:2002}.\footnote{Roytenberg \cite{Roytenberg:1998ku,Roytenberg99} first showed this for the Courant algebroid $\mbb{T}M$. For the general case, \v{S}evera realized the correspondence independently \cite{LetToWein,Severa:2005vla}.}
\end{remark}

\begin{proposition}\label{prop:Lie2VBCour}
 There is a one-to-one correspondence between Lie 2-algebroids and $\mc{VB}$-Courant algebroids. There is a one-to-one correspondence between wide Lie subalgebroids of a Lie 2-algebroid and $\mc{VB}$-Dirac structures in the corresponding $\mc{VB}$-Courant algebroid.
\end{proposition}
\begin{proof}
Let $X$ be a degree 2 $NQ$ manifold, then $T^*[2]X$ is a degree 2 symplectic $NQ$-manifold, which by the results of Roytenberg and \v{S}evera \cite{Roytenberg:2002,LetToWein,Severa:2005vla} corresponds to a Courant algebroid $\mbb{E}$ (see \cref{rem:SupCour} for a short summary of this correspondence). Now consider the graph of addition in the cotangent fibres $$\gr( +): T^*[2]X\times T^*[2]X\dasharrow T^*[2]X.$$ Since $\gr( +)\subseteq T^*[2]X\times\overline{T^*[2]X\times T^*[2]X}$ is a Lagrangian $NQ$-submanifold, it corresponds to a Courant relation $$\gr( +'):\mbb{E}\times\mbb{E}\dasharrow \mbb{E}$$ defining a second vector bundle structure on $\mbb{E}$. Therefore, by \cref{def:CAvb}, $\mbb{E}$ is a $\mc{VB}$-Courant algebroid, fitting into the diagram
\begin{equation*}\begin{tikzpicture}
\mmat{m}{\mbb{E}&V\\ E& M\\};
\path[->] (m-1-1)	edge (m-1-2)
				edge (m-2-1);
\path[<-] (m-2-2)	edge (m-1-2)
				edge (m-2-1);
\end{tikzpicture}\end{equation*} 
 where $E$ is the base of $T^*[2]X$,  $V[1]$ is the $1$-truncation of $X$, and $M$ is the base of $X$.
 
 Next, recall that a wide Lie subalgebroid of $X$ is a degree 1 $NQ$ submanifold $A[1]\subseteq X$ over the base $M$ (\emph{wide} refers to the fact that $A[1]$ has the same base space as $X$). Then the conormal bundle $\ann[2](TA[1])\subset T^*[2]X$ is a Lagrangian $NQ$-submanifold with base $E$. Hence, by the results of Roytenberg and \v{S}evera \cite{Roytenberg:2002,LetToWein,Severa:2005vla} (also see \cref{rem:SupCour}), it corresponds to a Dirac structure $L\subseteq \mbb{E}$. Moreover, since $\ann[2](TA[1])$ is a subbundle of $T^*[2]X$, $L$ is a subbundle of $\mbb{E}\to V$, and hence a $\mc{VB}$-Dirac structure.
 
The converse is proved by reversing the above construction.

 \end{proof}

\section{Further examples}

\begin{example}[String Lie 2-algebra]
Suppose $\g$ is a Lie algebra. Let $(S^2\g^*)^\g$ denote the space of invariant symmetric bilinear forms on $\g$. Then
$$\begin{tikzpicture}
\mmat{m}{(\g\times\g^*)\times (S^2\g^*)^\g &\mf{g}\\ (S^2\g^*)^\g&\ast\\};
\path[->]
	(m-1-1) edge (m-1-2)
		edge (m-2-1);
\path[<-] 
	(m-2-2) edge (m-1-2)
		edge (m-2-1);
\end{tikzpicture}$$
is a $\mc{VB}$-Courant algebroid. 
The left vertical arrow is a bundle of quadratic Lie algebras, where the quadratic form is  just the natural pairing between $\g$ and $\g^*$, and the bracket at $g\in(S^2\g^*)^\g$ is defined by 
$$[(\xi,\mu),(\eta,\nu)]=([\xi,\eta],\ad_\xi^\dagger\nu-\ad_\eta^\dagger\mu+g^\flat[\xi,\eta]),$$
where $\xi,\eta\in \g$, $\mu,\nu\in\g^*$, $\ad_\cdot^\dagger$ denotes the contragredient of the adjoint representation, and $g^\flat:\g\to\g^*$ is the map $$g^\flat(\xi)(\eta)=g(\xi,\eta).$$ 

Given a quadratic Lie algebra $(\mf{d},\la\cdot,\cdot\ra)$, Baez and Crans \cite{Baez:2004vl,Crans:2004ux} constructed a 1-parameter family of Lie 2-algebras, which is often called the string Lie 2-algebra in the literature. Let $\phi:\mbb{R}\to(S^2\mf{d}^*)^\mf{d}$ be given by $\phi(t)=t\la\cdot,\cdot\ra$.
The $\mc{VB}$-Courant algebroid corresponding to the string Lie 2-algebra via \cref{prop:Lie2VBCour} is the pull-back Courant algebroid $\phi^!\big((\mf{d}\times\mf{d}^*)\times (S^2\mf{d}^*)^\mf{d}\big)$,
%
$$\begin{tikzpicture}
\mmat{m}{(\mf{d}\oplus\mf{d}^*)\times\mbb{R} &\mf{d}\\ \mbb{R}&\ast\\};
\path[->]
	(m-1-1) edge (m-1-2)
		edge (m-2-1);
\path[<-] 
	(m-2-2) edge (m-1-2)
		edge (m-2-1);
\end{tikzpicture}$$
That is the bundle of quadratic Lie algebras over $\mbb{R}$ whose bracket at $t\in \mbb{R}$ is
$$[(\xi,\mu),(\eta,\nu)]=([\xi,\eta],[\xi,\nu]+[\mu,\eta]+t[\xi,\eta]),$$
for $\xi,\eta\in \mf{d}$, $\mu,\nu\in \mf{d}^*$, and we have used the quadratic form to identify $\mf{d}$ with $\mf{d}^*$. The quadratic form on $(\mf{d}\oplus\mf{d}^*)$ is given by the natural pairing.
\end{example}

\begin{example}[Crossed modules of Lie algebras]
A crossed module of Lie algebras consists of 
\begin{itemize}
\item a pair of Lie algebras $(\g,\h)$,
\item an action of $\g$ on $\h$, and
\item a $\g$-equivariant Lie algebra morphism $\delta:\h\to\g$,
\end{itemize} satisfying the \emph{Peiffer identity}
 $$\delta(\xi)\cdot \eta=[\xi,\eta], \quad \xi,\eta\in\h.$$

The abelian Lie algebra $\g^*$ acts by translations on the affine space $\h^*$ via the map $\delta^*:\g^*\to \h^*$. Since $\delta:\h\to\g$ is equivariant, this action is compatible with the contragredient action of $\g$ on $\h^*$. Hence, the quadratic Lie algebra $\mf{d}:=\g\ltimes \g^*$ acts on $\h^*$. Since $$\mbf{a}_\mu(\zeta,\alpha)=\zeta\cdot\mu+\delta^*\alpha,\quad \mu\in\h^*,\zeta\in\g,\alpha\in\g^*,$$ we see that $$\la\mbf{a}_\mu^*(\xi),\mbf{a}_\mu^*(\eta)\ra=\mu(\delta(\eta)\cdot \xi-\delta(\xi)\cdot \eta),\quad \xi,\eta\in\h.$$ Therefore, the Peiffer identity implies that the stabilizers are coisotropic. Let $\mbb{E}=(\g\ltimes\g^*)\times\h^*$ be the corresponding action Courant algebroid. 

As described in \cref{ex:DirectSumDVB}, $\mbb{E}$ is naturally a double vector bundle 
$$\begin{tikzpicture}
\mmat{m}{\mbb{E} &\g\\ \h^*&\ast\\};
\path[->]
	(m-1-1) edge (m-1-2)
		edge (m-2-1);
\path[<-] 
	(m-2-2) edge (m-1-2)
		edge (m-2-1);
\end{tikzpicture}$$
and it is clear that the pairing and the Courant bracket satisfy the conditions of \cref{prop:BasVBCAfacts}, so $\mbb{E}$ is a $\mc{VB}$-Courant algebroid.
\end{example}

\begin{example}
Let $A\to M$ be a Lie algebroid. Then $A\oplus A^*$ is a Courant algebroid over $M$, with bracket given by $$\Cour{(a_1,\mu_1),(a_2,\mu_2)}=([a_1,a_2],\Lied_{a_1}\mu_2-\iota_{a_2}d_A\mu_1),$$
for $a_i\in\Gamma(A)$ and $\mu_i\in\Gamma(A^*)$.  Here $d_A:\Gamma(\wedge^\bullet A^*)\to \Gamma(\wedge^{\bullet+1}A^*)$ is the Lie algebroid differential, $\Lied_{a}=[\iota_a,d_A]$, and  for any $a\in\Gamma(A)$,  $\iota_a:\Gamma(\wedge^\bullet A^*)\to\Gamma(\wedge^{\bullet-1} A^*)$ is the derivation extending the pairing between $A$ and $A^*$.

Viewed as a double vector bundle, $A\oplus A^*$ becomes a $\mc{VB}$-Courant algebroid,
$$\begin{tikzpicture}
\mmat{m}{A\oplus A^* &A\\ M&M\\};
\path[->]
	(m-1-1) edge (m-1-2)
		edge (m-2-1);
\path[<-] 
	(m-2-2) edge (m-1-2)
		edge (m-2-1);
\end{tikzpicture}$$
\end{example}

\begin{example}[Standard Courant algebroid over a vector bundle]\label{ex:StdVBCAoverVB} Suppose $E\to M$ is a vector bundle. Consider the double vector bundle
$$\begin{tikzpicture}
\mmat{m}{\mbb{T}E&TM\oplus E^*\\ E&M\\};
\path[->]
	(m-1-1) edge (m-1-2)
		edge (m-2-1);
\path[<-] 
	(m-2-2) edge (m-1-2)
		edge (m-2-1);
\end{tikzpicture}$$
constructed as the direct sum of the tangent and cotangent double vector bundles, $(TE)^{flip}$ and $(T^*E)^{flip}$, described in \cref{ex:TngDVB,ex:CotDVB}, respectively. The standard lift, described in \cref{ex:StdDiracStr}, of the relation $$\gr(+_{E/M}):E\times E\dasharrow E$$ is naturally identified with $$\gr(+_{\mbb{T}E/TM\oplus E^*}):\mbb{T}E\times\mbb{T}E\dasharrow\mbb{T}E.$$ Hence the latter  is a Courant relation. Therefore $\mbb{T}E$ is a $\mc{VB}$-Courant algebroid.
 
 The double vector subbundles $TE^{flip}$:
$$\begin{tikzpicture}
\mmat{m} at (-2.5,0) {TE&TM\\ E&M\\};
\path[->]
	(m-1-1) edge (m-1-2)
		edge (m-2-1);
\path[<-] 
	(m-2-2) edge (m-1-2)
		edge (m-2-1);

\draw (0,0) node {$\subseteq$};

\mmat{m1} at (3,0) {\mbb{T}E&TM\oplus E^*\\ E&M\\};
\path[->]
	(m1-1-1) edge (m1-1-2)
		edge (m1-2-1);
\path[<-] 
	(m1-2-2) edge (m1-1-2)
		edge (m1-2-1);
\end{tikzpicture}$$
 and $T^*E^{flip}$:
$$\begin{tikzpicture}
\mmat{m} at (-2.5,0) {T^*E&E^*\\ E&M\\};
\path[->]
	(m-1-1) edge (m-1-2)
		edge (m-2-1);
\path[<-] 
	(m-2-2) edge (m-1-2)
		edge (m-2-1);

\draw (0,0) node {$\subseteq$};

\mmat{m1} at (3,0) {\mbb{T}E&TM\oplus E^*\\ E&M\\};
\path[->]
	(m1-1-1) edge (m1-1-2)
		edge (m1-2-1);
\path[<-] 
	(m1-2-2) edge (m1-1-2)
		edge (m1-2-1);
\end{tikzpicture}$$
are both $\mc{VB}$-Dirac structures (cf. \cref{ex:StdCourAlg,ex:TngDVB,ex:CotDVB}).
\end{example}

\section{Tangent prolongation of a Courant algebroid}\label{sec:TngPro}

Let $\mbb{E}$ be a Courant algebroid over $M$. In \cite{Boumaiza:2009eg}  Boumaiza and  Zaalani showed that $T\mbb{E}\to TM$ carries a canonical Courant algebroid structure.  In this section we recall this so-called \emph{tangent prolongation of Courant algebroids} in the context of $\mc{VB}$-Dirac structures. 

Recall from \cref{ex:TngDVB2} that  any section $\sigma\in\Gamma(\mbb{E},M)$, defines two sections $\sigma_C,\sigma_T\in\Gamma(T\mbb{E},TM)$ called the core and tangent lift of $\sigma$, respectively. Note also that $\{\sigma^i_C,\sigma^i_T\}$ is a local basis for $T\mbb{E}\to TM$ whenever $\{\sigma^i\}$ is a local basis for $\mbb{E}\to M$.

\begin{proposition}\label{prop:TngProCA}
The tangent bundle $T\mbb{E}$ of a Courant algebroid $\mbb{E}\to M$,
$$\begin{tikzpicture}
\mmat{m}{T\mbb{E} &\mbb{E}\\ TM&M\\};
\path[->]
	(m-1-1) edge (m-1-2)
		edge (m-2-1);
\path[<-] 
	(m-2-2) edge (m-1-2)
		edge (m-2-1);
\end{tikzpicture}$$
carries a unique Courant algebroid structure over $TM$ such that the pairing and bracket satisfy
\begin{subequations}\label[pluralequation]{eq:TEPairBrk}
\begin{align}
\label{eq:TEPair}\la \sigma_T,\tau_T\ra&=\la\sigma,\tau\ra_T & \la\sigma_T,\tau_C\ra&=\la\sigma,\tau\ra_C &
\la\sigma_C,\tau_T\ra&=\la\sigma,\tau\ra_C & \la\sigma_C,\tau_C\ra&=0\\
\label{eq:TEBrk}\Cour{\sigma_T,\tau_T}&=\Cour{\sigma,\tau}_T & \Cour{\sigma_T,\tau_C}&=\Cour{\sigma,\tau}_C  & 
\Cour{\sigma_C,\tau_T}&=\Cour{\sigma,\tau}_C& \Cour{\sigma_C,\tau_C}&=0
\end{align}\end{subequations}
 and the anchor map satisfies
$$\mbf{a}(\sigma_T)=\mbf{a}(\sigma)_T \quad\quad \mbf{a}(\sigma_C) = \mbf{a}(\sigma)_C,$$
for any sections $\sigma,\tau\in\Gamma(\mbb{E},M)$.
With this structure, $T\mbb{E}$ is a $\mc{VB}$-Courant algebroid.
\end{proposition}
\begin{proof}
The fact that $T\mbb{E}\to TM$ is a Courant algebroid follows  from  \cref{prop:CAconst} in the Appendix. The fact that $T\mbb{E}$ is a $\mc{VB}$-Courant algebroid follows  from \cref{prop:BasVBCAfacts}.
\end{proof}
\begin{remark}
Let $X$ be the degree 2 symplectic $NQ$ manifold corresponding to the Courant algebroid $\mbb{E}$ via the equivalence described in \cite{Severa:2005vla,Roytenberg:2002} (the correspondence is also summarized in \cref{rem:SupCour}).
%
The symplectic structure on $X$ defines a bundle isomorphism $TX\cong T^*[2]X$. Since $T^*[2]X$ carries a canonical symplectic structure, this isomorphism endows $TX$ with a symplectic structure \cite{Courant:1999ho}.
 Moreover, the tangent lift of the $Q$ structure to $TX$ is compatible with this symplectic structure, so $TX$ is a degree 2 symplectic $NQ$ manifold. Under the equivalence described in \cite{Severa:2005vla, Roytenberg:2002}, it corresponds to the Courant algebroid $T\mbb{E}$.
\end{remark}

\begin{remark}\label{rem:TangLiftTngProDef}
One may define the pairing and anchor map without reference to core and tangent lift sections. If 
$$\begin{tikzpicture}
\mmat{m1} at (-1.5,0) {\mbb{E} \\ M\\};
\path[->]
	(m1-1-1) edge (m1-2-1);
\mmat{m2} at (1.5,0) {TM\\ M\\};
\path[->]
	(m2-1-1) edge (m2-2-1);
		
\path[->,bend left= 15] (m1-1-1) edge node {$\mbf{a}$} (m2-1-1);
\end{tikzpicture}$$
 is the anchor map for $\mbb{E}\to M$, the anchor map for $T\mbb{E}\to TM$ is obtained by applying the tangent functor,
$$
\begin{tikzpicture}[cross line/.style={preaction={draw=white, -,
           line width=6pt}},]
\mmat{m1} at (-2.5,0) {T\mbb{E}&\mbb{E}\\ TM& M\\};
\path[->] (m1-1-1)	edge (m1-1-2)
				edge (m1-2-1);
\path[<-] (m1-2-2)	edge (m1-1-2)
				edge (m1-2-1);
\mmat{m2} at (2.5,0) {T^2M&TM\\ TM& M\\};
\path[->] (m2-1-1)	edge (m2-1-2)
				edge (m2-2-1);
\path[<-] (m2-2-2)	edge (m2-1-2)
				edge (m2-2-1);
				
\path[->] (m1) edge node {$\mbf{a}_{T\mbb{E}}$} (m2);
\path[->, bend left =30] (m1-1-2) edge node {$\mbf{a}$} (m2-1-2);
\path[->, bend left =30] (m1-1-1) edge [cross line] node {$d\mbf{a}$} (m2-1-1);
\end{tikzpicture}
$$

Meanwhile, as reviewed in \cref{ex:SecDualTng}, $(T\mbb{E})^{*_{\!y}}\cong T(\mbb{E}^*)$. The fibre metric on $\mbb{E}$ defines an isomorphism $Q_{\la\cdot,\cdot\ra}:\mbb{E}\to \mbb{E}^*$, which lifts to an isomorphism of double vector bundles $$dQ_{\la\cdot,\cdot\ra}:T\mbb{E}\to T(\mbb{E}^*)\cong (T\mbb{E})^{*_{\!y}}$$ via the tangent functor. This latter isomorphism defines the metric on the fibres of $T\mbb{E}\to TM$.

\end{remark}

\begin{example}
Suppose that $\mbb{E}=\mf{d}$ is a quadratic Lie algebra. Then $T\mbb{E}=T\mf{d}:=\mf{d}\ltimes\mf{d}$, where the bracket and pairing are given by 
\begin{align*}
[(\xi,\xi'),(\eta,\eta')]_{\mf{d}\ltimes\mf{d}}&=([\xi,\eta]_{\mf{d}},[\xi,\eta']_{\mf{d}}+[\xi',\eta]_{\mf{d}})\\
\la(\xi,\xi'),(\eta,\eta')\ra_{\mf{d}\ltimes\mf{d}}&=\la \xi,\eta'\ra_{\mf{d}}+\la \xi',\eta\ra_{\mf{d}}
\end{align*}
for $\xi,\xi',\eta,\eta'\in\mf{d}$.
\end{example}

\begin{example}[Tangent Lift of a Dirac structure (with support)]\label{ex:DirStrTngLf}
Suppose that $R\subseteq\mbb{E}$ is a Dirac structure with support on a submanifold $S\subset M$, then it follows immediately from \cref{eq:TEPair,eq:TEBrk} that $TR\subseteq T\mbb{E}$ is a $\mc{VB}$-Dirac structure with support on $TS\subseteq TM$.
\end{example}

\begin{example}\label{ex:BCGred}Bursztyn, Cavalcanti and Gualtieri \cite{Bursztyn:2007ko} described a reduction procedure for an exact Courant algebroid $\mbb{E}\to M$. 
One may interpret part of their construction as defining a $\mc{VB}$-Dirac structure  $L\subseteq T\mbb{E}$ (with support), as we shall now explain.

Let $\g$ be the Lie algebra of a Lie group $G$ acting on $M$. Suppose there is a given vector space $\mf{a}$ carrying a bilinear bracket $[\cdot,\cdot]:\mf{a}\times\mf{a}\to \mf{a}$, together with a surjective map $\mbf{a}:\mf{a}\to \g$ satisfying
\begin{itemize}
\item $[a_1,[a_2,a_3]]=[[a_1,a_2],a_3]+[a_2,[a_1,a_3]],\quad$ for $a_1,a_2,a_3\in\mf{a}$
\item $\mbf{a}([a_1,a_2])=[\mbf{a}(a_1),\mbf{a}(a_2)],\quad$  for $a_1,a_2\in\mf{a}$, and
\item $[a_1,a_2]=0,\quad$ for $a_1,a_2\in \on{ker}(\mbf{a})$.
\end{itemize}
The pair $(\mf{a},[\cdot,\cdot])$ is called an \emph{exact Courant algebra} over $\g$ in \cite{Bursztyn:2007ko}. 
We denote the abelian Lie algebra $\on{ker}(\mbf{a})\subseteq \mf{a}$ by $\h$, and note that $\g\cong \mf{a}/\h$ acts naturally on $\h$.
Suppose further that there are given maps $\tilde \rho:\mf{a}\to\Gamma(\mbb{E})$ and $\nu:\h\to\Omega^1(M)$ such that the following diagram commutes
\begin{equation}\label{eq:CalgebraAct}\begin{tikzpicture}
\mmat{m}{0&\h &\mf{a}&\g&0\\0& \Omega^1(M)&\Gamma(\mbb{E})&\mf{X}^1(M)&0\\};
\path[->]
	(m-1-1) edge (m-1-2)
	(m-1-2) edge (m-1-3)
		edge node {$\nu$} (m-2-2)
	(m-1-3) edge (m-1-4)
		edge node {$\tilde \rho$}(m-2-3)
	(m-1-4) edge (m-1-5)
		edge node {$\rho$} (m-2-4)
	(m-2-1) edge (m-2-2)
	(m-2-2) edge (m-2-3)
	(m-2-3) edge (m-2-4)
	(m-2-4) edge (m-2-5);
\end{tikzpicture}\end{equation}
where $\rho:\g\to\mf{X}^1(M)$ is defined by the $G$ action, and $\tilde\rho:\mf{a}\to\Gamma(\mbb{E})$ intertwines the brackets.

As explained in \cite{Bursztyn:2007ko}, axioms c1) and c2) for the Courant bracket (\cref{def:CA}) imply that sections $\sigma\in\Gamma(\mbb{E})$ act on $\mbb{E}$ via the `adjoint action': $$\sigma\to\Cour{\sigma,\cdot}.$$
Thus \labelcref{eq:CalgebraAct} defines an action of the Courant algebra $\mf{a}$ on $\mbb{E}$.
Suppose that $\h$ acts trivially, that is $$\Cour{\nu(\xi),\cdot}=0$$ for all $\xi\in\h$ (or equivalently $\nu$ takes values in closed forms, i.e. $\nu(\h)\subseteq \Omega_{cl}^1(M)$), and the induced action of the Lie algebra $\g=\mf{a}/\h$ on $\mbb{E}$ integrates to an action of $G$. In \cite{Bursztyn:2007ko}, such data is called an \emph{extension of the action of $\g$ on $M$ to $\mbb{E}$}.

Let $K\subset\mbb{E}$ be the distribution spanned by $\tilde\rho(\mf{a})$, and  $K^\perp$ its orthogonal. For simplicity we shall assume that $K$ and $K^\perp$ are both of constant rank. Suppose $P$ is a leaf of the distribution $\mbf{a}(K+K^\perp)\subseteq TM$ on which $G$ acts freely and properly, and suppose that $\nu(\h)$ has constant rank along $P$. 

Let $F\subseteq TP$ be the integrable distribution spanned by $\rho(\g)$. Let 
$$
\begin{tikzpicture}
\mmat{m1} at (-2,0) {L&K^\perp\\ F& P\\};
\path[->] (m1-1-1)	edge (m1-1-2)
				edge (m1-2-1);
\path[<-] (m1-2-2)	edge (m1-1-2)
				edge (m1-2-1);
\mmat{m2} at (2,0) {T\mbb{E}&\mbb{E}\\ TM& M\\};
\path[->] (m2-1-1)	edge (m2-1-2)
				edge (m2-2-1);
\path[<-] (m2-2-2)	edge (m2-1-2)
				edge (m2-2-1);
				
\draw (0,0) node {$\subseteq$};
\end{tikzpicture}
$$
be the double vector subbundle which is spanned by the $G$-invariant linear sections $$\{\sigma_T\mid\sigma\in\Gamma(K^\perp\rvert_P)^G\}$$ and the core sections $$\{\tau_C\mid \tau\in\Gamma(K)\}.$$  Since $K^\perp\rvert_P$ is a $G$-invariant bundle of constant rank, it follows that $$\on{rank}(L,F)=\on{rank}(K)+\on{rank}(K^\perp)=\on{rank}(\mbb{E})=\frac{1}{2}\on{rank}(T\mbb{E},TM)$$ where $\on{rank}(L,F)$ and $\on{rank}(T\mbb{E},TM)$ denote the rank of the vector bundles $L\to F$ and $T\mbb{E}\to TM$.
As we shall explain, the proof of \cite[Theorem~3.3]{Bursztyn:2007ko} shows that $L$ is a $\mc{VB}$-Dirac structure with support on $F\subseteq TP\subseteq TM$.

First of all, $L$ is isotropic: it is clear that $\la\sigma_T,\tau_C\ra=\la\sigma,\tau\ra_C=0$ for any $\sigma\in\Gamma(K^\perp\rvert_P)^G$ and $\tau\in\Gamma(K)$. Moreover, if $\xi\in\g$, $x\in P$ and $\sigma,\sigma'\in\Gamma(K^\perp\rvert_P)^G$ then $$\la\sigma_T,\sigma'_T\ra(\rho(\xi)\rvert_x)=\rho(\xi)\cdot \la\sigma,\sigma'\ra\rvert_x=0,$$ by $G$-invariance. Since $F$ is spanned by $\rho(\g)$ and $L\to F$ is a half-rank subbundle of $T\mbb{E}\to TM$, it follows that $L$ is Lagrangian.

Next, we show that $L$ is involutive. Suppose that $\xi\in\mf{a}$ and $\tilde\sigma,\tilde\sigma'\in\Gamma(\mbb{E})$  extend $G$-invariant sections of $K^\perp\rvert_P$. Then 
\begin{align*}
\la \tilde\rho(\xi),\Cour{\tilde\sigma,\tilde\sigma'}\ra&=-\la\Cour{\tilde\sigma,\tilde\rho(\xi)},\tilde\sigma'\ra+\mbf{a}(\tilde\sigma)\la \tilde\rho(\xi),\tilde\sigma'\ra\\
&=\la\Cour{\tilde\rho(\xi),\tilde\sigma},\tilde\sigma'\ra+\mbf{a}(\tilde\sigma')\la\tilde\rho(\xi),\tilde\sigma\ra+\mbf{a}(\tilde\sigma)\la\tilde\rho(\xi),\tilde\sigma'\ra
\end{align*}
which vanishes along $P$ since $\tilde\sigma\rvert_P$ is $G$ invariant and $\la\tilde\rho(\xi),\sigma\ra\rvert_P=\la\tilde\rho(\xi),\sigma'\ra\rvert_P=0$ (as explained in \cite{Bursztyn:2007ko}). Hence $\Cour{\sigma,\sigma'}\rvert_P$ is again a $G$-invariant section of $K^\perp$. It follows that $$\Cour{\sigma_T,\sigma'_T}\rvert_F=\Cour{\sigma,\sigma'}_T\rvert_F$$ is a section of $L$. Moreover, since $\tilde\sigma\rvert_P$ is $G$-invariant, $$\Cour{\tilde \rho(\xi)_C,\tilde\sigma_T}\rvert_F=\Cour{\rho(\xi),\tilde\sigma}_C\rvert_F=0.$$ But $\Gamma(K\rvert_P)$ is spanned by the sections $\tilde\rho(\xi)\rvert_P$ for $\xi\in\mf{a}$. It follows that $L$ is involutive.

Bursztyn, Cavalcanti and Gualtieri prove \cite[Theorem~3.3]{Bursztyn:2007ko} that $L\cap TK^\perp$ defines a foliation of $K^\perp$, and that the leaf space naturally carries the structure of a Courant algebroid. This suggests that $\mc{VB}$-Dirac structures are closely related to the reduction of Courant algebroids. We shall explore this in more detail in Section~\ref{sec:reduc}.
\end{example}

\chapter{pseudo-Dirac structures}\label{chp:TngProAndLieSub}


\section{pseudo-Dirac structures}\label{sec:LieSubAlg}
In this section, we relate $\mc{VB}$-Dirac structures in $T\mbb{E}$ to structures in $\mbb{E}$ itself, which we call \emph{pseudo-Dirac structures}.

Let 
\begin{equation}\label{eq:VBDirTng}
\begin{tikzpicture}
\mmat{m1} at (-2,0) {L&W\\ TM& M\\};
\path[->] (m1-1-1)	edge (m1-1-2)
				edge (m1-2-1);
\path[<-] (m1-2-2)	edge (m1-1-2)
				edge (m1-2-1);
\mmat{m2} at (2,0) {T\mbb{E}&\mbb{E}\\ TM& M\\};
\path[->] (m2-1-1)	edge (m2-1-2)
				edge (m2-2-1);
\path[<-] (m2-2-2)	edge (m2-1-2)
				edge (m2-2-1);
				
\draw (0,0) node {$\subseteq$};
\end{tikzpicture}
\end{equation}
be a Lagrangian double vector subbundle of $T\mbb{E}$, and denote by $q:T\mbb{E}\rvert_W\to W^*$  the morphism  \labelcref{eq:LagQuo} of double vector bundles defined in \cref{prop:clasLagDB}. Define the map $$\nabla:\Omega^0(M,W)\to \Omega^1(M, W^*)$$ by \begin{equation}\label{eq:NabDef}\nabla\sigma=q(\sigma_T), \quad\sigma\in\Gamma(W).\end{equation}

\begin{proposition}\label{lem:qIsNab}
The construction above defines a one-to-one correspondence between Lagrangian double vector subbundles \labelcref{eq:VBDirTng} and pairs $(W,\nabla)$ consisting of a subbundle $W\subseteq \mbb{E}$ together with a map $$\nabla:\Omega^0(M,W)\to \Omega^1(M,W^*)$$ satisfying
\begin{subequations}\label[pluralequation]{eq:pseudoConProp}
\begin{align} 
\label{eq:NabDerProp}\nabla(f\sigma)&=f\nabla\sigma+df\otimes \la\sigma,\cdot\ra,\\
\notag d\la\sigma,\tau\ra&=\la\nabla\sigma,\tau\ra+\la\sigma,\nabla\tau\ra,
\end{align}
\end{subequations}
for any $\sigma,\tau\in\Gamma(W)$  and  $f\in C^\infty(M)$.
%
%
%
%
\end{proposition}
\begin{proof}
Suppose that $\nabla$ is defined by \cref{eq:NabDef}, and $\sigma\in\Gamma(W)$. Then
\begin{subequations}\label{eq:NabIsQ1}\begin{align}
\notag\nabla(f\sigma)&=q((f\sigma)_T)\\
\label{eq:UseTDer}&= q(f_C\sigma_T+f_T\sigma_C)\\
\notag&=fq(\sigma_T)+df\;q(\sigma_C)\\
\label{eq:UseRestCore}&=f\nabla\sigma+df\otimes \la\sigma,\cdot\ra
\end{align}\end{subequations}
where \labelcref{eq:UseTDer} follows from \cref{eq:TCDer}, and the equality in \labelcref{eq:UseRestCore} holds since the restriction of $q$ to the core, $q\rvert_\mbb{E}:\mbb{E}\to W^*$ is dual to the inclusion $W\subseteq \mbb{E}$. 

Next, suppose that $\sigma,\tau\in\Gamma(W)$. Then by \cref{eq:restPair}, we have
\begin{subequations}\label{eq:NabIsQ2}
\begin{align}
\label{eq:NabPDer1a}d\la\sigma,\tau\ra&=\la\sigma_T,\tau_T\ra\\
\label{eq:NabPDer1b}&=\la q(\sigma_T),\tau\ra+\la\sigma,q(\tau_T)\ra\\
\notag&=\la \nabla\sigma,\tau\ra+\la\sigma,\nabla\tau\ra
\end{align}\end{subequations}
where \labelcref{eq:NabPDer1a} follows from \cref{eq:TEPair}, while \labelcref{eq:NabPDer1b} follows from \cref{eq:restPair}. Thus, we have shown that \cref{eq:pseudoConProp} holds.

Conversely, suppose that $\nabla:\Omega^0(M,W)\to\Omega^1(M, W^*)$ satisfies \cref{eq:pseudoConProp}. We will apply \cref{prop:clasLagDB} to construct the Lagrangian subbundle $L\subseteq T\mbb{E}$:

Since \cref{eq:NabDerProp} holds, the calculation \labelcref{eq:NabIsQ1} shows we can extend the core map $\mbb{E}\to W^*$ (dual to the inclusion $W\subseteq \mbb{E}^*$)
 uniquely to a map $q:T\mbb{E}\rvert_W\to W^*$ via \cref{eq:NabDef}.

Next we, show that \cref{eq:restPair} holds. The calculation \labelcref{eq:NabIsQ2} shows that  \cref{eq:restPair} holds when $x=\sigma_T$ and $y=\tau_T$. Next, the fact that $q$ extends the core map $$\mbb{E}\xrightarrow{x\to\la x,\cdot\ra} W^*$$ implies that \cref{eq:restPair} holds when $x=\sigma_T$ and $y=\tau_C$. Finally, \cref{eq:restPair} holds trivially when $x=\sigma_C$ and $y=\tau_C$.

Thus the assumptions of \cref{prop:clasLagDB} are satisfied, so there exists a unique Lagrangian subbundle $L\subseteq T\mbb{E}$ corresponding to the map of double vector bundles $q:T\mbb{E}\rvert_W\to W^*$. Moreover,  by construction, $\nabla$ satisfies \cref{eq:NabDef} for the map $q:T\mbb{E}\rvert_W\to W^*$ corresponding to this $L$.
\end{proof}

\begin{definition}\label{def:pseuCon}
Suppose $\mbb{E}\to M$ is a vector bundle with a bundle metric $\la\cdot,\cdot\ra$, and $W\subseteq\mbb{E}$ is a subbundle. A map $\nabla:\Omega^0(M,W)\to \Omega^1(M, W^*)$ satisfying 
\begin{subequations}\label[pluralequation]{eq:pseudoCon}
\begin{align} 
\label{eq:NabDer}\nabla(f\sigma)&=f\nabla\sigma+df\otimes \la\sigma,\cdot\ra,\\
\label{eq:NabPairDer}d\la\sigma,\tau\ra&=\la\nabla\sigma,\tau\ra+\la\sigma,\nabla\tau\ra,
\end{align}
\end{subequations}
 is called a \emph{pseudo-connection} for $W$. 
 
 Given a smooth map $\phi:N\to M$,  the \emph{pull-back pseudo-connection} $\phi^*\nabla$ for $\phi^*W\subseteq \phi^*\mbb{E}$ is defined by $$(\phi^*\nabla)_X\phi^*\sigma=\phi^*(\nabla_{d\phi(X)}\sigma),$$ where $\sigma\in\Gamma(W)$ and $X\in TN$.
\end{definition}

\begin{remark} Let $E\to M$ be a fibre bundle.
The idea of interpreting connections on $E$ as subbundles of $TE$ goes back to Ehresmann \cite{Ehresmann:1995ts}. In the special case where $E\to M$ is a vector bundle,  Pradines \cite{Pradines:1974tc} (Ehresmann's student) 
noted that there was a second vector bundle structure on $TE$, and studied certain double vector subbundles of it (providing applications to the study of jets). Later, Konieczna
and Urba\'{n}ski \cite{Konieczna:1999vh}  interpreted certain double vector subbundles of $TE$ as linear connections on $E\to M$, and explored the compatibility of the connection with a metric on $E\to M$
\end{remark}

For a subbundle $W\subseteq\mbb{E}$ of a Courant algebroid, modifying the Courant bracket using a pseudo-connection, 
\begin{equation}\label{eq:pcModBrk}[\sigma,\tau]:=\Cour{\sigma,\tau}-\mbf{a}^*\la \nabla\sigma,\tau\ra,\quad\sigma,\tau\in\Gamma(W),\end{equation} 
defines a  $\Gamma(\mbb{E})$-valued bracket on sections of $W$.

\begin{lemma}\label{lem:ModBrkProp1}
Suppose that $\mbb{E}\to M$ is a Courant algebroid and $\nabla$ is a pseudo-connection  for the subbundle $W\subseteq \mbb{E}$. Then the modified bracket \labelcref{eq:pcModBrk} is skew symmetric.
\end{lemma}
\begin{proof}
 Axioms (c3) for the Courant bracket (see \cref{def:CA}) implies that, 
for any $\sigma,\tau\in\Gamma(W)$, we have
\begin{align*}
[\sigma,\tau]+[\tau,\sigma]&=\Cour{\sigma,\tau}+\Cour{\tau,\sigma}-\mbf{a}^*\la \nabla\sigma,\tau\ra-\mbf{a}^*\la \nabla\tau,\sigma\ra\\
&=\mbf{a}^*d\la\sigma,\tau\ra-\mbf{a}^*\la \nabla\sigma,\tau\ra-\mbf{a}^*\la \nabla\tau,\sigma\ra\\
&=0.
\end{align*}
Here the final equality follows from \cref{eq:NabPairDer}.
\end{proof}

Using the modified bracket, we define a `torsion' tensor for the pseudo-connection, \begin{equation}\label{eq:torsTens}T(\sigma,\tau,\upsilon)=\la\nabla_{\mbf{a}(\sigma)}\tau-\nabla_{\mbf{a}(\tau)}\sigma-[\sigma,\tau],\upsilon\ra,\quad\sigma,\tau,\upsilon\in\Gamma(W).\end{equation} 
\begin{proposition}
\Cref{eq:torsTens} defines a skew symmetric tensor. That is, $T\in\Gamma(\wedge^3W^*)$.
\end{proposition}
\begin{proof}
First we show that $T$ is cyclic in its arguments. Combining \cref{eq:pcModBrk,eq:torsTens} yields
\begin{align*}
T(\sigma,\tau,\upsilon)&=\la\nabla_{\mbf{a}(\sigma)}\tau-\nabla_{\mbf{a}(\tau)}\sigma-[\sigma,\tau],\upsilon\ra\\
&=\la\nabla_{\mbf{a}(\sigma)}\tau-\nabla_{\mbf{a}(\tau)}\sigma-\Cour{\sigma,\tau} +\mbf{a}^*\la\nabla\sigma,\tau\ra,\upsilon\ra\\
&=\la\nabla_{\mbf{a}(\sigma)}\tau-\nabla_{\mbf{a}(\tau)}\sigma+\mbf{a}^*\la\nabla\sigma,\tau\ra,\upsilon\ra  -\mbf{a}(\sigma)\la\tau,\upsilon\ra+\la\tau,\Cour{\sigma,\upsilon}\ra,
\end{align*}
where we used axiom (c2) for the Courant bracket in the last line. Hence
\begin{align*}
T(\sigma,\tau,\upsilon)
&=\la\nabla_{\mbf{a}(\sigma)}\tau-\nabla_{\mbf{a}(\tau)}\sigma,\upsilon\ra+\la\nabla_{\mbf{a}(\upsilon)}\sigma,\tau\ra  -\mbf{a}(\sigma)\la\tau,\upsilon\ra+\la\tau,\Cour{\sigma,\upsilon}\ra\\
&=\la\nabla_{\mbf{a}(\sigma)}\tau-\nabla_{\mbf{a}(\tau)}\sigma,\upsilon\ra+\la\nabla_{\mbf{a}(\upsilon)}\sigma,\tau\ra  -\la\nabla_{\mbf{a}(\sigma)}\tau,\upsilon\ra-\la\tau,\nabla_{\mbf{a}(\sigma)}\upsilon\ra+\la\tau,\Cour{\sigma,\upsilon}\ra\\
&=\la\nabla_{\mbf{a}(\upsilon)}\sigma-\nabla_{\mbf{a}(\sigma)}\upsilon+\Cour{\sigma,\upsilon},\tau\ra-\la\nabla_{\mbf{a}(\tau)}\sigma,\upsilon\ra,
\end{align*}
where we used \cref{eq:NabPairDer} in the second line. Finally, using \cref{eq:pcModBrk} again, we get
\begin{align*}
T(\sigma,\tau,\upsilon)
&=\la\nabla_{\mbf{a}(\upsilon)}\sigma-\nabla_{\mbf{a}(\sigma)}\upsilon+[\sigma,\upsilon],\tau\ra\\
&=\la\nabla_{\mbf{a}(\upsilon)}\sigma-\nabla_{\mbf{a}(\sigma)}\upsilon-[\upsilon,\sigma],\tau\ra\\
&=T(\upsilon,\sigma,\tau),
\end{align*}
where we used \cref{lem:ModBrkProp1} to obtain the second equality.
 
\Cref{lem:ModBrkProp1} implies that $T$ is skew symmetric in the first two arguments, and since it is also cyclic, it must be totally skew symmetric. Moreover, it is manifestly tensorial in the last argument, and hence tensorial in all its arguments.
\end{proof}

\begin{lemma}\label{lem:ModBrkProp2}
Suppose that $\mbb{E}\to M$ is a Courant algebroid and $\nabla$ is a pseudo-connection  for the subbundle $W\subseteq \mbb{E}$. Let $L\subseteq T\mbb{E}$ be the  Lagrangian double vector subbundle corresponding to $(W,\nabla)$ via \cref{lem:qIsNab}.
The modified bracket \labelcref{eq:pcModBrk} takes values in $\Gamma(W)$ if and only if $$\Cour{\Gamma_l(L,E),\Gamma_C(L,E)}\subseteq\Gamma(L,E).$$

\end{lemma}
\begin{proof}

Given $\sigma,\tau\in\Gamma(W)$, suppose $\tilde\sigma\in\Gamma_l(L,E)$ is $q_{L/W}$-related to $\sigma$. Since the core of $L$ is  $W^\perp$, any core section of $L\to E$ is of the form $\gamma_C\in\Gamma_C(L,E)$ for $\gamma\in\Gamma(W^\perp)$. Note that $\Cour{\tilde\sigma,\gamma_C}\in\Gamma(L,E)$ if and only if $\la\Cour{\tilde\sigma,\gamma_C},\tau_T\ra=0$. But
\begin{subequations}\label{eq:LinCoreInv}
\begin{align}
\label{eq:Wperp1}\la\Cour{\tilde\sigma,\gamma_C},\tau_T\ra&=\la q(\Cour{\tilde \sigma,\gamma_C},\tau\ra\\
\notag&=\la q(\Cour{\sigma,\gamma}_C+\Cour{\tilde\sigma-\sigma_T,\gamma_C}),\tau\ra\\
\notag&=\la \Cour{\sigma,\gamma}+\Lied_{\mbf{a}(\gamma)_C}q(\sigma_T-\tilde\sigma),\tau\ra\\
\label{eq:Wperp2}&=\la \Cour{\sigma,\gamma}+\Lied_{\mbf{a}(\gamma)_C}q(\sigma_T),\tau\ra\\
\notag&=\la \Cour{\sigma,\gamma}+\nabla_{\mbf{a}(\gamma)}\sigma,\tau\ra\\
\notag&=\la \gamma,-\Cour{\sigma,\tau}+\mbf{a}^*\la\nabla\sigma,\tau\ra\ra\\
\notag&=-\la \gamma,[\sigma,\tau]\ra
\end{align}\end{subequations}
where \labelcref{eq:Wperp1} follows from  \cref{eq:VBCALinP2,eq:restPair}, and the equality \labelcref{eq:Wperp2} holds since $\tilde\sigma\in\Gamma(L,E)$. So $\Cour{\Gamma_l(L,E),\Gamma_C(L,E)}\subseteq\Gamma(L,E)$ if and only if \labelcref{eq:pcModBrk} is a well defined bracket on $\Gamma(W)$.
\end{proof}

\begin{proposition}\label{lem:PsiProp}
Suppose that $\nabla$ is a pseudo-connection  for the subbundle $W\subseteq \mbb{E}$, and that the modified bracket \labelcref{eq:pcModBrk} takes values in $\Gamma(W)$. 
Let $L\subseteq T\mbb{E}$ be the Lagrangian double vector subbundle corresponding to $(W,\nabla)$ via \cref{lem:qIsNab}.
\begin{itemize}
\item
The following expression: \begin{equation*}\begin{split}
\Psi(\sigma,\tau,\upsilon)=&\la[\sigma,\tau],\nabla\upsilon\ra+\la[\upsilon,\sigma],\nabla\tau\ra+\la[\tau,\upsilon],\nabla\sigma\ra\\
&+\iota_{\mbf{a}(\sigma)}d\la\nabla\tau,\upsilon\ra+\iota_{\mbf{a}(\upsilon)}d\la\nabla\sigma,\tau\ra+\iota_{\mbf{a}(\tau)}d\la\nabla\upsilon,\sigma\ra\\
&+d T(\sigma,\tau,\upsilon),
\end{split}\end{equation*}
for $\sigma,\tau,\upsilon\in\Gamma(W)$, defines a skew symmetric tensor on $W$. That is,  $$\Psi\in\Omega^1(M, \wedge^3 W^*).$$
\item $L$ is a Dirac structure if and only if $\Psi=0$.
\end{itemize}
\end{proposition}
\begin{proof}
Let $L_C=T\mbb{E}_C\cap L$, where $T\mbb{E}_C$ is the vertical subbundle of $T\mbb{E}\to TM$, as defined in \cref{prop:FreeVBDir}. Then \cref{lem:ModBrkProp2} implies that $\Cour{\Gamma_l(L,E),\Gamma_C(L,E)}\subseteq \Gamma_l(L_C,E)$, and hence $\Cour{\Gamma_l(L,E),\Gamma_l(L_C,E)}\subseteq \Gamma_l(L_C,E)$. Since $L$ is Lagrangian, we also have $\Cour{\Gamma_l(L_C,E),\Gamma_l(L,E)}\subseteq \Gamma_l(L_C,E)$. Therefore, for $\tilde\sigma,\tilde\tau\in\Gamma_l(L,E)$, the expression $q_{L/W}\Cour{\tilde\sigma,\tilde\tau}$ only depends on $q_{L/W}\circ\tilde\sigma$ and $q_{L/W}\circ\tilde\tau$. Consequently, if $\tilde\sigma,\tilde\tau,\tilde\upsilon\in\Gamma_l(L,E)$ are $q_{L/W}$ related to $\sigma,\tau,\upsilon\in\Gamma(W)$,  the expression 
\begin{equation}\label{eq:LinvLinTens}
-\la\Cour{\tilde\sigma,\tilde\tau},\tilde\upsilon\ra\in \Gamma_l(TM\times \mbb{R},TM)\cong\Omega^1(M)
\end{equation}
 only depends on $\sigma,\tau,\upsilon\in\Gamma(W)$. Hence \labelcref{eq:LinvLinTens} defines a tensor  $\Psi\in\Omega^1(M, \wedge^3 W^*)$ measuring the involutivity of $L$, which we shall now calculate directly. To simplify notation, we let $\sigma'=\tilde\sigma-\sigma_T$, $\tau'=\tilde\tau-\tau_T$ and $\upsilon'=\tilde\upsilon-\upsilon_T$, and remark that $\sigma',\tau',\upsilon'\in\Gamma_l(L_C,E)$. Moreover, since $q(\tilde\sigma)=q(\tilde\tau)=q(\tilde\upsilon)=0$, by \cref{eq:NabDef} we have
\begin{align}\label{eq:qsig'}
q(\sigma')&=-\nabla(\sigma), &
q(\tau')&=-\nabla(\tau),&
q(\upsilon')&=-\nabla(\upsilon).
\end{align}
Plugging these into $\Psi$, we get
\begin{align}
-\Psi(\sigma,\tau,\upsilon)
=&\la\Cour{\tilde\sigma,\tilde\tau},\upsilon'+\upsilon_T\ra\notag\\
=&-\la[\sigma,\tau],\nabla\upsilon\ra+\la\Cour{\sigma,\tau}_T+\Cour{\sigma_T,\tau'}+\Cour{\sigma',\tau_T}+\Cour{\sigma',\tau'},\upsilon_T\ra\notag\\
=&-\la[\sigma,\tau],\nabla\upsilon\ra+\la\Cour{\sigma,\tau}_T,\upsilon_T\ra-\la\Cour{\sigma,\upsilon}_T,\tau'\ra+\la\Cour{\tau,\upsilon}_T,\sigma'\ra\notag\\
&+\mbf{a}(\sigma_T)\la\tau',\upsilon_T\ra-\mbf{a}(\tau_T)\la\sigma',\upsilon_T\ra+\mbf{a}(\upsilon_T)\la\sigma',\tau_T\ra+\la\Cour{\sigma',\tau'},\upsilon_T\ra. \label{eq:PsiLstLine1}
\end{align}
To obtain the last line we rearranged terms using axioms (c2) and (c3) for a Courant algebroid (see \cref{def:CA}). Using (c2) and (c3) again, we notice that $\la\Cour{\sigma',\tau'},\upsilon_T\ra=\la\mbf{a}^*d\la\tau',\upsilon_T\ra,\sigma'\ra-\la\mbf{a}^*d\la\sigma',\upsilon_T\ra,\tau'\ra+\la\tau',\Cour{\upsilon_T,\sigma'}\ra$, and the third term vanishes since $\Cour{\upsilon_T,\sigma'}\in\Gamma(T\mbb{E}_C,E)$. Next, using \cref{eq:restPair,eq:qsig'} notice that $\la\sigma',\upsilon_T\ra=\la\nabla\sigma,\upsilon\ra$ and hence $\mbf{a}^*d\la\sigma',\upsilon_T\ra=\mbf{a}^*\la\nabla\sigma,\upsilon\ra_T$.
Substituting the corresponding expressions for various permutations of $\sigma,\tau$ and $\upsilon$ into \cref{eq:PsiLstLine1}, we get
\begin{align*}
-\Psi(\sigma,\tau,\upsilon)=&-\la[\sigma,\tau],\nabla\upsilon\ra+\la\Cour{\sigma,\tau}_T,\upsilon_T\ra-\la\Cour{\sigma,\upsilon}_T,\tau'\ra+\la\Cour{\tau,\upsilon}_T,\sigma'\ra\\
&-\mbf{a}(\sigma_T)\la\nabla\tau,\upsilon\ra+\mbf{a}(\tau_T)\la\nabla\sigma,\upsilon\ra-\mbf{a}(\upsilon_T)\la\nabla\sigma,\tau_T\ra\\
&-\la\mbf{a}^*\la\nabla\tau,\upsilon\ra_T,\sigma'\ra+\la\mbf{a}^*\la\nabla\sigma,\upsilon\ra_T,\tau'\ra.
\end{align*}
Using the definition \labelcref{eq:pcModBrk} of the bracket, we get
\begin{align}
-\Psi(\sigma,\tau,\upsilon)=&-(\la[\sigma,\tau],\nabla\upsilon\ra+\la[\upsilon,\sigma],\nabla\tau\ra+\la[\tau,\upsilon],\nabla\sigma\ra)\notag\\
&+d\la[\sigma,\tau],\upsilon\ra+d\la\nabla_{\mbf{a}(\upsilon)}\sigma,\tau\ra\notag\\
&-\mbf{a}(\sigma_T)\la\nabla\tau,\upsilon\ra+\mbf{a}(\tau_T)\la\nabla\sigma,\upsilon\ra-\mbf{a}(\upsilon_T)\la\nabla\sigma,\tau\ra.\label{eq:PsiLstLine2}
\end{align}

Now $\mbf{a}(\sigma_T)\la\nabla\tau,\upsilon\ra=\Lied_{\mbf{a}(\sigma)}\la\nabla\tau,\upsilon\ra=d\la\nabla_{\mbf{a}(\sigma)}\tau,\upsilon\ra+\iota_{\mbf{a}(\sigma)}d\la\nabla\tau,\upsilon\ra$. Substituting the corresponding expression for various permutations of $\sigma,\tau$ and $\upsilon$ into \cref{eq:PsiLstLine2}, we get
%
%
\begin{equation}\label{eq:Psi_La}
\begin{split}
\Psi(\sigma,\tau,\upsilon)=&\la[\sigma,\tau],\nabla\upsilon\ra+\la[\upsilon,\sigma],\nabla\tau\ra+\la[\tau,\upsilon],\nabla\sigma\ra\\
&+\iota_{\mbf{a}(\sigma)}d\la\nabla\tau,\upsilon\ra+\iota_{\mbf{a}(\upsilon)}d\la\nabla\sigma,\tau\ra+\iota_{\mbf{a}(\tau)}d\la\nabla\upsilon,\sigma\ra\\
&+d T(\sigma,\tau,\upsilon),
\end{split}\end{equation}
\end{proof}

\begin{remark} If $W\subseteq\mbb{E}$ is quadratic, then a pseudo-connection is simply a metric connection on $W$. Using the identity $$d\la\nabla\tau,\upsilon\ra=\la R\tau,\upsilon\ra-\la\nabla\tau\wedge\nabla\upsilon\ra,$$ where $R\in\Omega^2(M,\mf{o}(W))$ is the curvature tensor, one may rewrite \cref{eq:Psi_La} as
\begin{equation}\label{eq:PsiMet}
\Psi(\sigma,\tau,\upsilon)=\iota_{\mbf{a}(\sigma)}\la R\tau,\upsilon\ra+\iota_{\mbf{a}(\upsilon)}\la R\sigma,\tau\ra+\iota_{\mbf{a}(\tau)}\la R\upsilon,\sigma\ra
+(\nabla T)(\sigma,\tau,\upsilon).
\end{equation}
So, in this case, $\Psi\in\Omega^1(M,\wedge^3 W^*)$ can be expressed entirely in terms of the curvature and torsion of the connection.
\end{remark}

\begin{remark} Suppose that $W\subseteq\mbb{E}$ is a subbundle carrying a pseudo-connection such that the modified bracket \labelcref{eq:pcModBrk} takes values in $\Gamma(W)$. The tensor $\Psi$ can be understood as an obstruction to the modified bracket \labelcref{eq:pcModBrk} satisfying the Jacobi identity. Indeed, for $\sigma,\tau,\upsilon\in\Gamma(W)$, we have
$$[\sigma,[\tau,\upsilon]]+[\tau,[\upsilon,\sigma]]+[\upsilon,[\sigma,\tau]]=-\mbf{a}^*\Psi(\sigma,\tau,\upsilon).$$

\end{remark}

\begin{definition}\label{def:LieSubalg}
Suppose $\mbb{E}\to M$ is a Courant algebroid.
A pair, $(W,\nabla)$, consisting of a subbundle $W\subseteq \mbb{E}$ together with pseudo-connection $\nabla$ for $W\subseteq \mbb{E}$ (cf. \cref{def:pseuCon}) is called a \emph{pseudo-Dirac structure} in $\mbb{E}$
if
\begin{itemize}
\item the modified bracket
\labelcref{eq:pcModBrk} takes values in $\Gamma(W)$, and
\item  the tensor $\Psi\in\Omega^1(M, \wedge^3W^*)$ defined by \begin{equation}\label{eq:Psi_L}
\begin{split}
\Psi(\sigma,\tau,\upsilon)=&\la[\sigma,\tau],\nabla\upsilon\ra+\la[\upsilon,\sigma],\nabla\tau\ra+\la[\tau,\upsilon],\nabla\sigma\ra\\
&+\iota_{\mbf{a}(\sigma)}d\la\nabla\tau,\upsilon\ra+\iota_{\mbf{a}(\upsilon)}d\la\nabla\sigma,\tau\ra+\iota_{\mbf{a}(\tau)}d\la\nabla\upsilon,\sigma\ra\\
&+d T(\sigma,\tau,\upsilon),
\end{split}\end{equation}
for $\sigma,\tau,\upsilon\in\Gamma(W)$ vanishes.
\end{itemize}
\end{definition}

The terminology in \cref{def:LieSubalg} will be justified by the following theorem, which summarizes many of the results above.
\begin{theorem}\label{thm:LieSubIsVBDir}
Suppose $\mbb{E}\to M$ is a Courant algebroid.
\begin{itemize}
\item If $(W,\nabla)$ is a pseudo-Dirac structure for $\mbb{E}$, then 
$$[\sigma,\tau]:=\Cour{\sigma,\tau}-\mbf{a}^*\la \nabla\sigma,\tau\ra,\quad\sigma,\tau\in\Gamma(W),$$
 defines a Lie algebroid bracket on $W$.
\item There is a one-to-one correspondence between $\mc{VB}$-Dirac structures 
\begin{equation}\label{eq:VBDirTngThm}
\begin{tikzpicture}
\mmat{m1} at (-2,0) {L&W\\ TM& M\\};
\path[->] (m1-1-1)	edge node {$q_{L/W}$} (m1-1-2)
				edge (m1-2-1);
\path[<-] (m1-2-2)	edge (m1-1-2)
				edge (m1-2-1);
\mmat{m2} at (2,0) {T\mbb{E}&\mbb{E}\\ TM& M\\};
\path[->] (m2-1-1)	edge (m2-1-2)
				edge (m2-2-1);
\path[<-] (m2-2-2)	edge (m2-1-2)
				edge (m2-2-1);
				
\draw (0,0) node {$\subseteq$};
\end{tikzpicture}
\end{equation} 
and pseudo-Dirac structures in $\mbb{E}$. 
Under this correspondence, the map $q_{L/W}:L\to W$ is a Lie algebroid morphism.
\end{itemize}

\end{theorem}
\begin{proof}
\Cref{lem:qIsNab,lem:PsiProp} establish a one-to-one correspondence between $\mc{VB}$-Dirac structures of the form \labelcref{eq:VBDirTngThm}, and pseudo-Dirac structures $(W,\nabla)$ in $\mbb{E}$. 

Now, suppose that $L\subseteq T\mbb{E}$ is a $\mc{VB}$-Dirac structure. \Cref{prop:VBDirIsLaVB} implies that $L$ is a $\mc{VB}$-Lie algebroid, and that there exists a unique Lie algebroid structure on $W\to M$ such that the map $q_{L/W}:L\to W$ is a Lie algebroid morphism. Let $$[\cdot,\cdot]:\Gamma(W)\times \Gamma(W)\to\Gamma(W)$$ denote the corresponding Lie bracket on sections of $W$.
For any sections $\sigma,\tau\in\Gamma(W)$, the sections $\sigma_T,\tau_T\in\Gamma(T\mbb{E})$ are $q_{L/W}$-related to $\sigma$ and $\tau$, respectively. Thus, by \cref{prop:qIndBrk},
\begin{align*}
[\sigma,\tau]&=q_{L/W}\Cour{\sigma_T,\tau_T}-\mbf{a}^*\la q\sigma_T,\tau_T\ra\\
&=q_{L/W}\Cour{\sigma,\tau}_T-\mbf{a}^*\la \nabla\sigma,\tau\ra\\
&=\Cour{\sigma,\tau}-\mbf{a}^*\la \nabla\sigma,\tau\ra,
\end{align*}
where  $q:L\to W^*$ is the map \labelcref{eq:LagQuo} defined by \cref{prop:clasLagDB}, and $\nabla$ is the pseudo-connection defined by \cref{eq:NabDef}.
Therefore \cref{eq:pcModBrk} defines a Lie algebroid structure on $W$, and $q_{L/W}:L\to W$ is a Lie algebroid morphism.
%
%
%
%
%
%
\end{proof}

\begin{remark}
Suppose $A$ is a Lie algebroid (with bracket $[\cdot,\cdot]$ and anchor $\mbf{a}$),  $\la\cdot,\cdot\ra$ is a metric on the fibres of $A$, and $\nabla$ is a metric connection on $A$. Then
$$\Cour{\sigma,\tau}:=[\sigma,\tau]+\mbf{a}^*\la \nabla\sigma,\tau\ra,\quad\sigma,\tau\in\Gamma(A),$$
 defines a Courant bracket on $A$ if and only if
\begin{itemize}
\item $\mbf{a}\circ\mbf{a}^*=0$ ($A$ acts with coisotropic stabilizers),
\item the `torsion' tensor \labelcref{eq:torsTens} is skew symmetric, and
\item $\mbf{a}^*\Psi=0$, (where $\Psi$  is defined  in terms of the curvature and torsion by \cref{eq:PsiMet}).
\end{itemize}
Indeed axiom (c3) for a Courant bracket (see \cref{def:CA}) holds since $\nabla$ is a metric connection, axiom (c2) holds if and only if the `torsion' tensor is skew symmetric, and axiom (c1) holds if and only if $\mbf{a}\circ\mbf{a}^*=0$ and $\mbf{a}^*\Psi=0$.

It is perhaps more natural to require that $\Psi=0$, in which case $(A,\nabla)$ is embedded as a pseudo-Dirac structure in the corresponding Courant algebroid.
\end{remark}

\begin{proposition}\label{prop:fltSec}
Suppose that $(W,\nabla)$ is a pseudo-Dirac structure  in the Courant algebroid $\mbb{E}$. If $\sigma,\tau\in\Gamma(W)$ satisfy $\nabla\sigma=\nabla\tau=0$  then 
\begin{itemize}
\item $\nabla\Cour{\sigma,\tau}=0$, and
\item $\Cour{\sigma,\tau}=[\sigma,\tau]$.
\end{itemize} That is, the `flat' sections of $W$ form a Lie algebra with respect to the Courant bracket.
\end{proposition}
\begin{proof}
First, if $\nabla\sigma=0$, then $[\sigma,\tau]:=\Cour{\sigma,\tau}-\mbf{a}^*\la\nabla\sigma,\tau\ra=\Cour{\sigma,\tau}$.

Next, let $L\subset T\mbb{E}$ be the $\mc{VB}$-Dirac structure corresponding to $(W,\nabla)$. By \cref{eq:NabDef}, we see that $\nabla\sigma=0$ if and only if $\sigma_T\in\Gamma(L,E)$. Therefore, $\sigma_T,\tau_T\in\Gamma(L,E)$, and thus  $\Cour{\sigma_T,\tau_T}=\Cour{\sigma,\tau}_T\in \Gamma(L,E)$. In turn, this shows that $\nabla\Cour{\sigma,\tau}=0$.
\end{proof}

\subsection{Examples}

\begin{example}
If $W\subseteq\mbb{E}$ is a Dirac structure, then after endowing $W$ with the trivial pseudo-connection, $(W,0)$ is a pseudo-Dirac structure corresponding to the $\mc{VB}$-Dirac structure $L=TW\subseteq T\mbb{E}$ (see \cref{ex:DirStrTngLf}).
\end{example}

\begin{example}\label{ex:BarLieSub}
If $(W,\nabla)$ is a pseudo-Dirac structure for the Courant algebroid $\mbb{E}$, then $(W,-\nabla)$ is a pseudo-Dirac structure for $\overline{\mbb{E}}$.
\end{example}

\begin{example}
If $\mf{d}$ is a quadratic Lie algebra, then any Lie subalgebra $\g\subset\mf{d}$ is a pseudo-Dirac structure.
\end{example}


\begin{example}[Action Courant algebroids]\label{ex:LieSubAct}
Suppose $\mf{d}$ is a quadratic Lie algebra which acts on a manifold $M$ with coisotropic stabilizers. In this case, as explained by Meinrenken and the author in \cite{LiBland:2009ul}, the bundle $\mf{d}\times M$ is naturally a Courant algebroid (see \cref{ex:ActCourAlg} for details). If $\h\subseteq\mf{d}$ is any subalgebra, then $(\h\times M,d)$ is a pseudo-Dirac structure in $\mf{d}\times M$, where $d$ is the standard connection on the trivial bundle $\mf{d}\times M$. As a Lie algebroid, $(\h\times M,d)$ is isomorphic to the action Lie algebroid, as can be seen by comparing \labelcref{eq:pcModBrk} with \labelcref{eq:actioncourant}.
\end{example}

The following is a special case of the last example.
\begin{example}[Dirac Lie groups]
Dirac Lie groups for which multiplication is a morphism of Manin pairs were classified in \cite{LiBland:2010wi, LiBland:2011vqa}  (see also \cite{Ortiz:2008bd, Jotz:2009va} for a different setting). There it was shown that the underlying Courant algebroid can be canonically trivialized as an action Courant algebroid  $\mbb{A}=\mf{q}\times H$ (see \cref{ex:ActCourAlg}), and the Dirac structure is a constant subbundle  $E=\g\times H$ under this trivialization.

As such, $\mf{d}\times H$ and $\g\times H$ both define pseudo-Dirac structures in $\mbb{A}$. Moreover, if $\mf{r}\subset\mf{q}$ is the Lie subalgebra transverse to $\g$ described in \cite[Section~3.2]{LiBland:2011vqa}, then $\mf{r}\times H$ describes a pseudo-Dirac structure which does not correspond to any Dirac structure in $\mbb{A}$. 
\end{example}

\begin{example}[Cotangent Lie algebroids]
Suppose that $T^*M$ carries the structure of a Lie algebroid with anchor map $\mbf{a}':T^*M\to TM$. Then $W=\gr(\mbf{a}')\subset \mbb{T}M$ is a pseudo-Dirac structure, where the pseudo-connection is defined by \cref{eq:pcModBrk}. That is
$$\la\nabla\alpha,\beta\ra=\Lied_{\mbf{a}'(\alpha)}\beta-\iota_{\mbf{a}'(\beta)}d\alpha-[\alpha,\beta].$$

Since $T^*M$ is a Lie algebroid, $TM$ carries a linear Poisson structure with a bivector field $\pi\in\mf{X}^2(TM)$. The graph of $\pi^\sharp:T^*TM\to TTM$, $$L=\gr(\pi^\sharp)\in\mbb{T}TM\cong T\mbb{T}M$$ is the $\mc{VB}$-Dirac structure corresponding to the pseudo-Dirac structure $(T^*M,\nabla)$.

This example arises in q-Poisson geometry. \v{S}evera and the author \cite{LiBland:2010wi} showed that given any q-Poisson $(\g\oplus\bar\g,\g_\Delta)$-structure on a manifold $M$ (see \cref{ex:qPstr} for details), $T^*M$ carries the structure of a Lie algebroid. Of significance is that $T^*M$ cannot generally be identified with a Dirac structure in any exact Courant algebroid, which contrasts the case of a Poisson structure on $M$. However, both Poisson and q-Poisson structures on $M$ endow $T^*M$ with the structure of pseudo-Dirac structure, as we shall soon explain in more detail in \cref{ex:qPCotLie}.
\end{example}

\subsection{Forward and backward images of  pseudo-Dirac structures}\label{sec:FBimage}
The properties of Dirac structures which allow you to compose them with Courant relations extend to pseudo-Dirac structures, as we shall explain in this section. Indeed, morally, 
any procedure for Dirac structures carries over to the more general pseudo-Dirac structures, since the latter \emph{are in fact just Dirac structures} in the tangent prolongation of the Courant algebroid.
Special cases of composing with Courant relations include forward and backward Dirac maps, or restricting a Dirac structure to a submanifold. 

Recall from \cite{LiBland:2011vqa} (or \cref{sec:CARel}) that a Courant relation $R:\mbb{E}\dasharrow \mbb{F}$ between two Courant algebroids $\mbb{E}\to M$ and $\mbb{F}\to N$ is a Dirac structure $R\subseteq \mbb{F}\times\overline{\mbb{E}}$ with support along a submanifold. In particular, a Dirac structure $L\subseteq \mbb{E}$ is a Courant relation $L:\ast\dasharrow\mbb{E}$, where $\ast$ is the trivial Courant algebroid.

If two Courant relations $R_1:\mbb{E}_1\dasharrow \mbb{E}_2$ and $R_2:\mbb{E}_2\dasharrow\mbb{R}_3$ compose \emph{cleanly}, then their composition $R_2\circ R_1:\mbb{E}_1\dasharrow \mbb{R}_3$ is a Courant relation (see \cite[Proposition~1.4]{LiBland:2011vqa}). Furthermore: 

\begin{proposition}[\!\!{\cite[Proposition~1.4]{LiBland:2011vqa}}]
Suppose that $\mbb{E}\to M$ and $\mbb{F}\to N$ are Courant algebroids and
$R:\mbb{E}\dasharrow\mbb{F}$ is a Courant relation.
\begin{itemize}
\item If $E\subseteq \mbb{E}$ is a Dirac structure which composes cleanly with $R$ and the subbundle $F\subseteq \mbb{F}$ supported on all of $N$ satisfies $F=R\circ E$, then $F\subseteq \mbb{F}$ is a Dirac structure.
\item If $F\subseteq \mbb{F}$ is a Dirac structure which composes cleanly with $R$ and the subbundle $E\subseteq \mbb{E}$ supported on all of $M$ satisfies $E=F\circ R$, then $E\subseteq \mbb{E}$ is a Dirac structure.
\end{itemize}
\end{proposition}

%

Suppose now that $(W,\nabla)$ is a pseudo-Dirac structure in the Courant algebroid $\mbb{E}$, and $R:\mbb{E}\dasharrow \mbb{F}$ is a Courant relation with support on $S:M\dasharrow N$. Let $L\subset T\mbb{E}$ be the $\mc{VB}$-Dirac structure corresponding to $(W,\nabla)$ (c.f. \cref{thm:LieSubIsVBDir}). From \cref{ex:DirStrTngLf} we see that $TR:T\mbb{E}\dasharrow T\mbb{F}$ is a $\mc{VB}$-Dirac structure with support on $TS:TM\dasharrow TN$. 
\begin{definition}
We say that $R:\mbb{E}\dasharrow \mbb{F}$ \emph{composes cleanly} with the pseudo-Dirac structure $(W,\nabla)$ if $TR:T\mbb{E}\dasharrow T\mbb{F}$ composes cleanly with $L$.
\end{definition}
Assume that the composition $TR\circ L$ is clean, and equal to a subbundle $L'\subseteq T\mbb{F}$ supported on all of $TN$. Then the Proposition above shows that $L'$ is a $\mc{VB}$-Dirac structure, which in turn corresponds to a pseudo-Dirac structure $(W',\nabla')$ in $\mbb{F}$ (c.f. \cref{thm:LieSubIsVBDir}). In this case, we write $$(W',\nabla')=R\circ(W,\nabla),$$ and call it the \emph{forward image} along $R:\mbb{E}\dasharrow \mbb{F}$.

Conversely, suppose we instead start with a pseudo-Dirac structure $(W',\nabla')$ in $\mbb{E}$ corresponding to a $\mc{VB}$-Dirac structure $L'\subseteq T\mbb{F}$ (c.f. \cref{thm:LieSubIsVBDir}). 
\begin{definition}
We say that $R:\mbb{E}\dasharrow \mbb{F}$ \emph{composes cleanly} with the pseudo-Dirac structure $(W',\nabla')$ if $TR:T\mbb{E}\dasharrow T\mbb{F}$ composes cleanly with $L'$.
\end{definition}
Assume that the composition $L'\circ TR$ is clean, and equal to a subbundle $L\subseteq T\mbb{E}$ supported on all of $TM$. Then the Proposition above shows that $L$ is a $\mc{VB}$-Dirac structure, which in turn corresponds to a pseudo-Dirac structure $(W,\nabla)$ in $\mbb{E}$ (c.f. \cref{thm:LieSubIsVBDir}). In this case, we write $$(W,\nabla)=(W',\nabla')\circ R,$$ and call it the \emph{backward image} along $R:\mbb{E}\dasharrow \mbb{F}$.

The following proposition describes a useful equation satisfied by the pseudo-connections associated  to forward and backward images of a pseudo-Dirac structure along a Courant relation.

\begin{proposition}
Suppose $(W,\nabla)$ and $(W',\nabla')$ are pseudo-Dirac structures in the Courant algebroids $\mbb{E}\to M$ and $\mbb{F}\to N$, respectively.
 Let $L\subseteq T\mbb{E}$ and $L'\subseteq T\mbb{F}$ be the respective $\mc{VB}$-Dirac structures corresponding to $(W,\nabla)$ and $(W',\nabla')$ (c.f. \cref{thm:LieSubIsVBDir}).
Suppose $R:\mbb{E}\dasharrow\mbb{F}$ is a Courant relation over $S:M\dasharrow N,$ and that the intersection $$(L'\times L)\cap TR$$ is clean.
Then 
for any section $(\sigma',\sigma)\in\Gamma(W'\times W)$ satisfying 
 $$(\sigma',\sigma)\rvert_S\in\Gamma\big((W'\times W)\cap R\big),$$ and any  element $$\big(X;(\lambda',\lambda)\big)\in TS\times_M \big((W'\times W)\cap R\big),$$ the following equation holds:
\begin{equation}\label{eq:nablaComposition}\la (p_M^*\nabla')_{X}\sigma',\lambda'\ra=\la (p_N^*\nabla)_{X}\sigma,\lambda\ra.\end{equation} Here $p_N^*\nabla$ and $p_M^*\nabla'$ are the pull-backs of the respective pseudo-connections along the natural projections $p_N:N\times M\to N$ and $p_M:N\times M\to M$ (c.f. \cref{def:pseuCon}).

\end{proposition}
\begin{proof}
%
%
%

%
%

 Let $$q:T(\mbb{F}\times\overline{\mbb{E}})\rvert_{W'\times W}\to (W'\times W)^*$$ be the map defined in \cref{prop:clasLagDB} corresponding to the Lagrangian double vector subbundle $L'\times L\subseteq T(\mbb{F}\times\overline{\mbb{E}})$.
Now the $\mc{VB}$-Dirac structure $L'\times L$ corresponds to the pseudo-Dirac structure $\big(W'\times W,(-\nabla')\oplus\nabla\big)$ (c.f. \cref{thm:LieSubIsVBDir,ex:BarLieSub}). Thus, by \cref{eq:NabDef}, for any section $(\sigma',\sigma)\in\Gamma(W'\times W)$ and any $X\in T(N\times M)$, we have \begin{equation}\label{eq:qLieSubCom}q(\sigma'_T,\sigma_T)\rvert_{X}=\big((-\nabla')\oplus\nabla\big)_{X}(\sigma',\sigma)=\big(-(p_N^*\nabla')_{X}\sigma',(p_M^*\nabla)_{X}\sigma\big).\end{equation}

Next, by the clean intersection assumption, $$\begin{tikzpicture} \mmat{m2}{\big(L'\times L\big)\cap TR &(W'\times W)\cap R\\ TS&S\\};
\path[->]
	(m2-1-1) edge  (m2-1-2)
		edge (m2-2-1);
\path[<-] 
	(m2-2-2) edge  (m2-1-2)
		edge (m2-2-1);
\end{tikzpicture}$$ is a double vector subbundle of $T(\mbb{F}\times\overline{\mbb{E}})$. Thus for any $$\big(X;(\lambda',\lambda)\big)\in TS\times_M \big((W'\times W)\cap R\big),$$ there exists a pair $(\tilde\lambda',\tilde\lambda)\in (L'\times L)\cap TR$ which is simultaneously mapped to both $X$ and $(\lambda',\lambda)$ by the respective bundle maps, as pictured in the following diagram:
$$\begin{tikzpicture}
\mmat{m1} at (-3,0) {(\tilde\lambda',\tilde\lambda) &(\lambda',\lambda)\\ X&\\};
\path[->]
	(m1-1-1) edge  (m1-1-2)
		edge (m1-2-1);

\draw (0,0) node {$\in$};

\mmat{m2} at (4,0) {\big(L'\times L\big)\cap TR &(W'\times W)\cap R\\ TS&S\\};
\path[->]
	(m2-1-1) edge  (m2-1-2)
		edge (m2-2-1);
\path[<-] 
	(m2-2-2) edge  (m2-1-2)
		edge (m2-2-1);
\end{tikzpicture}$$
Additionally, for any $(\sigma',\sigma)\in\Gamma(W'\times W)$ satisfying $(\sigma',\sigma)\rvert_S\in\Gamma\big((W'\times W)\cap R\big)$,  $$(\sigma'_T,\sigma_T)\rvert_{TS}\in\Gamma(TR,TS).$$ Consequently, since $TR$ is Lagrangian, \begin{equation}\label{eq:LieSubCompZero}\la(\sigma'_T,\sigma_T),(\tilde\lambda',\tilde\lambda)\ra=0.\end{equation}

 Since $(\tilde\lambda',\tilde\lambda)\in (L'\times L)$, we have $q(\tilde\lambda',\tilde\lambda)=0$. Hence, \cref{eq:LieSubCompZero,eq:restPair} imply that 
$$\la q(\sigma'_T,\sigma_T),(\lambda',\lambda)\ra=0.$$
 
 Combining this equality with \cref{eq:qLieSubCom} yields
 $$\la\big(-(p_N^*\nabla')_{X}\sigma',(p_M^*\nabla)_{X}\sigma\big), (\lambda',\lambda)\ra=0,$$
 which proves the proposition.
\end{proof}

Now we specialize to the case where $R:\mbb{E}\dasharrow \mbb{F}$ is a Courant morphism supported on the graph of a map $\phi:M\to N$. Suppose that $(W',\nabla')$ is a pseudo-Dirac structure in $\mbb{F}$ which composes cleanly with $R:\mbb{E}\dasharrow\mbb{F}$, and $\on{ran}(R)+ \phi^*W'^\perp=\phi^*\mbb{F}$. 

Since $(W',\nabla')$ composes cleanly with $R$, the intersection $\phi^*W'\cap \on{ran}(R)\subseteq \phi^*\mbb{F}$ is a smooth subbundle. Since $R$ is supported on the graph of a map $\phi:M\to N$, the subbundle
$$W:=W'\circ R\subseteq \mbb{E}$$ is well defined with support on all of $M$.
\begin{lemma}
There is a bundle map $\Psi:W\to \phi^* W'$, uniquely determined by the equation $$(\Psi(\lambda),\lambda)\in R_{(\phi(x),x)},$$ for every $\lambda\in W_x.$
\end{lemma}
\begin{proof}
The proof is entirely analogous to that of \cite[Lemma 7.2]{Bursztyn03-1} for Dirac realizations.
Consider the relation $R':W\dasharrow W'$, supported on the graph of $\phi:M\to N$, defined by
$$R':=R\cap (\phi^* W'\times W).$$ We will show that $R'$ is the graph of a bundle map $\Psi:W\to \phi^*W'$, i.e.  the natural projection $R'\to W$ is an isomorphism.

First, we establish surjectivity: let $\lambda\in W_x$. Since $W:=W'\circ R$, there exists some $\lambda'\in W'_{\phi(x)}$ such that the pair $(\lambda',\lambda)\in R_{(\phi(x),x)}$.

Next we establish injectivity. Suppose there exists $\lambda''\in W'_{\phi(x)}$ such that the pair $(\lambda'',\lambda)\in R_{(\phi(x),x)}$. Then $(\lambda''-\lambda',0)\in R\cap (\phi^*W'\times 0)$. However, 
$$\big(R\cap (\phi^*W'\times 0)\big)^\perp=(\on{ran}(R)+ \phi^*W'^\perp)\times \overline{\mbb{E}}=\mbb{F}\times\overline{\mbb{E}}.$$
Hence $R\cap (\phi^*W'\times 0)=0$ and $\lambda''=\lambda'$.
\end{proof}

%

In this case, \cref{eq:nablaComposition} specializes to 
$$\nabla=\Psi^*\circ(\phi^*\nabla')\circ\Psi.$$ 

Let $L\subseteq T\mbb{E}$ and $L'\subseteq T\mbb{F}$ be the $\mc{VB}$-Dirac structures corresponding to the pseudo-Dirac structures $(W,\nabla)$ and $(W',\nabla')$. As an intersection of involutive subbundles, it is clear that $$TR\cap (L'\times L)=\gr(T\Psi)$$ is a Lie subalgebroid of $L'\times L$. In particular, the intersection $\gr(\Psi)=\gr(T\Psi)\cap(W'\times W)$ is a Lie subalgebroid of $W'\times W$. Thus $$\Psi:W\to W'$$ is a morphism of Lie algebroids.

\begin{example}[q-Poisson structures and Cotangent Lie algebroids]\label{ex:qPCotLie}
Suppose that $\mf{d}$ is a quadratic Lie algebra, $\g\subset\mf{d}$ is a Lagrangian Lie subalgebra, and $$R:(\mbb{T}M,TM)\dasharrow (\mf{d},\g)$$ is a morphism of Manin pairs. That is to say, $R$ defines a q-Poisson $(\mf{d},\g)$-structure on $M$, as in \cref{ex:qPstr}. In particular, $R$ defines an action $\rho:\g\times M\to TM$ of $\g$ on $M$.

If $\h\subset\mf{d}$ is any subalgebra transverse to $\g$, then $\on{ran}(R)+\h^\perp=\mf{d}$, so $F=\h\circ R\subseteq \mbb{T}M$ is a subbundle transverse to $TM$. Moreover, $(F,\nabla)$ is a pseudo-Dirac structure, where
$$\nabla\sigma=\rho d\rho^*\sigma,\quad\sigma\in\Gamma(F),$$ and we have used the metric to identify $\h\cong \g^*$.
Thus, any choice of Lie subalgebra $\h\subset\mf{d}$ transverse to $\g$ endows $T^*M\cong F$ with the structure of a Lie algebroid so that both 
\begin{align*}
\rho^*:T^*M&\to \h,\\
\rho:\g\times M&\to TM
\end{align*}
are morphisms of Lie algebroids.

This extends \cite[Proposition~6.1]{Xu95} (see also \cite{Bursztyn:2009wi}) for the case where $\h$ is Lagrangian.
Notice that, if $\h$ is Lagrangian, then $\nabla=0$ and $F$ is Lagrangian, so $F$ is a Dirac structure. Moreover, since $F$ is transverse to $TM$, $F=\gr(\pi^\sharp)$ for a Poisson bivector field $\pi\in\mf{X}^2(M)$. In addition to this, $(\mf{d},\g,\h)$ forms a Manin triple, whence we see that $\g$ is a Lie-bialgebra. In fact, the action $\rho:\g\times M\to TM$ defines a Lie-bialgebra action on the Poisson manifold $M$, as shown in \cite{Xu95,Bursztyn:2009wi}.

This also extends \cite[Theorem~1]{LiBland:2010wi}, where $\mf{k}$ is a quadratic Lie algebra, $\mf{d}=\mf{k}\oplus\overline{\mf{k}}$, $\g=\mf{k}_\Delta$ is the diagonal subalgebra, and $\h=0\oplus\overline{\mf{k}}$. 
\end{example}

\begin{remark}[Supergeometric Interpretation]\label{rem:SupPsDir}
Recall from \cite{NonComDiffForm,Bursztyn:2009wi} that the category of Manin pairs is equivalent to to the category of Poisson principal $\mbb{R}[2]$ bundles, $P\to E^*[1]$, over degree 1 $N$-manifolds (see \cref{rem:SupMP} for a short summary).

The following table lists some one-to-one correspondences between structures associated with  a Manin pair $(\mbb{E},E)$ and the corresponding Poisson principal $\mbb{R}[2]$ bundle $P\to E^*[1]$.
\begin{center}
\begin{tabular}{p{.45\textwidth}|p{.45\textwidth}}
Manin pair: $(\mbb{E},E)$ & Poisson principal $\mbb{R}[2]$ bundle: $P\to E^*$\\\hline\hline
Lagrangian subbundles $F\subseteq \mbb{E}$ transverse to $E$. & Flat connections on $P$.\\\hline
Dirac subbundles $F\subseteq \mbb{E}$ transverse to $E$. & Flat connections on $P$ such that the connection one-form $\theta\in \Omega^1(P,\mbb{R}[2])$ satisfies $$[\theta,\theta]=0,$$ (where the bracket is the graded Koszul bracket \cref{eq:Koszul}).\\\hline
 pseudo-Dirac structures $(F,\nabla)\subseteq\mbb{E}$ transverse to $E$. & Connections on $P$ such that the connection one-form $\theta\in \Omega^1(P,\mbb{R}[2])$ satisfies $$[\theta,\theta]=0.$$
\end{tabular}\end{center}
The first two correspondences are explained in \cite{NonComDiffForm,Bursztyn:2009wi}. For the third correspondence, suppose that the section $\theta:P\to T^*[2]P$ is a connection one-form. Then $[\theta,\theta]=0$ if and only if the image of the section $$\tilde F:=\theta(P)\subseteq T^*[2]P$$ is $Q$ invariant. In this case, the reduced submanifold 
$$F[1]:=\tilde F/\mbb{R}[2]\subseteq T^*[2]P/\!\!/_1\mbb{R}[2]$$
 is $Q$ invariant, i.e. $F[1]$ corresponds to a pseudo-Dirac structure via \cref{prop:Lie2VBCour}. It is transverse to $E$ since $F[1]$ is the graph of a map $E^*[1]\to T^*[2]P/\!\!/_1\mbb{R}[2]$. 
 
 Conversely, suppose $\theta_0:E^*[1]\to T^*[2]P/\!\!/_1\mbb{R}[2]$ is a section whose image is $Q$-invariant. Let $\theta:P\to T^*[2]P$ be the unique lift to an $\mbb{R}[2]$ invariant section at moment level $1$. It is clear that the image of $\theta$ is $Q$ invariant, so it defines a connection satisfying $[\theta,\theta]=0$.
 
 Suppose that $\theta$ is a connection one-form on $P\to E^*[1]$ satisfying $[\theta,\theta]=0$. Then the corresponding  Hamiltonian vector field $X_\theta$ is an $\mbb{R}[2]$ invariant homological vector field. The induced $Q$ structure on $E^*[1]$ corresponds to the Lie algebroid structure on $F\cong E^*$. 
 Meanwhile, the curvature $d\theta\in\Omega^2(E^*[1],\mbb{R}[2])$ defines a metric on the fibres of $E^*$ via $$\la\sigma,\tau\ra=\iota_\sigma\iota_\tau d\theta, \quad \sigma,\tau\in\Gamma(E^*),$$
 where $\iota_\sigma\in\mf{X}(E^*[1])$ is the vector field defined in \cref{rem:SupLie}.
  This is precisely the restriction of the fibre metric on $\mbb{E}$ to  $F\cong E^*$.
  Finally, the pseudo-connection $\nabla:\Omega^0(M,E^*)\to \Omega^1(M,E)$ is given by the formula
  $$\la\nabla\sigma,\tau\ra=\iota_\tau d\iota_\sigma d\theta.$$

Suppose that 
$$\begin{tikzpicture}
\mmat[2em]{m}{P&P'\\ E^*[1]&E'^*[1]\\};
\path[->] (m-1-1) edge node {$\tilde\Phi$} (m-1-2)
			edge (m-2-1);
\path[<-] (m-2-2) edge node {$\Phi$} (m-2-1)
			edge (m-1-2);
\end{tikzpicture}$$
is a morphism of Poisson principal $\mbb{R}[2]$ bundles corresponding to the morphism of Manin pairs $$R:(\mbb{E},E)\dasharrow (\mbb{E}',E').$$
Additionally, suppose that $\theta'\in\Omega^1(P',\mbb{R}[2])$ is a connection one-form satisfying $$[\theta',\theta']=0,$$ corresponding to a pseudo-Dirac structure $(W',\nabla')$ in $\mbb{E}'$ transverse to $E'$.

Then $\theta:=\tilde\Phi^*\theta$ is a connection one-form on $P$ satisfying $[\theta,\theta]=0$. It corresponds to the pseudo-Dirac structure $(F,\nabla):=(F',\nabla')\circ R$.

\end{remark}

\begin{remark}
Recall that pseudo-Dirac structures of a Courant algebroid define $\mc{VB}$-Dirac structures in the tangent prolongation. This is analogous to the fact that if $\theta$ is a connection one-form on $P\to E^*[1]$, then $\theta$ defines a degree 2 function $t$ on $TP$ (a flat connection on $TP$). Moreover $[\theta,\theta]=0$ if and only if $[dt,dt]=0$, i.e. the Lagrangian subbundle of $T\mbb{E}$ corresponding to the function $t$ is involutive.
\end{remark}

\chapter{$\mc{LA}$-Courant algebroids}\label{chp:LACour}
\section{$\mc{LA}$-Courant algebroids}

Suppose that $\mbb{E}$ is a Courant algebroid and consider the tangent prolongation $T\mbb{E}$ of $\mbb{E}$. From \cref{ex:TngLAVB} we see that $T\mbb{E}$ is an $\mc{LA}$-vector bundle. In this section, we explore $\mc{VB}$-Courant algebroids which abstract this property. 

Suppose that 
\begin{equation}\label{eq:LACA}\begin{tikzpicture}
\mmat{m}{\mbb{A}&V\\ A& M\\};
\path[->] (m-1-1)	edge (m-1-2)
				edge (m-2-1);
\path[<-] (m-2-2)	edge (m-1-2)
				edge (m-2-1);
\end{tikzpicture}\end{equation} 
is both a $\mc{VB}$-Courant algebroid and an $\mc{LA}$-vector bundle. Since it is an $\mc{LA}$-vector bundle, dualizing $\mbb{A}$ over $V$ yields a double linear Poisson structure on 
\begin{equation}\label{eq:LACADLP}\begin{tikzpicture}
\mmat{m}{\mbb{A}^{*_{\!x}}&V\\ V& M\\};
\path[->] (m-1-1)	edge (m-1-2)
				edge (m-2-1);
\path[<-] (m-2-2)	edge (m-1-2)
				edge (m-2-1);
\end{tikzpicture}\end{equation}

Since the Poisson structure is double linear, the anchor map \begin{equation}\label{eq:LACApi}\pi^\sharp:T^*(\mbb{A}^{*_{\!x}})=T\mbb{A}^{*_{\!x}*_{\!z}}\to T\mbb{A}^{*_{\!x}}\end{equation} defines the morphism of triple vector bundles
\begin{equation}\label{eq:piAsharp}\begin{tikzpicture}[
        back line/.style={densely dotted},
        cross line/.style={preaction={draw=white, -,
           line width=6pt}}]
        
\mmat[1em]{T^*m} at (-3.5,0){
	&T^*\mbb{A}^{*_{\!x}}	&	&\mbb{A}\\
\mbb{A}^{*_{\!x}}	&	&V	&\\
	&\mbb{A}	&	&A\\
V	&	&M	&\\
};

\path[->]
        (T^*m-1-2) edge (T^*m-1-4)
        		edge (T^*m-2-1)
                edge [back line] (T^*m-3-2)
        (T^*m-1-4) edge (T^*m-3-4)
        		edge (T^*m-2-3)
        (T^*m-2-1) edge [cross line] (T^*m-2-3)
                edge (T^*m-4-1)
        (T^*m-3-2) edge [back line] (T^*m-3-4)
        		edge [back line] (T^*m-4-1)
        (T^*m-4-1) edge (T^*m-4-3)
        (T^*m-3-4) edge (T^*m-4-3)
        (T^*m-2-3) edge [cross line] (T^*m-4-3);
  
\mmat[1em]{Tm} at (3.5,0){
	&T\mbb{A}^{*_{\!x}}	&	&TV\\
\mbb{A}^{*_{\!x}}	&	&V	&\\
	&TV	&	&TM\\
V	&	&M	&\\
};

\path[->]
        (Tm-1-2) edge (Tm-1-4)
        		edge (Tm-2-1)
                edge [back line] (Tm-3-2)
        (Tm-1-4) edge (Tm-3-4)
        		edge (Tm-2-3)
        (Tm-2-1) edge [cross line] (Tm-2-3)
                edge (Tm-4-1)
        (Tm-3-2) edge [back line] (Tm-3-4)
        		edge [back line] (Tm-4-1)
        (Tm-4-1) edge (Tm-4-3)
        (Tm-3-4) edge (Tm-4-3)
        (Tm-2-3) edge [cross line] (Tm-4-3);
        
  \path[->] (T^*m) edge node{$\pi^\sharp$} (Tm);
        
\path [->] (8,.5) edge node {$x$} (9,.5)
			edge node {$y$} (8,-.5)
			edge node[swap] {$z$} (7.5,0);
\end{tikzpicture}\end{equation}
%

Dualizing the two triple vector bundles, $T\mbb{A}^{*_{\!x}*_{\!z}}$ and $T\mbb{A}^{*_{\!x}}$ (pictured in \labelcref{eq:piAsharp}), along the $x$-axis yields $T\mbb{A}^{*_{\!x}*_{\!z}*_{\!x}}$ and $T\mbb{A}$, respectively (cf. \cref{ex:dualX}):
%
\begin{equation}\label{eq:PiA}\begin{tikzpicture}[
        back line/.style={densely dotted},
        cross line/.style={preaction={draw=white, -,
           line width=6pt}}]
        
\mmat[1em]{T^*m} at (-3.5,0){
	&T\mbb{A}	&	&\mbb{A}\\
TV	&	&V	&\\
	&TA	&	&A\\
TM	&	&M	&\\
};

\path[->]
        (T^*m-1-2) edge (T^*m-1-4)
        		edge (T^*m-2-1)
                edge [back line] (T^*m-3-2)
        (T^*m-1-4) edge (T^*m-3-4)
        		edge (T^*m-2-3)
        (T^*m-2-1) edge [cross line] (T^*m-2-3)
                edge (T^*m-4-1)
        (T^*m-3-2) edge [back line] (T^*m-3-4)
        		edge [back line] (T^*m-4-1)
        (T^*m-4-1) edge (T^*m-4-3)
        (T^*m-3-4) edge (T^*m-4-3)
        (T^*m-2-3) edge [cross line] (T^*m-4-3);

\mmat[1em]{Tm} at (3.5,0){
	&T\mbb{A}	&	&TV\\
\mbb{A}	&	&V	&\\
	&TA	&	&TM\\
A	&	&M	&\\
};

\path[->]
        (Tm-1-2) edge (Tm-1-4)
        		edge (Tm-2-1)
                edge [back line] (Tm-3-2)
        (Tm-1-4) edge (Tm-3-4)
        		edge (Tm-2-3)
        (Tm-2-1) edge [cross line] (Tm-2-3)
                edge (Tm-4-1)
        (Tm-3-2) edge [back line] (Tm-3-4)
        		edge [back line] (Tm-4-1)
        (Tm-4-1) edge (Tm-4-3)
        (Tm-3-4) edge (Tm-4-3)
        (Tm-2-3) edge [cross line] (Tm-4-3);

  \path[dashed,->] (T^*m) edge node{$\Pi_\mbb{A}$} (Tm);
  
\path [->] (8,.5) edge node {$x$} (9,.5)
			edge node {$y$} (8,-.5)
			edge node[swap] {$z$} (7.5,0);
\end{tikzpicture}\end{equation}
%
Meanwhile, by considering the annihilator, $\Pi_\mbb{A}:=\ann^\natural\big(\gr(\pi^\sharp)\big)$, of the morphism \labelcref{eq:LACApi} (pictured in \labelcref{eq:piAsharp}), one obtains the relation  of triple vector bundles
$$\Pi_\mbb{A}:T\mbb{A}^{*_{\!x}*_{\!z}*_{\!x}}\dasharrow T\mbb{A}.$$

\begin{definition}[$\mc{LA}$-Courant algebroids]
The $\mc{VB}$-Courant algebroid $\mbb{A}$ pictured in \labelcref{eq:LACA}
is called an \emph{$\mc{LA}$-Courant algebroid} if
 the relation of triple vector bundles $$\Pi_\mbb{A}:T\mbb{A}^{*_{\!x}*_{\!z}*_{\!x}}\dasharrow T\mbb{A}$$ pictured in \labelcref{eq:PiA} is a Courant relation, where $T\mbb{A}$ is the tangent prolongation of $\mbb{A}$ (cf. \cref{ex:TngTVB}).

A $\mc{VB}$-Dirac structure $L\subseteq\mbb{A}$ (with support) is called an \emph{$\mc{LA}$-Dirac structure (with support)} if it is also a Lie subalgebroid of the Lie algebroid $\mbb{A}$.  A Courant relation $$R:\mbb{A}_1\dasharrow \mbb{A}_2$$ between two $\mc{LA}$-Courant algebroids is called an $\mc{LA}$-Courant relation if $R\subseteq \mbb{A}_2\times\overline{\mbb{A}_1}$ is an $\mc{LA}$-Dirac structure with support.

If $L$ has support on all of $A$, the pair $(\mbb{A},L)$ is called an \emph{$\mc{LA}$-Manin pair}.
A morphism $$R:(\mbb{A}_1,L_1)\dasharrow(\mbb{A}_2,L_2)$$ of Manin pairs is called a morphism of $\mc{LA}$-Manin pairs if $(\mbb{A}_i,L_i)$ are $\mc{LA}$-Manin pairs and $R$ is an $\mc{LA}$-Courant relation. In particular, $R$ is supported on the graph of a Lie algebroid morphism. 
\end{definition}
\begin{remark}
The compatibility condition  between the Lie algebroid and Courant algebroid structures in the definition above is based quite closely on the compatibility condition for Lie bialgebroids described in \cite[Theorem~6.3]{Mackenzie97}.
\end{remark}

\begin{proposition}\label{prop:LieMetComp}
Suppose that $\mbb{A}$ is an $\mc{LA}$-Courant algebroid. Then the map $\mbb{A}\to \mbb{A}^*_A$ induced by the fibre metric is a morphism of Lie algebroids.
\end{proposition}
\begin{proof}
The proof, which makes use of the duality functors defined by Gracia-Saz and Mackenzie \cite{GraciaSaz:2009ck}, is found in \cref{app:PairIsLAPair}.
\end{proof}

\subsection{Examples}

\begin{example}[Tangent Prolongation of a Courant algebroid]\label{ex:TngProIsLACA}
Suppose that $\mbb{E}$ is a Courant algebroid over $M$, then the tangent prolongation $\mbb{A}=T\mbb{E}$ is a $\mc{LA}$-Courant algebroid. 

Indeed this amounts to showing that the canonical isomorphism $S:TT\mbb{E}\to (TT\mbb{E})^{flip}$ which identifies the double vector bundle 
$$\begin{tikzpicture}
\mmat{m}{TT\mbb{E}&T\mbb{E}\\ T\mbb{E}& \mbb{E}\\};
\path[->] (m-1-1)	edge (m-1-2)
				edge (m-2-1);
\path[<-] (m-2-2)	edge (m-1-2)
				edge (m-2-1);
\end{tikzpicture}$$
with its diagonal reflection, is an automorphism of the Courant algebroid $TT\mbb{E}\to TTM$.

For any $\sigma\in\Gamma(\mbb{E})$, $S$ preserves the sections $\sigma_{TT},\sigma_{CC}\in\Gamma(TT\mbb{E},TTM)$ and interchanges the sections $\sigma_{TC}$ with $\sigma_{CT}$. By \cref{prop:TngProCA}, these sections determine the Courant algebroid structure for the double tangent prolongation  $TT\mbb{E}\to TTM$. Consequently, \cref{eq:TEBrk} implies that $S$ preserves the Courant bracket on $TT\mbb{E}\to TTM$; while \cref{eq:TEPair} implies that $S$ is an isomorphism of quadratic vector bundles. It follows that $S$ is an automorphism  of the Courant algebroid $TT\mbb{E}\to TTM$.

If $L\subseteq \mbb{E}$ is a Dirac structure with support on $S\subseteq M$,  then 
$$\begin{tikzpicture}
\mmat{m} at (-2,0) {TL&L\\ TS& S\\};
\path[->] (m-1-1)	edge (m-1-2)
				edge (m-2-1);
\path[<-] (m-2-2)	edge (m-1-2)
				edge (m-2-1);
		
\draw (0,0) node {$\subseteq$};
				
\mmat{m2} at (2,0) {T\mbb{E}&\mbb{E}\\ TM& M\\};
\path[->] (m2-1-1)	edge (m2-1-2)
				edge (m2-2-1);
\path[<-] (m2-2-2)	edge (m2-1-2)
				edge (m2-2-1);
\end{tikzpicture}$$
is a $\mc{VB}$-Dirac structure with support on $TS\subseteq TM$ (cf. \cref{ex:DirStrTngLf}). Since $TL\subseteq T\mbb{E}$ is also a Lie subalgebroid, it follows that $TL$ is a $\mc{LA}$-Dirac structure (with support).
\end{example}

\begin{example}\label{ex:AtiyahLACA}
Suppose a Lie group $G$ acts freely and properly by automorphisms on a Courant algebroid $\mbb{E}$. Then $G$ acts freely and properly by automorphisms on the $\mc{LA}$-Courant algebroid
$$\begin{tikzpicture}
\mmat{m2}{T\mbb{E}&\mbb{E}\\ TM& M\\};
\path[->] (m2-1-1)	edge (m2-1-2)
				edge (m2-2-1);
\path[<-] (m2-2-2)	edge (m2-1-2)
				edge (m2-2-1);
\end{tikzpicture}$$
 So the quotient,
$$\begin{tikzpicture}
\mmat{m2}{T\mbb{E}/G&\mbb{E}/G\\ TM/G& M/G\\};
\path[->] (m2-1-1)	edge (m2-1-2)
				edge (m2-2-1);
\path[<-] (m2-2-2)	edge (m2-1-2)
				edge (m2-2-1);
\end{tikzpicture}$$
is an $\mc{LA}$-Courant algebroid. Here $T\mbb{E}/G\to \mbb{E}/G$ and $TM/G\to M/G$ are the Atiyah Lie algebroids corresponding to the principal $G$ bundles $\mbb{E}$ and $M$ (cf. \cref{eq:AtiyahLieAlg}).
\end{example}

\begin{example}\label{ex:FreeLADir}
Suppose that
$$\begin{tikzpicture}
\mmat{m}{\mbb{A}&V\\ A& M\\};
\path[->] (m-1-1)	edge node {$q_{\mbb{A}/V}$} (m-1-2)
				edge (m-2-1);
\path[<-] (m-2-2)	edge (m-1-2)
				edge (m-2-1);
\end{tikzpicture}$$
is a $\mc{LA}$-Courant algebroid. \Cref{prop:FreeVBDir} shows that $V\subseteq\mbb{A}$ is a $\mc{VB}$-Dirac structure. Since it is also (trivially) a Lie subalgebroid, $V\subseteq\mbb{A}$ is an $\mc{LA}$-Dirac structure.

\Cref{prop:FreeVBDir} also shows that $\mbb{A}_C:=\on{ker}(q_{\mbb{A}/V})$ is a $\mc{VB}$-Dirac structure. However, in general, $\mbb{A}_C\subseteq\mbb{A}$ is not a Lie subalgebroid, so it is not an $\mc{LA}$-Dirac structure.
\end{example}

\begin{example}[$\mc{LA}$-Courant algebroids over a point]\label{ex:LACApt}
Suppose $\g$ is a Lie algebra and $\beta\in S^2(\g)^\g$ is an invariant symmetric bilinear form on $\g^*$. 
Let 
\begin{equation}\label{eq:LACApt1}\begin{tikzpicture}
\mmat{m}{\g\ltimes\g^*&\g\\ \ast&\ast\\};
\path[->] (m-1-1)	edge (m-1-2)
				edge (m-2-1);
\path[<-] (m-2-2)	edge (m-1-2)
				edge (m-2-1);
\end{tikzpicture}\end{equation}
be the $\mc{VB}$-Courant algebroid
from \cref{ex:VBCAoverpt}.
The vector space $\g^*$ acts on the affine space $\g$ via the map  $$\beta^\sharp:\g^*\to \g$$ defined by $$\beta^\sharp(\mu):=\beta(\mu,\cdot),\quad\mu\in\g^*.$$
Thus $\g\ltimes\g^*\to\g$ carries the structure of an action Lie algebroid. With these structures, \labelcref{eq:LACApt1} is a $\mc{LA}$-Courant algebroid.

Suppose $\mf{h}\subseteq\g$ is a Lie subalgebra, and 
\begin{equation}\label{eq:LADSpt1}\begin{tikzpicture}
\mmat{m1} at (-3,0){\h\ltimes\on{ann}(\h)&\h\\ \ast& \ast\\};
\path[->] (m1-1-1)	edge (m1-1-2)
				edge (m1-2-1);
\path[<-] (m1-2-2)	edge (m1-1-2)
				edge (m1-2-1);
\mmat{m2} at (2,0) {\g\ltimes\g^*&\g\\ \ast& \ast\\};
\path[->] (m2-1-1)	edge (m2-1-2)
				edge (m2-2-1);
\path[<-] (m2-2-2)	edge (m2-1-2)
				edge (m2-2-1);
\draw (0,0) node {$\subseteq$};
\end{tikzpicture}\end{equation}
is the $\mc{VB}$-Dirac structure described in \cref{ex:VBDSoverpt}. The Lie subalgebra $\h$ is said to be \emph{$\beta$-coisotropic} \cite{LiBland:2011vqa} if $$\beta^\sharp(\ann(\h))\subseteq\h.$$  In this case, $$\h\ltimes\on{ann}(\h)\to \h$$ is a Lie subalgebroid of the action Lie algebroid $\g\ltimes\g^*\to\g$; and \labelcref{eq:LADSpt1} is an $\mc{LA}$-Dirac structure.
\end{example}

\begin{remark}
An invariant symmetric element $\beta\in S^2(\g)^\g$ is called a quasi-triangular structure for $\g$ and the pair $(\g,\beta)$ is called a \emph{quasi-triangular Lie algebra}.
Drinfel'd \cite{Drinfeld:1989tu} showed that quasi-triangular Lie algebras are in one-to-one correspondence with Manin pairs $(\mf{d},\g)$ together with a Lie algebra ideal $\h\subset\mf{d}$ complementary to $\g$ (i.e. $\mf{d}=\g\oplus\h$ as a vector space). 
\end{remark}

\begin{proposition}
$\mc{LA}$-Courant algebroids $$\begin{tikzpicture}
\mmat{m}{\mbb{A}&V\\ \ast& \ast\\};
\path[->] (m-1-1)	edge (m-1-2)
				edge (m-2-1);
\path[<-] (m-2-2)	edge (m-1-2)
				edge (m-2-1);
\end{tikzpicture}$$ over a point are all of the form described in \cref{ex:LACApt}. Thus there is a one-to-one correspondence between $\mc{LA}$-Courant algebroids over a point and quasi-triangular Lie algebras $(\g,\beta)$.

Similarly, any $\mc{LA}$-Dirac structure
$$\begin{tikzpicture}
\mmat{m1} at (-2,0){L&W\\ \ast& \ast\\};
\path[->] (m1-1-1)	edge (m1-1-2)
				edge (m1-2-1);
\path[<-] (m1-2-2)	edge (m1-1-2)
				edge (m1-2-1);
\mmat{m2} at (2,0) {\g\ltimes\g^*&\g\\ \ast& \ast\\};
\path[->] (m2-1-1)	edge (m2-1-2)
				edge (m2-2-1);
\path[<-] (m2-2-2)	edge (m2-1-2)
				edge (m2-2-1);
\draw (0,0) node {$\subseteq$};
\end{tikzpicture}$$
is of the form described in \cref{ex:LACApt} for a $\beta$-coisotropic Lie subalgebra $\h\subseteq\g$.
\end{proposition}
The proof can be found in \cref{app:PairIsLAPair}.

\begin{example}[Standard Courant algebroid over a Lie algebroid]\label{ex:StdLAC}
Let $A$ be a Lie algebroid. Consider the $\mc{VB}$-Courant algebroid $\mbb{A}:=\mbb{T}A$ described in \cref{ex:StdVBCAoverVB}. From \cref{ex:TngLAVB,ex:CotLAVB} we see that both $TA^{flip}$ and $(T^*A)^{flip}$ are $\mc{LA}$-vector bundles, so their direct sum,
$$\begin{tikzpicture}
\mmat{m}{\mbb{T}A&TM\oplus A^*\\ A& M\\};
\path[->] (m-1-1)	edge (m-1-2)
				edge (m-2-1);
\path[<-] (m-2-2)	edge (m-1-2)
				edge (m-2-1);
\end{tikzpicture}$$
is naturally both an $\mc{LA}$-vector bundle and a $\mc{VB}$-Courant algebroid. 

Since $A$ is a Lie algebroid, $A^*$ carries a  linear Poisson structure. Let 
$$\begin{tikzpicture}
\mmat{m} at (-3,0) {T^*A^*&A\\ A^*& M\\};
\path[->] (m-1-1)	edge (m-1-2)
				edge (m-2-1);
\path[<-] (m-2-2)	edge (m-1-2)
				edge (m-2-1);

\mmat{m2} at (3,0) {TA^*&TM\\ A^*& M\\};
\path[->] (m2-1-1)	edge (m2-1-2)
				edge (m2-2-1);
\path[<-] (m2-2-2)	edge (m2-1-2)
				edge (m2-2-1);
  \path[->] (m) edge node{$\pi^\sharp$} (m2);
\end{tikzpicture}$$
denote the anchor map for the corresponding cotangent Lie algebroid (cf. \cref{ex:CotLieAlg}). Taking the horizontal dual yields the relation
$$\begin{tikzpicture}
\mmat{m} at (-3,0) {TA&A\\ TM& M\\};
\path[->] (m-1-1)	edge (m-1-2)
				edge (m-2-1);
\path[<-] (m-2-2)	edge (m-1-2)
				edge (m-2-1);

\mmat{m2} at (3,0) {TA&TM\\ A& M\\};
\path[->] (m2-1-1)	edge (m2-1-2)
				edge (m2-2-1);
\path[<-] (m2-2-2)	edge (m2-1-2)
				edge (m2-2-1);
  \path[dashed,->] (m) edge node{$\Pi_0$} (m2);
\end{tikzpicture}$$
Finally, taking $\Pi_{\mbb{T}A}:=R_{\Pi_0}$ to be the standard lift of $\Pi_0$ (cf. \cref{ex:StdDiracStr}) defines the Courant relation \labelcref{eq:PiA}
$$\begin{tikzpicture}[
        back line/.style={densely dotted},
        cross line/.style={preaction={draw=white, -,
           line width=6pt}}]
        
\mmat[.5em]{T^*m} at (-4.5,0) {
	&T\mbb{T}A	&	&\mbb{T}A\\
T(TM\oplus A^*)	&	&TM\oplus A^*	&\\
	&TA	&	&A\\
TM	&	&M	&\\
};

\path[->]
        (T^*m-1-2) edge (T^*m-1-4)
        		edge (T^*m-2-1)
                edge [back line] (T^*m-3-2)
        (T^*m-1-4) edge (T^*m-3-4)
        		edge (T^*m-2-3)
        (T^*m-2-1) edge [cross line] (T^*m-2-3)
                edge (T^*m-4-1)
        (T^*m-3-2) edge [back line] (T^*m-3-4)
        		edge [back line] (T^*m-4-1)
        (T^*m-4-1) edge (T^*m-4-3)
        (T^*m-3-4) edge (T^*m-4-3)
        (T^*m-2-3) edge [cross line] (T^*m-4-3);

\mmat[.5em]{Tm} at (4.5,0) {
	&T\mbb{T}A	&	&T(TM\oplus A^*)\\
\mbb{T}A	&	&TM\oplus A^*	&\\
	&TA	&	&TM\\
A	&	&M	&\\
};

\path[->]
        (Tm-1-2) edge (Tm-1-4)
        		edge (Tm-2-1)
                edge [back line] (Tm-3-2)
        (Tm-1-4) edge (Tm-3-4)
        		edge (Tm-2-3)
        (Tm-2-1) edge [cross line] (Tm-2-3)
                edge (Tm-4-1)
        (Tm-3-2) edge [back line] (Tm-3-4)
        		edge [back line] (Tm-4-1)
        (Tm-4-1) edge (Tm-4-3)
        (Tm-3-4) edge (Tm-4-3)
        (Tm-2-3) edge [cross line] (Tm-4-3);

  \path[dashed,->] (T^*m) edge node{$\Pi_{\mbb{T}A}$} (Tm);
\end{tikzpicture}$$
Thus the relation \labelcref{eq:PiA} (for $\mbb{A}=\mbb{T}A$) is a Courant relation, which shows that $\mbb{T}A$ is an $\mc{LA}$-Courant algebroid.

The double vector subbundles $TA^{flip}$:
$$\begin{tikzpicture}
\mmat{m} at (-2.5,0) {TA&TM\\ A&M\\};
\path[->]
	(m-1-1) edge (m-1-2)
		edge (m-2-1);
\path[<-] 
	(m-2-2) edge (m-1-2)
		edge (m-2-1);

\draw (0,0) node {$\subseteq$};

\mmat{m1} at (3,0) {\mbb{T}A&TM\oplus A^*\\ A&M\\};
\path[->]
	(m1-1-1) edge (m1-1-2)
		edge (m1-2-1);
\path[<-] 
	(m1-2-2) edge (m1-1-2)
		edge (m1-2-1);
\end{tikzpicture}$$
 and $T^*A^{flip}$:
$$\begin{tikzpicture}
\mmat{m} at (-2.5,0) {T^*A&A^*\\ A&M\\};
\path[->]
	(m-1-1) edge (m-1-2)
		edge (m-2-1);
\path[<-] 
	(m-2-2) edge (m-1-2)
		edge (m-2-1);

\draw (0,0) node {$\subseteq$};

\mmat{m1} at (3,0) {\mbb{T}A&TM\oplus A^*\\ A&M\\};
\path[->]
	(m1-1-1) edge (m1-1-2)
		edge (m1-2-1);
\path[<-] 
	(m1-2-2) edge (m1-1-2)
		edge (m1-2-1);
\end{tikzpicture}$$
are both $\mc{LA}$-Dirac structures (cf. \cref{ex:TngLAVB,ex:CotLAVB,ex:StdVBCAoverVB}).
\end{example}

\begin{example}\label{ex:LADirMan1}
Suppose that $(\mf{d},\beta)$ is a quasi-triangular Lie algebra and $\g\subseteq \mf{d}$ is a $\beta$-coisotropic Lie subalgebra. As explained by Meinrenken and the author in \cite[Section 3.2]{LiBland:2011vqa}, there exists a unique Manin pair $(\mf{q},\g)$ together with a morphism of Lie algebras $f:\mf{q}\to \mf{d}$ such that
$f$ restricts to the identity map on $\g\subseteq\mf{q}$, and
\begin{equation}\label{eq:fPoisMap}f(\gamma)=\beta,\end{equation} where $\gamma\in S^2(\mf{q})^\mf{q}$ is the dual metric on $\mf{q}^*$.

We now describe a $\mc{LA}$-Courant algebroid structure on 
\begin{equation}\label{eq:DiracManUnRes}\begin{tikzpicture}
\mmat[1em]{m}{T\mf{q}\times \mf{d}&\mf{q}\\ \mf{d}& \ast\\};
\path[->] (m-1-1)	edge (m-1-2)
				edge (m-2-1);
\path[<-] (m-2-2)	edge (m-1-2)
				edge (m-2-1);
\end{tikzpicture}\end{equation}
for which
$$\begin{tikzpicture}
\mmat[1em]{m1} at (-2,0) {T\mf{g}\times \mf{d}&\mf{g}\\ \mf{d}& \ast\\};
\path[->] (m1-1-1)	edge (m1-1-2)
				edge (m1-2-1);
\path[<-] (m1-2-2)	edge (m1-1-2)
				edge (m1-2-1);

\draw (0,0) node {$\subseteq$};

\mmat[1em]{m} at (2,0) {T\mf{q}\times \mf{d}&\mf{q}\\ \mf{d}& \ast\\};
\path[->] (m-1-1)	edge (m-1-2)
				edge (m-2-1);
\path[<-] (m-2-2)	edge (m-1-2)
				edge (m-2-1);
\end{tikzpicture}$$ is an $\mc{LA}$-Dirac structure.



First we describe the $\mc{VB}$-Courant algebroid structure on \labelcref{eq:DiracManUnRes}. The tangent lifts, $\beta_T\in S^2(T\mf{d})^{T\mf{d}}$ and $\gamma_T\in S^2(T\mf{q})^{T\mf{q}}$,  of $\beta$ and $\gamma$, are determined by the equations 
\begin{align}\label{eq:TngJLift}\beta_T^\sharp&=d\beta^\sharp:T\mf{d}^*\to T\mf{d},  &\gamma_T^\sharp=d\gamma^\sharp:T\mf{q}^*\to T\mf{q}.\end{align} They satisfy 
\begin{equation}\label{eq:TfPoisMap}df(\gamma_T)=\beta_T.\end{equation}

Let $D$ be the simply connected Lie group with Lie algebra $\mf{d}$. The stabilizer of the natural $T\mf{d}$ action on $\mf{d}\cong TD/D$ at the origin is $\mf{d}\subset T\mf{d}$.  \Cref{eq:TngJLift} implies that $$\beta_T^\sharp(\on{ann}(\mf{d}))\subseteq \mf{d},$$ so the stabilizer at the origin is $\beta_T$-coisotropic. Since all the stabilizers are conjugate, they are all $\beta_T$-coisotropic. As a result, \cref{eq:TfPoisMap} implies that the induced action of $T\mf{q}$ on $\mf{d}\cong TD/D$ via the map $df:T\mf{q}\to T\mf{d}$ has coisotropic stabilizers. 
%
%
%
%
In conclusion,  \labelcref{eq:DiracManUnRes}
is naturally an action Courant algebroid (cf. \cref{ex:ActCourAlg}), whose anchor map we denote by $\mbf{a}$.


Next we describe the compatible $\mc{LA}$-vector bundle structure on \labelcref{eq:DiracManUnRes}: both
$$\begin{tikzpicture}
\mmat[1em]{m}{T\mf{q}&\mf{q}\\ \ast& \ast\\};
\path[->] (m-1-1)	edge (m-1-2)
				edge (m-2-1);
\path[<-] (m-2-2)	edge (m-1-2)
				edge (m-2-1);
\end{tikzpicture}$$
(cf. \cref{ex:LACApt}) and the trivial $\mc{VB}$-Courant algebroid over the vector space $\mf{d}$,
$$\begin{tikzpicture}
\mmat[1em]{m}{\mf{d}&\ast\\ \mf{d}& \ast\\};
\path[->] (m-1-1)	edge (m-1-2)
				edge (m-2-1);
\path[<-] (m-2-2)	edge (m-1-2)
				edge (m-2-1);
\end{tikzpicture}$$
 are $\mc{LA}$-Courant algebroids. Thus their direct product, \labelcref{eq:DiracManUnRes}, carries a natural $\mc{LA}$-vector bundle structure. Moreover, the action of $T\mf{q}$ on $\mf{d}$ is compatible with this $\mc{LA}$-vector bundle structure, that is $(\mbf{a}\times\mbf{a})(\Pi_{T\mf{q}})\subseteq T(\Pi_{\mf{d}})$. 
\end{example}

\begin{example}[$\mc{LA}$-Courant algebroids associated to Dirac Manin triples]\label{ex:LADirMan}
A Dirac Manin triple \cite{LiBland:2011vqa,LiBland:2010wi} is a triple $(\mf{d},\g;\h)_\beta$ of Lie algebras together with a symmetric invariant element $\beta\in S^2(\mf{d})^\mf{d}$ such that 
\begin{itemize}
\item $\g$ is $\beta$-coisotropic, (i.e. $\beta^\sharp(\ann(\g))\subseteq\mf{g}$), and
\item $\g,\h\subset\mf{d}$ are complementary Lie subalgebras (i.e. $\mf{d}=\g\oplus\h$ as a vector space).
\end{itemize}
Let $(\mf{q},\g)$ be the Manin pair and $f:\mf{q}\to\mf{d}$ be the morphism of Lie algebras discussed in \cref{ex:LADirMan1} and described in  \cite[Section 3.2]{LiBland:2011vqa}. 
As explained in \cref{ex:LADirMan1},
$T\mf{q}\times \mf{d}$ is naturally an $\mc{LA}$-Courant algebroid, and
$$\begin{tikzpicture}
\mmat[1em]{m1} at (-2,0) {T\mf{g}\times \mf{d}&\mf{g}\\ \mf{d}& \ast\\};
\path[->] (m1-1-1)	edge (m1-1-2)
				edge (m1-2-1);
\path[<-] (m1-2-2)	edge (m1-1-2)
				edge (m1-2-1);

\draw (0,0) node {$\subseteq$};

\mmat[1em]{m} at (2,0) {T\mf{q}\times \mf{d}&\mf{q}\\ \mf{d}& \ast\\};
\path[->] (m-1-1)	edge (m-1-2)
				edge (m-2-1);
\path[<-] (m-2-2)	edge (m-1-2)
				edge (m-2-1);
\end{tikzpicture}$$ is an $\mc{LA}$-Dirac structure.

We now construct an $\mc{LA}$-Courant algebroid over $\h$, using the pull-back construction from Section~\ref{sec:pullback}.
Let $i:\h\to\mf{d}$ be the inclusion, and let 
$$\mbb{A}_{(\mf{d},\g;\h)_\beta}:=i^!(T\mf{q}\times \mf{d}),\quad\quad E_{(\mf{d},\g;\h)_\beta}:=(T\g\times\mf{d})\circ P_i$$ be the pull-back to $\h$, where $$P_i:i^!(T\mf{q}\times \mf{d})\dasharrow T\mf{q}\times\mf{d}$$ is the canonically associated Courant relation.

In summary, we have associated to the Dirac Manin triple $(\mf{d},\g;\h)_\beta$ the $\mc{LA}$-Manin pair $$(\mbb{A}_{(\mf{d},\g;\h)_\beta},E_{(\mf{d},\g;\h)_\beta})$$ over the Lie algebra $\h$. When $\beta$ is non-degenerate, this $\mc{LA}$-Manin pair will be central to our integration procedure for q-Poisson $(\mf{d},\g)$-structures. More generally, as we shall explain in \cref{ex:DirLieInt}, this $\mc{LA}$-Manin pair is the infinitesimal version of the Dirac Lie groups associated to $(\mf{d},\g;\h)_\beta$, described in \cite{LiBland:2010wi,LiBland:2011vqa}.

\end{example}


\begin{proposition}\label{prop:CoLieSub}
Let $\mbb{E}$ be a Courant algebroid.
 $\mc{LA}$-Dirac structures  $L\subseteq T\mbb{E}$ are in one-to-one correspondence with pseudo-Dirac structures $(W,\nabla)$ in $\mbb{E}$ such that $W$ is \emph{coisotropic} (i.e. $W^\perp\subseteq W$) and $\nabla$ restricts to a flat connection on $W/W^\perp$.
\end{proposition}
\begin{proof}
Suppose first that a $\mc{VB}$-Dirac structure 
$$\begin{tikzpicture}
\mmat{m1} at (-2,0) {L&W\\ TM& M\\};
\path[->] (m1-1-1)	edge (m1-1-2)
				edge (m1-2-1);
\path[<-] (m1-2-2)	edge (m1-1-2)
				edge (m1-2-1);
\draw (-2,0) node (c) {$W^\perp$};
\path[left hook->] (c) edge (m1-1-1);

\draw (0,0) node {$\subseteq$};

\mmat{m} at (2,0) {T\mbb{E}&\mbb{E}\\ TM& M\\};
\path[->] (m-1-1)	edge (m-1-2)
				edge (m-2-1);
\path[<-] (m-2-2)	edge (m-1-2)
				edge (m-2-1);
\draw (2,0) node (E) {$\mbb{E}$};
\path[left hook->] (E) edge (m-1-1);
\end{tikzpicture}$$ 
is also a Lie subalgebroid over $W\subseteq \mbb{E}$. Equivalently, $L\subseteq TW$, 
$$\begin{tikzpicture}
\mmat{m1} at (-2,0) {L&W\\ TM& M\\};
\path[->] (m1-1-1)	edge (m1-1-2)
				edge (m1-2-1);
\path[<-] (m1-2-2)	edge (m1-1-2)
				edge (m1-2-1);
\draw (-2,0) node (c) {$W^\perp$};
\path[left hook->] (c) edge (m1-1-1);

\draw (0,0) node {$\subseteq$};

\mmat{m} at (2,0) {TW&W\\ TM& M\\};
\path[->] (m-1-1)	edge (m-1-2)
				edge (m-2-1);
\path[<-] (m-2-2)	edge (m-1-2)
				edge (m-2-1);
\draw (2,0) node (W) {$W$};
\path[left hook->] (W) edge (m-1-1);
\end{tikzpicture}$$ 
and $L$
 is an involutive distribution.
 
 Since the core of $L$ is $W^\perp$ and the core of $TW$ is $W$, we must have $W^\perp\subseteq W$. Moreover, since $L$ is involutive, $TW^\perp\subset L$. Hence for any $\sigma\in\Gamma(W^\perp)$, $\sigma_T\in\Gamma(L)$, and therefore $\nabla(\sigma)=0$. If $\tau\in\Gamma(W)$ then $$\la\nabla\tau,\sigma\ra=\la\nabla\tau,\sigma\ra+\la\tau,\nabla\sigma\ra=d\la\tau,\sigma\ra=0,$$ so $\nabla\tau\in\Omega(M,W/W^\perp)$. Therefore $\nabla$ descends to a connection on $W/W^\perp$. Finally, since $L\subset TW$ is involutive, the connection must be flat. This proves one direction.

Conversely, suppose that $(W,\nabla)$ is a pseudo-Dirac structure, that $W$ is coisotropic, and that the restriction of $\nabla$ to $W/W^\perp$ is flat. Let $H\subseteq T(W/W^\perp)$ be the horizontal distribution. Then $L$ must be the preimage of $H$ along the map $TW\to T(W/W^\perp)$. As such, it defines an involutive distribution in $TW$, which implies that $L\subset T\mbb{E}$ is an $\mc{LA}$-Dirac structure.
\end{proof}


\begin{remark}
From \cref{prop:fltSec}, we see that, if $M$ is simply connected, the flat sections of $W/W^\perp$ form a quadratic Lie algebra. That is to say, $(W,\nabla)$ enables us to reduce the Courant algebroid $\mbb{E}$ to a Courant algebroid over a point. More generally, in \cref{sec:reduc} we will see that $\mc{LA}$-Dirac structures in $T\mbb{E}$ (with support provide (infinitesimal) reduction data for $\mbb{E}$.
\end{remark}

\section{Poisson Lie 2-algebroids and multiplicative Courant algebroids}\label{sec:Integr}
In this section, we briefly discuss the supergeometric interpretation of $\mc{LA}$-Courant algebroids.

\begin{definition}
A Poisson Lie 2-algebroid is a degree 2 $NQ$ manifold carrying a compatible Poisson structure of degree -2. (Compatibility means the flow of the homological vector field, $Q$, preserves the Poisson structure).
\end{definition}

\begin{proposition}\label{prop:PoisLie2LACour}
 There is a one-to-one correspondence between Poisson Lie 2-algebroids and $\mc{LA}$-Courant algebroids; and a one-to-one correspondence between coisotropic $NQ$ submanifolds of a Poisson Lie 2-algebroid and $\mc{LA}$-Dirac structures with support in the corresponding $\mc{LA}$-Courant algebroid. Moreover, the coisotropic $NQ$-submanifold is a wide Lie subalgebroid if and only if the corresponding $\mc{LA}$-Dirac structure has full support.
\end{proposition}
\begin{proof}
Suppose first that 
$$\begin{tikzpicture}
\mmat[1em]{m}{\mbb{A}&V\\ A& M\\};
\path[->] (m-1-1)	edge (m-1-2)
				edge (m-2-1);
\path[<-] (m-2-2)	edge (m-1-2)
				edge (m-2-1);
\end{tikzpicture}$$ 
 is an $\mc{LA}$-Courant algebroid.   Let 
$$\begin{tikzpicture}
\mmat[1em]{m}{\mbb{A}[1,1]&V[0,1]\\ A[1,0]& M\\};
\path[->] (m-1-1)	edge (m-1-2)
				edge (m-2-1);
\path[<-] (m-2-2)	edge (m-1-2)
				edge (m-2-1);
\end{tikzpicture}$$ 
 denote the double vector bundle whose horizontal and vertical fibre degrees are shifted by $(1,0)$ and $(0,1)$ respectively (see \cite{Grabowski:2009dc,Mehta06} for details). (Note that the degree of the core fibres is shifted by $(1,1)$). As explained in \cref{rem:SupCour}, the metric on the fibres of $\mbb{A}\to A$ defines a degree $(-1,-2)$ Poisson structure on $\mbb{A}[1,1]$. 

Since $\mbb{A}$ is a $\mc{VB}$-Courant algebroid it corresponds to a degree (0,2) $NQ$-manifold $X$, by \cref{prop:Lie2VBCour}.
By construction, the Courant algebroid $\mbb{A}$ corresponds to the degree $(0,2)$ symplectic $NQ$-manifold $T^*[0,2]X$, and $T^*[0,2]\to\mbb{A}[0,1]$ is a minimal symplectic realization  (see \cite{Roytenberg:2002, Severa:2005vla} or \cref{rem:SupCour} for a details). Taking the degree shift into consideration, we see that the map $T^*[1,2]X\to \mbb{A}[1,1]$ is a minimal symplectic realization. 

Now, the Lie algebroid structure on $\mbb{A}$ induces a degree $(1,0)$ homological vector field, $Q_{\pi}'$, on $\mbb{A}[1,1]$ (see \cite{LieAlgebroidsH, Alexandrov:1997jj} or  \cref{rem:SupLie} for details).
Now, by \cref{prop:LieMetComp} the Lie algebroid structure on $\mbb{A}\to V$ is compatible with the metric on the fibres of $\mbb{A}\to A$, that is $Q_{\pi}'$ preserves the Poisson structure on $\mbb{A}[1,1]$. Since the correspondence between the symplectic manifold $T^*[1,2]X$ and $\mbb{A}$ is natural, this implies that $Q_{\pi}$ lifts canonically to a degree $(1,0)$ homological vector field on $T^*[1,2]X$, compatible with the symplectic structure. The corresponding Hamiltonian $\pi\in C^\infty(T^*[1,2])$ is a degree $(2,2)$ function satisfying $$\{\pi,\pi\}=0,$$ namely it defines a degree $-2$ Poisson structure on $X$.

It remains to show that the Poisson structure on $X$ is compatible with the homological vector field on $X$. To keep track of the gradings, we view $X$ as a degree $(0,2,0)$ $NQ$-manifold. We need to show that the anchor map for the cotangent Lie algebroid $$\pi^\sharp:T^*[0,2,1]X\to T[0,0,1]X$$ is an $NQ$-map, with respect to the canonical lifts of the $Q$ structure on $X$.
Equivalently the relation $$\pi_R:=\ann\big(\on{gr}(\pi^\sharp)\big)\subseteq (T^*[1,2,1]T[0,0,1]X)\times\overline{(T^*[1,2,1]T^*[0,2,1]X)}$$ is an $NQ$-submanifold (with respect to the canonical lifts of the $Q$ structure). However, $T^*[1,2,1]T^*[0,2,1]X\cong T[1,0,0]T^*[0,2,1]X$ while $T^*[1,2,1]T[0,0,1]X\cong T[0,0,1]T^*[1,2,0]X$. Visualizing the first, second, and third gradings as the $x$, $y$ and $z$ axis, we have
$$\begin{tikzpicture}[
        back line/.style={densely dotted},
        cross line/.style={preaction={draw=white, -,
           line width=6pt}}]
        
\mmat[0.5em]{T^*m}{
	&	&\scriptscriptstyle{T[1,0,0]T^*[0,2,1]X}	&	&\scriptscriptstyle{T^*[0,2,1]X}\\
	&\scriptscriptstyle{T[1,0,0](T^*[0,2,1]X)_{(0,1,1)}}	&	&\scriptscriptstyle{(T^*[0,2,1]X)_{(0,1,1)}}&\\
\scriptscriptstyle{T[1,0,0](T^*[0,2,1]X)_{(0,0,1)}}	&	&\scriptscriptstyle{(T^*[0,2,1]X)_{(0,0,1)}}	&	&\scriptscriptstyle{X}\\
	&\scriptscriptstyle{T[1,0,0]X_{(0,1,0)}}	&	&\scriptscriptstyle{X_{(0,1,0)}}&\\
\scriptscriptstyle{T[1,0,0]X_{(0,0,0)}}	&	&\scriptscriptstyle{X_{(0,0,0)}}	&&\\
};

\path[->]
	(T^*m-1-3) edge [dotted] (T^*m-2-2)
			edge (T^*m-1-5)
	(T^*m-1-5) edge [dotted] (T^*m-2-4)
			edge (T^*m-3-5)
	(T^*m-3-5) edge [dotted] (T^*m-4-4)
        (T^*m-2-2) edge (T^*m-2-4)
        		edge (T^*m-3-1)
                edge [back line] (T^*m-4-2)
        (T^*m-2-4) edge (T^*m-4-4)
        		edge (T^*m-3-3)
        (T^*m-3-1) edge [cross line] (T^*m-3-3)
                edge (T^*m-5-1)
        (T^*m-4-2) edge [back line] (T^*m-4-4)
        		edge [back line] (T^*m-5-1)
        (T^*m-5-1) edge (T^*m-5-3)
        (T^*m-4-4) edge (T^*m-5-3)
        (T^*m-3-3) edge [cross line] (T^*m-5-3);
        
  \path[dashed,->] (0,-70pt) edge node{$\pi_R$} (0,-110pt);

\mmat[0.5em]{Tm}[yshift=-180pt]{
	&	&\scriptscriptstyle{T[0,0,1]T^*[1,2,0]X}	&	&\scriptscriptstyle{T[0,0,1]X}\\
	&\scriptscriptstyle{T[0,0,1](T^*[1,2,0]X)_{(1,1,0)}}	&	&\scriptscriptstyle{T[0,0,1]X_{(0,1,0)}}&\\
\scriptscriptstyle{T[0,0,1](T^*[1,2,0]X)_{(1,0,0)}}	&	&\scriptscriptstyle{T[0,0,1]X_{(0,0,0)}}	&	&\scriptscriptstyle{X}\\
	&\scriptscriptstyle{(T^*[1,2,0]X)_{(1,1,0)}}	&	&\scriptscriptstyle{X_{(0,1,0)}}&\\
\scriptscriptstyle{(T^*[1,2,0]X)_{(1,0,0)}}	&	&\scriptscriptstyle{X_{(0,0,0)}}	&&\\
};

\path[->]
	(Tm-1-3) edge [dotted] (Tm-2-2)
			edge (Tm-1-5)
	(Tm-1-5) edge [dotted] (Tm-2-4)
			edge (Tm-3-5)
	(Tm-3-5) edge [dotted] (Tm-4-4)
        (Tm-2-2) edge (Tm-2-4)
        		edge (Tm-3-1)
                edge [back line] (Tm-4-2)
        (Tm-2-4) edge (Tm-4-4)
        		edge (Tm-3-3)
        (Tm-3-1) edge [cross line] (Tm-3-3)
                edge (Tm-5-1)
        (Tm-4-2) edge [back line] (Tm-4-4)
        		edge [back line] (Tm-5-1)
        (Tm-5-1) edge (Tm-5-3)
        (Tm-4-4) edge (Tm-5-3)
        (Tm-3-3) edge [cross line] (Tm-5-3);
\end{tikzpicture}$$
Where $Y_{(i,j,k)}$ denotes the truncation of $Y$ to an $(i,j,k)$ degree manifold. The identifications 
\begin{align*}
\mbb{A}[1,1,0]&\cong (T^*[1,2,0]X)_{(1,1,0)}&V[0,1,0]&\cong X_{(0,1,0)}\\
A[1,0,0]&\cong(T^*[1,2,0]X)_{(1,0,0)}& M&\cong X_{(0,0,0)}
\end{align*}
and
\begin{align*}
\mbb{A}[0,1,1]&\cong (T^*[0,2,1]X)_{(0,1,1)}&V[0,1,0]&\cong X_{(0,1,0)}\\
A[0,0,1]&\cong(T^*[0,2,1]X)_{(0,0,1)}& M&\cong X_{(0,0,0)}
\end{align*}
identify $\pi_R$ with the relation $\Pi_{\mbb{A}}$ shown in \labelcref{eq:PiA} under the correspondence described in \cite{Severa:2005vla} (and summarized in \cref{rem:SupCour}). In particular, $\pi_R$ is a Lagrangian $NQ$ submanifold if and only if $\Pi_{\mbb{A}}$ is a Courant relation. This proves the first direction.

On the other hand, suppose $X$ is a Poisson $NQ$-manifold, and let $\mbb{A}$ be the $\mc{VB}$-Courant algebroid which the equivalence described in \cref{prop:Lie2VBCour} associates to it.
The cotangent Lie algebroid $T^*[2]X$ is a graded Lie algebroid. Thus $\mbb{A}[1]=(T^*[2]X)_{(1)}$ is a graded Lie subalgebroid over $V[1]$. 
Equivalently,  $\mbb{A}$ is a $\mc{LA}$-vector bundle. The compatibility between the Lie algebroid and Courant algebroid structures on $\mbb{A}$ follows from the same argument as above.

Next, suppose that 
$$\begin{tikzpicture}
\mmat[1em]{m}{L&W\\ A'& N\\};
\path[->] (m-1-1)	edge (m-1-2)
				edge (m-2-1);
\path[<-] (m-2-2)	edge (m-1-2)
				edge (m-2-1);
\end{tikzpicture}$$
  is a $\mc{VB}$-Dirac structure in $\mbb{A}$ with support on $A'\subseteq A$. As explained in \cite{Severa:2005vla} (see also \cref{rem:SupCour}), $L$ corresponds to a Lagrangian $NQ$ submanifold $\tilde Y\subseteq T^*[2]X$. Since $L\subseteq \mbb{A}$ is a double vector subbundle, $\tilde Y\subseteq T^*[2]X$ must be a subbundle, which implies that it is the conormal bundle $$\tilde Y=\ann[2](TY)\subseteq T^*[2]X$$ for some $NQ$-submanifold $Y\subseteq X$. Now, $L$ is a Lie subalgebroid of the Lie algebroid $\mbb{A}\to V$ if and only if $\ann[2](TY)\subseteq T^*[2]X$ is a Lie subalgebroid of the cotangent Lie algebroid. Equivalently, $Y\subseteq X$ is coisotropic.
  
  Finally, the last statement follows from \cref{prop:Lie2VBCour}.
  
\end{proof}

\begin{corollary}\label{cor:LACourIntMultCour}
There is a one-to-one correspondence between integrable $\mc{LA}$-Courant algebroids and source-simply-connected multiplicative Courant algebroids; and a one-to-one correspondence between $\mc{LA}$-Dirac structures (with support) and multiplicative Dirac structures (with support) in the corresponding multiplicative Courant algebroid. Moreover, morphisms of $\mc{LA}$-Manin pairs integrate to morphisms of multiplicative Manin pairs via this correspondence.
\end{corollary}
\begin{proof}
As described in \cite{Ortiz:2009ux}, this follows directly from \cref{prop:PoisLie2LACour} and the general integration theory for Poisson manifolds. In more detail, suppose $X$ is an integrable Poisson Lie-2-algebroid. The source simply connected groupoid $\mc{G}$ integrating $X$ is a degree 2 symplectic manifold. Since the flow of the homological vector field on $X$ preserves the Poisson structure, it integrates to a  symplectic $\mbb{R}[-1]$ action on $\mc{G}$ by groupoid automorphisms, namely a compatible $Q$ structure. By the correspondence described in \cite{Roytenberg:2002,Severa:2005vla} between degree 2 symplectic $NQ$ manifolds and Courant algebroids, $\mc{G}$ corresponds to a multiplicative Courant algebroid (see \cite[Section~5.1]{LiBland:2010wi} for details).

Similarly, the general integration theory for Poisson manifolds shows that coisotropic $NQ$-submanifolds of $X$ integrate to Lagrangian $NQ$-subgroupoids of $\mc{G}$, i.e. multiplicative Dirac structures.

Finally, the proof of the last statement is found in \cite[Proposition~11]{LiBland:2010wi}.

\begin{definition}
In this case, we say that the corresponding multiplicative Courant algebroid $\mbb{G}$ \emph{integrates} the $\mc{LA}$-Courant algebroid $\mbb{A}$, or \emph{$\mbb{G}$ differentiates to $\mbb{A}$}. Similarly terminology is used for corresponding $\mc{LA}$-Dirac structures and multiplicative Dirac structures.
\end{definition}
\end{proof}

\begin{example}[Standard Courant algebroids]\label{ex:StdCourInt}
Let $G\rightrightarrows M$ be a (source simply connected) Lie groupoid, and $A\to M$ the corresponding Lie algebroid. Then the multiplicative Courant algebroid $\mbb{T}G$ from \cref{ex:StdCAGr} integrates the $\mc{LA}$-Courant algebroid $\mbb{T}A$ from \cref{ex:StdLAC}.
\end{example}

\begin{example}[Dirac Lie groups]\label{ex:DirLieInt}
Let $(\mf{d},\g;\h)_\beta$ be a Dirac Manin triple, and $(\mbb{A}_{(\mf{d},\g;\h)_\beta},E_{(\mf{d},\g;\h)_\beta})$ the $\mc{LA}$-Manin pair from \cref{ex:LADirMan}. Then $(\mbb{A}_{(\mf{d},\g;\h)_\beta},E_{(\mf{d},\g;\h)_\beta})$ integrates to the simply connected Dirac Lie group described in \cite{LiBland:2011vqa} corresponding to $(\mf{d},\g;\h)_\beta$.

In more detail, the $\mc{LA}$-Courant algebroid 
\labelcref{eq:DiracManUnRes}
$$\begin{tikzpicture}
\mmat[1em]{m}{T\mf{q}\times \mf{d}&\mf{q}\\ \mf{d}& \ast\\};
\path[->] (m-1-1)	edge (m-1-2)
				edge (m-2-1);
\path[<-] (m-2-2)	edge (m-1-2)
				edge (m-2-1);
\end{tikzpicture}$$
integrates to the action Courant algebroid $(\mf{q}\oplus\overline{\mf{q}})\times D$ where $D$ is the simply connected Lie groupoid integrating $\mf{d}$ and $(\mf{q}\oplus\overline{\mf{q}})$ is the pair groupoid. $(\mf{q}\oplus\overline{\mf{q}})$ acts on $D$ via
$$(\xi,\eta)\to f(\eta)^L-f(\xi)^R,\quad \xi,\eta\in\mf{q}$$ where $f(\xi)^L$ and $f(\xi)^R$ are the left and right invariant vector fields on $D$ with values $f(\xi)$ at the identity.

Meanwhile, the $\mc{LA}$-Dirac structure $T\g\times\mf{d}$ integrates to the multiplicative Dirac structure $$(\g\oplus\g)\times D\subseteq (\mf{q}\oplus\overline{\mf{q}})\times D.$$

Finally, the $\mc{LA}$-Manin pair $(\mbb{A}_{(\mf{d},\g;\h)_\beta},E_{(\mf{d},\g;\h)_\beta})$ integrates to the pullback
$$(\mbb{H}_{(\mf{d},\g;\h)_\beta}, F_{(\mf{d},\g;\h)_\beta}):=\bigg(i^!\big((\mf{q}\oplus\overline{\mf{q}})\times D\big),\big((\g\oplus\g)\times D\big)\circ P_i\bigg),$$ where $i:H\to D$ is the morphism of simply connected Lie groups integrating the inclusion $\h\to \mf{d}$ and $P_i$ and $i^!$ are as described in Section~\ref{sec:pullback}.
\end{example}

\begin{example}[Tangent prolongation of a Courant algebroid]\label{ex:TngProInt}
The tangent prolongation $T\mbb{E}$ of a Courant algebroid $\mbb{E}$ integrates to the multiplicative Courant algebroid $$\mbb{E}\times\overline{\mbb{E}},$$ endowed with the pair groupoid structure (or, more precisely, to the `fundamental groupoid', the source simply connected cover of $\mbb{E}\times\overline{\mbb{E}}$).

Suppose 
$$\begin{tikzpicture}
\mmat[1em]{m}{L&W\\ TM& M\\};
\path[->] (m-1-1)	edge (m-1-2)
				edge (m-2-1);
\path[<-] (m-2-2)	edge (m-1-2)
				edge (m-2-1);
\end{tikzpicture}$$
 is an $\mc{LA}$-Dirac structure in $T\mbb{E}$ such that  $L\subseteq TW$ is the foliation associated to a surjective submersion $W\to W_0$. Let $R:=W\times_{W_0} W$ be the equivalence relation.
Then $R\subseteq \mbb{E}\times\overline{\mbb{E}}$ is the multiplicative Dirac structure integrating $L$.
\end{example}

\begin{example}
Suppose that a Lie group $G$ acts freely and properly on the Courant algebroid $\mbb{E}$, by automorphisms. Then the Courant algebroid $$(\mbb{E}\times\overline{\mbb{E}})/G,$$ where $G$ acts diagonally, 
 is a multiplicative Courant algebroid over $\mbb{E}/G$. Here the multiplication is 
$$[e_1,e_2]\cdot [e_2,e_3]=[e_1,e_3]$$ 
for any $(e_1,e_2),(e_2,e_3)\in\mbb{E}\times\overline{\mbb{E}}$, where $[e_1,e_2]\in(\mbb{E}\times\overline{\mbb{E}})/G$ and $[e_2,e_3]\in(\mbb{E}\times\overline{\mbb{E}})/G$ denote the respective equivalence classes.

The corresponding $\mc{LA}$-Courant algebroid is $T\mbb{E}/G$, as described in \cref{ex:AtiyahLACA}.
\end{example}

\begin{remark}[A second supergeometric perspective]\label{rem:MPalg}
$\mc{LA}$-Manin pairs are also equivalent to MP-algebroids, as introduced in \cite{LiBland:2010wi}. We recall their definition here.
\begin{definition}[MP-algebroid]\label{MPA}
An MP-algebroid is a graded Lie algebroid $P$, such that $P$ is also an MP-manifold, and
\begin{enumerate}
\renewcommand{\labelenumi}{MPA-\arabic{enumi}}
\item the Poisson structure on $P$ is linear, defining a Lie algebroid structure on $P^*$ (see \cite{LieAlgebroidsH}), 
\item the Lie algebroid structures on $P$ and $P^*$ are compatible, so that $P$ is a Lie bialgebroid (see \cite{Mackenzie-Xu94,Voronov:2002wl,Voronov:2006wh,Voronov:2007tf}), and
\item the action map $P\times\mathbb{R}[2]\to P$ is a Lie algebroid morphism, where $\mathbb{R}[2]$ is viewed as a trivial Lie algebra.
\end{enumerate}
Morphisms of MP-algebroids are morphisms of Lie algebroids which are also morphisms of MP-manifolds.
\end{definition}
\end{remark}

\chapter{Outlook}\label{chp:Outlook}

\section{Reduction of Courant algebroids}\label{sec:reduc}
In this section, we apply the theory of $\mc{LA}$-Dirac structures and $\mc{VB}$-Dirac structures to the reduction of Courant algebroids. Reduction of Courant algebroids was first studied in \cite{Bursztyn:2007ko,Stienon:2008cl,Vaisman:2007gg,Hu:2009wl, Lin:2006ku,Zambon:2008wj,Yoshimura:2007gw}, and further studied in \cite{Calvo:2010bj,Jotz:2011kt, Jotz:2011cz, Goldberg:2010wj, Baird:2010ge,Vaisman:2010bt}. 

Recently, a very general and novel approach to reduction in terms of supergeometry has emerged in the work of  Bursztyn,  Cattaneo,  Mehta, and  Zambon \cite{Mehta:2010ux,Cattaneo:2010th,Cattaneo:2010vr,Bursztyn:gcFguuB1}. From such a supergeometric perspective, in \cref{sec:LAred} we reduce by coisotropic $NQ$-submanifolds of the degree 2-symplectic $NQ$-manifold corresponding to the Courant algebroid,  while in \cref{sec:VBred} we reduce by presymplectic $NQ$-submanifolds. 

%
%
%



\subsection{Reduction via $\mc{LA}$-Dirac structures in $T\mbb{E}$}\label{sec:LAred}
Reduction is a process which lowers the dimensions of a given system through a combination of imposing constraints and quotienting out by symmetries. For a Courant algebroid $\mbb{E}$, constraints correspond to specifying a subbundle $W\subseteq \mbb{E}$, while symmetries correspond to specifying a linear foliation of $W$: that is, an $\mc{LA}$-subbundle\footnote{By the Frobenius theorem.},
$$\begin{tikzpicture}
\mmat[1em]{m} at (-2,0) {L&W\\ F& S\\};
\path[->] (m-1-1)	edge (m-1-2)
				edge (m-2-1);
\path[<-] (m-2-2)	edge (m-1-2)
				edge (m-2-1);
				
\draw (0,0) node {$\subseteq$};
				
\mmat[1em]{m2} at (2,0) {T\mbb{E}&\mbb{E}\\ TM& M\\};
\path[->] (m2-1-1)	edge (m2-1-2)
				edge (m2-2-1);
\path[<-] (m2-2-2)	edge (m2-1-2)
				edge (m2-2-1);
\end{tikzpicture}$$
(Here $L\subseteq T\mbb{E}\rvert_W$ is the subbundle tangent to the leaves of the foliation). We say that $L$ is \emph{regular} if the leaf space, $\mbb{F}$, is a manifold and the quotient map, $p:W\to \mbb{F}$, is a surjective submersion. In this section, we show that when $L$ is an $\mc{LA}$-Dirac structure (with support on $F$), $\mbb{F}$ is naturally a Courant algebroid. Moreover,  the composition
\begin{equation}\label{eq:CRR1}\gr(p)\circ \gr(i)^\top:\mbb{E}\dasharrow \mbb{F}\end{equation} is a Courant relation, where $i:W\to \mbb{E}$ denotes the inclusion and $\gr(i)^\top:\mbb{E}\dasharrow W$ is the transpose relation (cf. \cref{sec:LinRel}).


%
%

%

Before proceeding, we make the following definition, abstracting the properties of the Courant relation \labelcref{eq:CRR1}

\begin{definition}
Suppose $\mbb{E}$ and $\mbb{F}$ are Courant algebroids, and let $p_\mbb{E}:\mbb{F}\times\overline{\mbb{E}}\to\overline{\mbb{E}}$ and $p_\mbb{F}:\mbb{F}\times\overline{\mbb{E}}\to\mbb{F}$ denote the two projections.
A Courant relation $$Q:\mbb{E}\dasharrow\mbb{F}$$ is called a \emph{Courant reduction relation} if $p_{\mbb{F}}\rvert_Q:Q\to\mbb{F}$ is a surjective submersion, while $p_{\mbb{E}}\rvert_Q:Q\to\overline{\mbb{E}}$ is an embedding. 

We say that $Q$ has \emph{connected fibres} if the fibres of $p_{\mbb{F}}\rvert_Q:Q\to\mbb{F}$ are connected.
\end{definition}

\begin{remark}
If two Courant reduction relations $Q_1:\mbb{E}_1\dasharrow\mbb{E}_2$ and $Q_2:\mbb{E}_2\dasharrow\mbb{E}_3$ compose cleanly, then their composition $$Q_2\circ Q_1:\mbb{E}_1\dasharrow\mbb{E}_2$$ is a Courant reduction relation.
\end{remark}

\begin{proposition}\label{prop:LAred}
Suppose $$\begin{tikzpicture}
\mmat[1em]{m} at (-2,0) {L&W\\ F& S\\};
\path[->] (m-1-1)	edge (m-1-2)
				edge (m-2-1);
\path[<-] (m-2-2)	edge (m-1-2)
				edge (m-2-1);
				
\draw (0,0) node {$\subseteq$};
				
\mmat[1em]{m2} at (2,0) {T\mbb{E}&\mbb{E}\\ TM& M\\};
\path[->] (m2-1-1)	edge (m2-1-2)
				edge (m2-2-1);
\path[<-] (m2-2-2)	edge (m2-1-2)
				edge (m2-2-1);
\end{tikzpicture}$$
is a regular $\mc{LA}$-Dirac structure (with support). Let $\mbb{F}$ denote the leaf space, and $i:W\to \mbb{E}$ and $p:W\to \mbb{F}$ denote the inclusion and the quotient map, respectively. Then there is a unique Courant algebroid structure on $\mbb{F}$ such that
\begin{equation}\label{eq:QL}Q_L:=\gr(p)\circ \gr(i)^\top:\mbb{E}\dasharrow \mbb{F}\end{equation} is a Courant reduction relation with connected fibres. This defines a one-to-one correspondence between regular $\mc{LA}$-Dirac structures (with support) and Courant reduction relations with connected fibres.
\end{proposition}
\begin{lemma}\label{lem:CARRdetCAStr}
Suppose that $Q:\mbb{E}\dasharrow\mbb{F}$ is a Courant reduction relation. Then the Courant algebroid structure on $\mbb{F}$ is uniquely determined by the Courant algebroid structure on $\mbb{E}$ and the underlying vector bundle relation $Q:\mbb{E}\dasharrow\mbb{F}$.
\end{lemma}
\begin{proof}
By definition, 
the restriction of $p_\mbb{E}$ to $Q$ defines an isomorphism $$j:=p_\mbb{E}\rvert_Q:Q\to W_Q$$ where the subbundle $W_Q\subseteq\mbb{E}$ is the image of $Q$ under $p_\mbb{E}$. Similarly, the restriction of $p_\mbb{F}$ to $Q$ defines a surjective submersion $$p:=p_{\mbb{F}}\rvert_Q\circ j^{-1}:W_Q\to \mbb{F}.$$

Suppose that $\sigma,\tau\in\Gamma(\mbb{F})$ are any two sections, and let $\tilde\sigma,\tilde\tau\in\Gamma(\mbb{E})$ be any extensions of $p^*\sigma,p^*\tau\in\Gamma(W_Q/W_Q^\perp)$ to all of $\mbb{E}$ i.e.
	$$ \tilde\sigma\sim_Q\sigma,\quad \tilde\tau\sim_Q\tau.$$
Since $Q$ is a Courant relation, 
\begin{align*} \Cour{\tilde\sigma,\tilde\tau}&\sim_Q\Cour{\sigma,\tau},\\ \la\tilde\sigma,\tilde\tau\ra&\sim_Q\la\sigma,\tau\ra,\end{align*}
which shows that the Courant bracket and metric on $\mbb{F}$ are uniquely determined by the corresponding structures on $\mbb{E}$.
\end{proof}
\begin{proof}[Proof of \cref{prop:LAred}]
%

\begin{description}
\item[$\Rightarrow$] In light of \cref{lem:CARRdetCAStr}, it suffices to show that there exists a Courant algebroid structure on $\mbb{F}$ such that \labelcref{eq:QL} is a Courant reduction relation. We prove this fact in \cref{prop:VBred}. However, here we offer an alternative proof which takes advantage of the theory presented in \cref{sec:Integr}:

Let $N$ denote the base of the vector bundle $\mbb{F}\to N$. $L$ integrates to the multiplicative Dirac structure $R=W\times_{\mbb{F}} W\subseteq \mbb{E}\times\overline{\mbb{E}}$ with support on $S\times_{N} S\subset M\times M$.

Let $\sigma,\tau\in\Gamma(\mbb{F})$ be any two sections, and suppose that $\tilde\sigma_1,\tilde\sigma_2,\tilde\tau_1,\tilde\tau_2\in \Gamma(\mbb{E})$ are such that their restrictions to $S$ are $p:W\to \mbb{F}$ related to $\sigma$ and $\tau$, respectively, i.e. 
$$\tilde\sigma_i\sim_{Q_L}\sigma,\quad \tilde\tau_i\sim_{Q_L}\tau, \quad (i=1,2).$$
Then $(\tilde\sigma_1,\tilde\sigma_2)\rvert_{S\times_{N} S}\in\Gamma(R)$ and $(\tilde\tau_1,\tilde\tau_2)\rvert_{S\times_{N} S}\in\Gamma(R)$. Hence
$$(\Cour{\tilde\sigma_1,\tilde\tau_1},\Cour{\tilde\sigma_2,\tilde\tau_2})\rvert_{S\times_{N} S}\in\Gamma(R)$$
namely, $\Cour{\tilde\sigma_1,\tilde\tau_1}\rvert_S$ and $\Cour{\tilde\sigma_2,\tilde\tau_2}\rvert_S$ are both $p:W\to \mbb{F}$ related to the same section, which we denote by $\Cour{\sigma,\tau}\in\Gamma(\mbb{F})$. In particular, $\Cour{\sigma,\tau}$ doesn't depend on the extensions of $\sigma,\tau\in\Gamma(\mbb{F})$ to sections of $\mbb{E}$.

 Similarly, 
$$\la (\tilde\sigma_1,\tilde\sigma_2),(\tilde\tau_1,\tilde\tau_2)\ra\rvert_{S\times_{N} S}=0,$$
 so there exists a unique function, which we denote by $\la\sigma,\tau\ra\in C^\infty(N)$, such that 
 $$\la\tilde\sigma_i,\tilde\tau_i\ra\rvert_S=p^*\la\sigma,\tau\ra$$
  for $i=1,2$. In particular, $\la\sigma,\tau\ra$ doesn't depend on the extensions of $\sigma,\tau\in\Gamma(\mbb{F})$ to sections of $\mbb{E}$.

Since $\mbb{E}$ satisfies the axioms of a Courant algebroid, it follows directly that $\mbb{F}$ endowed with this Courant bracket and fibre metric does too. Moreover, by construction, \labelcref{eq:QL} is a Courant reduction relation.

\item[$\Leftarrow$] Suppose next that $$Q:\mbb{E}\dasharrow \mbb{F}$$ is a Courant reduction relation with connected fibres. Let $p:W_Q\to \mbb{F}$ be the surjective submersion described in the proof of \cref{lem:CARRdetCAStr}.

Taking the tangent lift of $Q$ yields the $\mc{LA}$-Courant relation $TQ:T\mbb{E}\dasharrow T\mbb{F}$ (cf. \cref{ex:TngProIsLACA}). The zero section $\mbb{F}\subseteq T\mbb{F}$ is an $\mc{LA}$-Dirac structure with support (cf. \cref{ex:FreeLADir}). Thus the composition $L_Q:=\mbb{F}\circ TQ$ defines an $\mc{LA}$-Dirac structure with support,
$$\begin{tikzpicture}
\mmat[1em]{m} at (-2,0) {L_Q&W_Q\\ F_Q& S_Q\\};
\path[->] (m-1-1)	edge (m-1-2)
				edge (m-2-1);
\path[<-] (m-2-2)	edge (m-1-2)
				edge (m-2-1);
				
\draw (0,0) node {$\subseteq$};
				
\mmat[1em]{m2} at (2,0) {T\mbb{E}&\mbb{E}\\ TM& M\\};
\path[->] (m2-1-1)	edge (m2-1-2)
				edge (m2-2-1);
\path[<-] (m2-2-2)	edge (m2-1-2)
				edge (m2-2-1);
\end{tikzpicture}$$

Since $L_Q$ is defined as the composition of $TQ$ with the zero section of $T\mbb{F}\to \mbb{F}$, we have $$L_Q=\on{ker}(dp)\subseteq TW_Q.$$ Namely, $L_Q$ is the subbundle tangent to the fibres of the surjective submersion $$p:W_Q\to\mbb{F}.$$
By assumption, this submersion has connected fibres, so $\mbb{F}$ is canonically identified with the leaf space of $L_Q\subseteq TW_Q$.

\end{description}
 By construction, the $\mc{LA}$-vector subbundle $L\subseteq T\mbb{E}$ entirely determines the vector bundle relation $Q_L:\mbb{E}\dasharrow\mbb{F}$, and vice versa. Thus \cref{lem:CARRdetCAStr} shows that the constructions above invert each other, establishing the one-to-one correspondence.
\end{proof}

\begin{example}
Suppose $L\subseteq T\mbb{E}$ is supported on all of $TM$, and let  $(W,\nabla)$ be the corresponding pseudo-Dirac structure in $\mbb{E}$ (cf. \cref{thm:LieSubIsVBDir}). \Cref{prop:CoLieSub} shows that $\nabla$ defines a flat connection on the pseudo-euclidean vector bundle $W/W^\perp$. Moreover if $\nabla$ defines a trivialization of $W/W^\perp$ with typical fibre $\mf{d}$,  then $$W/W^\perp\cong \mf{d}\times M$$ is isomorphic to an action Courant algebroid (cf. \cref{ex:ActCourAlg}) for some quadratic Lie algebra structure on $\mf{d}$.

In particular, the corresponding reduction of $\mbb{E}$ is the quadratic Lie algebra $\mf{d}$.
\end{example}

The following example is originally due to Courant \cite{Courant:1990uy}, but was extended to arbitrary Courant algebroids in \cite{Bursztyn:2003ud,LiBland:2009ul}.
\begin{example}
Suppose $\mbb{E}\to M$ is a Courant algebroid, and  $N\subseteq M$ is an embedded submanifold transverse to the anchor map. Then $C=\mbf{a}^{-1}(TN)\subseteq\mbb{E}$ is a well defined subbundle. 
Let $i:N\to M$ denote the embedding, and consider the Courant morphism $$P_i:i^!\mbb{E}\dasharrow \mbb{E},$$ canonically associated to the pull-back Courant algebroid $i^!\mbb{E}\cong C/C^\perp$ \cite{LiBland:2009ul} (cf. \cref{def:CAPullback}). 

The transpose relation 
\begin{equation}\label{eq:RiCourant}P_i^\top:\mbb{E}\dasharrow i^!\mbb{E}\end{equation} is a Courant reduction relation. 

The corresponding $\mc{LA}$-Dirac structure is
$$\begin{tikzpicture}
\mmat{m} at (-2,0) {L&C\\ N& N\\};
\path[->] (m-1-1)	edge (m-1-2)
				edge (m-2-1);
\path[<-] (m-2-2)	edge (m-1-2)
				edge (m-2-1);
\draw (-2,0) node (c) {$C^\perp$};
\path[left hook->] (c) edge (m-1-1);
				
\draw (0,0) node {$\subseteq$};
				
\mmat{m2} at (2,0) {T\mbb{E}&\mbb{E}\\ TM& M\\};
\path[->] (m2-1-1)	edge (m2-1-2)
				edge (m2-2-1);
\path[<-] (m2-2-2)	edge (m2-1-2)
				edge (m2-2-1);
\draw (2,0) node (e) {$\mbb{E}$};
\path[left hook->] (e) edge (m2-1-1);
\end{tikzpicture}$$
where $L$ is the tangent bundle to the fibres of the quotient map $C\to C/C^\perp$.
\end{example}

\begin{example}
Suppose that a Lie group $G$ acts freely and properly by automorphisms on the Courant algebroid $\mbb{E}\to M$. Let $\mbb{F}=\mbb{E}/G$ denote the quotient space. Then the quotient map $$p:\mbb{E}\dasharrow\mbb{F}$$ is a Courant reduction relation. The corresponding $\mc{LA}$-Dirac structure,
$$\begin{tikzpicture}
\mmat{m} at (-2,0) {L&\mbb{E}\\ F& M\\};
\path[->] (m-1-1)	edge (m-1-2)
				edge (m-2-1);
\path[<-] (m-2-2)	edge (m-1-2)
				edge (m-2-1);
				
\draw (0,0) node {$\subseteq$};
				
\mmat{m2} at (2,0) {T\mbb{E}&\mbb{E}\\ TM& M\\};
\path[->] (m2-1-1)	edge (m2-1-2)
				edge (m2-2-1);
\path[<-] (m2-2-2)	edge (m2-1-2)
				edge (m2-2-1);
\end{tikzpicture}$$
is the tangent bundle of the $G$-orbits.
\end{example}

The following example is due to Stienon and Xu \cite{Stienon:2008cl} in the context of reducing generalized complex structures.
\begin{example}
Suppose that $G$ is a Lie group acting on a manifold $M$, and $M_0\subseteq M$ is an embedded submanifold on which $G$ acts freely and properly. Let $N:=M_0/G$ denote the quotient space. Let 
$$R_i:\mbb{T}M_0\dasharrow \mbb{T}M$$ and $$R_p:\mbb{T}M_0\dasharrow\mbb{T}N$$  denote the standard lifts of the embedding and the quotient maps,
$i:M_0\to M$ and $p:M_0\to N$ (cf. \cref{ex:StdDiracStr}). 
Then their composition, 
$$R_p\circ R_i^\top:\mbb{T}M\dasharrow\mbb{T}N,$$
 is a Courant reduction relation. 
\end{example}

\subsection{Reduction via $\mc{VB}$-Dirac structures in $T\mbb{E}$}\label{sec:VBred}

In the last section we considered a reduction procedure for a Courant algebroid $\mbb{E}$ based on looking at a $\mc{LA}$-subbundle $L\subset T\mbb{E}$, and described a sufficient condition which guaranteed that the leaf space inherit the structure of a Courant algebroid. However, this set up is not general enough for certain applications. For example, the reduction procedure described by Bursztyn, Cavalcanti and Gualtieri \cite{Bursztyn:2007ko} does not fit into this framework.

In this section we will describe a slightly more general framework, the analogue of Marsden-Ratiu reduction \cite{Marsden:1986vs} for Poisson structures. Consider a double vector subbundle $L\subseteq T\mbb{E}$:
\begin{equation}\label{eq:VBDirForRed}
\begin{tikzpicture}
\mmat[1em]{m} at (-2,0) {L&W\\ E& S\\};
\path[->] (m-1-1)	edge (m-1-2)
				edge (m-2-1);
\path[<-] (m-2-2)	edge (m-1-2)
				edge (m-2-1);
				

\draw (0,0) node {$\subseteq$};
				
\mmat[1em]{m2} at (2,0) {T\mbb{E}&\mbb{E}\\ TM& M\\};
\path[->] (m2-1-1)	edge (m2-1-2)
				edge (m2-2-1);
\path[<-] (m2-2-2)	edge (m2-1-2)
				edge (m2-2-1);
\end{tikzpicture}\end{equation}
Unlike in the previous section, we will no longer assume that $L$ is a Lie subalgebroid (i.e. describes a foliation of $W$). Instead we will assume that  $L\cap TW$ is a constant rank involutive subbundle, and the corresponding leaf space, $\mbb{F}$, is a smooth vector bundle with base $N$. We  let $p:W\to \mbb{F}$ and $p_0:S\to N$ denote the quotient maps.  As we shall show, if $L$ is a $\mc{VB}$-Dirac structure, then $\mbb{F}$ naturally inherits the structure of a Courant algebroid.

\begin{proposition}\label{prop:VBred}

Let $\sigma,\tau\in\Gamma(\mbb{F})$ and suppose $\tilde\sigma,\tilde\tau\in\Gamma(\mbb{E})$ are extensions of $p^*\sigma$ and $p^*\tau$ respectively, which are tangent to $L\subset\mbb{E}$  i.e.
 \begin{subequations}\label[pluralequation]{eq:Lext}
\begin{equation}\label{eq:extCond}p\circ\tilde\sigma\rvert_S=\sigma\circ p_0, \quad p\circ\tilde\tau\rvert_S=\tau\circ p_0,\end{equation}
  and 
\begin{equation}\label{eq:TangencyCond}\tilde\sigma_T\rvert_E,\tilde\tau_T\rvert_E\in\Gamma(L,E).\end{equation}
\end{subequations}

If $\mc{L}$ is a $\mc{VB}$-Dirac structure, then there exists a unique section $\Cour{\sigma,\tau}\in\Gamma(\mbb{F})$ and a unique function $\la\sigma,\tau\ra\in C^\infty (N)$ such that 
\begin{subequations}
\begin{equation}\label{eq:RedCourCond}p\circ \Cour{\tilde\sigma,\tilde\tau}\rvert_S=\Cour{\sigma,\tau}\circ p_0\end{equation} 
and 
\begin{equation}\label{eq:RedCourBrkCond}p\circ \la\tilde\sigma,\tilde\tau\ra=\la\sigma,\tau\ra\circ p_0.\end{equation}
\end{subequations}
Moreover, neither $\Cour{\sigma,\tau}$ nor $\la\sigma,\tau\ra$ depend on the choice of the extensions $\tilde\sigma,\tilde\tau\in\Gamma(\mbb{E})$. Finally, the resulting bracket and pairing on $\Gamma(\mbb{F})$ endow $\mbb{F}$ with the structure of a Courant algebroid.
\end{proposition}

In this reduction procedure, the double vector subbundle $L\cap TW\subseteq \mbb{E}$ serves to define the foliation whose leaf space, $\mbb{F}$, interests us. The purpose of the larger double vector subbundle $L\subseteq \mbb{E}$ is to control how one extends sections of $\mbb{F}$ to sections of $\mbb{E}$. 

The following proof of \cref{prop:VBred} is an adaptation of the proof of \cite[Theorem~3.3]{Bursztyn:2007ko}.

\begin{proof}
We call any section $\tilde\sigma\in\Gamma(\mbb{E})$ an $L$-controlled extension of $\sigma\in\Gamma(\mbb{F})$ if \cref{eq:Lext} hold.

First, we show that a section $\Cour{\sigma,\tau}\in\Gamma(\mbb{F})$ satisfying \cref{eq:RedCourCond} exists. The sections $\tilde\sigma,\tilde\tau\in\Gamma(\mbb{E})$ are tangent to $L$ if and only if their respective tangent lifts are sections of $L\to E$, i.e. \cref{eq:TangencyCond} holds.
Consequently, since $L$ is a $\mc{VB}$-Dirac structure, \cref{eq:TangencyCond} implies $$\Cour{\tilde\sigma,\tilde\tau}_T\rvert_E:=\Cour{\tilde\sigma_T,\tilde\tau_T}\rvert_E\in\Gamma(L,E).$$ Since $\Cour{\tilde\sigma,\tilde\tau}$ is tangent to $L$ (and thus also to $L\cap TW$), it follows that there exists a unique section $\Cour{\sigma,\tau}\in\Gamma(\mbb{F})$ satisfying \cref{eq:RedCourCond}.

Similarly, since $L$ is Lagrangian, \cref{eq:TangencyCond} implies  $$\la\tilde\sigma,\tilde\tau\ra_T\rvert_E:=\la\tilde\sigma_T,\tilde\tau_T\ra\rvert_E=0.$$ In particular, $d\la\tilde\sigma,\tilde\tau\ra=\la\tilde\sigma,\tilde\tau\ra_T$ vanishes on $E\cap TS$, so there exists a unique function $\la\sigma,\tau\ra\in C^\infty(N)$ satisfying \cref{eq:RedCourBrkCond}.

Next, we need to show that $\Cour{\sigma,\tau}$ and $\la\sigma,\tau\ra$ do not depend on the choice of $L$-controlled extensions $\tilde\sigma,\tilde\tau\in\Gamma(\mbb{E})$. Suppose that $\tilde\tau'\in\Gamma(\mbb{E})$ satisfies \cref{eq:Lext}. Since $L$ is Lagrangian, \cref{prop:VBDirIsLaVB} shows its core must be $W^\perp$:
$$\begin{tikzpicture}
\mmat{m} at (-2,0){L&W\\ E& S\\};
\path[->] (m-1-1)	edge (m-1-2)
				edge (m-2-1);
\path[<-] (m-2-2)	edge (m-1-2)
				edge (m-2-1);
\draw (-2,0) node (c) {$W^\perp$};
\path[left hook->] (c) edge (m-1-1);
\draw (0,0) node {$\subseteq$};
				
\mmat{m2} at (2,0) {T\mbb{E}&\mbb{E}\\ TM& M\\};
\path[->] (m2-1-1)	edge (m2-1-2)
				edge (m2-2-1);
\path[<-] (m2-2-2)	edge (m2-1-2)
				edge (m2-2-1);
\draw (2,0) node (e) {$\mbb{E}$};
\path[left hook->] (e) edge (m2-1-1);
\end{tikzpicture}$$
Meanwhile the core of $TW$ is $W$, the subbundle of vertical vectors along $S\subseteq W$ (cf. \cref{ex:TngDVB2,fig:TngCore}):
$$\begin{tikzpicture}
\mmat{m} at (-2,0){TW&W\\ TS& S\\};
\path[->] (m-1-1)	edge (m-1-2)
				edge (m-2-1);
\path[<-] (m-2-2)	edge (m-1-2)
				edge (m-2-1);
\draw (-2,0) node (c) {$W$};
\path[left hook->] (c) edge (m-1-1);
\draw (0,0) node {$\subseteq$};
				
\mmat{m2} at (2,0) {T\mbb{E}&\mbb{E}\\ TM& M\\};
\path[->] (m2-1-1)	edge (m2-1-2)
				edge (m2-2-1);
\path[<-] (m2-2-2)	edge (m2-1-2)
				edge (m2-2-1);
\draw (2,0) node (e) {$\mbb{E}$};
\path[left hook->] (e) edge (m2-1-1);
\end{tikzpicture}$$
 Thus the core of $L\cap TW$ is $W\cap W^\perp$. It follows that $p^*\mbb{F}\cong W/(W\cap W^\perp)$. Hence $\hat\tau:=\tilde\tau-\tilde\tau'$ satisfies $\hat\tau\rvert_S\in\Gamma(W\cap W^\perp)$ and $\hat\tau_T\in\Gamma(L,E)$. Let $\upsilon\in\Gamma(\mbb{E})$ be any section such that $\upsilon\rvert_S\in\Gamma(W)$ and $\upsilon_T\in\Gamma(L,E)$. Then 
\begin{equation}\label{eq:MstVanishRed}\la\Cour{\tilde\sigma,\hat\tau},\upsilon\ra\rvert_S=\mbf{a}(\tilde\sigma) \la\hat\tau,\upsilon\ra\rvert_S-\la\hat\tau,\Cour{\tilde\sigma,\upsilon}\ra\rvert_S.\end{equation}
 Now $\mbf{a}(W)\subseteq TS$ and $\la\hat\tau,\upsilon\ra\rvert_S=0$, so the first term on the right hand side of \cref{eq:MstVanishRed} vanishes. Furthermore, since $\Cour{\tilde\sigma,\upsilon}_T\rvert_E:=\Cour{\tilde\sigma_T,\upsilon_T}\rvert_E\in\Gamma(L,E)$, it follows that $\Cour{\tilde\sigma,\upsilon}\rvert_S\in\Gamma(W)$, thus the last term in \cref{eq:MstVanishRed} vanishes. Since the right hand side of \cref{eq:MstVanishRed} vanishes, we see that 
 $$\Cour{\tilde\sigma,\hat\tau}\rvert_S\in\Gamma(W^\perp\cap W).$$
 Since the map $p:W\to \mbb{F}$ factors through the map $W\to W/(W\cap W^\perp)$, we conclude that $$p\circ\Cour{\tilde\sigma,\tilde\tau}\rvert_S=p\circ\Cour{\tilde\sigma,\tilde\tau'}\rvert_S.$$ This shows that $\Cour{\sigma,\tau}$ doesn't depend on the choice of $L$-controlled extension $\tilde\tau$.
 
 To show that $\Cour{\sigma,\tau}$ doesn't depend on the choice $\tilde\sigma$, suppose that $\tilde\sigma'\in\Gamma(\mbb{E})$ satisfies  \cref{eq:Lext}. Let $\hat\sigma=\tilde\sigma-\tilde\sigma'$, so that $\hat\sigma\rvert_S\in\Gamma(W^\perp\cap W)$.  Then $$\la \upsilon,\mbf{a}^* d\la\tilde\tau,\hat\sigma\ra\ra\rvert_S=\mbf{a}(\upsilon)\la\tilde\tau,\hat\sigma\ra\rvert_S=0$$ for any section $\upsilon\in\Gamma(\mbb{E})$ satisfying $\upsilon\rvert_S\in\Gamma(W)$. This shows that $\mbf{a}^* d\la\tilde\tau,\hat\sigma\ra\rvert_S\in\Gamma(W^\perp)$.  Now the left hand side of  $$\Cour{\hat\sigma,\tilde\tau}\rvert_S=-\Cour{\tilde\tau,\hat\sigma}\rvert_S+\mbf{a}^* d\la\tilde\tau,\hat\sigma\ra\rvert_S$$ lies in $\Gamma(W)$ while both terms on the right hand side lie in $\Gamma(W^\perp)$, so $$\Cour{\hat\sigma,\tilde\tau}\rvert_S\in\Gamma(W\cap W^\perp).$$  We conclude that $$p\circ\Cour{\tilde\sigma,\tilde\tau}\rvert_S=p\circ\Cour{\tilde\sigma',\tilde\tau}\rvert_S.$$ This shows that $\Cour{\sigma,\tau}$ doesn't depend on the choice of $L$-controlled extension $\tilde\sigma$.
 
A similar argument also shows that $\la\sigma,\tau\ra$ doesn't depend on the choice of $L$-controlled extensions $\tilde\sigma$ or $\tilde\tau$.

We have shown that $\Gamma(\mbb{F})$ carries both a well defined bracket and a well defined pairing. It remains to show that these structures endow $\mbb{F}$ with the structure of a Courant algebroid. Since $p^*\mbb{F}\cong W/(W\cap W^\perp)$, the bundle metric on $\mbb{F}$ is non-degenerate. To show that the bracket on $\mbb{F}$ satisfies (c1) of \cref{def:CA} (the Jacobi identity), note that for any $L$-controlled extensions $\tilde\sigma$ and $\tilde\tau$ of $\sigma\in\Gamma(\mbb{F})$ and $\tau\in\Gamma(\mbb{F})$, respectively, $$\Cour{\tilde\sigma,\tilde\tau}_T\rvert_E=\Cour{\tilde\sigma_T,\tilde\tau_T}\rvert_E\in\Gamma(L,E)$$ so that $\Cour{\tilde\sigma,\tilde\tau}$ is an $L$-controlled extension of $\Cour{\sigma,\tau}$. Since the Jacobi identity holds for sections of $\mbb{E}$, it holds (in particular) for $L$-controlled extensions of sections of $\mbb{F}$.
In turn, since it holds for their $L$-controlled extensions, the Jacobi identity must hold for sections of $\mbb{F}$. Similar arguments establish (c2) and (c3) of \cref{def:CA}.
 \end{proof}
 
 \begin{remark}[Supergeometric perspective on \cref{prop:VBred}]
 $\mc{VB}$-Dirac structures in $T\mbb{E}$ are in one-to-one correspondence with $NQ$-submanifolds of the degree 2 symplectic $NQ$-manifold $X$ corresponding to $\mbb{E}$ (see \cref{prop:PoisLie2LACour} for details). An analogous argument to the one Cattaneo and Zambon developed for degree 1 symplectic $NQ$-manifolds \cite{Cattaneo:2010th,Cattaneo:2010vr} should show that presymplectic $NQ$-submanifolds of $X$ are in one-to-one correspondence with $\mc{VB}$-Dirac structures of the form \cref{eq:VBDirForRed} such that $L\cap TW$ is an involutive subbundle. Thus the reduction procedure described in this section can be interpreted as presymplectic reduction of supermanifolds, as is done in \cite{Cattaneo:2010th,Cattaneo:2010vr} for the reduction of Poisson structures.
 
 The supergeometric perspective on the reduction of Courant algebroids is developed in great generality by Bursztyn, Cattaneo, Mehta and Zambon \cite{Bursztyn:gcFguuB1}.
 \end{remark}
 
 \begin{example}
 Recall from \cref{ex:BCGred} that Bursztyn, Cavalcanti and Gualtieri's reduction procedure \cite{Bursztyn:2007ko} for an exact Courant algebroid $\mbb{E}$ defines a $\mc{VB}$-Dirac structure $L\subseteq T\mbb{E}$. \Cref{prop:VBred} shows that the resulting quotient, $$\big(K^\perp/(K\cap K^\perp)\big)/G,$$ inherits the structure of a Courant algebroid. (In fact, as mentioned above, the proof of \cref{prop:VBred} is just an adaptation of their proof of \cite[Theorem~3.3]{Bursztyn:2007ko}).
 \end{example}

\section{Integration of q-Poisson $(\mf{d},\g)$-structures}
There is a one-to-one correspondence between integrable Poisson structures and source-simply-connected symplectic groupoids \cite{Weinstein:1987ua,Coste:1987ui,Crainic02,Lie-Algebroids,Mackenzie97}. 
q-Poisson $(\mf{d},\g)$-structures (see \cref{ex:qPstr}) are a slight generalization of Poisson structures, and it is natural to ask if they integrate to some generalization of a symplectic groupoid. For the special case of q-Poisson $(\mf{d}\oplus\overline{\mf{d}},\mf{d}_\Delta)$-structures, this was accomplished in \cite{LiBland:2010wi}. In this section, we will treat the general case, but first we develop some background. 

\subsection{A canonical morphism}
Suppose $(\mbb{E},A)$ is a Manin pair. There are two $\mc{LA}$-Manin pairs canonically associated to it, $(\mbb{T} A,TA^{flip})$ (cf. \cref{ex:StdLAC}):
$$\begin{tikzpicture}
\mmat[1em]{m} at (-1.5,0) {TA&TM\\ A& M\\};
\path[->] (m-1-1)	edge (m-1-2)
				edge (m-2-1);
\path[<-] (m-2-2)	edge (m-1-2)
				edge (m-2-1);
				
\draw (0,0) node {$\subseteq$};
				
\mmat[1em]{m2} at (2,0) {\mbb{T}A&TM\oplus A^*\\ A& M\\};
\path[->] (m2-1-1)	edge (m2-1-2)
				edge (m2-2-1);
\path[<-] (m2-2-2)	edge (m2-1-2)
				edge (m2-2-1);
\end{tikzpicture}$$
and $(T\mathbb{E},TA)$ (cf. \cref{ex:TngProIsLACA}):
$$\begin{tikzpicture}
\mmat[1em]{m} at (-1.5,0) {TA&A\\ TM& M\\};
\path[->] (m-1-1)	edge (m-1-2)
				edge (m-2-1);
\path[<-] (m-2-2)	edge (m-1-2)
				edge (m-2-1);
				
\draw (0,0) node {$\subseteq$};
				
\mmat[1em]{m2} at (1.5,0) {T\mbb{E}&\mbb{E}\\ TM& M\\};
\path[->] (m2-1-1)	edge (m2-1-2)
				edge (m2-2-1);
\path[<-] (m2-2-2)	edge (m2-1-2)
				edge (m2-2-1);
\end{tikzpicture}$$

\begin{lemma}\label{lem:BurszLem1}

For any Manin pair $(\mathbb{E},A)$, there is a canonical $\mc{LA}$-Courant morphism
$$\begin{tikzpicture}
\mmat[1em]{m} at (-3,0) {\mbb{T}A&TM\oplus A^*\\ A& M\\};
\path[->] (m-1-1)	edge (m-1-2)
				edge (m-2-1);
\path[<-] (m-2-2)	edge (m-1-2)
				edge (m-2-1);

\mmat[1em]{m2} at (3,0) {T\mbb{E}&\mbb{E}\\ TM& M\\};
\path[->] (m2-1-1)	edge (m2-1-2)
				edge (m2-2-1);
\path[<-] (m2-2-2)	edge (m2-1-2)
				edge (m2-2-1);
\path[->,bend right=30] (m-2-1) edge node{$\mbf{a}\rvert_A$} (m2-2-1);
\path[dashed,->] (m) edge node {$K$}(m2);
\end{tikzpicture}$$
over the restriction of the anchor map $\mathbf{a}:\mbb{E}\to TM$ to $A$, such that
 \begin{equation}\label{eq:MPM1}K:(\mbb{T} A,TA^{flip})\dasharrow (T\mathbb{E},TA),\end{equation}  is a  morphism of $\mathcal{LA}$-Manin pairs.
\end{lemma} 

\begin{proof}
Let $P$ be the Poisson principal $\mbb{R}[2]$ bundle corresponding to the Manin pair $(\mbb{E},A)$, as described in \cite{NonComDiffForm,Bursztyn:2009wi} (or see \cref{rem:SupMP}).

The Poisson structure $\pi_P$ on $P$ defines an $\mathbb{R}[0,2]$-equivariant morphism of Lie bialgebroids $\pi_P^\sharp:T^*[1,1]P\to T[1,0]P$. Therefore, \begin{equation}\label{eq:piP}\pi_P^\sharp:T^*[1,1]P/\!\!/_0 \mathbb{R}[0,2]\to T[1,0]P/\mathbb{R}[0,2]\end{equation} is a morphism of Lie bialgebroids. But \begin{equation}\label{eq:Leg}T^*[1,1]P/\!\!/_0 \mathbb{R}[0,2]\cong T^*[1,1]A[1,0]\end{equation} by the Legendre transform. Therefore, \labelcref{eq:piP} is is equivalent to a morphism of $\mathcal{LA}$-Manin pairs $$(\mbb{T} A,TA^{flip})\dasharrow (T\mathbb{E},TA),$$
as described in \cref{rem:MPalg}.
\end{proof}

\begin{remark}
The global version of \cref{lem:BurszLem1} is explained in \cite[\S~4.2]{Bursztyn:2009wi}, following \cite{Bursztyn03-1,Ponte:2005txa}. Namely, if the Lie algebroid $A$ integrates to a source simply connected groupoid $G$, then there is a morphism of multiplicative Manin pairs $$(\mbb{T}G,TG)\dasharrow (\mathbb{E},A)\times(\bar{\mathbb{E}},A)$$ over $(t,s):G\to M\times M$, where $s$ and $t$ are the source and target maps respectively.
\end{remark}

It will be useful to describe this morphism of $\mathcal{LA}$-Manin pairs in classical terms, which we shall do in \cref{prop:formOfK} after introducing some new notation in the next section.

\subsubsection{Some vector fields}

For any $f\in C^\infty(M)$, $\tau\in\Gamma(\mathbb{E})$,  the functions $q_{\mbb{E}/M}^*f$ and $\la \tau,\cdot\ra$ are respectively constant and linear along the fibres of $\mathbb{E}$. Any section $\sigma\in\Gamma(\mathbb{E})$ defines two infinitesimal flows $\on{Drf}_\sigma,\on{Trl}_\sigma\in\mf{X}(\mathbb{E})$ on $\mathbb{E}$ \cite{Stienon:2010ws,Roytenberg:2002} (see also \cite{Hu:2007wq, Bursztyn:2007ko, Hu:2009wl, Stienon:2008cl}), by 
\begin{equation}\label[pluralequation]{eq:DrfTrlDef}\begin{array}{rcl}
\on{Drf}_\sigma\cdot q_{\mbb{E}/M}^*f=q_{\mbb{E}/M}^*(\mbf{a}(\sigma)\cdot f),&&
\on{Drf}_\sigma\cdot \la \tau,\cdot\ra=\la\Cour{\sigma,\tau},\cdot\ra,\\
\on{Trl}_\sigma\cdot q_{\mbb{E}/M}^*f=0, &\text{ and }&
\on{Trl}_\sigma\cdot \la\tau,\cdot\ra=q_{\mbb{E}/M}^*\la\sigma,\tau\ra.
\end{array}\end{equation}
(The notation $\on{Drf}$ refers to the Dorfman bracket \cite{Dorfman:1993us} and $\on{Trl}$ refers to fibrewise translation). For $f\in C^\infty(M)$, we have \begin{equation}\label{eq:DrfCinfty}\on{Trl}_{f\sigma}=q_{\mbb{E}/M}^*f\on{Trl}_\sigma, \text{ and } \on{Drf}_{f\sigma}=q^*_{\mbb{E}/M}f\on{Drf}_\sigma+\la\mbf{a}^*df,\cdot\ra\on{Trl}_\sigma-\la\sigma,\cdot\ra\on{Trl}_{\mbf{a}^*df},\end{equation} as can be checked from \cref{def:CA,eq:DrfTrlDef}.

\begin{lemma}\label{lem:DrfTrlCom}
The vector fields $\on{Drf}_\tau$ and $\on{Trl}_\tau$ satisfy the following commutation relations
\begin{align}
[\on{Drf}_\sigma,\on{Drf}_\tau]&=\on{Drf}_{\Cour{\sigma,\tau}} & [\on{Drf}_\sigma,\on{Trl}_\tau]&=\on{Trl}_{\Cour{\sigma,\tau}}\\
[\on{Trl}_\sigma,\on{Drf}_\tau]&=\on{Trl}_{(\Cour{\sigma,\tau}-\mbf{a}^*d\la\sigma,\tau\ra)}& [\on{Trl}_\sigma,\on{Trl}_\tau]&=0
\end{align}
for any $\sigma,\tau\in\Gamma(\mathbb{E})$.
\end{lemma}
\begin{proof}
Let $f\in C^\infty(M)$ and $\upsilon\in\Gamma(\mathbb{E})$.
$$\begin{array}{rcl}
[\on{Drf}_\sigma,\on{Drf}_\tau](\la\upsilon,\cdot\ra+q_{\mbb{E}/M}^*f)&=&\la \Cour{\sigma,\Cour{\tau,\upsilon}},\cdot\ra-\la \Cour{\tau,\Cour{\sigma,\upsilon}},\cdot\ra\\
&&+q_{\mbb{E}/M}^*[\mbf{a}(\sigma),\mbf{a}(\tau)]f\\
&=&\la \Cour{ \Cour{\sigma,\tau},\upsilon},\cdot\ra+q_{\mbb{E}/M}^*\mbf{a}\big(\Cour{\sigma,\tau}\big)f\\
&=&\on{Drf}_{\Cour{\sigma,\tau}}(\la\upsilon,\cdot\ra+q_{\mbb{E}/M}^*f)
\end{array}$$

$$\begin{array}{rcl}
[\on{Drf}_\sigma,\on{Trl}_\tau](\la\upsilon,\cdot\ra+q_{\mbb{E}/M}^*f)
&=&q_{\mbb{E}/M}^*\big(\mbf{a}(\sigma)\la\upsilon,\tau\ra-\la\Cour{\sigma,\upsilon},\tau\ra\big)\\
&=&q_{\mbb{E}/M}^*\la\upsilon,\Cour{\sigma,\tau}\ra\\
&=&\on{Trl}_{\Cour{\sigma,\tau}}(\la\upsilon,\cdot\ra+q_{\mbb{E}/M}^*f)
\end{array}$$
From this it follows that
$$[\on{Trl}_\sigma,\on{Drf}_\tau]=\on{Trl}_{-\Cour{\tau,\sigma}}=\on{Trl}_{(\Cour{\sigma,\tau}-\mbf{a}^*d\la\sigma,\tau\ra)}.$$
And the fact that  $[\on{Trl}_\sigma,\on{Trl}_\tau]=0$ is immediate from \cref{eq:DrfTrlDef}.
\end{proof}

\subsubsection{The morphism of Manin pairs $(\mathbb{T}A,TA^{flip})\dasharrow (T\mathbb{E},TA)$}

\begin{proposition}\label{prop:formOfK}
The morphism of Manin pairs \begin{equation}\label{eq:MPM2}K:(\mbb{T} A,TA^{flip})\dasharrow (T\mathbb{E},TA),\end{equation} described in \cref{lem:BurszLem1} is given explicitly by a Dirac structure $K\subset T\mathbb{E}\times\overline{\mbb{T} A}$ with support on $\gr_{\mbf{a}\rvert_A}$, the graph of the anchor map $\mbf{a}\rvert_A:A\to TM$. For any $a\in A$, the fibre of $K$ at the point $(\mbf{a}(a),a)$ is spanned by the elements
\begin{equation}\label{eq:FormOfK}
\big((\mbf{a}^*\alpha)_C;(0,\alpha)\big),\text{ and }
\big(\sigma_T+\tau_C;(\on{Drf}_\sigma+\on{Trl}_\tau,d\la\sigma,\cdot\ra\rvert_A)\big)\rvert_{(\mbf{a}(a),a)}
\end{equation}
subject to the requirement $(\on{Drf}_\sigma+\on{Trl}_\tau)\rvert_a\in T_aA$. Here $\alpha\in T^*_{\mathbf{a}(a)}M$, and $\sigma,\tau\in\Gamma(\mathbb{E})$.
\end{proposition}

\begin{proof}
Let $K'$ denote the Dirac structure with support on $\gr(\mbf{a})$ defined in \cref{lem:BurszLem1}. Once a Lagrangian subbundle $B\subset \mathbb{E}$ transverse to $A$ has been chosen,  the morphisms of Manin pairs \labelcref{eq:MPM1,eq:MPM2} are both determined by a comorphism of Lie algebroids \begin{equation}\label{eq:comorph}TA^{flip}\dasharrow TA\end{equation} together  with the respective Lagrangian subbundles \begin{equation}\label{eq:complement}TB\circ K' \subset \mbb{T} A \text{ and } TB\circ K\subset \mbb{T} A\end{equation} transverse to $TA^{flip}$.

We first show that $K'$ determines the same comorphism \labelcref{eq:comorph} as $K$. Let $\sigma\in\Gamma(A)$. For the sake of brevity, we identify $\sigma$ with the degree $(0,1)$ function $\la\sigma,\cdot\ra$ on $P\to A^*[0,1]$. Let $p:T[1,0]P\to P$ and $p':T^*[1,1]P\to P$ denote the canonical projections. Then the pull back of $p^*\sigma$ under \labelcref{eq:piP} is just $$(\pi^\sharp_P)^*p^*\sigma=p'^*\sigma,$$ which  is identified with the vector field $\on{Trl}_\sigma$ by the Legendre transform \labelcref{eq:Leg}. Meanwhile, the degree $(0,1)$ function $p^*\sigma$ is identified with $\sigma_C\in\Gamma_C(T\mbb{E},TM)$.
So, we have \begin{equation}\label{eq:K'com_C}(\on{Trl}_\sigma,0)\sim_{K'}\sigma_C.\end{equation}
On the other hand, the pull back under \labelcref{eq:piP} of the degree $(1,1)$ function $d\sigma$ is the Hamiltonian vector field $$\{\sigma,\cdot\}_{\pi_P}:=(\pi^\sharp_P)^*d\sigma,$$ which under the Legendre transform \labelcref{eq:Leg} is identified with the vector field $\on{Drf}_\sigma$. Meanwhile the degree $(1,1)$ function $d\sigma$ is identified with $\sigma_T\in\Gamma_l(T\mbb{E},TM)$, so we have \begin{equation}\label{eq:K'com_T}(\on{Drf}_\sigma,0)\sim_{K'}\sigma_T.\end{equation}

It follows from \cref{eq:K'com_C,eq:K'com_T,eq:FormOfK} that $K$ and $K'$ both define the same comorphism of Lie algebroids $TA^{flip}\dasharrow TA$.

Next we show that $TB\circ K' =TB\circ K$. Choosing the Lagrangian complement $B\subset \mathbb{E}$ endows $A$ with the structure of a Lie quasi-bialgebroid. This determines the following data:
\begin{itemize}
\item A trivialization $\on{triv}:P\to \mathbb{R}[0,2]\times A^*[0,1]$ (cf.  \cite{NonComDiffForm,Bursztyn:2009wi} or \cref{rem:SupPsDir}). If $t$ is the standard coordinate on $\mathbb{R}[0,2]$, let $s:=\on{triv}^*t$ be the corresponding degree $(0,2)$ function on $P$.
\item The degree $(0,1)$ vector field $Q:=\{s,\cdot\}_{\pi_P}$ on $A^*[0,1]$ (see \cite{NonComDiffForm,Bursztyn:2009wi}).
\item A bivector field $\Pi_B$ on $A$, which is identified with $Q$ under the Legendre transform \labelcref{eq:Leg}. Hence
\begin{subequations}\label[pluralequation]{eq:PiB}
\begin{align*}
\Pi_B(d\la\sigma,\cdot\ra,d\la\tau,\cdot\ra)&=\la\Cour{\sigma,\tau},\cdot\ra,&&\sigma,\tau\in\Gamma(B)\\
\Pi_B(df,d\la\sigma,\cdot\ra)&=-\mbf{a}(\sigma)\cdot f,&&f\in C^\infty(M),\sigma\in\Gamma(B)\\
\Pi_B(df,dg)&=0&&f,g\in C^\infty(M).
\end{align*}
\end{subequations}
(See Roytenberg's original paper on quasi-Lie bialgebroids \cite{Roytenberg02}, or \cite[Section~3.2]{Bursztyn:2009wi})
\end{itemize}
The function $ds:T[1,0]P/\mathbb{R}[0,2]\to\mathbb{R}[1,2]$ trivializes the MP-algebroid $T[1,0]P/\mathbb{R}[0,2]$ and pulls back to $$(\pi_P^\sharp)^*ds=\{s,\cdot\}=Q,$$ which is identified with $\Pi_B$ under the Legendre transform \labelcref{eq:Leg}. Therefore 
\begin{equation}\label{eq:BK'}TB\circ K'=\gr_{\Pi_B}:=\{\big(\Pi_B^\sharp(\alpha),\alpha\big)\mid \alpha\in T^*A\}\subset \mbb{T} A.\end{equation}

We must calculate $TB\circ K$ to compare the two. Let $\on{pr}_A:\mathbb{E}\to A$ and $\on{pr}_B:\mathbb{E}\to B$ denote the two projections along $B$ and $A$ respectively. Then from \cref{eq:FormOfK} we see that for any $f\in C^\infty(M)$, \begin{subequations}\begin{equation}\label{eq:BcK1}(-\on{Trl}_{\on{pr}_A(\mbf{a}^*df)},df)\sim_K \on{pr}_B(\mbf{a}^*df)_C.\end{equation} 

To proceed further, we need the following Lemma.
\begin{lemma}\label{lem:tecLemTA}
For any $\sigma\in\Gamma(\mathbb{E})$, and any extension of $a\in A$ to a section $\tilde a\in\Gamma(A)$, the vector field $\on{Drf}_\sigma+\on{Trl}_{\Cour{\sigma,\tilde a}}$ is tangent to $A$ at $a$.
 \end{lemma}
\begin{proof}
To show this, let $\tau\in\Gamma(A)$, then
$$\begin{array}{rcl}
(\on{Drf}_\sigma+\on{Trl}_{\Cour{\sigma,\tilde a}})\la \tau,\cdot\ra\rvert_a
&=&\la \Cour{\sigma,\tau},a\ra+\la \tau,\Cour{\sigma,\tilde a}\ra\\
&=&\mbf{a}(\sigma)\la \tau,\tilde a\ra\\
&=&0
\end{array}$$
In the last line we used the fact that $A$ is a Lagrangian subbundle of $\mathbb{E}$. Since the ideal of functions vanishing on $A\subset\mathbb{E}$ is generated by the functions $\la \tau,\cdot\ra$ for $\tau\in\Gamma(A)$, the claim follows.
\end{proof}
It follows from \cref{lem:tecLemTA} and \labelcref{eq:FormOfK} that for any $\sigma\in\Gamma(B)$
\begin{equation}\label{eq:BcK2}(\on{Drf}_{\sigma}+\on{Trl}_{\on{pr}_B\Cour{\sigma,\tilde a}},d\la\sigma,\cdot\ra)\sim_K\sigma_T+\big(\on{pr}_B\Cour{\sigma,\tilde a}\big)_C.\end{equation} 
(Notice that $\on{pr}_B\Cour{\sigma,\tilde a}\rvert_{q_{\mbb{E}/M}(a)}$ does not depend on the extension $\tilde a\in\Gamma(A)$ of $a\in A$).

Since the right hand side of \cref{eq:BcK1,eq:BcK2} are in $TB$, we may deduce that $(TB\circ K)_a$ is spanned by elements of the form
\begin{equation}\label{eq:BcK}(-\on{Trl}_{\on{pr}_A(\mbf{a}^*df)},df),\text{ and }(\on{Drf}_{\sigma}+\on{Trl}_{\on{pr}_B\Cour{\sigma,\tilde a}},d\la\sigma,\cdot\ra).\end{equation}\end{subequations}
This shows that $B\circ K$ is the graph of a bivector $\Pi_K$ satisfying 
\begin{align*}
\Pi_K(d\la\sigma,\cdot\ra,d\la\tau,\cdot\ra)&:=(\on{Drf}_{\sigma}+\on{Trl}_{\on{pr}_B\Cour{\sigma,\tilde a}})\la\tau,\cdot\ra\rvert_a=
\la\Cour{\sigma,\tau},a\ra,&& \sigma,\tau\in\Gamma(B)\\
\Pi_K(df,d\la\sigma,\cdot\ra)&:=-\on{Trl}_{\on{pr}_A(\mbf{a}^*df)}\la\sigma,\cdot\ra\rvert_a
=-\mbf{a}(\sigma)\cdot f,&& f\in C^\infty(M), \sigma\in\Gamma(B)\\
\Pi_K(df,dg)&:=-\on{Trl}_{\on{pr}_A(\mbf{a}^*df)}dg\rvert_a=0,&& f,g\in C^\infty(M).
\end{align*}
Comparison with \cref{eq:PiB} shows that $\Pi_K=\Pi_B$ and thus $B\circ K=B\circ K'$, which concludes the proof.
\end{proof}

\subsection{Transfer of data}Let $\mbb{F}\to N$ be a Courant algebroid. The Manin pair $(\mbb{T} M,TM)$ is canonically associated to the manifold $M$. Consequently, in a morphism of Manin pairs 
\begin{equation}\label{eq:trivLS}K:(\mbb{T} M,TM)\dasharrow (\mbb{F},A)\end{equation}
over the map $\phi:M\to N$,
   all the interesting information is stored in the morphism $K$ itself. However, we can reorganize the information stored in \labelcref{eq:trivLS}, and replace it with \begin{equation}\label{eq:trivMorp}P_\phi:(\phi^!\mbb{F},K\rvert_{\gr(\phi)})\dasharrow(\mbb{F},A).\end{equation} The advantage of \labelcref{eq:trivMorp} is that the morphism $P_\phi$ is canonically associated to the map $\phi$ and the Courant algebroid $\mbb{F}$ (see Section~\ref{sec:pullback} for details). Therefore, in \labelcref{eq:trivMorp},  all the interesting information is stored in the Manin pair $(\phi^!\mbb{F},K\rvert_{\gr(\phi)})$.

%
%

In more detail, assume that $d\phi:TM\to TN$ is transverse to the anchor map $\mbf{a}:\mbb{F}\to TN$. Let $\phi^!\mbb{F}$ denote the pull-back of $\mbb{F}$ along $\phi$   (\cite{LiBland:2009ul}, cf. \cref{def:CAPullback}),
and $$P_\phi:\phi^!\mbb{F}\dasharrow \mbb{F}$$ denote the canonically associated Courant morphism over the graph of $\phi$ (\cite{LiBland:2009ul}, cf. \cref{sec:pullback}).

Recall from \cref{sec:pullback} that $\phi^!\mbb{F}:=C/C^\perp$, where \begin{equation}\label{eq:Ctransfer}C=\mbf{a}^{-1}(T\gr(\phi))\subseteq \mbb{E}\times\overline{\mbb{T}M}.\end{equation}
Since $K$ has support on the graph of $\phi$, \begin{equation}\label{eq:Ktransfer}K\rvert_{\gr(\phi)}:=K/C^\perp\subseteq \phi^!\mbb{F}\end{equation} is 
a Dirac structure \cite[Proposition~1.4]{LiBland:2011vqa}.

\begin{lemma}\label{BurszLem2}

\begin{equation}\label{eq:PphiTransfer}P_\phi:(\phi^!\mbb{F},K\rvert_{\gr(\phi)})\dasharrow(\mbb{F},F)\end{equation} is a morphism of Manin pairs. 
\end{lemma}

\begin{proof}

Note, that both conditions m1) and m2) in \cref{def:MorphMP} make sense for arbitrary quadratic vector bundles. Thus, we make the following provisional definition:
Suppose that $\mbb{E}_1$ and $\mbb{E}_2$ are two quadratic vector bundles over $M_1$ and $M_2$, respectively, $E_1\subset \mbb{E}_1$ and $E_2\subset\mbb{E}_2$ are two Lagrangian subbundles. A \emph{weak morphism of Manin pairs} 
$$R:(\mbb{E}_1,E_1)\dasharrow(\mbb{E}_2,E_2)$$ 
is a vector bundle relation $R:\mbb{E}_1\dasharrow\mbb{E}_2$, supported on the graph of a map $M_1\to M_2$, such that
\begin{enumerate} 
\item[wm1)] $E_1\cap \on{ker}(R)=0$, 
\item[wm2)] $R\circ E_1\subseteq E_2$. 
\end{enumerate}

%


Consider the  direct product $$\tilde I:= \mbb{F}_\Delta\times TM\subseteq \mbb{F}\times\overline{\mbb{F}}\times \mbb{T} M,$$ of the diagonal Dirac structure (with support) $\mbb{F}_\Delta\subseteq\mbb{F}\times\overline{\mbb{F}}$ and the Dirac structure $TM\subseteq \mbb{T}M$  (cf. \cref{ex:diagDirStr,ex:StdCourAlg}). It defines a Courant morphism $$\tilde I:\mbb{F}\times\overline{\mbb{T} M}\dasharrow \mbb{F}$$ over the projection $N\times M\to N$. 



Let $I:=\tilde I\rvert_{ N\times\gr(\phi)}$ and $\mbb{B}=(\mbb{F}\times\overline{\mbb{T}M})\rvert_{\gr(\phi)}$.

\begin{claim}
$I:(\mbb{B},K)\dasharrow (\mbb{F},F)$ is a weak morphism of Manin pairs
\end{claim}
First we show that (wm1) holds: $$(\{0\}\times K)\cap  I= (\{0\}\times K)\cap (\mbb{F}_\Delta\times TM)
= \{0\}\times \big(K\cap (\{0\}\times TM)\big).$$
 Since $K:(\mbb{T} M,TM)\dasharrow (\mbb{F},F)$ is a morphism of Manin pairs, $K\cap (\{0\}\times TM)=\{0\}$, and therefore $(\{0\}\times K)\cap \tilde I=\{0\}.$ This shows that (wm1) holds.
 
 Next we show that (wm2) holds:
 $$I\circ K=\mbb{F}_\Delta\circ K\circ TM=K\circ TM.$$ Since $K:(\mbb{T} M,TM)\dasharrow (\mbb{F},F)$ is a morphism of Manin pairs, $K\circ TM\subseteq F$. Hence $$I\circ K\subseteq F,$$ which shows that (wm2) holds.

Let $Q:\mbb{B}\to \phi^! \mbb{F}$ be the Lagrangian relation defined by the quotient map $C\to C/C^\perp=:\phi^!\mbb{F}$, where $C\subseteq \mbb{F}\times\overline{\mbb{T}M}$ is given in \cref{eq:Ctransfer}.
\begin{claim}
$Q^\top:(\phi^!\mbb{F},K\rvert_{\gr(\phi)})\dasharrow (\mbb{B},K)$ is a weak morphism of Manin pairs.
\end{claim}
Since $K$ is supported on the graph of $\phi:M\to N$, $K\subseteq C$, and hence $C^\perp\subseteq K$. It follows that $$Q^\top\circ K\rvert_{\gr(\phi)}=K,$$ so (wm2) holds. Moreover, $\on{ker}(Q^\top)=0$, so (wm1) holds.

From \cref{sec:pullback}, we have 
$P_\phi=I\circ Q^\top.$ Therefore, $$P_\phi:(\phi^!\mbb{F},K\rvert_{\gr(\phi)})\dasharrow(\mbb{F},F)$$ is a weak morphism of Manin pairs. Since $P_\phi\subseteq \mbb{F}\times\overline{\phi^!\mbb{F}}$ is also a Dirac structure (with support), \cref{eq:PphiTransfer} is a morphism of Manin pairs.

\end{proof}
\begin{remark}
In the special case where $\mbb{F}$ is exact, \cref{BurszLem2} was proven in \cite[Theorem~3.7]{Bursztyn:2009wi}.
\end{remark}

If one wishes to recover the \labelcref{eq:trivLS} from \labelcref{eq:trivMorp}, one may use the following lemma:
\begin{lemma}
Let $(\mbb{F},F)$ be a Manin pair over the manifold $N$, and let $$K:(\mbb{T} M, TM)\dasharrow (\mbb{F},F)$$ be a morphism of Manin pairs over $\phi:M\to N$. Then $K=P_\phi\circ J$, where $P_\phi$ is as Section~\ref{sec:pullback} and $J:=J_{K\rvert_{\gr(\phi)}}$ is described in \cref{ex:canonicalmorphism}.
\end{lemma}
\begin{proof}
Let $C=\{(x;v,\mu)\in \mbb{F}\times \overline{\mbb{T} M}\mid (d\phi)(v)=\mbf{a}_{\mbb{F}}(x)\}$, where $\mbf{a}_{\mbb{F}}:\mbb{F}\to TN$ is the anchor map.

Suppose that $$(v,\mu)\sim_K x$$ for $x\in \mbb{F}$ and $(v,\mu)\in\mbb{T} M$. Suppose that $(x;v,\mu)\in K\subset C$ maps to $$\widetilde{(x,v,\mu)}\in K\rvert_{\gr(\phi)}\subset \phi^!\mbb{F}=C/C^\perp,$$ while
$(x;v,0),(0;0,\mu)\in C$ map to the elements $\widetilde{(x;v,0)},\widetilde{(0;0,\mu)}\in \phi^!\mbb{F}$. By \labelcref{eq:ysimIx}, we have
\begin{equation}\label{eq:xisimy}\widetilde{(x;v,0)}\sim_{P_\phi} x.\end{equation}

Furthermore,  
$$\widetilde{(x,v,\mu)}=\widetilde{(x;v,0)}+\widetilde{(0;0,\mu)}.$$
But $\widetilde{(0;0,\mu)}=-\mbf{a}_{\phi^!\mbb{F}}^*\mu$ (where the minus sign comes from the fact that we are considering $\mbb{F}\times\overline{\mbb{T} M}$ and not $\mbb{F}\times\mbb{T} M$).
This gives us 
$$\widetilde{(x;v,0)}=\widetilde{(x,v,\mu)}+\mbf{a}_{\phi^!\mbb{F}}^*\mu,\quad\text{ with }\widetilde{(x,v,\mu)}\in K\rvert_{\gr(\phi)}$$
Since $\mbf{a}_{\phi^!\mbb{F}}(\widetilde{(x,v,\mu)})=v$, by the definition of $J:=J_{K\rvert_{\gr(\phi)}}$ (described in \cref{ex:canonicalmorphism}), we have $$(v,\mu)\sim_J\widetilde{(x;v,0)}.$$ Combining this with \labelcref{eq:xisimy}, we get $$(v,\mu)\sim_{P_\phi\circ J} x.$$ Therefore $K\subset P_\phi\circ J$, which concludes the proof.
\end{proof}

\subsection{The morphism of Manin pairs $(\mbb{T} L, TL^{flip})\dasharrow (\mf{d}\times T\mf{d},\mf{d}\times T\g)$}

Let $$L:(\mbb{T} M, TM)\dasharrow (\mf{d},\g)$$ be a morphism of Manin pairs defined by a Dirac structure $L\subset \mf{d}\times\overline{\mbb{T} M}$. Then the projection $\mf{d}\times\overline{\mbb{T}M}\to\mf{d}$ restricts to $L$ to define a morphism of Lie algebroids $$\mu:L\to\mf{d}.$$

\begin{proposition}\label{prop:IntegrationProp}
Let $$L:(\mbb{T} M, TM)\dasharrow (\mf{d},\g)$$ be a morphism of Manin pairs defined by a Dirac structure $L\subset \mf{d}\times\overline{\mbb{T} M}$. 
Then there is a morphism of $\mathcal{LA}$-Manin pairs \begin{equation}\label{eq:IntegrationProp}R:(\mbb{T} L, TL^{flip})\dasharrow (T\mf{d}\times\mf{d},T\g\times \mf{d}),\end{equation} over the Lie algebroid morphism $\mu:L\to \mf{d}$, where $T\mf{d}\times \mf{d}$ is the $\mc{LA}$-Courant algebroid described in \cref{eq:DiracManUnRes}
\end{proposition}
\begin{proof}
Let $P_\ast:\mf{d}\times \overline{\mbb{T} M}\dasharrow \mf{d}$ be the Courant morphism over the trivial map $M\to\ast$, defined by \cref{eq:ysimIx}.
By \cref{BurszLem2}, $$P_\ast:(\mf{d}\times \overline{\mbb{T} M},L)\dasharrow (\mf{d},\g)$$ is a morphism of Manin pairs. Taking the tangent prolongation yields a morphism of $\mc{LA}$-Manin pairs (cf. \cref{ex:TngProIsLACA}),
 $$T{P_\ast}:(T(\mf{d}\times\overline{\mbb{T} M}),TL)\dasharrow (T\mf{d},T\g).$$ 

Let  $$R_1:(\mbb{T} L,TL^{flip})\dasharrow (T(\mf{d}\times\overline{\mbb{T} M}),TL)$$
denote the morphism of $\mc{LA}$-Manin pairs defined in \cref{lem:BurszLem1}.
   
   We will be interested in the morphism of $\mc{LA}$-Manin pairs
   \begin{equation}\label{eq:R}R:=T{P_\ast}\circ R_1:(\mbb{T} L,TL^{flip})\dasharrow(T\mf{d},T\g).\end{equation}

Recall the $\mc{LA}$-Courant algebroid $T\mf{d}\times \mf{d}$ associated to the quadratic Lie algebra $\mf{d}$ in \cref{ex:LADirMan1}. $T\mf{d}\times \mf{d}$ is defined as an action Courant algebroid. Therefore if we can show that $R$ and $\mu$ are compatible with this action, then \begin{equation}\label{eq:muR}R\times\mu:(\mbb{T} L,TL^{flip})\dasharrow(T\mf{d}\times \mf{d} ,T\g\times \mf{d})\end{equation} will be a morphism of Manin pairs.

 More precisely, if $(X,\alpha)\sim_R y$ for sections $(X,\alpha)\in\Gamma(\mbb{T} L)$ and $y\in T\mf{d}$, then we need to show that \begin{equation}\label{eq:muEquRho}\mu_*X=\rho(y).\end{equation} Here $\rho:T\mf{d}\to\mathfrak{X}(\mf{d})$ is the Lie algebra morphism associated with the natural action of $T\mf{d}$ on the homogeneous space $TD/D\cong \mf{d}$ (where $D$ is any Lie group integrating $\mf{d}$, cf. \cref{ex:LADirMan1}). Thus $$\rho(\xi_T+\eta_C)=\on{Drf}_\xi+\on{Trl}_\eta.$$
 
 Recall from \cref{eq:ysimIx} that the relation $TP_\ast$, is given by $$(\xi_T+\eta_C;(X,0)_T+(Y,0)_C)\sim_{T{P_\ast}} \xi_T+\eta_C$$ for $\xi,\eta\in\mf{d}$ and $X,Y\in\mathfrak{X}(M)$. Meanwhile, by \cref{prop:formOfK}, 
  $R_1$ describes the relation \begin{multline*}(\on{Drf}_{(\xi;(X,\alpha))}+\on{Trl}_{(\eta;(Y,\beta))},d\langle(\xi;(X,\alpha)),\cdot\rangle+\gamma)\\ 
 \sim_{R_1}(\xi_T+\eta_C;(X,\alpha)_T+(Y,\beta)_C+(0,\gamma)_C),\end{multline*} 
 for $\xi,\eta,\in\mf{d}$, $X,Y\in\mathfrak{X}(M)$ and $\alpha,\beta,\gamma\in\Omega^1(M)$. 
  
 Therefore, 
 $R$ describes the relation
 $$(\on{Drf}_{(\xi;(X,0))}+\on{Trl}_{(\eta;(Y,\beta))},d\langle(\xi;(X,0)),\cdot\rangle-\beta)\sim_R(\xi_T+\eta_C).$$ 
 
 Now, on the one hand, $\rho(\xi_T+\eta_C)=\on{Drf}_\xi+\on{Trl}_\eta$. On the other hand, $$\mu_*(\on{Drf}_{(\xi;(X,0))}+\on{Trl}_{(\eta;(Y,\beta))})=\on{Drf}_{\mu(\xi;(X,0))}+\on{Trl}_{\mu(\eta;(Y,\beta))}=\on{Drf}_\xi+\on{Trl}_\eta,$$ where $\mu:\mf{d}\times\overline{\mbb{T} M}\to \mf{d}$ is the natural projection. Hence \cref{eq:muEquRho} holds. If follows that \labelcref{eq:muR} is a morphism of Manin pairs.

Next, we recall that as a Lie algebroid, $T\mf{d}\times \mf{d}$ is just the cross product of  the tangent bundle Lie algebroid $T\mf{d}:=T\mf{d}$ with the Lie algebra $\mf{d}$. Since the map $\mu:L\to \mf{d}$ is a morphism of Lie algebroids, and $R$ is a morphism of $\mc{LA}$-Manin pairs, the cross product $R$ is automatically a morphism of $\mc{LA}$-Manin pairs.
\end{proof}

\subsection{Integration corollaries}

From \cref{prop:IntegrationProp}, we immediately get the following set of Corollaries. 

\begin{corollary}\label{cor:basicInt}
Let $L\subset \mf{d}\times\overline{\mbb{T} M}$ be the Dirac structure corresponding to a morphism of Manin pairs $$(\mbb{T} M, TM)\dasharrow (\mf{d},\g),$$ and $\mu:L\to\mf{d}$ the natural projection. If $L$ integrates to a source simply connected groupoid $\Gamma$, then there exists a morphism of Manin pairs
\begin{equation}\label{eq:basicInt}K:(\mbb{T} \Gamma,T\Gamma)\dasharrow \big((\mf{d}\oplus\overline{\mf{d}})\times D,(\g\oplus\g)\times D\big),\end{equation} where the latter multiplicative Manin pair is described in \cref{ex:DirLieInt}.
\end{corollary}
\begin{proof}
\labelcref{eq:basicInt} follows from \cref{cor:LACourIntMultCour} by integrating the morphism of $\mathcal{LA}$-Manin pairs \labelcref{eq:IntegrationProp} defined in \cref{prop:IntegrationProp}.
\end{proof}

\begin{corollary}
Let $\h\subset\mf{d}$ be a Lie subalgebra transverse to $\g$, and assume that $\h$ integrates to a closed subgroup $H\subset D$ transverse to $\mu:\Gamma\to D$. Let $\Gamma_H:=\mu^{-1}(H)$, then the restriction of \labelcref{eq:basicInt} to $\Gamma_H$ defines a morphism of multiplicative Manin pairs
$$K\rvert_{H\times\Gamma_H}:(\mbb{T}\Gamma_H,T\Gamma_H)\dasharrow (\mbb{H}_{(\mf{d},\g;\h)}, F_{(\mf{d},\g;\h)}),$$ where $(\mbb{H}_{(\mf{d},\g;\h)}, F_{(\mf{d},\g;\h)})$ is the Dirac Lie group described in \cref{ex:DirLieInt}.
\end{corollary}

\begin{remark}
In the special case where $\h\subseteq\mf{d}$ is a Lagrangian Lie subalgebra, $(\mbb{H}_{(\mf{d},\g;\h)}, F_{(\mf{d},\g;\h)})$ is a Poisson Lie group corresponding to the Manin triple $(\mf{d},\g,\h)$. Moreover, $K\rvert_{H\times\Gamma_H}$ describes a multiplicative Hamiltonian action of the Lie bialgebra, $\g$. In particular, one recovers \cite[Theorem 6.2]{Xu95}.

On the other hand, if $\h\subseteq\mf{d}$ is a quadratic ideal, then $\mf{d}\cong\h\oplus\overline{\h}$, $\g\cong\h_\Delta$ is the diagonal subalgebra, and $\h\cong\h\oplus 0$. So $(\mbb{H}_{(\mf{d},\g;\h)}, F_{(\mf{d},\g;\h)})$ is the Cartan Dirac structure (see \cref{ex:CartanDirac}). Thus, $K\rvert_{H\times\Gamma_H}$ describes a multiplicative quasi-Hamiltonian $\h$-structure on $\Gamma$ \cite{Alekseev:2009tg,Bursztyn03-1,Xu03,Bursztyn:2005te,Alekseev97,Behrend03}.
Thus we recover a result found in \cite{LiBland:2010wi}.
\end{remark}

\begin{corollary}
Suppose that $\g\subset \mf{d}$ integrates to a closed subgroup $G$ of $D$, and the right action of $G$ on $\Gamma$ is free and proper. Then the quotient of \labelcref{eq:basicInt} by the right action of $G$ defines a morphism of Manin pairs
$$\big(\mbb{T}(\Gamma/G),T(\Gamma/G)\big)\dasharrow \big( \mf{d}\times (D/G),\g\times(D/G)\big),$$ where $\mf{d}\times (D/G)$ is the natural action Courant algebroid.
\end{corollary}

\begin{remark}
The Manin pair $\big( \mf{d}\times (D/G),\g\times (D/G)\big)$ was first introduced in \cite{Alekseev:2002tn,LetToWein}. In \cite{Bursztyn:2009vq,Bursztyn:2009wi} morphisms of Manin pairs taking values in $\big( \mf{d}\times (D/G),\g\times (D/G)\big)$ were studied, and shown to be equivalent to the Hamiltonian spaces described in \cite{Alekseev99}. Thus, this result can be interpreted as the analogue of \cite[Theorem 6.2]{Xu95} for the case where one has a Manin pair rather than a Manin triple.
\end{remark}

\begin{appendix}
\chapter{Technical proof for Courant algebroids}
We have the following Proposition, which we will use on occasion to define Courant algebroids.
\begin{proposition}\label{prop:CAconst}
Let $\mbb{E}\to M$ be a vector bundle, $\mbf{a}\colon \mbb{E}\to TM$ be a bundle map, $\la\cdot,\cdot\ra$ a bundle metric on $\mbb{E}$, and let $W\subseteq\Gamma(\mbb{E})$ be a subspace of sections which generates $\Gamma(\mbb{E})$ as a $C^\infty(M)$ module. Suppose that $\Cour{\cdot,\cdot}:W\to W$ is a bracket which satisfies
\begin{enumerate}
\item[c1)] $\Cour{\sigma_1,\Cour{\sigma_2,\sigma_3}}=\Cour{\Cour{\sigma_1,\sigma_2},\sigma_3}
+\Cour{\sigma_2,\Cour{\sigma_1,\sigma_3}}$, 
\item[c2)] $\mbf{a}(\sigma_1)\la \sigma_2,\sigma_3\ra=\la \Cour{\sigma_1,\sigma_2},\,\sigma_3\ra+\la \sigma_2,\,\Cour{\sigma_1,\sigma_3}\ra$,
\item[c3)] $\Cour{\sigma_1,\sigma_2}+\Cour{\sigma_2,\sigma_1}=\mbf{a}^*(d \la \sigma_1,\sigma_2\ra)$,
\item[c6)] $\mbf{a}(\Cour{\sigma_1,\sigma_2})=[\mbf{a}(\sigma_1),\mbf{a}(\sigma_2)]$,
\end{enumerate}
for any $\sigma_i\in W$, and that $\mbf{a}\circ\mbf{a}^*=0$. Then there is a unique extension of $\Cour{\cdot,\cdot}$ to a Courant bracket on all of $\Gamma(\mbb{E})$.
\end{proposition}
\begin{lemma}\label{lem:TechCAconst}
Suppose the assumptions of \cref{prop:CAconst} are satisfied, and for some $f^i\in C^\infty(M)$, $\sigma_i\in W$, the section $f^i\sigma_i\in W$ (where we use Einstein's summation convention). Then for any $\sigma\in W$, we have
\begin{align}
\label{eq:TechCAconst1}\Cour{\sigma,f^i\sigma_i}&=f^i\Cour{\sigma,\sigma_i}+(\mbf{a}(\sigma)f^i)\sigma_i,\text{ and }\\
\label{eq:TechCAconst2}\Cour{f^i\sigma_i,\sigma}&=f^i\Cour{\sigma_i,\sigma}-(\mbf{a}(\sigma)f^i)\sigma_i+\la\sigma_i,\sigma\ra\mbf{a}^*df^i.
\end{align}
In other words, $\Cour{\cdot,\cdot}$ extends to a unique bilinear map from $\Gamma(\mbb{E})\times\Gamma(\mbb{E})$ to $\Gamma(\mbb{E})$ satisfying conditions 
\begin{enumerate}
\item[c4)] $\Cour{\sigma_1,f\sigma_2}=f\Cour{\sigma_1,\sigma_2}+\mbf{a}(\sigma_1)(f)\sigma_2$, and
\item[c5)] $\Cour{f\sigma_1,\sigma_2}=f\Cour{\sigma_1,\sigma_2}-\mbf{a}(\sigma_2)(f)\sigma_1+\la \sigma_1,\sigma_2\ra \mbf{a}^*(d f)$ 
\end{enumerate} from \cref{def:CA}.

\end{lemma}
\begin{proof}
Let $\tau$ be an arbitrary element of $W$. Then, by (c2), we have
$$\mbf{a}(\sigma)\la f^i\sigma_i,\tau\ra=\la\Cour{\sigma,f^i\sigma_i},\tau\ra+\la f^i\sigma_i,\Cour{\sigma,\tau}\ra.$$ 
Applying the Leibniz rule and then (c2) to the left hand side, we get 
\begin{multline*}\mbf{a}(\sigma)\la f^i\sigma_i,\tau\ra=\la(\mbf{a}(\sigma)f^i)\sigma_i,\tau\ra+f^i\mbf{a}(\sigma)\la \sigma_i,\tau\ra
\\=\la(\mbf{a}(\sigma)f^i)\sigma_i,\tau\ra+f_i\la\Cour{\sigma,\sigma_i},\tau\ra+f_i\la\sigma_i,\Cour{\sigma,\tau}\ra.\end{multline*} Taking the difference of these two equations and rearranging, we obtain
$$ \la\Cour{\sigma,f^i\sigma_i},\tau\ra=\la f_i\Cour{\sigma,\sigma_i}+(\mbf{a}(\sigma)f^i)\sigma_i,\tau\ra.$$ And \labelcref{eq:TechCAconst1} follows since $\tau\in W$ was arbitrary. Finally, \labelcref{eq:TechCAconst2} follows from \labelcref{eq:TechCAconst1} and (c3).
\end{proof}
\begin{proof}[Proof of Proposition]
Uniqueness follows from \cref{lem:TechCAconst}. It remains to prove that the unique extension of $\Cour{\cdot,\cdot}$ to all of $\Gamma(\mbb{E})$ defined in \cref{lem:TechCAconst} satisfies the axioms of \cref{def:CA}.

First we show that (c2) is satisfied for arbitrary sections of $\mbb{E}$. Indeed let $f\in C^\infty(M)$, and $\sigma_i\in W$. Then
\begin{subequations}
\begin{align}
\notag&\la\Cour{f\sigma_1,\sigma_2},\sigma_3\ra+\la\sigma_2,\Cour{f\sigma_1,\sigma_3}\ra\\
\label{eq:c2a1}=&\la f\Cour{\sigma_1,\sigma_2}-(\mbf{a}(\sigma_2)f)\sigma_1+\la\sigma_1,\sigma_2\ra\mbf{a}^*df,\sigma_3\ra\\
\notag&+\la\sigma_2, f\Cour{\sigma_1,\sigma_3}-(\mbf{a}(\sigma_3)f)\sigma_1+\la\sigma_1,\sigma_3\ra\mbf{a}^*df\ra\\
\notag=&\la f\Cour{\sigma_1,\sigma_2},\sigma_3\ra+\la\sigma_2, f\Cour{\sigma_1,\sigma_3}\ra\\
\label{eq:c2a2}=&f\mbf{a}(\sigma_1)\la\sigma_2,\sigma_3\ra
\end{align}
\end{subequations}
where \labelcref{eq:c2a1} follows by (c5) and \labelcref{eq:c2a2} follows since (c2) holds for elements of $W$. With similar manipulations, one shows that $$\mbf{a}(\sigma_1)\la f\sigma_2,\sigma_3\ra=\la\Cour{\sigma_1,f\sigma_2},\sigma_3\ra+\la f\sigma_2,\Cour{\sigma_1,\sigma_3}\ra.$$ Therefore, c2) holds for arbitrary elements $\sigma_i\in \Gamma(\mbb{E})$.


Next we show that (c3) holds for arbitrary sections of $\mbb{E}$. Let $f\in C^\infty(M)$, and $\sigma_i\in W$. Then, 
\begin{subequations}
\begin{align}
\label{eq:c31}\Cour{f\sigma_1,\sigma_2}+\Cour{\sigma_2,f\sigma_1}=&f\Cour{\sigma_1,\sigma_2}-\big(\mbf{a}(\sigma_2)f\big)\sigma_1+\la\sigma_1,\sigma_2\ra\mbf{a}^*df+f\Cour{\sigma_2,\sigma_1}+\big(\mbf{a}(\sigma_2)f\big)\sigma_1\\
\label{eq:c32}=&f\mbf{a}^*d\la\sigma_1,\sigma_2\ra+\la\sigma_1,\sigma_2\ra\mbf{a}^*df\\
\notag=&\mbf{a}^*d\la f\sigma_1,\sigma_2\ra,
\end{align}\end{subequations}
where \labelcref{eq:c31} follows from (c4) and (c5), and \labelcref{eq:c32} follows since (c3) holds for elements of $W$.

Finally, we show that (c1) holds for arbitrary sections. As before, let $f\in C^\infty(M)$, and $\sigma_i\in W$. Then (c4) implies that
$$\Cour{\sigma_i,\Cour{\sigma_j,f\sigma_k}}=f\Cour{\sigma_i,\Cour{\sigma_j,\sigma_k}}+(\mbf{a}(\sigma_i)f)\Cour{\sigma_j,\sigma_k}+(\mbf{a}(\sigma_j)f)\Cour{\sigma_i,\sigma_k}+\big(\mbf{a}(\sigma_i)\mbf{a}(\sigma_j)f\big)\sigma_k.$$
Meanwhile, (c4) and (c6) imply
$$\Cour{\Cour{\sigma_1,\sigma_2},f\sigma_3}=f\Cour{\Cour{\sigma_1,\sigma_2},\sigma_3}+([\mbf{a}(\sigma_1),\mbf{a}(\sigma_2)]f)\sigma_3.$$
Combining these equations, we get
\begin{equation}\label{eq:c1case3}
\begin{split}
&\Cour{\sigma_1,\Cour{\sigma_2,f\sigma_3}}-\Cour{\Cour{\sigma_1,\sigma_2},f\sigma_3}-\Cour{\sigma_2,\Cour{\sigma_1,f\sigma_3}}\\
=&f\Cour{\sigma_1,\Cour{\sigma_2,\sigma_3}}+(\mbf{a}(\sigma_1)f)\Cour{\sigma_2,\sigma_3}+(\mbf{a}(\sigma_2)f)\Cour{\sigma_1,\sigma_3}+\big(\mbf{a}(\sigma_1)\mbf{a}(\sigma_2)f\big)\sigma_3\\
&-f\Cour{\Cour{\sigma_1,\sigma_2},\sigma_3}-([\mbf{a}(\sigma_1),\mbf{a}(\sigma_2)]f)\sigma_3\\
&-f\Cour{\sigma_2,\Cour{\sigma_1,\sigma_3}}-(\mbf{a}(\sigma_2)f)\Cour{\sigma_1,\sigma_3}-(\mbf{a}(\sigma_1)f)\Cour{\sigma_2,\sigma_3}-\big(\mbf{a}(\sigma_2)\mbf{a}(\sigma_1)f\big)\sigma_3\\
=&f\Cour{\sigma_1,\Cour{\sigma_2,\sigma_3}}-f\Cour{\Cour{\sigma_1,\sigma_2},\sigma_3}-f\Cour{\sigma_2,\Cour{\sigma_1,\sigma_3}}\\
=&0
\end{split}\end{equation}
Where the last line follows since (c1) holds for sections of $W$.

Next, suppose $\sigma_i\in W$ and $f\in C^\infty(M)$. By assumption, we have
$[\mbf{a}(\sigma_1),\mbf{a}(\sigma_2)]f=\la \Cour{\sigma_1,\sigma_2},\mbf{a}^*df\ra$. Since (c2) holds for arbitrary sections of $\mbb{E}$, we have
$$[\mbf{a}(\sigma_1),\mbf{a}(\sigma_2)]f = \mbf{a}(\sigma_1)\la\sigma_2,\mbf{a}^*df\ra-\la\sigma_2,\Cour{\sigma_1,\mbf{a}^*df}\ra.$$ Thus $$\la\sigma_2,\Cour{\sigma_1,\mbf{a}^*df}\ra=\mbf{a}(\sigma_1)\mbf{a}(\sigma_2)f-[\mbf{a}(\sigma_1),\mbf{a}\sigma_2]f=\la\sigma_2,\mbf{a^*}d(\mbf{a}(\sigma_1)f)\ra.$$ Since $\sigma_2\in W$ was arbitrary, we have \begin{equation}\label{eq:CAconstTmp}\Cour{\sigma_1,\mbf{a}^*df}=\mbf{a^*}d(\mbf{a}(\sigma_1)f).\end{equation}

Next we see that
\begin{subequations}
\begin{align}
\notag&\Cour{\sigma_1,\Cour{f\sigma_2,\sigma_3}}-\Cour{\Cour{\sigma_1,f\sigma_2},\sigma_3}-\Cour{f\sigma_2,\Cour{\sigma_1,\sigma_3}}\\
\notag=&-\Cour{\sigma_1,\Cour{\sigma_3,f\sigma_2}}+\Cour{\sigma_3,\Cour{\sigma_1,f\sigma_2}}+\Cour{\Cour{\sigma_1,\sigma_3},f\sigma_2}\\
\label{eq:c1case2CAconstTmp}&+\mbf{a}^*d(\mbf{a}(\sigma_1)\la f\sigma_2,\sigma_3\ra-\la\Cour{\sigma_1,f\sigma_2},\sigma_3\ra-\la f\sigma_2,\Cour{\sigma_1,\sigma_3}\ra)\\
\notag=&0.
\end{align}\end{subequations}
Here the first equality follows by applying (c3) three times, and the first term on line \labelcref{eq:c1case2CAconstTmp} follows from \labelcref{eq:CAconstTmp}. The second equality follows from \labelcref{eq:c1case3} and (c2).

Finally, a similar calculation shows that 
$$\Cour{f\sigma_1,\Cour{\sigma_2,\sigma_3}}-\Cour{\Cour{f\sigma_1,\sigma_2},\sigma_3}-\Cour{\sigma_2,\Cour{f\sigma_1,\sigma_3}}=0,$$
which concludes the proof.
\end{proof}

\chapter{Technical proof for double vector bundles}\label{app:dvbSymProof}
This appendix is devoted to the proof of the following.
\begin{proposition}\label{prop:DVBSymB}
Let $D$ be a manifold which is canonically identified with the total space of two vector bundles over manifolds $A\subseteq D$ and $B\subseteq D$, respectively. We let 
\begin{align*}\gr(+_{D/A}):D\times D&\dasharrow D,\\
\gr(+_{D/B}):D\times D&\dasharrow D
\end{align*}
denote the graph of the two additions. Let $s_{(1324)}:D^4\to D^4$ denote the permutation $$s_{(1324)}(d_1,d_2,d_3,d_4)=(d_1,d_3,d_2,d_4),$$  and $M=A\cap B$. 
Then the following are equivalent. 
\begin{enumerate}
\item[E1)] The following diagram of relations commutes
\begin{equation}\label{eq:DVBComDiagRelB}
\begin{tikzpicture}
\mmat[2em]{m}{D^4& &&D^2&&&\\
&&&&&&D\\
D^4&&&D^2&&&\\};
\path [dashed,<->] (m-1-1) edge node [swap] {$\gr(s_{(1324)})$} (m-3-1);
\path [dashed,->] (m-1-1) edge node {$\gr(+_{D/B})^2$} (m-1-4);
\path [dashed,->] (m-3-1) edge node [swap]{$\gr(+_{D/A})^2$} (m-3-4);
\path [dashed,<-] (m-2-7) edge node  {$\gr(+_{D/B})$} (m-3-4)
			edge node [swap] {$\gr(+_{D/A})$} (m-1-4);
\end{tikzpicture}
\end{equation}
\item[E2)] The following diagram   is a double vector bundle
\begin{equation}\label{eq:DVB1B}\begin{tikzpicture}
\mmat{m}{
D&B\\
A&M\\
};
\path[->] (m-1-1) edge node {$q_{D/B}$} (m-1-2);
\path[->] (m-1-1) edge node[swap] {$q_{D/A}$} (m-2-1);
\path[->] (m-1-2) edge node {$q_{B/M}$} (m-2-2);
\path[->] (m-2-1) edge node[swap] {$q_{A/M}$} (m-2-2);
\end{tikzpicture}\end{equation}
where $q_{B/M}:=(q_{D/A})\rvert_B$ and $q_{A/M}:=(q_{D/B})\rvert_A$.
\item[E3)] The following diagram (the `diagonal flip' of \labelcref{eq:DVB1B}) is a double vector bundle
\begin{equation}\label{eq:DVBflipB}\begin{tikzpicture}
\mmat{m}{
D&A\\
B&M\\
};
\path[->] (m-1-1) edge node {$q_{D/A}$} (m-1-2);
\path[->] (m-1-1) edge node[swap] {$q_{D/B}$} (m-2-1);
\path[->] (m-1-2) edge node {$q_{A/M}$} (m-2-2);
\path[->] (m-2-1) edge node[swap] {$q_{B/M}$} (m-2-2);
\end{tikzpicture}\end{equation}
where $q_{B/M}:=(q_{D/A})\rvert_B$ and $q_{A/M}:=(q_{D/B})\rvert_A$.
\end{enumerate}
%
\end{proposition}
\begin{lemma}\label{lem:SymDVB}
The following are equivalent.
\begin{enumerate}
\item[e1)] The diagram \labelcref{eq:DVBComDiagRelB} of relations commutes. \label{enum:DVBcom1}
\item[e2)] $\gr(+_{D/A})\subseteq D^3$ is a vector subbundle of $D^3\to B^3$. \label{enum:DVBcom2}
\item[e3)] $\gr(+_{D/B})\subseteq D^3$ is a vector subbundle of $D^3\to A^3$. \label{enum:DVBcom3}
\end{enumerate}
\end{lemma}
\begin{proof}
Note that e1) can also be written as the point-set equality \cref{eq:DVBinterchg}. Also, since the diagram \labelcref{eq:DVBComDiagRelB} is symmetric with respect to the interchange of $A$ and $B$, it is sufficient to prove that e1) and e2) are equivalent.
\begin{description}
\item[e1)$\Rightarrow$e2)] Without loss of generality, we assume that $D$ is connected. Since e1) holds, it follows that $\gr(+_{D/A})\subseteq D^3$ is a subset of the vector bundle $D^3\to B^3$, which is closed under both addition and subtraction. Since it is also topologically closed and connected, it must be a vector subbundle.
\item[e1)$\Leftarrow$e2)] Since $\gr(+_{D/A})\subseteq D^3$ is a vector subbundle of $D^3\to B^3$, the two additions $+_{D/A}$ and $+_{D/B}$ commute, and hence e1) holds.
\end{description}
\end{proof}

\begin{proof}[Proof of \cref{prop:DVBSymB}]
Since the diagram \labelcref{eq:DVBComDiagRelB} is symmetric with respect to the interchange of $A$ and $B$, it is sufficient to show that E1)$\Leftrightarrow$E2).
\begin{description}
\item[E1)$\Rightarrow$E2)]  If \labelcref{eq:DVBComDiagRelB} commutes, then by \cref{lem:SymDVB}, $\gr(+_{D/B})\subseteq D^3$ is a vector subbundle of $D^3\to A^3$. Since $$\on{gr}(-_{D/B})=\{(d;d_1,d_2)\mid d+_{D/B}d_2=d_1\}\subseteq D^3$$ is obtained from $\on{gr}(+_{D/A})$ by permuting factors, it is a vector subbundle of $D^3\to A^3$. Moreover, the diagonal $D_\Delta\subset D\times D$ is vector subbundle of $D^2\to A^2$. Therefore, the composition $$\gr(-_{D/B})\circ D_\Delta$$ is a vector subbundle of $D\to A$. But we have the equality $$\gr(-_{D/B})\circ D_\Delta=\{d\in D\mid d=(d'-_{D/B}d') \text{ for some }d'\in D\}=B,$$ which shows that $B$ is a vector subbundle of $D\to A$.
It is clear that the base space of $B$ is just $M=A\cap B$, the intersection of $B$ with the base space of $D\to A$.

Next, since the diagonal embedding $\Delta_D:D\to D\times D$, given by $d\to(d,d)$ is a morphism of vector bundles from $D\to A$ to $D^2\to A^2$, the composition of relations $$\gr(-_{D/B})\circ \gr(\Delta_D):D\dasharrow D$$ is a $\mc{VB}$-relation. However $$\gr(-_{D/B})\circ \gr(\Delta_D)=\gr(0_{D/B}\circ q_{D/B}),$$ which shows that $q_{D/B}:D\to B$ is a morphism of vector bundles. This concludes the proof that \labelcref{eq:DVB1B} is a double vector bundle.
\item[E1)$\Leftarrow$E2)] This follows directly from \cref{def:DLACAVB} and \cref{lem:SymDVB}.
\end{description}
\end{proof}

\chapter{Technical proofs for $\mc{LA}$-Courant algebroids}\label{app:PairIsLAPair}
\begin{proposition}
Suppose that $\mbb{A}$ is an $\mc{LA}$-Courant algebroid. Then the map $\mbb{A}\to \mbb{A}^{*_{\!y}}$ induced by the fibre metric is a morphism of Lie algebroids.
\end{proposition}
\begin{proof}
Following \cite{GraciaSaz:2009ck}, we consider the Wehrheim-Woodward category whose objects are triple vector bundles, and whose morphisms are relations of triple vector bundles \cite{Wehrheim:2010cg, Weinstein:2010wm}\footnote{i.e. if $T$ and $T'$ are triple vector bundles, a morphism $R:T\dasharrow T'$ is a sub-triple vector bundle $R\subseteq T'\times T$}. 

Suppose that
$$\begin{tikzpicture}[
        back line/.style={densely dotted},
        cross line/.style={preaction={draw=white, -,
           line width=6pt}}]
        
\mmat[1em]{T^*m}{
	&T	&	&A\\
C	&	&D	&\\
	&B	&	&F\\
E	&	&M	&\\
};

\path[->]
        (T^*m-1-2) edge (T^*m-1-4)
        		edge (T^*m-2-1)
                edge [back line] (T^*m-3-2)
        (T^*m-1-4) edge (T^*m-3-4)
        		edge (T^*m-2-3)
        (T^*m-2-1) edge [cross line] (T^*m-2-3)
                edge (T^*m-4-1)
        (T^*m-3-2) edge [back line] (T^*m-3-4)
        		edge [back line] (T^*m-4-1)
        (T^*m-4-1) edge (T^*m-4-3)
        (T^*m-3-4) edge (T^*m-4-3)
        (T^*m-2-3) edge [cross line] (T^*m-4-3);
        
  \path[dashed,->] (100pt,0) edge node{$R$} (140pt,0);
  
\mmat[1em]{Tm}[xshift=240pt]{
	&T'	&	&A'\\
C'	&	&D'	&\\
	&B'	&	&F'\\
E'	&	&M'	&\\
};

\path[->]
        (Tm-1-2) edge (Tm-1-4)
        		edge (Tm-2-1)
                edge [back line] (Tm-3-2)
        (Tm-1-4) edge (Tm-3-4)
        		edge (Tm-2-3)
        (Tm-2-1) edge [cross line] (Tm-2-3)
                edge (Tm-4-1)
        (Tm-3-2) edge [back line] (Tm-3-4)
        		edge [back line] (Tm-4-1)
        (Tm-4-1) edge (Tm-4-3)
        (Tm-3-4) edge (Tm-4-3)
        (Tm-2-3) edge [cross line] (Tm-4-3);
\end{tikzpicture}$$
 is a relation of triple vector bundles. Take the horizontal duals of both $T$ and $T'$, and consider the relation $R^{*_{\!x}}:T^{*_{\!x}}\dasharrow T'^{*_{\!x}}$ between them, defined by 
 $$R^{*_{\!x}}:=\ann^\natural(R)=\{(\beta,\alpha)\in T'^{*_{\!x}}\times T^{*_{\!x}}\text{ such that } \la \beta,y\ra=\la \alpha,x\ra \text{ for all } (y,x)\in R\},$$
 as in \cref{sec:LinRel}.
  Thus, as explained by Gracia-Saz and Mackenzie \cite{GraciaSaz:2009ck}, dualizing along the $x$-axis defines an endofunctor $*_{\!x}$ on our category. Similarly dualizing along the $y$ and $z$-axes defines endofunctors $*_{\!y}$ and $*_{\!z}$, respectively \cite{GraciaSaz:2009ck}.

The fibrewise metric on $\mbb{A}\to A$ defines an isomorphism $Q:\mbb{A}\to \mbb{A}^{*_{\!y}}$, which we shall prove to be a morphism of Lie algebroids. Applying the functor $*_{\!x}$, we 
get an isomorphism $Q^{*_{\!x}}:\mbb{A}^{*_{\!x}}\to \mbb{A}^{*_{\!y}*_{\!x}}\cong (\mbb{A}^{*_{\!x}})^{*_{\!x}*_{\!y}*_{\!x}}.$ Let $\pi$ be the bivector field for the double linear Poisson structure on $\mbb{A}^{*_{\!x}}$. As explained in \cref{sec:DLPV}, $\pi$ also defines the Lie algebroid structure on $\mbb{A}^{*_{\!y}}$ under the canonical isomorphism $\mbb{A}^{*_{\!y}*_{\!x}}\cong (\mbb{A}^{*_{\!x}})^{flip}$. So $Q$ is a morphism of Lie algebroids if and only if the following diagram commutes,
\begin{equation}\label{eq:TQdiagcom}\begin{tikzpicture}
\mmat[4em]{m}{T\mbb{A}^{*_{\!x}*_{\!z}}&\big((T\mbb{A}^{*_{\!x}})^{*_{\!x}*_{\!y}*_{\!x}}\big)^{*_{\!z}}\\ T\mbb{A}^{*_{\!x}}& (T\mbb{A}^{*_{\!x}})^{*_{\!x}*_{\!y}*_{\!x}}\\};
\path[dashed,<->] (m-1-1)	edge node {$TQ^{*_{\!x}*_{\!z}}$} (m-1-2);
\path[dashed,<->] (m-2-1)	edge node {$TQ^{*_{\!x}}$} (m-2-2);
\path[dashed,->] (m-1-1)	edge node {$\pi^\sharp$} (m-2-1);
\path[dashed,->] (m-1-2)	edge node {$(\pi^\sharp)^{*_{\!x}*_{\!y}*_{\!x}}$} (m-2-2);
\end{tikzpicture}\end{equation} 
where $T\mbb{A}$ is the triple vector bundle 
$$\begin{tikzpicture}[
        back line/.style={densely dotted},
        cross line/.style={preaction={draw=white, -,
           line width=6pt}}]
        
\mmat[1em]{Tm}{
	&T\mbb{A}	&	&TV\\
\mbb{A}	&	&V	&\\
	&TA	&	&TM\\
A	&	&M	&\\
};

\path[->]
        (Tm-1-2) edge (Tm-1-4)
        		edge (Tm-2-1)
                edge [back line] (Tm-3-2)
        (Tm-1-4) edge (Tm-3-4)
        		edge (Tm-2-3)
        (Tm-2-1) edge [cross line] (Tm-2-3)
                edge (Tm-4-1)
        (Tm-3-2) edge [back line] (Tm-3-4)
        		edge [back line] (Tm-4-1)
        (Tm-4-1) edge (Tm-4-3)
        (Tm-3-4) edge (Tm-4-3)
        (Tm-2-3) edge [cross line] (Tm-4-3);
\end{tikzpicture}$$
Applying $*_{\!x}$ to \labelcref{eq:TQdiagcom} we get 
\begin{equation}\label{eq:TQdiagcom2}\begin{tikzpicture}
\mmat[4em]{m}{T\mbb{A}^{*_{\!x}*_{\!z}*_{\!x}}&(T\mbb{A}^{*_{\!y}})^{*_{\!x}*_{\!z}*_{\!x}}\\ T\mbb{A}& T\mbb{A}^{*_{\!y}}\\};
\path[dashed,<->] (m-1-1)	edge node {$TQ^{*_{\!x}*_{\!z}*_{\!x}}$} (m-1-2);
\path[dashed,<->] (m-2-1)	edge node {$TQ$} (m-2-2);
\path[dashed,->] (m-1-1)	edge node {$(\pi^\sharp)^{*_{\!x}}$} (m-2-1);
\path[dashed,->] (m-1-2)	edge node {$(\pi^\sharp)^{*_{\!x}*_{\!y}}$} (m-2-2);
\end{tikzpicture}\end{equation}
However, by definition $(\pi^\sharp)^{*_{\!x}}=\Pi_{\mbb{A}}$, and $TQ:T\mbb{A}\to T\mbb{A}^{*_{\!y}}$ is the isomorphism defined by the fibrewise metric $T\mbb{A}\to TA$ (see \cref{rem:TangLiftTngProDef}). Meanwhile $$\big((\pi^\sharp)^{*_{\!x}}\big)^{*_{\!y}}=\ann^\natural(\Pi_{\mbb{A}})=\Pi_{\mbb{A}}^\perp\subseteq (T\mbb{A}\times \overline{T\mbb{A}^{*_{\!x}*_{\!z}*_{\!x}}}).$$
So \labelcref{eq:TQdiagcom2} commutes if and only if $$\Pi_{\mbb{A}}^\perp=\Pi_{\mbb{A}},$$ namely $\Pi_{\mbb{A}}\subseteq T\mbb{A}\times \overline{T\mbb{A}^{*_{\!x}*_{\!z}*_{\!x}}}$ is Lagrangian. 

\end{proof}

\begin{proposition}
$\mc{LA}$-Courant algebroids $$\begin{tikzpicture}
\mmat{m}{\mbb{A}&V\\ \ast& \ast\\};
\path[->] (m-1-1)	edge (m-1-2)
				edge (m-2-1);
\path[<-] (m-2-2)	edge (m-1-2)
				edge (m-2-1);
\end{tikzpicture}$$ over a point are all of the form described in \cref{ex:LACApt}. Thus they are in one-to-one correspondence with pairs $(\g,\beta)$ consisting of a Lie algebra $\g$ together with an invariant symmetric bilinear form, $\beta\in S^2(\g)^\g$, on $\g^*$.

Similarly, any $\mc{LA}$-Dirac structure
$$\begin{tikzpicture}
\mmat{m1} at (-3,0){L&W\\ \ast& \ast\\};
\path[->] (m1-1-1)	edge (m1-1-2)
				edge (m1-2-1);
\path[<-] (m1-2-2)	edge (m1-1-2)
				edge (m1-2-1);
\mmat{m2} at (2,0) {\g\ltimes\g^*&\g\\ \ast& \ast\\};
\path[->] (m2-1-1)	edge (m2-1-2)
				edge (m2-2-1);
\path[<-] (m2-2-2)	edge (m2-1-2)
				edge (m2-2-1);
\draw (0,0) node {$\subseteq$};
\end{tikzpicture}$$
is of the form described in \cref{ex:LACApt} for a $\beta$-coisotropic Lie subalgebra $\h\subseteq\g$.
\end{proposition}
\begin{proof}
Suppose that $$\begin{tikzpicture}
\mmat{m}{\mbb{A}&V\\ \ast& \ast\\};
\path[->] (m-1-1)	edge (m-1-2)
				edge (m-2-1);
\path[<-] (m-2-2)	edge (m-1-2)
				edge (m-2-1);
\end{tikzpicture}$$ is an $\mc{LA}$-Courant algebroid over a point. Then by \cref{prop:VBCAoverpt}, we see that $V\cong\g$ for some Lie algebra $\g$ and $\mbb{A}\cong\g^*\rtimes\g$, where $\g^*$ is the core of $\mbb{A}$.
Since $\g^*\rtimes\g$ is a linear Lie algebroid over $\g$, it must be an action Lie algebroid for an affine action of $\g^*$ on $\g$ given by a linear map $\beta^\sharp:\g^*\to \g$. 

Let $\beta\in\g\otimes\g$ be the element defined by $\nu\big(\beta^\sharp(\mu)\big)=\beta(\mu,\nu)$. Let $\rho:\g\to\mf{X}(\g)$ be the map which takes an element of $\g$ to the corresponding constant vector field. Then $$\pi:=(\rho\otimes\rho)(\beta)\in \mf{X}(\g)\otimes\mf{X}(\g)\subseteq \mf{X}^2(\g\times\g)$$ is the bivector field defining the double linear Poisson structure on $\mbb{A}^*_V=\g\times\g$. 

In this case, the map $\pi^\sharp:T^*\g\times T^*\g\to T\g\times T\g$ shown in \labelcref{eq:piAsharp} becomes 
$$\begin{tikzpicture}[
        back line/.style={densely dotted},
        cross line/.style={preaction={draw=white, -,
           line width=6pt}}]
        
\mmat[.5em]{T^*m}{
	&\big((\mu;\xi),(\nu;\eta)\big)	&	&(\nu;\eta)\\
(\xi,\eta)	&	&\eta	&\\
	&(\mu;\xi)	&	&\ast\\
\xi	&	&\ast	&\\
};

\path[->]
        (T^*m-1-2) edge (T^*m-1-4)
        		edge (T^*m-2-1)
                edge [back line] (T^*m-3-2)
        (T^*m-1-4) edge (T^*m-3-4)
        		edge (T^*m-2-3)
        (T^*m-2-1) edge [cross line] (T^*m-2-3)
                edge (T^*m-4-1)
        (T^*m-3-2) edge [back line] (T^*m-3-4)
        		edge [back line] (T^*m-4-1)
        (T^*m-4-1) edge (T^*m-4-3)
        (T^*m-3-4) edge (T^*m-4-3)
        (T^*m-2-3) edge [cross line] (T^*m-4-3);
        
  \path[->] (85pt,0) edge node{$\pi^\sharp$} (100pt,0);
  
\mmat[.5em]{Tm}[xshift=240pt]{
	&\big((-(\beta^\sharp)^*\mu;\xi),((\beta^\sharp\nu);\eta\big)	&	&(\beta^\sharp\nu;\eta)\\
(\xi,\eta)	&	&\eta	&\\
	&(-(\beta^\sharp)^*\mu;\xi)	&	&\ast\\
\xi	&	&\ast	&\\
};

\path[->]
        (Tm-1-2) edge (Tm-1-4)
        		edge (Tm-2-1)
                edge [back line] (Tm-3-2)
        (Tm-1-4) edge (Tm-3-4)
        		edge (Tm-2-3)
        (Tm-2-1) edge [cross line] (Tm-2-3)
                edge (Tm-4-1)
        (Tm-3-2) edge [back line] (Tm-3-4)
        		edge [back line] (Tm-4-1)
        (Tm-4-1) edge (Tm-4-3)
        (Tm-3-4) edge (Tm-4-3)
        (Tm-2-3) edge [cross line] (Tm-4-3);
\end{tikzpicture}$$
for $\xi,\eta\in\g$ and $\mu,\nu\in\g^*$
Hence the relation $\Pi:T\g^*\rtimes T\g \dasharrow T\g^*\rtimes T\g$ depicted in \labelcref{eq:PiA} is given by 
$$\begin{tikzpicture}[
        back line/.style={densely dotted},
        cross line/.style={preaction={draw=white, -,
           line width=6pt}}]
        
\mmat[1em]{T^*m}{
	&\big((\upsilon,\nu);(\beta^\sharp\mu,\eta)\big)	&	&(\nu;\eta)\\
(\beta^\sharp\mu;\eta)	&	&\eta	&\\
	&\ast	&	&\ast\\
\ast	&	&\ast	&\\
};

\path[->]
        (T^*m-1-2) edge (T^*m-1-4)
        		edge (T^*m-2-1)
                edge [back line] (T^*m-3-2)
        (T^*m-1-4) edge (T^*m-3-4)
        		edge (T^*m-2-3)
        (T^*m-2-1) edge [cross line] (T^*m-2-3)
                edge (T^*m-4-1)
        (T^*m-3-2) edge [back line] (T^*m-3-4)
        		edge [back line] (T^*m-4-1)
        (T^*m-4-1) edge (T^*m-4-3)
        (T^*m-3-4) edge (T^*m-4-3)
        (T^*m-2-3) edge [cross line] (T^*m-4-3);
        
  \path[->] (100pt,0) edge node{$\Pi$} (115pt,0);
  
\mmat[1em]{Tm}[xshift=240pt]{
	&\big((\upsilon,\mu);(\beta^\sharp\nu,\eta)\big)	&	&(\beta^\sharp\nu;\eta)\\
(\mu;\eta)	&	&\eta	&\\
	&\ast	&	&\ast\\
\ast	&	&\ast	&\\
};

\path[->]
        (Tm-1-2) edge (Tm-1-4)
        		edge (Tm-2-1)
                edge [back line] (Tm-3-2)
        (Tm-1-4) edge (Tm-3-4)
        		edge (Tm-2-3)
        (Tm-2-1) edge [cross line] (Tm-2-3)
                edge (Tm-4-1)
        (Tm-3-2) edge [back line] (Tm-3-4)
        		edge [back line] (Tm-4-1)
        (Tm-4-1) edge (Tm-4-3)
        (Tm-3-4) edge (Tm-4-3)
        (Tm-2-3) edge [cross line] (Tm-4-3);
\end{tikzpicture}$$
for $\eta\in\g$ and $\mu,\nu,\upsilon\in\g^*$. Where the pairing is the tangent lift of the canonical pairing between $\g$ and $\g^*$.
Thus, $\Pi$   is compatible with the pairing if and only if, for any $\mu,\nu\in\g^*$,
$$\la \big((0,0);(\beta^\sharp\mu,0)\big),\big((0,\nu);(0,0)\big)\ra=\la\big((0,\mu);(0,0)\big),\big((0,0);(\beta^\sharp\nu,0)\big)\ra.$$
That is, $\la \beta^\sharp\mu,\nu\ra=\la\mu,\beta^\sharp\nu\ra$.
 Namely, $\beta\in\g\otimes\g$ is symmetric.
 
Similarly, the relation $\Pi$ is compatible with the Courant bracket if and only if for any $\eta\in\g$ and $\nu\in\g^*$, the Lie bracket
$$\big((0,[\eta,\nu]);(0,0\big)$$
of $\big((0,0);(0,\eta)\big)$ with $\big((0,\nu);(0,0)\big)$ 
 is $\Pi$ related to the Lie bracket
 $$\big((0,0);([\eta,\beta^\sharp\nu],0\big)$$
 of $\big((0,0);(0,\eta)\big)$ with $\big((0,0);(\beta^\sharp\nu,0)\big)$.
 Equivalently, $\beta^\sharp[\eta,\nu]=[\eta,\beta^\sharp\nu]$,
 or $\beta$ is $\g$ invariant.
In conclusion, $\Pi$ defines a Courant relation if and only if $\beta$ is symmetric and  $\g$ invariant.

Thus there is a one-to-one correspondence between $\mc{LA}$-Courant algebroids over a point and Lie algebras $\g$ together with a symmetric invariant element $\beta\in (\g\otimes\g)$. 
 
Recall from \cref{prop:VBCAoverpt} that any $\mc{VB}$-Dirac structure
$$\begin{tikzpicture}
\mmat{m1} at (-2,0){L&W\\ \ast& \ast\\};
\path[->] (m1-1-1)	edge (m1-1-2)
				edge (m1-2-1);
\path[<-] (m1-2-2)	edge (m1-1-2)
				edge (m1-2-1);
\mmat{m2} at (2,0) {\g\ltimes\g^*&\g\\ \ast& \ast\\};
\path[->] (m2-1-1)	edge (m2-1-2)
				edge (m2-2-1);
\path[<-] (m2-2-2)	edge (m2-1-2)
				edge (m2-2-1);
\draw (0,0) node {$\subseteq$};
\end{tikzpicture}$$
must be of the form $L=\h\ltimes\on{ann}(\h)$ for some Lie subalgebra $\mf{h}\subset\g$. However, $$\h\ltimes\on{ann}(\h)\to \h$$ is a Lie subalgebroid of the action Lie algebroid $$\g\ltimes\g^*\to \g$$ if and only if $\beta^\sharp(\on{ann}(\h))\subseteq \h$.
Thus $\mc{LA}$-Dirac structures in $\mbb{A}$ are in one-to-one correspondence with \emph{$\beta$-coisotropic} Lie subalgebras $\mf{h}\subset\g$.
\end{proof}

\end{appendix}

\addcontentsline{toc}{chapter}{Bibliography}
\bibliographystyle{plain}
\bibliography{basicbib.bib}

\begin{thebibliography}{100}

\bibitem{Affane:2009dx}
Atallah Affane.
\newblock {Caract\'erisation et existence de structures de Dirac
  multiplicatives}.
\newblock {\em Comptes Rendus Math{\'e}matique. Acad{\'e}mie des Sciences.
  Paris}, 347(21-22):1299--1304, 2009.

\bibitem{Alekseev:2009tg}
Anton Alekseev, Henrique Bursztyn, and Eckhard Meinrenken.
\newblock {Pure spinors on Lie groups}.
\newblock {\em Ast\'erisque}, (327):131--199 (2010), 2009.

\bibitem{Alekseev99}
Anton Alekseev and Yvette Kosmann-Schwarzbach.
\newblock {Manin pairs and moment maps}.
\newblock {\em Journal of Differential Geometry}, 56(1):133--165, 2000.

\bibitem{Alekseev00}
Anton Alekseev, Yvette Kosmann-Schwarzbach, and Eckhard Meinrenken.
\newblock {Quasi-Poisson manifolds}.
\newblock {\em Canadian Journal of Mathematics}, 54(1):3--29, 2002.

\bibitem{Alekseev97}
Anton Alekseev, Anton~Z. Malkin, and Eckhard Meinrenken.
\newblock {Lie group valued moment maps}.
\newblock {\em Journal of Differential Geometry}, 48(3):445--495, 1998.

\bibitem{Alekseev:2002tn}
Anton Alekseev and Ping Xu.
\newblock {Derived brackets and Courant algebroids}.
\newblock 2002.

\bibitem{Alexandrov:1997jj}
Mikhail Alexandrov, Maxim Kontsevich, Albert Schwarz, and Oleg Zaboronsky.
\newblock {The geometry of the master equation and topological quantum field
  theory}.
\newblock {\em International Journal of Modern Physics A. Particles and Fields.
  Gravitation. Cosmology. Nuclear Physics}, 12(7):1405--1429, 1997.

\bibitem{Baez:2004vl}
John~C. Baez and Alissa~Susan Crans.
\newblock {Higher-dimensional algebra. VI. Lie 2-algebras}.
\newblock {\em Theory and Applications of Categories}, 12:492--538, 2004.

\bibitem{Baird:2010ge}
Thomas Baird and Yi~Lin.
\newblock {Topology of generalized complex quotients}.
\newblock {\em Journal of Geometry and Physics}, 60(10):1539--1557, 2010.

\bibitem{Behrend03}
Kai Behrend, Ping Xu, and Bin Zhang.
\newblock {Equivariant gerbes over compact simple Lie groups}.
\newblock {\em Comptes Rendus Math{\'e}matique. Acad{\'e}mie des Sciences.
  Paris}, 336(3):251--256, 2003.

\bibitem{Boumaiza:2009eg}
Mohamed Boumaiza and Nadhem Zaalani.
\newblock {Rel{\`e}vement d'une alg{\'e}bro{\"\i}de de Courant}.
\newblock {\em Comptes Rendus Math{\'e}matique. Acad{\'e}mie des Sciences.
  Paris}, 347(3-4):177--182, February 2009.

\bibitem{Bursztyn:gcFguuB1}
Henrique Bursztyn, Alberto~Sergio Cattaneo, Rajan Mehta, and Marco Zambon.
\newblock {Reduction of Courant algebroids via supergeometry}.

\bibitem{Bursztyn:2007ko}
Henrique Bursztyn, Gil~R. Cavalcanti, and Marco Gualtieri.
\newblock {Reduction of Courant algebroids and generalized complex structures}.
\newblock {\em Advances in Mathematics}, 211(2):726--765, 2007.

\bibitem{Bursztyn:2005te}
Henrique Bursztyn and Marius Crainic.
\newblock {Dirac structures, momentum maps, and quasi-Poisson manifolds}.
\newblock In {\em The breadth of symplectic and Poisson geometry}, pages 1--40.
  Birkh\"auser Boston, Boston, MA, 2005.

\bibitem{Bursztyn:2009vq}
Henrique Bursztyn and Marius Crainic.
\newblock {Dirac geometry, quasi-Poisson actions and $D/G$-valued moment maps}.
\newblock {\em Journal of Differential Geometry}, 82(3):501--566, 2009.

\bibitem{Bursztyn03-1}
Henrique Bursztyn, Marius Crainic, Alan Weinstein, and Chenchang Zhu.
\newblock {Integration of twisted Dirac brackets}.
\newblock {\em Duke Mathematical Journal}, 123(3):549--607, 2004.

\bibitem{Bursztyn:2009wi}
Henrique Bursztyn, David Iglesias~Ponte, and Pavol {\v S}evera.
\newblock {Courant morphisms and moment maps}.
\newblock {\em Mathematical Research Letters}, 16(2):215--232, 2009.

\bibitem{Bursztyn:2003ud}
Henrique Bursztyn and Olga Radko.
\newblock {Gauge equivalence of Dirac structures and symplectic groupoids}.
\newblock {\em Universit\'e de Grenoble. Annales de l'Institut Fourier},
  53(1):309--337, 2003.

\bibitem{Bursztyn:2005wi}
Henrique Bursztyn and Alan Weinstein.
\newblock {Poisson geometry and Morita equivalence}.
\newblock In {\em Poisson geometry, deformation quantisation and group
  representations}, pages 1--78. Cambridge University Press, Cambridge, 2005.

\bibitem{Bursztyn07}
Henrique Bursztyn and Chenchang Zhu.
\newblock {Morita equivalence of Poisson manifold via stacky groupoids}.
\newblock {\em Oberwolfach Reports}, 4:1243--1298, 2007.

\bibitem{Calvo:2010bj}
Iv{\'a}n Calvo, Fernando Falceto, and Marco Zambon.
\newblock {Deformation of Dirac structures along isotropic subbundles}.
\newblock {\em Reports on Mathematical Physics}, 65(2):259--269, 2010.

\bibitem{SCattaneo:2004fe}
Alberto~Sergio Cattaneo.
\newblock {Integration of twisted Poisson structures}.
\newblock {\em Journal of Geometry and Physics}, 49(2):187--196, February 2004.

\bibitem{Cattaneo:2000du}
Alberto~Sergio Cattaneo and Giovanni Felder.
\newblock {A path integral approach to the Kontsevich quantization formula}.
\newblock {\em Communications in Mathematical Physics}, 212(3):591--611, 2000.

\bibitem{Cattaneo:2008vf}
Alberto~Sergio Cattaneo and Giovanni Felder.
\newblock {Poisson sigma models and symplectic groupoids}.
\newblock In {\em Quantization of singular symplectic quotients}, pages 61--93.
  Birkh\"auser Verlag, Basel, 2001.

\bibitem{Cattaneo:2010vr}
Alberto~Sergio Cattaneo and Marco Zambon.
\newblock {Graded geometry and Poisson reduction}.
\newblock In {\em Special metrics and supersymmetry}, pages 48--56. American
  Institute of Physics, Melville, NY, 2009.

\bibitem{Cattaneo:2010th}
Alberto~Sergio Cattaneo and Marco Zambon.
\newblock {A Supergeometric Approach to Poisson Reduction}.
\newblock pages 1--40, September 2010.

\bibitem{Coste:1987ui}
Alain Cost{\'e}, Pierre Dazord, and Alan Weinstein.
\newblock {Groupo\"\i des symplectiques}.
\newblock In {\em Publications du D\'epartement de Math\'ematiques. Nouvelle
  S\'erie. A. Vol. 2}, pages i--ii, 1--62. Universit\'e Claude-Bernard
  D\'epartement de Math\'ematiques, Lyon, 1987.

\bibitem{Courant:1990uy}
Theodore~James Courant.
\newblock {Dirac manifolds}.
\newblock {\em Transactions of the American Mathematical Society},
  319(2):631--661, 1990.

\bibitem{Courant:1999ho}
Theodore~James Courant.
\newblock {Tangent Dirac structures}.
\newblock {\em Journal of Physics A: Mathematical and General},
  23(22):5153--5168, January 1999.

\bibitem{Courant:tm}
Theodore~James Courant and Alan Weinstein.
\newblock {Beyond Poisson structures}.
\newblock In {\em Action hamiltoniennes de groupes. Troisi\`eme th\'eor\`eme de
  Lie (Lyon, 1986)}, pages 39--49. Hermann, Paris, 1988.

\bibitem{Lie-Algebroids}
Marius Crainic and Rui~Loja Fernandes.
\newblock {Integrability of Lie brackets}.
\newblock {\em Annals of Mathematics. Second Series}, 157(2):575--620, 2003.

\bibitem{Crainic02}
Marius Crainic and Rui~Loja Fernandes.
\newblock {Integrability of Poisson brackets}.
\newblock {\em Journal of Differential Geometry}, 66(1):71--137, 2004.

\bibitem{Crans:2004ux}
Alissa~Susan Crans.
\newblock {\em {Lie 2-algebras}}.
\newblock PhD thesis, University of California Riverside, Riverside, 2004.

\bibitem{Dirac:1967ug}
Paul Adrien~Maurice Dirac.
\newblock {\em {Lectures on quantum mechanics}}, volume~2 of {\em Belfer
  Graduate School of Science Monographs Series}.
\newblock Belfer Graduate School of Science, New York, 1967.

\bibitem{Dorfman:1993us}
Irene Dorfman.
\newblock {\em {Dirac structures and integrability of nonlinear evolution
  equations}}.
\newblock Nonlinear Science - theory and applications. John Wiley {\&} Son Ltd,
  July 1993.

\bibitem{Drinfeld83}
Vladimir~Gershonovich Drinfel'd.
\newblock {Hamiltonian structures on Lie groups, Lie bialgebras and the
  geometric meaning of classical Yang-Baxter equations}.
\newblock {\em Doklady Akademii Nauk SSSR}, 268(2):285--287, 1983.

\bibitem{Drinfeld:1988fq}
Vladimir~Gershonovich Drinfel'd.
\newblock {Quantum groups}.
\newblock {\em Journal of Soviet Mathematics}, 41(2):898--915, April 1988.

\bibitem{Drinfeld:1989tu}
Vladimir~Gershonovich Drinfel'd.
\newblock {\em {Quasi-Hopf Algebras}}.
\newblock Algebra i Analiz, 1989.

\bibitem{Ehresmann:1995ts}
Charles Ehresmann.
\newblock {Les connexions infinit\'esimales dans un espace fibr\'e
  diff\'erentiable}.
\newblock In {\em S\'eminaire Bourbaki. Vol. 1}, pages Exp.\ No.\ 24, 153--168.
  Soc. Math. France, Paris, 1995.

\bibitem{Falceto:2008gg}
Fernando Falceto and Marco Zambon.
\newblock {An extension of the Marsden-Ratiu reduction for Poisson manifolds}.
\newblock {\em Letters in Mathematical Physics}, 85(2-3):203--219, 2008.

\bibitem{Goldberg:2010wj}
Timothy~E. Goldberg.
\newblock {Singular reduction of generalized complex manifolds}.
\newblock {\em SIGMA. Symmetry, Integrability and Geometry. Methods and
  Applications}, 6:17, 2010.

\bibitem{Grabowski:2009dc}
Janusz Grabowski and Mikolaj Rotkiewicz.
\newblock {Higher vector bundles and multi-graded symplectic manifolds}.
\newblock {\em Journal of Geometry and Physics}, 59(9):1285--1305, 2009.

\bibitem{Grabowski:1999ce}
Janusz Grabowski and Pawel Urba{\'n}ski.
\newblock {Tangent lifts of Poisson and related structures}.
\newblock {\em Journal of Physics A: Mathematical and General},
  28(23):6743--6777, January 1999.

\bibitem{GraciaSaz:2009ck}
Alfonso Gracia-Saz and Kirill Charles~Howard Mackenzie.
\newblock {Duality Functors for Triple Vector Bundles}.
\newblock {\em Letters in Mathematical Physics}, 90(1-3):175--200, September
  2009.

\bibitem{gracia2010lie}
Alfonso Gracia-Saz and Rajan~Amit Mehta.
\newblock {Lie algebroid structures on double vector bundles and representation
  theory of Lie algebroids}.
\newblock {\em Advances in Mathematics}, 223(4):1236--1275, 2010.

\bibitem{Gualtieri:2004wh}
Marco Gualtieri.
\newblock {\em {Generalized complex geometry}}.
\newblock PhD thesis, University of Oxford, November 2003.

\bibitem{Hawkins:2006vl}
Eli Hawkins.
\newblock {A groupoid approach to quantization}.
\newblock {\em The Journal of Symplectic Geometry}, 6(1):61--125, 2008.

\bibitem{Higgins:1990gq}
Philip~J. Higgins and Kirill Charles~Howard Mackenzie.
\newblock {Algebraic constructions in the category of Lie algebroids}.
\newblock {\em Journal of Algebra}, 129(1):194--230, 1990.

\bibitem{Higgins:1993bc}
Philip~J. Higgins and Kirill Charles~Howard Mackenzie.
\newblock {Duality for base-changing morphisms of vector bundles, modules, Lie
  algebroids and Poisson structures}.
\newblock {\em Mathematical Proceedings of the Cambridge Philosophical
  Society}, 114(3):471--488, 1993.

\bibitem{Hitchin:2003kx}
Nigel Hitchin.
\newblock {Generalized Calabi-Yau manifolds}.
\newblock {\em The Quarterly Journal of Mathematics}, 54(3):281--308, 2003.

\bibitem{Hu:2007wq}
Shengda Hu.
\newblock {Reduction and duality in generalized geometry}.
\newblock {\em The Journal of Symplectic Geometry}, 5(4):439--473, 2007.

\bibitem{Hu:2009wl}
Shengda Hu.
\newblock {Hamiltonian symmetries and reduction in generalized geometry}.
\newblock {\em Houston Journal of Mathematics}, 35(3):787--811, 2009.

\bibitem{Ponte:2005txa}
David Iglesias~Ponte, Camille Laurent-Gengoux, and Ping Xu.
\newblock {Universal lifting theorem and quasi-Poisson groupoids}.
\newblock pages 1--46, July 2005.

\bibitem{PonteXu:08}
David Iglesias~Ponte and Ping Xu.
\newblock {Hamiltonian spaces for Manin pairs over manifolds}.
\newblock September 2008.

\bibitem{Jotz:2009va}
Madeleine Jotz.
\newblock {Dirac Lie groups, Dirac homogeneous spaces and the Theorem of
  Drinfeld}.
\newblock 2009.

\bibitem{Jotz:2008wn}
Madeleine Jotz and Tudor~S. Ratiu.
\newblock {Dirac structures, nonholonomic systems and reduction}.
\newblock page~33, June 2008.

\bibitem{Jotz:2011kt}
Madeleine Jotz and Tudor~S. Ratiu.
\newblock {Induced Dirac structures on isotropy-type manifolds}.
\newblock {\em Transformation Groups}, 16(1):175--191, 2011.

\bibitem{Jotz:2011cz}
Madeleine Jotz, Tudor~S. Ratiu, and Jedrzej {\'S}niatycki.
\newblock {Singular reduction of Dirac structures}.
\newblock {\em Transactions of the American Mathematical Society},
  363(6):2967--3013, 2011.

\bibitem{Karasev:1986vg}
Mikhail~V. Karas{\"e}v.
\newblock {Analogues of objects of the theory of Lie groups for nonlinear
  Poisson brackets}.
\newblock {\em Izvestiya Akademii Nauk SSSR. Seriya Matematicheskaya},
  50(3):508--538, 638, 1986.

\bibitem{Kirillov:1976ud}
Aleksandr~Aleksandrovich Kirillov.
\newblock {Local Lie algebras}.
\newblock {\em Rossi\u\i skaya Akademiya Nauk. Moskovskoe Matematicheskoe
  Obshchestvo. Uspekhi Matematicheskikh Nauk}, 31(4(190)):57--76, 1976.

\bibitem{Konieczna:1999vh}
Katarzyna Konieczna and Pawel Urba{\'n}ski.
\newblock {Double vector bundles and duality}.
\newblock {\em Universitatis Masarykianae Brunensis. Facultas Scientiarum
  Naturalium. Archivum Mathematicum}, 35(1):59--95, 1999.

\bibitem{Koszul:1985wb}
Jean-Louis Koszul.
\newblock {Crochet de Schouten-Nijenhuis et cohomologie}.
\newblock {\em Ast\'erisque}, (Numero Hors Serie):257--271, 1985.

\bibitem{Kotov:2005fe}
Alexei Kotov, Peter Schaller, and Thomas Strobl.
\newblock {Dirac Sigma Models}.
\newblock {\em Communications in Mathematical Physics}, 260(2):455--480, August
  2005.

\bibitem{LiBland:2009ul}
David Li-Bland and Eckhard Meinrenken.
\newblock {Courant algebroids and Poisson geometry}.
\newblock {\em International Mathematics Research Notices},
  2009(11):2106--2145, 2009.

\bibitem{LiBland:2011vqa}
David Li-Bland and Eckhard Meinrenken.
\newblock {Dirac Lie Groups}.
\newblock pages 1--46, October 2011.

\bibitem{LiBland:2010wi}
David Li-Bland and Pavol {\v S}evera.
\newblock {Quasi-Hamiltonian groupoids and multiplicative Manin pairs}.
\newblock {\em International Mathematics Research Notices},
  2011(10):2295--2350, 2011.

\bibitem{Lichnerowicz:1977wp}
Andr{\'e} Lichnerowicz.
\newblock {Les vari{\'e}t{\'e}s de Poisson et leurs alg{\`e}bres de Lie
  associ{\'e}es}.
\newblock {\em Journal of Differential Geometry}, 12(2):253--300, 1977.

\bibitem{Lin:2006ku}
Yi~Lin and Susan Tolman.
\newblock {Symmetries in generalized K\"ahler geometry}.
\newblock {\em Communications in Mathematical Physics}, 268(1):199--222, 2006.

\bibitem{ManinTriplesBi}
Zhang-Ju Liu, Alan Weinstein, and Ping Xu.
\newblock {Manin triples for Lie bialgebroids}.
\newblock {\em Journal of Differential Geometry}, 45(3):547--574, 1997.

\bibitem{Liu:2000vf}
Zhang-Ju Liu and Qilin Yang.
\newblock {Reduced Poisson actions}.
\newblock {\em Science in China Series A: Mathematics}, 43(10):1026--1034,
  2000.

\bibitem{thesis-3}
Jiang-Hua Lu.
\newblock {\em {Multiplicative and affine Poisson structures on Lie groups}}.
\newblock PhD thesis, University of California, Berkeley, 1990.

\bibitem{lu90}
Jiang-Hua Lu and Alan Weinstein.
\newblock {Poisson Lie groups, dressing transformations, and Bruhat
  decompositions}.
\newblock {\em Journal of Differential Geometry}, 31(2):501--526, 1990.

\bibitem{Mackenzie:2003wj}
Kirill Charles~Howard Mackenzie.
\newblock {Double Lie algebroids and second-order geometry. I}.
\newblock {\em Advances in Mathematics}, 94(2):180--239, 1992.

\bibitem{Mackenzie:1998te}
Kirill Charles~Howard Mackenzie.
\newblock {Double Lie algebroids and the double of a Lie bialgebroid}.
\newblock August 1998.

\bibitem{Mackenzie:1998ge}
Kirill Charles~Howard Mackenzie.
\newblock {Drinfel'd doubles and Ehresmann doubles for Lie algebroids and Lie
  bialgebroids}.
\newblock {\em Electronic Research Announcements of the American Mathematical
  Society}, 4:74--87 (electronic), 1998.

\bibitem{Mackenzie:1999vk}
Kirill Charles~Howard Mackenzie.
\newblock {On symplectic double groupoids and the duality of Poisson
  groupoids}.
\newblock {\em International Journal of Mathematics}, 10(4):435--456, 1999.

\bibitem{Mackenzie:2000ut}
Kirill Charles~Howard Mackenzie.
\newblock {A unified approach to Poisson reduction}.
\newblock {\em Letters in Mathematical Physics}, 53(3):215--232, 2000.

\bibitem{Mackenzie:2008tz}
Kirill Charles~Howard Mackenzie.
\newblock {Double Lie algebroids and second-order geometry. II}.
\newblock {\em Advances in Mathematics}, 154(1):46--75, 2000.

\bibitem{Mackenzie:2005tc}
Kirill Charles~Howard Mackenzie.
\newblock {Duality and triple structures}.
\newblock In {\em The breadth of symplectic and Poisson geometry}, pages
  455--481. Birkh\"auser Boston Inc., Boston, MA, 2005.

\bibitem{Mackenzie05}
Kirill Charles~Howard Mackenzie.
\newblock {\em {General theory of Lie groupoids and Lie algebroids}}, volume
  213 of {\em London Mathematical Society Lecture Note Series}.
\newblock Cambridge University Press, Cambridge, 2005.

\bibitem{Mackenzie-Xu94}
Kirill Charles~Howard Mackenzie and Ping Xu.
\newblock {Lie bialgebroids and Poisson groupoids}.
\newblock {\em Duke Mathematical Journal}, 73(2):415--452, 1994.

\bibitem{Mackenzie97}
Kirill Charles~Howard Mackenzie and Ping Xu.
\newblock {Integration of Lie bialgebroids}.
\newblock {\em Topology}, 39(3):445--467, 2000.

\bibitem{Marsden:1986vs}
Jerrold~Eldon Marsden and Tudor~S. Ratiu.
\newblock {Reduction of Poisson manifolds}.
\newblock {\em Letters in Mathematical Physics}, 11(2):161--169, 1986.

\bibitem{Yoshimura:2007gw}
Jerrold~Eldon Marsden and Hiroaki Yoshimura.
\newblock {Reduction of Dirac structures and the Hamilton-Pontryagin
  principle}.
\newblock {\em Reports on Mathematical Physics}, 60(3):381--426, 2007.

\bibitem{Mehta06}
Rajan~Amit Mehta.
\newblock {\em {Supergroupoids, double structures, and equivariant
  cohomology}}.
\newblock PhD thesis, University of California, Berkeley, 2006.

\bibitem{Mehta:2009to}
Rajan~Amit Mehta.
\newblock {Q-algebroids and their cohomology}.
\newblock {\em The Journal of Symplectic Geometry}, 7(3):263--293, 2009.

\bibitem{Mehta:2009js}
Rajan~Amit Mehta.
\newblock {Q-groupoids and their cohomology}.
\newblock {\em Pacific Journal of Mathematics}, 242(2):311--332, 2009.

\bibitem{Mehta:2010ux}
Rajan~Amit Mehta.
\newblock {On homotopy Poisson actions and reduction of symplectic
  $Q$-manifolds}.
\newblock {\em Differential Geometry and its Applications}, 29(3):319--328,
  2011.

\bibitem{Mikami:1988tv}
Kentaro Mikami and Alan Weinstein.
\newblock {Moments and reduction for symplectic groupoids}.
\newblock {\em Kyoto University. Research Institute for Mathematical Sciences.
  Publications}, 24(1):121--140, 1988.

\bibitem{moerdijk03}
Ieke Moerdijk and Janez Mr{\v c}un.
\newblock {\em {Introduction to foliations and Lie groupoids}}.
\newblock Cambridge University Press, 2003.

\bibitem{Nijenhuis:1955td}
Albert Nijenhuis.
\newblock {Jacobi-type identities for bilinear differential concomitants of
  certain tensor fields. I, II}.
\newblock {\em Nederl. Akad. Wetensch. Proc. Ser. A. \bf 58 = Indag. Math.},
  17:390--397, 398--403, 1955.

\bibitem{Ortiz:2008bd}
Cristi{\'a}n Ortiz.
\newblock {Multiplicative Dirac structures on Lie groups}.
\newblock {\em Comptes Rendus Math{\'e}matique}, 346(23-24):1279--1282,
  December 2008.

\bibitem{Ortiz:2009ux}
Cristi{\'a}n Ortiz.
\newblock {\em {Multiplicative Dirac structures}}.
\newblock PhD thesis, Instituto de Matem{\'a}tica Pura e Aplicada, April 2009.

\bibitem{Pradines:1966ux}
Jean Pradines.
\newblock {Th\'eorie de Lie pour les groupo\"\i des diff\'erentiables.
  Relations entre propri\'et\'es locales et globales}.
\newblock {\em Comptes Rendus Hebdomadaires des S{\'e}ances de l'Acad{\'e}mie
  des Sciences. S{\'e}ries A et B}, 263:A907--A910, 1966.

\bibitem{Pradines:1967wn}
Jean Pradines.
\newblock {Th\'eorie de Lie pour les groupo\"\i des diff\'erentiables. Calcul
  diff\'erenetiel dans la cat\'egorie des groupo\"\i des infinit\'esimaux}.
\newblock {\em Comptes Rendus Hebdomadaires des S{\'e}ances de l'Acad{\'e}mie
  des Sciences. S{\'e}ries A et B}, 264:A245--A248, 1967.

\bibitem{Pradines:1968wi}
Jean Pradines.
\newblock {Troisi\`eme th\'eor\`eme de Lie les groupo\"\i des
  diff\'erentiables}.
\newblock {\em Comptes Rendus Hebdomadaires des S{\'e}ances de l'Acad{\'e}mie
  des Sciences. S{\'e}ries A et B}, 267:A21--A23, 1968.

\bibitem{Pradines:1974tc}
Jean Pradines.
\newblock {Repr\'esentation des jets non holonomes par des morphismes
  vectoriels doubles soud\'es}.
\newblock {\em Comptes Rendus Hebdomadaires des S{\'e}ances de l'Acad{\'e}mie
  des Sciences. S{\'e}ries A}, 278:1523--1526, 1974.

\bibitem{Pradines:1988td}
Jean Pradines.
\newblock {Remarque sur le groupo\"\i de cotangent de Weinstein-Dazord}.
\newblock {\em Comptes Rendus des S\'eances de l'Acad\'emie des Sciences.
  S\'erie I. Math\'ematique}, 306(13):557--560, 1988.

\bibitem{Roytenberg99}
Dmitry Roytenberg.
\newblock {\em {Courant algebroids, derived brackets and even symplectic
  supermanifolds}}.
\newblock PhD thesis, University of California, Berkeley, 1999.

\bibitem{Roytenberg:2002}
Dmitry Roytenberg.
\newblock {On the structure of graded symplectic supermanifolds and Courant
  algebroids}.
\newblock In {\em Quantization, Poisson brackets and beyond (Manchester,
  2001)}, pages 169--185. Amer. Math. Soc., Providence, RI, 2002.

\bibitem{Roytenberg02}
Dmitry Roytenberg.
\newblock {Quasi-Lie bialgebroids and twisted Poisson manifolds}.
\newblock {\em Letters in Mathematical Physics}, 61(2):123--137, 2002.

\bibitem{Roytenberg:1998ku}
Dmitry Roytenberg and Alan Weinstein.
\newblock {Courant algebroids and strongly homotopy Lie algebras}.
\newblock {\em Letters in Mathematical Physics}, 46(1):81--93, 1998.

\bibitem{SanchezdeAlvarez:1989vf}
Gloria S{\'a}nchez~de Alvarez.
\newblock {Controllability of Poisson control systems with symmetries}.
\newblock In {\em Dynamics and control of multibody systems (Brunswick, ME,
  1988)}, pages 399--410. Amer. Math. Soc., Providence, RI, 1989.

\bibitem{Schlessinger:1985fv}
Michael Schlessinger and James Stasheff.
\newblock {The Lie algebra structure of tangent cohomology and deformation
  theory}.
\newblock {\em Journal of Pure and Applied Algebra}, 38(2-3):313--322, 1985.

\bibitem{Schouten:1954uj}
Jan~Arnoldus Schouten.
\newblock {On the differential operators of first order in tensor calculus}.
\newblock In {\em Convegno Internazionale di Geometria Differenziale, Italia,
  1953}, pages 1--7. Edizioni Cremonese, Roma, 1954.

\bibitem{Schwarz:1993}
Albert Schwarz.
\newblock {Semiclassical approximation in Batalin-Vilkovisky formalism}.
\newblock {\em Communications in Mathematical Physics}, 158(2):373--396,
  November 1993.

\bibitem{Semenov-Tian-Shansky85}
Michael~A. Semenov-Tian-Shansky.
\newblock {Dressing transformations and Poisson group actions}.
\newblock {\em Kyoto University. Research Institute for Mathematical Sciences.
  Publications}, 21(6):1237--1260, 1985.

\bibitem{LetToWein}
Pavol {\v S}evera.
\newblock {Letters to A. Weinstein}.

\bibitem{NonComDiffForm}
Pavol {\v S}evera.
\newblock {Noncommutative differential forms and quantization of the odd
  symplectic category}.
\newblock {\em Letters in Mathematical Physics}, 68(1):31--39, 2004.

\bibitem{Severa:2005vla}
Pavol {\v S}evera.
\newblock {Some title containing the words ``homotopy'' and ``symplectic'',
  e.g. this one}.
\newblock In {\em Travaux math\'ematiques. Fasc. XVI}, pages 121--137. Univ.
  Luxemb., Luxembourg, 2005.

\bibitem{PoissonActions}
Pavol {\v S}evera.
\newblock {Poisson actions up to homotopy and their quantization}.
\newblock {\em Letters in Mathematical Physics}, 77(2):199--208, 2006.

\bibitem{Severa:2001}
Pavol {\v S}evera and Alan Weinstein.
\newblock {Poisson geometry with a 3-form background}.
\newblock {\em Progress of Theoretical Physics. Supplement}, (144):145--154,
  2001.

\bibitem{Stienon:2010ws}
Mathieu Stienon.
\newblock {Generalized Moser lemma}.
\newblock {\em Transactions of the American Mathematical Society},
  362(10):5107--5123, 2010.

\bibitem{Stienon:2008cl}
Mathieu Stienon and Ping Xu.
\newblock {Reduction of generalized complex structures}.
\newblock {\em Journal of Geometry and Physics}, 58(1):105--121, 2008.

\bibitem{Uchino02}
Kyousuke Uchino.
\newblock {Remarks on the definition of a Courant algebroid}.
\newblock {\em Letters in Mathematical Physics}, 60(2):171--175, 2002.

\bibitem{LieAlgebroidsH}
Arkady~Yu Va{\u \i}ntrob.
\newblock {Lie algebroids and homological vector fields}.
\newblock {\em Rossi\u\i skaya Akademiya Nauk. Moskovskoe Matematicheskoe
  Obshchestvo. Uspekhi Matematicheskikh Nauk}, 52(2(314)):161--162, 1997.

\bibitem{Vaisman:2007gg}
Izu Vaisman.
\newblock {Reduction and submanifolds of generalized complex manifolds}.
\newblock {\em Differential Geometry and its Applications}, 25(2):147--166,
  2007.

\bibitem{Vaisman:2010bt}
Izu Vaisman.
\newblock {Foliated Lie and Courant algebroids}.
\newblock {\em Mediterranean Journal of Mathematics}, 7(4):415--444, 2010.

\bibitem{Voronov:2002wl}
Theodore Voronov.
\newblock {Graded manifolds and Drinfeld doubles for Lie bialgebroids}.
\newblock In {\em Quantization, Poisson brackets and beyond (Manchester,
  2001)}, pages 131--168. Amer. Math. Soc., Providence, RI, 2002.

\bibitem{Voronov:2006wh}
Theodore Voronov.
\newblock {Mackenzie Theory and {Q}-Manifolds}.
\newblock 2006.

\bibitem{Voronov:2007tf}
Theodore Voronov.
\newblock {Q-manifolds and Mackenzie theory: an overview}.
\newblock 2007.

\bibitem{Wehrheim:2010cg}
Katrin Wehrheim and Christopher~Thomas Woodward.
\newblock {Functoriality for Lagrangian correspondences in Floer theory}.
\newblock {\em Quantum Topology}, 1(2):129--170, 2010.

\bibitem{Weinstein:1987ua}
Alan Weinstein.
\newblock {Symplectic groupoids and Poisson manifolds}.
\newblock {\em American Mathematical Society. Bulletin. New Series},
  16(1):101--104, 1987.

\bibitem{weinstein87}
Alan Weinstein.
\newblock {Coisotropic calculus and Poisson groupoids}.
\newblock {\em Journal of the Mathematical Society of Japan}, 40(4):705--727,
  1988.

\bibitem{Weinstein:2010wm}
Alan Weinstein.
\newblock {A note on the Wehrheim-Woodward category}.
\newblock December 2010.

\bibitem{Xu:1991vb}
Ping Xu.
\newblock {Morita equivalence of Poisson manifolds}.
\newblock {\em Communications in Mathematical Physics}, 142(3):493--509, 1991.

\bibitem{Xu:1992tb}
Ping Xu.
\newblock {Morita equivalence and symplectic realizations of Poisson
  manifolds}.
\newblock {\em Annales Scientifiques de l'\'Ecole Normale Sup\'erieure.
  Quatri\`eme S\'erie}, 25(3):307--333, 1992.

\bibitem{Xu95}
Ping Xu.
\newblock {On Poisson groupoids}.
\newblock {\em International Journal of Mathematics}, 6(1):101--124, 1995.

\bibitem{Xu03}
Ping Xu.
\newblock {Momentum maps and Morita equivalence}.
\newblock {\em Journal of Differential Geometry}, 67(2):289--333, 2004.

\bibitem{Zambon:2008wj}
Marco Zambon.
\newblock {Reduction of branes in generalized complex geometry}.
\newblock {\em The Journal of Symplectic Geometry}, 6(4):353--378, 2008.

\end{thebibliography}


\end{document}